\documentclass[11pt,reqno]{amsart}

\usepackage{amsmath, amsthm, amssymb}
\usepackage{amsfonts}

\usepackage{hyperref}

\usepackage[ansinew]{inputenc}
\usepackage[dvips]{epsfig}
\usepackage{graphicx}
\usepackage[english]{babel}

\usepackage{thmtools}
\theoremstyle{plain}
\declaretheorem[title=Theorem, parent=section]{theorem}
\declaretheorem[title=Lemma,sibling=theorem]{lemma}
\declaretheorem[title=Proposition,sibling=theorem]{proposition}
\declaretheorem[title=Corollary,sibling=theorem]{corollary}

\theoremstyle{definition}
\declaretheorem[title=Definition,sibling=theorem]{definition}
\declaretheorem[title=Remark,sibling=theorem]{remark}
\declaretheorem[title=Remark, numbered=no]{remark*}
\declaretheorem[title=Example, sibling=theorem]{example}
\declaretheorem[title=Assumption, numbered=no]{assumption*}

\numberwithin{equation}{section}

\usepackage[backgroundcolor=white, bordercolor=blue,
linecolor=blue]{todonotes}

\usepackage{dsfont}
\usepackage{bbm}

\newcommand{\N}{\mathbb{N}}
\newcommand{\R}{\mathbb{R}}

\newcommand{\cP}{\mathcal{P}}

\newcommand{\eps}{\varepsilon}

\newcommand{\1}{\mathbbm{1}}

\DeclareMathOperator{\dist}{dist}
\DeclareMathOperator{\sgn}{sgn}

\DeclareMathOperator{\supp}{supp}

\DeclareMathOperator*{\osc}{osc}

\renewcommand{\d}{\textnormal{\,d}}

\newcommand{\average}{{\mathchoice {\kern1ex\vcenter{\hrule height.4pt
width 6pt depth0pt} \kern-9.7pt} {\kern1ex\vcenter{\hrule
height.4pt width 4.3pt depth0pt} \kern-7pt} {} {} }}
\newcommand{\dashint}{\average\int}

\begin{document}
\allowdisplaybreaks
\title[Optimal regularity for kinetic Fokker-Planck equations in domains]{Optimal regularity for kinetic Fokker-Planck equations\\ in domains}

\author{Xavier Ros-Oton}
\author{Marvin Weidner}

\address{ICREA, Pg. Llu\'is Companys 23, 08010 Barcelona, Spain \& Universitat de Barcelona, Departament de Matem\`atiques i Inform\`atica, Gran Via de les Corts Catalanes 585, 08007 Barcelona, Spain \& Centre de Recerca Matem\`atica, Barcelona, Spain}
\email{xros@icrea.cat}
\urladdr{https://www.ub.edu/pde/xros/}

\address{Departament de Matem\`atiques i Inform\`atica, Universitat de Barcelona, Gran Via de les Corts Catalanes 585, 08007 Barcelona, Spain}
\email{mweidner@ub.edu}
\urladdr{https://sites.google.com/view/marvinweidner/}



\allowdisplaybreaks

\begin{abstract}
We study the smoothness of solutions to linear kinetic Fokker-Planck equations in domains $\Omega\subset \R^n$ with specular reflection condition, including Kolmogorov's equation $\partial_t f +v\cdot\nabla_x f-\Delta_v f=h$.
Our main results establish the following:
\begin{itemize}
\item Solutions are always $C^\infty$ in $t,v,x$ away from the grazing set $\{x\in\partial\Omega,\ v\cdot n_x=0\}$.
\item They are $C^{4,1}_{\text{kin}}$ up to the grazing set.
\item This regularity is optimal, i.e. we show that that they are in general not $C^5_{\text{kin}}$.
\end{itemize}
These results show for the first time that solutions are classical up to boundary, i.e. $C^1_{t,x}$ and $C^2_v$.
\end{abstract}

\allowdisplaybreaks

\maketitle

\section{Introduction}

Kinetic theory emerged in the 19th century through the foundational works of Maxwell and Boltzmann to statistically describe large systems of particles via their distribution function $f(t,x,v)$, where $t\in (0,T)$ represents the time, $x\in \Omega\subset \R^n$ the position in space, and $v\in \R^n$ the velocity.
The Boltzmann equation is the most famous example of such an equation, and reads as
\[ \partial_t f + v\cdot \nabla_x f = Q(f,f) \quad \textrm{in}\quad (0,T)\times \Omega\times \R^n, \]
where $Q$ is the Boltzmann collision operator, a nonlinear and nonlocal operator acting in the $v$ variable.

Another well-known kinetic model is the so-called Kolmogorov equation, 
\begin{align}
\label{eq:Kolmogorov0}
\partial_t f + v \cdot \nabla_x f -\Delta_v f = h \quad \textrm{in}\quad (0,T)\times \Omega\times \R^n,
\end{align}
which is much simpler because it is both linear and local.
It belongs to the wider class of linear kinetic Fokker-Planck equations.
Here, we may assume $h\in C^\infty$, or even $h\equiv0$.

From the PDE perspective, an important property of kinetic equations is that they are degenerate in the sense that there is a transport term in the $x$ variable, combined with a diffusion term that acts only in the $v$ variable.
This makes the equation hypoelliptic.

A fundamental challenge in kinetic theory is to develop a well-posedness theory for these equations, which means to establish a priori estimates for solutions, ideally showing that they are $C^\infty$.
This great challenge has attracted a lot of attention in the last decades \cite{GS11,GHTT17,GS25,ImSi22,ISV24,HST25,FRW24}, and is also crucial in the study of convergence to equilibrium as $t\to\infty$ (see \cite{DV01,DV05}).

Of course, any kinetic PDE has to be supplemented with boundary conditions. 
The most natural boundary condition (see e.g. \cite{Vil02}) is the \emph{specular reflection}
\begin{equation}\label{specular-ref}
 f(t,x,\mathcal R_x v)=f(t,x,v) \quad \textrm{for}\quad x\in \partial\Omega, \qquad \text{ where } \quad \mathcal R_x v = v-2(v\cdot n_x)n_x,
\end{equation}
where $n_x$ is the outward unit normal vector at $x$.\footnote{The physical interpretation is clear: particles bounce back on the boundary wall with a post-collision angle equal to the pre-collision angle.}

There are of course other possible boundary conditions, including the in-flow, bounce-back, or diffuse condition. We will discuss them later on and focus for now on the most natural and well-known one, namely specular reflection.
 
Even in the simplest case of Kolmogorov's equation \eqref{eq:Kolmogorov0}-\eqref{specular-ref}, the regularity of solutions in domains is far from understood.
Solutions are known to be $C^\infty$ inside $\Omega$ \cite{Kol34}, however it is not even known if solutions are classical (i.e. $C^2$ in $v$, and $C^1$ in $t$ and $x$) up to the boundary or not.

The best known results in this direction are the following:
\begin{itemize}
\item When $\Omega=\{x_n>0\}$ and $h\equiv0$, solutions to  \eqref{eq:Kolmogorov0}-\eqref{specular-ref} are $C^\infty$ up to the boundary.

\item In general smooth domains $\Omega\subset \R^n$, solutions are $C^{0,\alpha}$ up to the boundary for some $\alpha\in(0,1)$ \cite{Sil22, Zhu22}. When $n=3$, they are $L^{\infty}_t C^{\alpha/3,\alpha}_{x,v}$ and $C^{1,\alpha}$ in $v$ for any $\alpha\in(0,1)$ \cite{DGY22}.

\item For other types of boundary conditions (in-flow), solutions are in general \emph{not} $C^1$ up to the boundary, not even in case of a flat boundary \cite{Guo95,GJW99,HJV14,HwPh17,HLW24}.

\item For the Boltzmann equation with cutoff in strictly convex domains of $\R^3$, solutions are in a weighted $C^1$ space, and are in general not $C^2$ \cite{GHTT17}.
\end{itemize}

The proofs of the estimates in \cite{DGY22,Sil22, Zhu22} (see also \cite{DGO22}) are all based on reflecting the solution (after flattening the boundary) and using interior regularity estimates. 
This technique is useful to get H\"older estimates, but does not seem to give higher order smoothness of solutions. 
In particular, it does not yield that solutions are classical up to the boundary, i.e. $C^1$ in $t,x$ and $C^2$ in $v$.

On the other hand, it is not clear at all if the existence of non-$C^\infty$ solutions is a particular phenomenon of the in-flow boundary condition (and of the cutoff Boltzmann equation), or if it also happens for the Kolmogorov equation (or for the Boltzmann equation without cutoff) with specular reflection.
In this direction, as explained above, solutions to the Kolmogorov equation \eqref{eq:Kolmogorov0} with specular reflection \eqref{specular-ref} are always $C^\infty$ in the special case of a flat boundary and $h\equiv0$. In line with this, the famous result of Desvillettes-Villani \cite{DV05} for the Boltzmann equation without cutoff assumes that solutions are $C^\infty$ up to the boundary.

Summarizing, the following problem has been open for a long time:
\[\begin{array}{c}
\textit{Are solutions to  \eqref{eq:Kolmogorov0}-\eqref{specular-ref} always $C^1_{x,t}\cap C^2_v$ up to the boundary? Are they $C^\infty$?}
\end{array}\]
This is the question we tackle in this paper.

Due to the scaling of the PDE, the most natural way to measure regularity for kinetic Fokker-Planck equations is the kinetic H\"older space $C^\beta_{\ell}$, which essentially consists of functions that are $C^\beta$ in $v$, $C^{\beta/3}$ in $x$, and $C^{\beta/2}$ in $t$ (see Subsection \ref{kinetic-dist} for the precise definition).
Thus, in order for solutions to \eqref{eq:Kolmogorov0} to be classical up to the boundary, we would like to prove that they belong to the space $C^3_{\ell}$.
Notice that, however, solutions are not even known to belong to $C^1_{\ell}$ so far.

\subsection{Main results}

Our main result for the Kolmogorov equation gives a definitive answer to the long-standing open question described above, and completely characterizes the smoothness of solutions.

We state our main result below in a simplified form. For a more general version, we refer to \autoref{thm:global-weighted-SR-reg} (see also \eqref{eq:weighted-Holder-spaces} for the definition of the weighted H\"older space $C^4_{\ell,4}$).

\begin{theorem}
\label{thm1}
Let $\Omega \subset \R^n$ be a bounded smooth domain, and let $h \in C^\infty((0,T) \times \Omega \times \R^n)$. Let $f$ be any solution of the Kolmogorov equation  \eqref{eq:Kolmogorov0} with specular reflection \eqref{specular-ref} and assume that both, $f$ and $h$ have fast decay as $|v|\to\infty$.

Let us denote $\gamma_0 = (0,T)\times \{x\in \partial\Omega, v\in \R^n : \ v\cdot n_x=0\}$.
Then, for any $[t_0,t_1] \subset (0,T)$,
\begin{itemize} 
\item $f$ is smooth away from the grazing set, i.e. $f\in C^\infty(([t_0,t_1]\times \overline\Omega\times \R^n)\setminus \gamma_0)$,
\item $f \in C^{4,1}_{\ell}([t_0,t_1]\times\overline\Omega \times \R^n)$, and more precisely we have
\[\|f\|_{C^{4,1}_{\ell}([t_0,t_1]\times\overline\Omega \times \R^n)} \leq C\left( \big\|(1+|v|)^p f\|_{L^1((0,T)\times \Omega\times \R^n)} + \|h\|_{C^4_{\ell,4}((0,T)\times \overline\Omega\times \R^n)} \right)\]
with $p=73+36n$ and $C$ depending only on $n$, $t_0$, $t_1$, and $\Omega$.
\end{itemize}

Moreover, this regularity is optimal, i.e. there exist bounded smooth domains $\Omega$ and functions $h\in C^\infty$ for which $f\notin C^5_{\ell}([t_0,t_1]\times\overline\Omega \times \R^n)$.
\end{theorem}

Notice that, in particular, solutions are $C^{2,\frac12}$ in $t$, $C^{1,\frac23}$ in $x$, and $C^{4,1}$ in $v$, so that our result yields for the first time that solutions are classical up to the boundary.
Moreover, both the regularity in $x$ and in $v$ are optimal.
On the other hand, concerning the regularity in $t$, a standard bootstrap argument yields that solutions are $C^\infty$ in the $t$ variable.

Finally, notice also that the regularity on the grazing set $\gamma_0$ is the most delicate part of our result, but even the $C^\infty$ regularity away from $\gamma_0$ is far from trivial and was not known before.

The counterexamples we construct are stationary solutions, and we can even take $\partial\Omega$ flat when $h\not\equiv0$.
These examples are constructed in Section \ref{sec:counterex}, and show that $f\notin C^5_v$ and $f\notin C^{1,\frac23+\varepsilon}_x$ for any $\varepsilon>0$.
We actually expect these counterexamples to be ``generic'', in some sense (see Section \ref{sec:counterex} for more details).

As said above, all previous regularity results for \eqref{eq:Kolmogorov0}-\eqref{specular-ref} were based in one way or another on reflecting the solution to then use interior regularity estimates.
Our proof of \autoref{thm1} is completely different, and the optimal regularity that we prove cannot be inferred without a delicate analysis of the possible boundary behaviors near $\partial\Omega$.
A key result in this direction is a Liouville-type theorem in a half-space, a simplified version of which we state in the following.
This is one of the core results of our paper, and its proof requires several new ideas (even in dimension $n=1$).

\begin{theorem}
\label{thm2}
Let $\beta \ge 0$. Let $p$ be any polynomial in $t,x,v$, and $f$ be any solution of 
\[ \partial_t f + v \cdot \nabla_x f -\Delta_v f = p \quad \textrm{in}\quad \R \times \{x_n>0\} \times \R^n \]
satisfying $f(t,x,v)=f(t,x,\mathcal R_x v)$ on $\R \times \{x_n=0\} \times \R^n$ and the growth condition
\[|f(t,x,v)| \leq C \big(1+|t|^{1/2}+|x|^{1/3}+|v|\big)^\beta.\]
\begin{itemize}
\item If $\beta<5$, then $f$ must be a polynomial.

\item If  $\beta=5$, then $f$ must be a linear combination of a polynomial and a (unique) non-trivial solution $\mathcal T$, which satisfies $\mathcal T(\lambda^3 x,\lambda v)=\lambda^5 \mathcal T(x,v)$ for all $\lambda>0$.
\end{itemize}
\end{theorem}

We believe it to be quite remarkable that we are able to establish a complete classification of global solutions, since we do not have any reflection principle nor any monotonicity formula that allows us to show a priori that all solutions must be homogeneous.
We refer to \autoref{thm:Liouville-higher-half-space-SR} for a more precise version of this result.

This classification result is necessary to prove the optimal $C^{4,1}_\ell$ regularity on the grazing set $\gamma_0$ stated in \autoref{thm1}.
Still, the proof of \autoref{thm1} requires much more than that, as explained in more detail later on.

\subsection{Regularity estimates for general kinetic Fokker-Planck equations}

The Kolmogorov equation is only a particular case in the class of linear kinetic Fokker-Planck equations
\begin{equation}\label{eq:kinetic-eq-main}
\partial_t f + v \cdot \nabla_x f -a^{i,j} \partial_{v_i,v_j}f = - b \cdot \nabla_v f - c f + h ~~ \text{ in }  (-1,1) \times \Omega \times \R^n.
\end{equation}
Throughout this article we will assume that $a^{i,j}$ is uniformly elliptic,
\begin{align}
\label{eq:unif-ell}
\Lambda |\xi|^2 \ge a^{i,j}(z) \xi_i \xi_j \ge \lambda |\xi|^2 ~~ \forall \xi \in \R^n, ~~ z=(t,x,v) \in \R^{1+2n}.
\end{align}
Here, and throughout the rest of the paper, the time interval $(-1,1)$ is chosen for convenience, so that the analysis happens at $t=0$.

The kinetic boundary $\gamma := (-1,1) \times \partial \Omega \times \R^n$ is split into the following three parts $\gamma = \gamma_+ \cup \gamma_- \cup \gamma_0$:
\begin{align*}
\gamma_{\pm} = \{ (t,x,v) \in (-1,1) \times \partial \Omega \times \R^n : \pm n_x \cdot v > 0 \}, ~~ \gamma_0 = \big( (-1,1) \times \partial \Omega \times \R^n \big) \setminus (\gamma_+ \cup \gamma_-),
\end{align*}
where $n_x \in \mathbb{S}^{n-1}$ denotes the outward unit normal vector of $\Omega$ at $x \in \partial \Omega$.

%The specular reflection boundary condition then reads as
%\begin{equation}\label{specular}
%f(t,x,v) = f(t,x,\mathcal{R}_x v) ~~ \forall (t,x,v) \in \gamma_-,
%\end{equation}
%where $\mathcal{R}_x$ denotes the reflection of $v$ with respect to the hyperplane spanned by $n_x$, namely
%\begin{align*}
%\mathcal{R}_x v = v - 2(n_x \cdot v)n_x.
%\end{align*}

All our results hold for general operators of the form \eqref{eq:kinetic-eq-main}-\eqref{eq:unif-ell}. We establish several results that are more precise versions of \autoref{thm1}.
More precisely, we
\begin{itemize}
\item prove global $C^{4,1}_\ell$ estimates in bounded domains $\Omega$ of class $C^{4,\frac12}$ for general equations \eqref{specular-ref}-\eqref{eq:kinetic-eq-main}-\eqref{eq:unif-ell} with $a,b,c,h\in C^{3,\varepsilon}_{\ell}$ (see \autoref{thm:global-weighted-SR-reg}).

\item establish localized version of these estimates (see \autoref{thm:regularity-SR-halfspace} and \autoref{thm:regularity-SR-constants}).

\item prove $C^{k,\alpha}_\ell$ estimates away from $\gamma_0$ for flat boundaries (see \autoref{prop:regularity-SR}) and for general domains (see \autoref{thm:regularity-SR-constants-nongrazing}).

\item prove $C^{k,\alpha}_\ell$ estimates away from $\gamma_0$ also in case of in-flow boundary conditions (see \autoref{prop:regularity-inflow}).

\item find an open condition under which solutions to \eqref{eq:Kolmogorov0}-\eqref{specular-ref} with $h\equiv0$ in non-flat domains $\Omega$ are not $C^{5}_{\ell}$ (see \autoref{thm:counterexample}).
\end{itemize}

All these results are entirely new, and open the road to proving optimal regularity estimates for other kinetic equations.
In particular, we believe that using the techniques from this paper it should be possible to establish optimal regularity estimates for the other typical boundary conditions: in-flow, bounce-back, diffuse, and inelastic.

Moreover, in a future work we plan to study the smoothness of solutions to the Boltzmann equation without cutoff in bounded domains.

\subsection{Strategy of the proof}
\label{subsec:strategy}

In stark contrast to previous works on the higher order boundary regularity for kinetic equations \cite{HJV14,HJJ15,HJJ18} and \cite{GHJO20,DGY22}, our proof of \autoref{thm:global-weighted-SR-reg} is based on \emph{localized} regularity estimates in kinetic cylinders 
\begin{align*}
Q_r(z_0) &= \big\{ (t,x,v) \in \R \times \R^n \times \R^n : |t-t_0| < r^2 , |x - x_0 - (t-t_0)v_0| < r^{3}, |v-v_0| < r \big\}.
\end{align*}
These cylinders respect the kinetic scaling and translation invariance and are commonly used in the study of interior regularity estimates (see for instance \cite{PaPo04,GuMo22,ImSi21,Loh23}). They also appear in \cite{Sil22}, where boundary $C^{\alpha}$ regularity has been established. 

A key advantage of proving localized estimates is that they allow us to consider the different boundary regions $\gamma = \gamma_+ \cup \gamma_- \cup \gamma_0$ in a separate way, and to obtain precise control on the behavior of our estimates for large velocities $v_0$.

Let us start by explaining how to prove localized regularity estimates in the simplest case of flat boundaries. Later, we will elaborate on the generalization of our results to curved domains.

\subsubsection{Flat domains}

Let $\Omega = \{ x_n > 0 \}$. In this case, we have $\gamma = (-1,1) \times \{ x_n = 0 \} \times \R^n$ and 
\begin{align*}
\gamma_+ = (-1,1) \times \{ x_n = 0 \} \times \{ v_n < 0 \},\\
\gamma_- = (-1,1) \times \{ x_n = 0 \} \times \{ v_n > 0 \},\\
\gamma_0 = (-1,1) \times \{ x_n = 0 \} \times \{ v_n = 0 \}.
\end{align*}

For $z_0 = (t_0,x'_0, 0, v_0) \in \gamma$ we define
\begin{align*}
\mathcal{S}(H_R(z_0)) &= \mathcal{S}(Q_R(z_0)) \cap ((-1,1) \times \{ x_n > 0 \} \times \R^n), \\
\mathcal{S}(Q_R(z_0)) &= Q_R(z_0) \cup \{ (t,x,v',-v_n) : (t,x,v) \in Q_R(z_0) \}.
\end{align*}
The set $\mathcal{S}(Q_R(z_0))$ consists of the kinetic cylinder $Q_R(z_0)$ itself and the cylinder reflected in $v$ with respect to $\gamma_0$.

The following result is a localized estimate in flat domains for operators with coefficients.

\begin{theorem}
\label{thm:regularity-SR-halfspace}
Let $\Omega = \{ x_n > 0 \}$ and $z_0 \in \gamma$. Let $R \in (0,1]$, $\eps \in (0,1)$, and $a^{i,j},b,c,h \in C^{3+\eps}_{\ell}(\mathcal{S}(H_R(z_0)))$ and assume that $a^{i,j}$ satisfies \eqref{eq:unif-ell}. Let $f$ be a solution to 
\begin{equation}
\label{eq:intro-PDE}
\left\{\begin{array}{rcl}
\partial_t f + v \cdot \nabla_x f + (-a^{i,j} \partial_{v_i,v_j})f &=& - b \cdot \nabla_v f - c f + h ~~ \text{ in }  \mathcal{S}(Q_R(z_0)) \cap ((-1,1) \times \{ x_n > 0 \} \times \R^n) , \\
f(t,x,v) &=& f(t,x,\mathcal{R}_x v) ~~ \qquad\quad \text{ on } \gamma_- \cap  \mathcal{S}(H_R(z_0)) .
\end{array}\right.
\end{equation}
Then, it holds:
\begin{align*}
[ f ]_{C^{4,1}_{\ell}(H_{R/16}(z_0))} \le C  R^{-5} \big( \Vert f \Vert_{L^{\infty}(\mathcal{S}(H_R(z_0)))} + R^{5} [ h ]_{C^{3+\eps}_{\ell}(\mathcal{S}(H_R(z_0)) )} \big).
\end{align*}
The constant $C$ depends only on $n,\eps,\lambda,\Lambda$, $\Vert a^{i,j} \Vert_{C^{3+\eps}_{\ell}(\mathcal{S}(H_R(z_0)))}, \Vert b \Vert_{C^{3+\eps}_{\ell}(\mathcal{S}(H_R(z_0)))}$, and \\$\Vert c \Vert_{C^{3+\eps}_{\ell}(\mathcal{S}(H_R(z_0)))}$, but not on $z_0$ and $R$.
\end{theorem}

Note that \autoref{thm:regularity-SR-halfspace} immediately implies a global estimate for solutions in $(-1,1) \times \{ x_n > 0 \} \times \R^n$. We state this result in \autoref{cor:regularity-SR-halfspace}.

Let us explain how to prove \autoref{thm:regularity-SR-halfspace} in case $R = 1$. Due to the hypoelliptic nature of kinetic equations of type \eqref{eq:kinetic-eq-main}, the regularity of $f$ in the interior of the solution domains can be achieved by perturbation arguments of Schauder-type. Such a result was established for instance in \cite{Loh23} and it implies that $f \in C^{5+\eps}_{\ell}$ locally inside the solution domain under the assumptions of \autoref{thm:regularity-SR-halfspace} (see also \autoref{lemma:interior-reg}).
In the light of these interior regularity results, we can reduce the proof of \autoref{thm:regularity-SR-halfspace} to the study of the local behavior of $f$ at boundary points $z^{\ast} \in \gamma \cap Q_{1/2}(z_0)$. Since the behavior differs significantly depending on whether $z^{\ast} \in \gamma_{\pm}$, or $z^{\ast} \in \gamma_0$, we need to treat these cases separately.

In case $z^{\ast} \in \gamma_{\pm}$, we prove that there is a polynomial $P_{z^{\ast}} \in \cP_5$ of degree five, such that
\begin{align}
\label{eq:intro-expansion-gamma-pm}
|f(z) - P_{z^{\ast}}(z)| \lesssim d_{\ell}(z,z_0)^{5 + \eps} ~~ \forall z \in Q_{1/2}(z^{\ast}) \cap H_1(z_0),
\end{align}
where $d_\ell$ is the kinetic distance defined in Subsection \ref{kinetic-dist}.
By the definition of the kinetic H\"older spaces, this property implies that $f$ is of class $C^{5+\eps}_{\ell}$ at $z^{\ast}$ (see \autoref{lemma:higher-reg-gamma_+} and \autoref{lemma:higher-reg-gamma_-SR}). 

For $z^{\ast} \in \gamma_0$ the situation is more complicated since it turns out that \eqref{eq:intro-expansion-gamma-pm} \emph{is in general false} (see also \autoref{thm:counterexample}). In fact, for $z^{\ast} \in \gamma_0$ we prove that there are a polynomial $P_{z^{\ast}} \in \cP_5$ and $\tau \in \R$ such that
\begin{align}
\label{eq:intro-expansion-gamma-0}
\big|f(z) - P_{z^{\ast}}(z) - \tau \mathcal{T}_{z^{\ast}}(z)\big| \lesssim d_{\ell}(z,z_0)^{5 + \eps} ~~ \forall z \in Q_{1/2}(z^{\ast}) \cap H_1(z_0),
\end{align}
where $\mathcal{T}_{z^{\ast}}$ is a so-called ``Tricomi solution'', given by
\begin{align*}
\mathcal{T}_{z^{\ast}}(z) = A^{-\frac{5}{2}} v_n^{5} - 2 \cdot 9^5 A^{-\frac{5}{6}} x_n^{\frac{5}{3}} U \left( - \frac{5}{3} ; \frac{2}{3} ; \frac{-v_n^3}{9Ax_n} \right), \qquad \text{ where } \qquad A = a^{n,n}(z^{\ast}),
\end{align*}
and $U$ denotes the Tricomi confluent hypergeometric function (see \autoref{lemma:higher-reg-gamma_0-SR-halfspace}). The function $\mathcal{T}_{z^{\ast}}$ is a $5$-homogeneous solution of the one-dimensional, stationary kinetic equation with specular reflection condition
\begin{equation}
\label{eq:intro-Tricomi-PDE}
\left\{\begin{array}{rcll}
v \partial_x \mathcal{T}_{z^{\ast}} -A \partial_{v,v} \mathcal{T}_{z^{\ast}} &=& p_{z^{\ast}} & \text{in } \{ x > 0 \} \times \R,\\
\mathcal{T}_{z^{\ast}}(0,v) &=& \mathcal{T}_{z^{\ast}}(0,-v) & \text{for } v \in \R,
\end{array}\right.
\end{equation}
where $p_{z^\ast}(x,v) = -40 A v^3 \in \cP_3$.
The existence of a non-polynomial solution to a one-dimensional kinetic equation with a polynomial right-hand side and subject to the specular reflection condition, such as \eqref{eq:intro-Tricomi-PDE}, was not known before, and it is the main reason why solutions to \eqref{eq:intro-PDE} are not smooth. 
In fact, the function $\mathcal{T}_{z^{\ast}}$ belongs to $C^{4,1}_{\ell}$, but it is not of class $C^5_{\ell}$ (see \autoref{prop:Tricomi}). This  proves the sharpness of \autoref{thm:regularity-SR-halfspace} and is sufficient to deduce from \eqref{eq:intro-expansion-gamma-0} that $f$ is $C^{4,1}_{\ell}$ at $z^{\ast}$.

We show the expansions \eqref{eq:intro-expansion-gamma-pm} and \eqref{eq:intro-expansion-gamma-0} at boundary points $z^{\ast} \in \gamma$ by blow-up arguments. Variants of such techniques are quite common nowadays in PDE theory and have also recently been applied to the study of the interior and behavior for solutions to nonlocal equations (see e.g. \cite{Ser15a,Ser15b,RoSe16a,RoSe16b,AbRo20,ImSi21,RoWe24}). In this paper, we use it for the first time in the context of boundary estimates for kinetic equations, which requires several new ideas and nontrivial adaptations, as we explain below.

The aforementioned blow-up arguments, which are used to establish the expansions \eqref{eq:intro-expansion-gamma-pm} and \eqref{eq:intro-expansion-gamma-0} at boundary points $z^{\ast} \in \gamma$, rely on ``zooming into'' the equation \eqref{eq:intro-PDE} at $z^{\ast}$. This is achieved by a proper rescaling of the equation, using the kinetic shift and scaling operator
\begin{align*}
z^{\ast} \circ S_r z := (t^{\ast} + r^2 t , x^{\ast} + r^3 x + r^2 t v^{\ast}, v^{\ast} + rv).
\end{align*}
Note that the rescaling $z^{\ast} \circ S_r z$ respects the kinetic geometry in the sense that it maps the kinetic cylinder $Q_1(z^{\ast})$ to $Q_{r^{-1}}(0)$. However, the rescaling of the solution domain $(-1,1) \times \{ x_n > 0 \} \times \R^n$ depends heavily on the position of $z^{\ast}$ on the kinetic boundary. It holds
\begin{align*}
(z^{\ast} \circ S_r) ((-1,1) \times \Omega \times \R^n) &= \{(t,x): t \in (r^{-2}(-1-t^{\ast}) , r^{-2}(1 - t^{\ast})) ~,~ x_n > -r^{-1} t v^{\ast}_n \} \times \R^n \\
& \to D_{\infty} := \begin{cases}
\{t<0, \ x,v\in\R^n\} ~~ \text{ if } z^{\ast} \in \gamma_+,\\
\{t>0, \ x,v\in\R^n\} ~~ \text{ if } z^{\ast} \in \gamma_-,\\
\{t\in\R,\ x_n > 0,\ v\in\R^n \} ~~ \text{ if } z^{\ast} \in \gamma_0.
\end{cases}
\end{align*}
This means that the shape of the rescaled domains changes dramatically depending on the boundary point! At $\gamma_\pm$ we get a full space (in $x$) but only a half-line in time, while at $\gamma_0$ we get a half-space in~$x$.

Using a rather delicate contradiction compactness argument, we can reduce  the proof of the expansions at $z^{\ast}$ to a classification of solutions to a translation invariant kinetic equation of the form
\begin{align*}
\partial_t f_{\infty} + v \cdot \nabla_x f_{\infty} + (-\bar{a}^{i,j} \partial_{v_i,v_j})f_{\infty} = p ~~ \text{ in } D_{\infty}
\end{align*} 
for some polynomial $p \in \cP_3$ of degree three and a constant matrix $\bar{a}^{i,j}$.

In case $z^{\ast} \in \gamma_{\pm}$, we achieve this by proving a Liouville theorem for such equations in $(\mp\infty,0) \times \R^n \times \R^n$ (see \autoref{prop:Liouville-higher-neg-times} and \autoref{prop:Liouville-higher-pos-times}). However, note that Liouville theorems do not hold true in general for positive times, even in the parabolic case, without suitable initial conditions at $t = 0$. Hence, some additional ideas are needed in case $z^{\ast} \in \gamma_-$. 

In order to prove \eqref{eq:intro-expansion-gamma-pm} for $z^{\ast} \in \gamma_-$ we need to use the boundary condition on $f$ in order to get an initial condition for the limiting equation. Unfortunately, the specular reflection condition degenerates under the kinetic rescaling $z^{\ast} \circ S_r \cdot$ and does not seem to provide any additional information in the limit. To solve this issue, we first establish the expansion \eqref{eq:intro-expansion-gamma-pm} for solutions to kinetic equations satisfying an in-flow boundary condition, i.e.
\begin{equation*}
\left\{\begin{array}{rcl}
\partial_t f + v \cdot \nabla_x f + (-a^{i,j} \partial_{v_i,v_j})f &=& - b \cdot \nabla_v f - c f + h ~~ \text{ in }  Q_{1/2}(z^{\ast}) \cap ((-1,1) \times \{ x_n > 0 \} \times \R^n) , \\
f &=& F ~~ \qquad\qquad\qquad\quad \text{ on } \gamma_- \cap Q_{1/2}(z^{\ast}),
\end{array}\right.
\end{equation*}
and then, in a second step, we recast \eqref{eq:intro-PDE} as a kinetic equation with in-flow condition
\begin{align*}
F(t,x,v) := f(t,x,\mathcal{R}_x v) ~~ \text{ in } \gamma_- \cap Q_{1/2}(z^{\ast}).
\end{align*}
Then, since $\mathcal{R}_{\cdot}(\gamma_- \cap Q_{1/2}(z^{\ast})) = \gamma_+ \cap Q_{1/2}(\mathcal{R}_x(z^{\ast}))$, we can use the expansion for $\mathcal{R}_x(z^{\ast}) \in \gamma_+$ to deduce that $F \in C^{5,\eps}_{\ell}$ and obtain the expansion \eqref{eq:intro-expansion-gamma-pm} for $z^{\ast} \in \gamma_-$.

We emphasize that the proof of the boundary regularity near $\gamma_-$, even under the in-flow condition is much more involved than the proof at $\gamma_+$, because the boundary condition has to be passed to the limit, as well. For this, we need to analyze polynomials that are restricted to the boundary $\partial \Omega$, which comes with certain additional difficulties (see Step 4 in the proof of \autoref{lemma:higher-reg-gamma_--Dirichlet}).

Finally, in case $z^{\ast} \in \gamma_0$, we prove that the specular reflection condition is preserved under the kinetic rescaling (which heavily uses that $\{ x_n > 0 \}$ is flat and stays flat under the rescaling). This reduces the proof of \eqref{eq:intro-expansion-gamma-0} to establishing a Liouville-type classification result for solutions to 
\begin{equation*}
\left\{\begin{array}{rcl}
\partial_t f_{\infty} + v \cdot \nabla_x f_{\infty} + (-\bar{a}^{i,j} \partial_{v_i,v_j})f_{\infty} &=& p \in \cP_3 ~~~ \qquad \text{ in } \R \times \{ x_n > 0 \} \times \R^n,\\
f_\infty(t,x,\mathcal{R}_x v) &=& f_\infty(t,x,\mathcal{R}_x v) ~~ \text{ in } \R \times \{ x_n = 0 \} \times \R^n.
\end{array}\right.
\end{equation*}
This is essentially \autoref{thm2}, stated above. We have already seen that not all solutions to this problem are polynomials, but that Tricomi solutions might appear. Therefore, the proof of such Liouville-type theorem is rather involved. We establish it in its most general form in \autoref{thm:Liouville-higher-half-space-SR} and believe it to be of independent interest.

\subsubsection{General domains}

\autoref{thm:regularity-SR-halfspace} implies the optimal $C^{4,1}_{\ell}$ regularity for solutions in the half-space $\Omega = \{ x_n > 0 \}$. In order to prove \autoref{thm1} in general domains, we need to flatten the boundary. 
Note that this is a nontrivial task for kinetic equations, since one needs to ensure that the kinetic distance and the boundary condition are preserved under the change of variables. 
We use a modification of the diffeomorphism $\Phi$ introduced in \cite{GHJO20,DGY22} (see \eqref{eq:Phi-def}) that flattens the boundary of $\Omega$ locally, near any point $x_0 \in \partial \Omega$, and keeps the specular reflection condition. Note that, in contrast to \cite{GHJO20,DGY22}, our diffeomorphism preserves the regularity of $\partial \Omega$ in $x$.

Given a solution $f$ to \eqref{eq:kinetic-eq-main} with specular reflection condition on $\gamma$, the function $\tilde{f}(z) := f(\Phi^{-1}(z))$ is a solution to
\begin{equation*}
\left\{\begin{array}{rcl}
\partial_t \tilde{f} + v \cdot \nabla_x \tilde{f} + (- \tilde{a}^{i,j}\partial_{v_i,v_j}) \tilde{f} &=& - \tilde{b} \cdot \nabla_v \tilde{f} + \tilde{h} ~~ \text{ in } (-1,1) \times (\{x_n > 0\} \cap B_1(x_0',0) ) \times \R^n,\\
\tilde f(t,x,v) &=&\tilde f(t,x,\mathcal{R}_x v) ~~~ \text{ in } (-1,1) \times (\{x_n = 0\} \cap B_1(x_0',0) ) \times \R^n.
\end{array}\right.
\end{equation*}

In order to prove \autoref{thm1}, we need to apply \autoref{thm:regularity-SR-halfspace} in appropriate kinetic cylinders. This approach comes with several technical issues:
\begin{itemize}
\item Since $\Phi$ is a local change of variables in $x$, one needs to choose the size of the kinetic cylinders $Q_r(z_0)$ depending on $|v_0|$. This leads to the appearance of weights in term of $(1 + |v_0|)$ in the estimate from \autoref{thm1}.

\item The norms of the coefficients $\tilde{a}^{i,j}$ and $\tilde{b}^i$ depend on $\Phi$ and therefore also on $|v_0|$. Therefore it is crucial to track the dependence of the constants in \autoref{thm:regularity-SR-halfspace} on the coefficients. For $\tilde{b}^i$ we can obtain such a refined estimate a posteriori by comparison with the equation without drift and an absorption argument (see \autoref{prop:regularity-SR-halfspace-constants}). For $\tilde{a}^{i,j}$ this approach does not work, since solutions are only in $C^{4,1}_{\ell}$ (and not in $C^{5,\eps}_{\ell}$). We work around this issue by assuming a little more regularity in $a^{i,j}$ (see \autoref{lemma:SR-preserved}(iii)).
\end{itemize}

\subsection{Organization of the paper}
This article is structured as follows. In Section \ref{sec:prelim} we collect several preliminary definitions and basic results. Section \ref{sec:Liouville} is dedicated to the proof of Liouville theorems in the full space and in the half-space. In particular, we establish \autoref{thm2}. In Sections \ref{sec:reg-outside} and \ref{sec:reg-gamma0} we prove boundary regularity estimates away from $\gamma_0$ and at~$\gamma_0$, respectively, and prove our main result \autoref{thm1}. We prove the optimality of \autoref{thm1} in Section \ref{sec:counterex}.

\subsection{Acknowledgments}

The authors were supported by the European Research Council under the Grant Agreement No. 101123223 (SSNSD), and by AEI project PID2021-125021NA-I00 (Spain).
Moreover, X.R was supported by the AEI grant RED2024-153842-T (Spain),  by the AEI--DFG project PCI2024-155066-2 (Spain--Germany), and by the Spanish State Research Agency through the Mar\'ia de Maeztu Program for Centers and Units of Excellence in R{\&}D (CEX2020-001084-M).

\section{Preliminaries}
\label{sec:prelim}

In this section we collect several auxiliary definitions and results. Most of them are well-known in slightly different settings. We provide proofs whenever necessary to make the article self-contained.

\subsection{Scaling properties of kinetic equations}

Our strategy to establish the boundary regularity will be by a blow-up and contradiction compactness argument. Hence, it is crucial to understand the behavior of the aforementioned kinetic equation under scaling. 

First, we introduce kinetic cylinders which respect the geometry of the PDE under consideration, namely given $z_0 = (t_0,x_0,v_0) \in \R \times \R^n \times \R^n$ and $r > 0$, we define
\begin{align*}
\tilde{Q}_r(z_0) &= \big\{ (t,x,v) \in \R \times \R^n \times \R^n : t_0 - r^2 < t \le t_0 , |x - x_0 - (t-t_0)v_0| < r^{3}, |v-v_0| < r \big\} \\
&= \big\{ (t,x,v) \in \R \times \R^n \times \R^n : t \in (t_0 - r^2 , t_0) , x \in B_{r^{3}}(x_0 + (t - t_0)v_0), v \in B_r(v_0) \big\}.
\end{align*}

In this article, we will mostly work with two-sided kinetic cylinders, defined as follows:
\begin{align*}
Q_r(z_0) &= \big\{ (t,x,v) \in \R \times \R^n \times \R^n : |t-t_0| < r^2 , |x - x_0 - (t-t_0)v_0| < r^{3}, |v-v_0| < r \big\} \\
&= \tilde{Q}_r(z_0) \cup \tilde{Q}_r(z_0 + (r^2, r^2 v_0, 0)).
\end{align*}

We observe that 
\begin{align*}
Q_r(z_0) \subset \tilde{Q}_{2r}\big(z_0 + (r^2 , r^2 v_0 , 0)\big), \quad \tilde{Q}_r(z_0) = Q_{r}\big(z_0 - (r^2/2,r^2v_0/2,0)\big).
\end{align*}

Moreover, whenever $D \subset \R^{1+2n}$, then we define
\begin{align*}
\tilde{H}_r(z_0) = D \cap \tilde{Q}_r(z_0), \qquad H_r(z_0) = D \cap Q_r(z_0).
\end{align*}
We will only use this notation whenever there is no doubt about the definition of $D$, and in case we are given $\Omega \subset \R^n$, we apply this definition with $D = (-1,1) \times \overline{\Omega} \times \R^n$.

We also introduce the following inertial change of variables and a scaling operator
\begin{align*}
(s,y,w) \circ (t,x,v) = (s+t,x+y+tw,v+w), \qquad S_r(t,x,v) = (r^2 t , r^3 x , r v).
\end{align*}
Note that the operation $\circ$ is not commutative.

\begin{remark}
The inertial change of variables $\circ$ and scaling operator $S_r$ are compatible with the kinetic cylinders in the following sense:
Given $z_0 = (t_0,x_0,v_0) \in \R \times \R^n \times \R^n$ and $R,r > 0$, then 
\begin{align}
\label{eq:trafo-cylinders}
z \in Q_{R/r}(0) ~~ \Leftrightarrow ~~ z_0 \circ S_r z \in Q_R(z_0).
\end{align}
Indeed, writing $z = (t,x,v)$ we have
\begin{align*}
(t_0 + r^2 t , x_0 + r^3 x + r^2 t v_0, v_0 + rv) = (t_0 , x_0 , v_0 ) \circ (r^2 t , r^3 x , r v) = z_0 \circ S_r z,
\end{align*}
and moreover,
\begin{align*}
\begin{cases}
|t| < (R/r)^2  ~~ &\Leftrightarrow ~~ \big|(t_0 + r^2 t) - t_0 \big| = r^2 |t| < R^2 , \\
|x| < (R/r)^3 ~~ &\Leftrightarrow ~~ \big|(x_0 + r^3 x + r^2 t v_0) - x_0 - ((t_0 + r^2 t) - t_0)v_0 \big| = r^3 |x| < R^3, \\
|v| < (R/r) ~~ &\Leftrightarrow ~~ \big|(v_0 + rv) - v_0 \big| = r|v| < R.
\end{cases}
\end{align*}
Moreover, an analogous equivalence holds true for one-sided cylinders.

We also have a slightly more general version of \eqref{eq:trafo-cylinders}: for any $z_1 = (0,x_1,v_1) \in \R^{2n+1}$ it holds
\begin{align}
\label{eq:trafo-cylinders-2}
z \in Q_{R/r}(z_1) ~~ \Leftrightarrow ~~ z_0 \circ S_r z \in Q_R(z_0 \circ S_r z_1) = Q_R(z_0 + S_r z_1).
\end{align}
Indeed, 
\begin{align*}
\begin{cases}
|t| < (R/r)^2  ~~ &\Leftrightarrow ~~ \big|(t_0 + r^2 t) - t_0 \big| = r^2 |t - t_1| < R^2 , \\
|x - x_1 - t v_1 | < (R/r)^3 ~~ &\Leftrightarrow ~~\big |(x_0 + r^3 x + r^2 t v_0) - (x_0 + r^3 x_1) \\
&\qquad \qquad \qquad - ((t_0 + r^2 t) - t_0(v_0 + r v_1)\big| = r^3 |x - x_1 - t v_1| < R^3, \\
|v - v_1| < (R/r) ~~ &\Leftrightarrow ~~ \big|(v_0 + rv) - (v_0 + r v_1) \big| = r|v - v_1| < R.
\end{cases}
\end{align*}
\end{remark}

We also have the following elementary lemma:
\begin{lemma}
\label{lemma:kinetic-inclusion}
Let $r > 0$ and $R \ge 1$ and $z_0,z_1 \in \R^{2n+1}$ with $t_1 = 0$. Then,
\begin{align*}
 Q_{Rr}(z_0 + S_{r} z_1) \subset Q_{2Rr(1 + |z_1|)}(z_0).
\end{align*}
\end{lemma}

\begin{proof}
Clearly, $z \in Q_{Rr}(z_0 + S_{r} z_1)$ implies
\begin{align*}
|t - t_0| < (Rr)^2, \quad |x - (x_0 + r^3 x_1) - (t-t_0) (v_0 + r v_1)| < (Rr)^3, \quad |v - (v_0 + r v_1)| < Rr.
\end{align*}
Then, it holds since $R \ge 1$:
\begin{align*}
|x - x_0 - (t-t_0)v_0| &< |x - (x_0 + r^3 x_1) - (t-t_0) (v_0 + r v_1)| + |r^3 x_1 - (t-t_0)rv_1| \\
&< (Rr)^3 + r^3 |x_1| + R^2 r^3 |v_1| < 2(Rr)^3(1 + |z_1|).
\end{align*}
This implies the desired result.
\end{proof}

Given a domain $\Omega \subset \R^n$, $z_0 = (t_0,x_0,v_0) \in \R \times \R^n \times \R^n$ and $r > 0$, we define
\begin{align}
\label{eq:Tr-def}
\begin{split}
T_{z_0,r} &= \big\{ (t,x) : t \in \big( r^{-2}(-1 - t_0) , r^{-2}(1 - t_0) \big],~ x \in \Omega_{z_0,r}(t)  \big\}, \\
\Omega_{z_0,r}(t) &= \{ x \in \R^n : x_0 + r^3 x + r^2 t v_0 \in \Omega \}.
\end{split}
\end{align}
Then, by construction we have
\begin{align}
\label{eq:trafo-domains}
z \in T_{z_0,r} \times \R^n ~~ \Leftrightarrow ~~ z_0 \circ S_r z \in (- 1 , 1) \times \Omega \times \R^n.
\end{align}

The following lemma is a straightforward corollary of the previous remark and indicates how kinetic equations behave under the change of variables $z_0 \circ S_r$.

\begin{lemma}
\label{lemma:scaling}
Let $z_0 = (t_0,x_0,v_0) \in \R \times \R^n \times \R^n$ and $r,R > 0$. Let $\Omega \subset \R^n$. Let $f$ be a solution to
\begin{align*}
\partial_t f + v \cdot \nabla_x f + (- a^{i,j}\partial_{v_i,v_j}) f = - b \cdot \nabla_v f - c f + h ~~ \text{ in } Q_R(z_0) \cap \big( (-1,1) \times \Omega \times \R^n \big).
\end{align*}
Then 
\begin{align*}
f_{z_0,r}(t,x,v) := f(z_0 \circ S_r z) = f(t_0 + r^2 t , x_0 + r^3 x + r^2 t v_0, v_0 + rv)
\end{align*}
is a solution to
\begin{align*}
\partial_t f_{z_0,r} + v \cdot \nabla_x f_{z_0,r} + (-\tilde{a}^{i,j}\partial_{v_i,v_j}) f_{z_0,r} = - \tilde{b} \cdot \nabla_v f_{z_0,r} - \tilde{c} f_{z_0,r} + \tilde{h} ~~ \text{ in } Q_{R/r}(0) \cap \big( T_{z_0,r} \times \R^n \big),
\end{align*}
where $\tilde{a}^{i,j}(z) = a^{i,j}(z_0 \circ S_r z)$, $\tilde{b}^i(z) = r b^i(z_0 \circ S_r z)$, $\tilde{c}(z) = r^2 c(z_0 \circ S_r)$, and $\tilde{h}(z) = r^2 h(z_0 \circ S_r z)$.
\end{lemma}

\begin{proof}
The proof follows immediately from \eqref{eq:trafo-cylinders}, the construction of $T_{z_0,r}$ \eqref{eq:trafo-domains}, and the chain rule.
\end{proof}

As we mentioned before, our proof of the boundary regularity goes by blow-up techniques, i.e. zooming into the solution at a boundary point $z_0 \in \gamma$. This corresponds to investigating the limit of solutions $f_{z_0,r}$ as $r \searrow 0$. The following lemma determines the limit of the corresponding solution domains $T_r \times \R^n$, which heavily depends on the position of $z_0 \in \gamma$.

\begin{lemma}
\label{lemma:domains-convergence}
Let $\partial \Omega \in C^{1}$, $R > 0$, and $(z_m^0)_m \subset (-1,1) \times \partial \Omega \times \R^n$, where we denote $z_m^0 = (t_m^0,x_m^0,v_m^0)$ and assume that $x_m^0 \to x^0 \in \partial \Omega$. Let $r_m \searrow 0$ and in case $z_m^0 \in \gamma_{\pm}$ for all $m \in \N$ we assume in addition that $|v_m^0 \cdot n_{x_m^0}| r_m^{-1} \ge M_m$ for some sequence $M_m \nearrow \infty$. Then, it holds
\begin{align*}
Q_{R/r_m}(0) \cap (T_{r_m} \times \R^n) \to (T_0 \times \R^n),
\end{align*}
as $m \to \infty$, where
\begin{align*}
T_0 := 
\begin{cases}
(-\infty,0) \times \R^n ~~ &\text{ if } z_m^0 \in \gamma_+ ~ \forall m \in \N,\\
(0,\infty) \times \R^n ~~ &\text{ if } z_m^0 \in \gamma_- ~ \forall m \in \N,\\
\R \times \{ x \cdot n_{x^0} < 0 \}  ~~ &\text{ if } z_m^0 \in \gamma_0 ~ \forall m \in \N.
\end{cases}
\end{align*}
Moreover, note that if $z_m^0 \in \gamma_-$ for any $m \in \N$, then, as $m \to \infty$
\begin{align*}
\gamma_-^{(r_m)} := \{ (t,x,v) \in T_{r_m} \times \R^n : (z_m^0 \circ S_{r_m} z) \in \gamma_- \} \to \{ 0 \} \times \R^n \times \R^n.
\end{align*}
\end{lemma}

All convergences of sets are to be understood with respect of the Hausdorff distance.

Note that $\tilde{Q}_{r^{-1}}(0) \to ((-\infty,0] \times \R^n \times \R^n)$ as $r \to 0$. Hence, if we considered one-sided cylinders, there would be no reasonable limiting domain whenever $z_0 \in \gamma_-$. Moreover, note that all the assumptions in the previous lemma are satisfied, if $(z_m^0)_m$ is constant and $r_m \searrow 0$.

\begin{proof}
It is easy to see that 
\begin{align*}
\big( r_m^{-2}(-1 - t_0) , r_m^{-2}(1 - t_0) \big) \to \R, \qquad Q_{R/r_m}(0) \to (\R \times \R^n \times \R^n)
\end{align*}
Hence, it remains to determine the limit of $\Omega_{r_m}(t)$ for any fixed $t \in \R$. We have
\begin{align*}
\Omega_{z_m^0,r_m}(t) = r_m^{-3}(\Omega - x_m^0) - r_m^{-1}t v_m^0 = r_m^{-3}(\Omega - x_m^0 - r_m^2 t v_m^0).
\end{align*}
Note that $r_m^{-3}(\Omega - x_m^0) \to \{ x \cdot n_{x_0} < 0 \}$ since $\partial \Omega \in C^{1}$. Therefore, for any $t \in \R$ it holds
\begin{align*}
\lim_{m \to \infty} \Omega_{r_m}(t) = \lim_{m \to \infty} r_m^{-3} (\{ x \cdot n_{x_m^0} < 0 \} - r_m^2 t v_m^0 ) = \lim_{m \to \infty} \{ x \cdot n_{x_m^0} < (-t) (v_m^0 \cdot n_{x_m^0}) r_m^{-1} \},
\end{align*}
and since by assumption we have $|v_m^0 \cdot n_{x_m^0}| r_m^{-1} \ge M_m \to \infty$ in case $z_m^0 \in \gamma_{\pm}$, we deduce
\begin{align}
\label{eq:domains-convergence-half-space}
\lim_{m \to \infty} \{ x \cdot n_{x_m^0} < -r_m^{-1} t v_m^0 \cdot n_{x_m^0} \} =
\begin{cases}
\R^n ~~ &\text{ if } t v_m^0 \cdot n_{x_m^0} < 0 ~\forall m \in \N,\\
\emptyset ~~ &\text{ if } t v_m^0 \cdot n_{x_m^0} > 0 ~\forall m \in \N,\\
\{ x \cdot n_{x^0} < 0\} ~~ &\text{ if } t v_m^0 \cdot n_{x_m^0} = 0 ~\forall m \in \N.
\end{cases}
\end{align}
The first case occurs if and only if $\mp t > 0$ and $z_m^0 \in \gamma_{\pm}$ and the second case occurs if and only if $\pm t > 0$ and $z_m^0 \in \gamma_{\mp}$. The third case occurs if and only if $z_m^0 \in \gamma_0$. Note that all of these convergences can be lifted immediately to convergence in the sense of appropriate norms (e.g. in the Gromov-Hausdorff sense).
This concludes the proof of the first claim.

The second claim follows immediately from \eqref{eq:domains-convergence-half-space}.
\end{proof}

\subsection{Kinetic distance and H\"older spaces}\label{kinetic-dist}

The goal of this subsection is to introduce a distance, which respects the kinetic scaling and is suitable in order to define natural H\"older spaces. Our definitions are in align with those in \cite{ImSi22} and respect the geometry induced by the kinetic cylinders.

The distance in \cite{ImSi22} is defined as the infimum of $r$ s.t. $z_1, z_2 \in \tilde{Q}_r(z)$ for some $z$ or equivalently
\begin{align*}
d_{\ell}(z_1,z_2) = \min_{w \in \R^n} \Big\{ \max \big( |t_1 - t_2|^{1/2} , |x_1 - x_2 - (t_1 - t_2)w|^{1/3} , |v_1 - w| , |v_2 - w| \big) \Big\}.
\end{align*} 
Note that $d_{\ell}$ defines a metric on $\R^{1+2n}$ and that we have for some universal $C > 0$
\begin{align}
\label{eq:dist-comp}
\begin{split}
C^{-1} \max \big( |t_1 - t_2|^{1/2} , |x_1 - x_2 - & (t_1 - t_2)v_1|^{1/3} , |v_1 - v_2| \big) \le d_{\ell}(z_1,z_2) \\
& \le C \max \big( |t_1 - t_2|^{1/2} , |x_1 - x_2 - (t_1 - t_2)v_2|^{1/3} , |v_1 - v_2| \big).
\end{split}
\end{align}
The kinetic distance $d_{\ell}$ is left-invariant with respect to the  translation $\circ$ and homogeneous with respect to the scaling $S_r$ operation that was introduced before:
\begin{align}
\label{eq:distance-invariance}
d_{\ell}(S_r z_1,S_r z_2) = r d_{\ell}(z_1,z_2), \qquad d_{\ell}(z \circ z_1 , z \circ z_2) = d_{\ell}(z_1,z_2).
\end{align}
These two properties are in align with the translation and scaling properties of kinetic equations (see \autoref{lemma:scaling}) and are the main motivations to work with $d_{\ell}$.

We have the following connection between kinetic cylinders $Q_r(z_0)$ and the kinetic distance $d_{\ell}$:

\begin{lemma}
\label{lemma:kinetic-balls}
Let $z_1 = (t_1,x_1,z_1), z_2 = (t_2,x_2,z_2) \in \R^{1+2n}$ and $r > 0$. Then, the following hold true:
\begin{itemize}
\item[(i)] When $z_1 \in Q_r(z_2)$, then $d_{\ell}(z_1,z_2) \le r$.
\item[(ii)] When $d_{\ell}(z_1,z_2) < r$, then $z_1 \in Q_{2 r}(z_2)$.
\end{itemize}
In particular, when $z_1 \in Q_r(z_2)$, then $z_2 \in Q_{2r}(z_1)$, and
\begin{align*}
Q_r(z_1) \subset \{ z \in \R^{1+2n} : d_{\ell}(z,z_1) < r \} \subset Q_{2r}(z_1),
\end{align*}
and moreover, when $z_1 \in Q_{r/4}(z_0)$, then $Q_{r/4}(z_1) \subset Q_r(z_0)$. 
\end{lemma}

Analogous statements hold true for one-sided kinetic cylinders $\tilde{Q}$, under additional restrictions on the ordering of $t_1,t_2$.

\begin{proof}
The proof of (i) is immediate upon choosing $w = v_2$ in the definition of $d_{\ell}(z_1,z_2)$. To see (ii), note that $d_{\ell}(z_1,z_2) < r$ implies the existence of $w \in \R^n$ such that
\begin{align*}
|t_1 - t_2| < r^2, \quad |x_1 - x_2 - (t_1-t_2)w| < r^3, \quad |v_1 - w| < r, \quad |v_2 - w| < r.
\end{align*}
To prove that $z_1 \in Q_{2r}(z_2)$, we need to verify 
\begin{align*}
|t_1 - t_2| < (2r)^2, \quad |x_1 - x_2 - (t_1 - t_2)v_2| < (2r)^3, \quad |v_1 - v_2| < 2r.
\end{align*}
Let now $z_1 \in Q_{r}(z_2)$. Then it is immediate that $z_2 \in Q_{2r}(z_1)$, and the remaining properties follow directly from the triangle inequality:
\begin{align*}
|x_1 - x_2 - (t_1 - t_2)v_2| < r^3 + |t_1 - t_2||w - v_2| < 2 r^3, \quad |v_1 - v_2| \le |v_1 - w| + |v_2 - w| < 2r.
\end{align*}
\end{proof}

Let us now introduce kinetic H\"older spaces. These spaces are defined with respect to the kinetic distance and therefore also respect the invariances of kinetic equations.

First, we define kinetic polynomials. Given a multi-index $\beta = (\beta_t,\beta_x,\beta_v) \in (\N \cup \{ 0 \})^{1+2n}$ with $\beta_t \in \N \cup \{ 0 \}$ and $\beta_x, \beta_v \in (\N \cup \{0\})^n$, we define its kinetic degree as
\begin{align*}
|\beta| = 2|\beta_t| + 3|\beta_x| + |\beta_v|.
\end{align*}

Given $k \in \N \cup \{ 0 \}$, we introduce the space of kinetic polynomials of kinetic degree $k$, as follows:
\begin{align*}
\cP_k = \Big\{ \sum_{|\beta| \le k} \alpha^{(\beta)} t^{\beta_t} x_1^{\beta_{x_1}} \cdot \dots \cdot x_n^{\beta_{x_n}} v_1^{\beta_{v_1}} \cdot \dots \cdot v_n^{\beta_{v_n}} ~:~  \alpha^{(\beta)} \in \R \Big\}.
\end{align*}

\begin{definition}
For $k \in \N \cup \{0\}$ and $\eps \in (0,1]$ and $D \subset \R^{1+2n}$, we say that $f :D  \to \R$ is $C^{k,\eps}_{\ell}$ at $z_0 \in \R^{1+2n}$ if there is $p \in \cP_{k}$ such that for any $z \in D$:
\begin{align*}
|f(z) - p(z)| \le C d_{\ell}^{k+\eps}(z,z_0).
\end{align*}
We say that $f \in C^{k,\eps}_{\ell}(D)$ if this property holds for any $z_0 \in D$ and define
\begin{align*}
[f]_{C^{k,\eps}_{\ell}(D)} = \sup_{z_0 \in D} \inf_{p_{z_0} \in \cP_k} \sup_{z \in D} \frac{|f(z) - p_{z_0}(z)|}{d_{\ell}^{k+\eps}(z,z_0)}, \qquad \Vert f \Vert_{C^{k,\eps}_{\ell}(D)} = [f]_{C^{k,\eps}_{\ell}(D)} + \Vert f \Vert_{L^{\infty}(D)}.
\end{align*}
\end{definition}

Whenever $\eps \not= 1$ we also write $C^{k+\eps}_{\ell}(D) := C^{k,\eps}_{\ell}(D)$. When $\eps = 1$, we need to distinguish the spaces $C^{k,1}_{\ell}(D)$ and $C^{k+1}_{\ell}(D)$. In fact:

\begin{definition}
For $k \in \N \cup \{0\}$ and $D \subset \R^{1+2n}$, we say that $f :D  \to \R$ is $C^k_{\ell}$ at $z_0 \in \R^{1+2n}$, if there is $p \in \cP_k$ such that for any $z \in D$:
\begin{align*}
\frac{|f(z) - p(z)|}{d_{\ell}^k(z,z_0)} \to 0 ~~ \text{ as } z_0 \to z.
\end{align*}
We say that $f \in C^k_{\ell}(D)$ if this property holds true for any $z_0 \in D$.
\end{definition}

We have the following scaling property of kinetic H\"older norms.

\begin{lemma}
\label{lemma:Holder-scaling}
Let $k \in \N \cup \{ 0 \}$ and $\eps \in (0,1]$, $z_0 \in \R^{1+2n}$, and $r,R > 0$. Then, for $f_{z_0}(z) = f(z_0 \circ S_r z)$
\begin{align*}
[f]_{C^{k,\eps}_{\ell}(Q_R(z_0))} = r^{-k-\eps} [f_{z_0,r}]_{C^{k,\eps}_{\ell}(Q_{R/r}(0))}.
\end{align*}
\end{lemma}

The same result holds true for one-sided cylinders.

\begin{proof}
By \autoref{lemma:scaling} and the invariances of the kinetic distance (see \eqref{eq:distance-invariance}), we have
\begin{align*}
[f]_{C^{k,\eps}_{\ell}(Q_R(z_0))} &=  \sup_{z_1 \in Q_R(z_0)} \inf_{p \in \cP_k} \sup_{z_2 \in Q_R(z_0)} \frac{|f(z_2) - p(z_2)|}{d_{\ell}^{k+\eps}(z_1,z_2)} \\
&=  \sup_{z_r^{(1)} \in Q_{R/r}(0)} \inf_{p \in \cP_k} \sup_{z_r^{(2)} \in Q_{R/r}(0)} \frac{|f(z_0 \circ S_r z_r^{(2)}) - p(z_0 \circ S_r z_r^{(2)})|}{d_{\ell}^{k+\eps}(z_0 \circ S_r z_r^{(1)},z_0 \circ S_r z_r^{(2)})} \\
&= \sup_{z_r^{(1)} \in Q_{R/r}(0)} \inf_{p_{z_0,r} \in \cP_k} \sup_{z_r^{(2)} \in Q_{R/r}(0)} \frac{|f_{z_0,r}(z_r^{(2)}) - p_{z_0,r}(z_r^{(2)})|}{\big( r d_{\ell}(z_r^{(1)},z_r^{(2)}) \big)^{k+\eps}} \\
&= r^{-k-\eps} [f_{z_0,r}]_{C^{k,\eps}_{\ell}(Q_{R/r}(0))}.
\end{align*}
In the last step, we used that whenever $p \in \cP_k$, then also $p_{z_0,r} \in \cP_k$.
\end{proof}

Note that one can choose $p_{z_0} \equiv f(z_0)$ if $k = 0$, and therefore we have in this special case
\begin{align*}
[f]_{C^{0,\eps}_{\ell}(D)} = \sup_{z_1,z_2 \in D} \frac{|f(z_1) - f(z_2)|}{d_{\ell}(z_1,z_2)^{\eps}}.
\end{align*}

The following lemma is a convenient sufficient condition for functions to belong to a kinetic H\"older space $C^{k,\eps}_{\ell}(D)$.

\begin{lemma}
\label{lemma:covering}
Let $D \subset \R^{1+2n}$ be a bounded Lipschitz domain. Let $f \in L^{\infty}(D)$ be such that for some $k \in \N \cup \{ 0 \}$ and $\eps \in (0,1]$:
\begin{align*}
[f]_{C^{k,\eps}_{\ell}(Q_r(z))} \le C ~~ \forall r,z : Q_{4r}(z) \subset D.
\end{align*}
Then, $f \in C^{k,\eps}_{\ell}(D)$ and
\begin{align*}
[f]_{C^{k,\eps}_{\ell}(D)} \le cC
\end{align*}
for some $c$ depending only on $n,D,k,\eps$.
\end{lemma}

\begin{proof}
The lemma is standard for domains $\Omega \subset \R^n$ and H\"older spaces with the Euclidean metric (see for instance \cite[Lemma A.1.4]{FeRo24}). Since by \autoref{lemma:kinetic-balls} the kinetic distance $d_{\ell}$ is compatible with kinetic cylinders $Q$, an analogous proof works in the kinetic setting. Note that an interpolation result for kinetic H\"older spaces is given in \cite{ImSi21}.
\end{proof}

We finish this subsection by recalling three results on kinetic H\"older spaces from \cite{ImSi21,ImSi22}.

We have the following H\"older interpolation result:

\begin{lemma}[H\"older interpolation]
\label{lemma:Holder-interpol}
Let $k_1,k_2 \in \N \cup \{ 0 \}$, $\eps_1,\eps_2 \in (0,1]$ such that $0 < k_1 + \eps_1 < k_2 + \eps_2$, and $f \in C^{k_2,\eps_2}_{\ell}(Q_r(z_0))$ for some $r \in (0,1]$ and $z_0 \in \R^{1+2n}$. Let $\delta \in (0,1]$. Then, it holds
\begin{align*}
[f]_{C^{k_1,\eps_1}_{\ell}(Q_r(z_0))} &\le (\delta r)^{(k_2 + \eps_2) - (k_1 + \eps_1)} [f]_{C^{k_2,\eps_2}_{\ell}(Q_r(z_0))} + C (\delta r)^{-(k_1 + \eps_1)} \Vert f \Vert_{L^{\infty}(Q_r(z_0))},\\
\Vert f \Vert_{L^{\infty}(Q_r(z_0))} &\le (\delta r)^{k_2+\eps_2} [f]_{C^{k_2,\eps_2}_{\ell}(Q_r(z_0))} + C (\delta r)^{-(2+4n)} \Vert f \Vert_{L^1(Q_r(z_0))},
\end{align*}
where $C$ depends only on $n,k_1 + \eps_1, k_2 + \eps_2$.
\end{lemma}

\begin{proof}
The first estimate was proved in \cite[Proposition 2.10, Remark 2.11]{ImSi21} on scale one and the result on general scales following immediately by scaling (see \autoref{lemma:Holder-scaling}). See also \cite[Proposition 2.12, Appendix C]{Loh23} for a clarified version of the proof. The second estimate can be proved as follows (see also \cite[Lemma 6.6]{FRW24}). First, we have for any $z \in Q_r(z_0)$ and $\delta \in (0,1]$,
\begin{align*}
|f(z)| &\le |f(z) - (f)_{Q}| + (f)_{Q} \le \sup_{z' \in Q} |f(z) - f(z')| + C (\delta r)^{-(2+4n)} \Vert f \Vert_{L^1(Q)} \\
&\le (2\delta r)^{\eps_2} [f]_{C^{\eps_2}_{\ell}(Q_r(z_0))} + C (\delta r)^{-(2+4n)} \Vert f \Vert_{L^1(Q_r(z_0))} ,
\end{align*}
where we denoted $Q = Q_r(z_0) \cap Q_{\delta r}(z)$, $(f)_{Q} = \dashint_{Q_r(z_0) \cap Q_{\delta r}(z)} f(z') \d z'$ and used that $|Q| \ge C (\delta r)^{2+4n}$, and $d_{\ell}(z,z') \le 2 \delta r$. By taking the supremum over $z \in Q_r(z_0)$, and choosing $\delta \in (0,1)$ appropriately, we deduce the second claim with $k_2 = 0$. The second claim for $k_2 \in \N$ follows by combination with the first claim.
\end{proof}

The following lemma is a product rule for kinetic H\"older spaces.

\begin{lemma}[Lemma 3.7 in \cite{ImSi22}]
\label{lemma:product-rule}
Let $k \in \N \cup \{ 0 \}$, $\eps \in (0,1]$.
Let $f,g \in C^{k,\eps}_{\ell}(Q_r(z_0))$ for some $r \in (0,1]$ and $z_0 \in \R^{1+2n}$. Then,
\begin{align*}
\Vert fg \Vert_{C^{k+\eps}_{\ell}(Q_r(z_0))} \le C \Vert f \Vert_{C^{k+\eps}_{\ell}(Q_r(z_0))}\Vert g \Vert_{C^{k+\eps}_{\ell}(Q_r(z_0))},
\end{align*}
where $C$ depends only on $n,k,\eps$.
\end{lemma}

Next, we recall the relation between derivatives and kinetic H\"older norms.

\begin{lemma}[Lemma 2.7 in \cite{ImSi21}]
\label{lemma:der-Holder}
Let $k \in \N \cup \{ 0 \}$, $\eps \in (0,1]$ and $f \in C^{k,\eps}_{\ell}(Q_r(z_0))$ for some $r \in (0,1]$ and $z_0 \in \R^{1+2n}$. Let $\beta = (\beta_t,\beta_x,\beta_v) \in (\N \cup \{0\})^{1+2n}$ with $|\beta| := 2|\beta_t| + 3|\beta_x| + |\beta_v| = l$ for some $l \in \N$, $l \le k$, and
\begin{align}
\label{eq:Dbeta-f}
D^{\beta} f = (\partial_t + v \cdot \nabla_x)^{\beta_t} \partial_{x_1}^{\beta_{x_1}} \dots \partial_{x_n}^{\beta_{x_n}} \partial_{v_1}^{\beta_{v_1}} \dots \partial_{v_n}^{\beta_{v_n}} f.
\end{align}
Then it holds
\begin{align*}
[D^{\beta} f]_{C^{k-l,\eps}_{\ell}(Q_r(z_0))} \le C [f]_{C^{k,\eps}_{\ell}(Q_r(z_0))},
\end{align*}
where $C$ depends only on $n,k,l,\eps$.
\end{lemma}

We will also use the following reverse version of \autoref{lemma:der-Holder}:

\begin{lemma}
\label{lemma:reverse-der-Holder}
Let $k \in \N \cup \{ 0 \}$, $\eps \in (0,1]$ and $f$ be such that $D^{\beta} f \in C^{\eps}_{\ell}(Q_r(z_0))$ for some $r \in (0,1]$ and $z_0 \in \R^{1+2n}$ and any $\beta \in (\N \cup \{0\})^{1+2n}$ with $|\beta| = k$, where $D^{\beta} f$ is defined as in \eqref{eq:Dbeta-f}. Then, it holds $f \in C^{k,\eps}_{\ell}(Q_r(z_0))$ and
\begin{align*}
[f]_{C^{k+\eps}_{\ell} (Q_r(z_0))} \le C \sum_{|\beta| = k}[D^{\beta} f]_{C^{\eps}_{\ell} (Q_r(z_0))},
\end{align*}
where $C > 0$ only depends on $n,k,\eps$.
\end{lemma}

\begin{proof}
By Taylor's formula and the mean value theorem, which are applicable since $D^{\beta} f$ is H\"older continuous by assumption, we have for any $z_1 \in Q_r(z_0)$ such that $z_1 \circ z \in Q_r(z_0)$:
\begin{align*}
f(z_1 \circ z) = \sum_{0 \le |\alpha| < k } D^{\alpha}f(z_1)z^{\alpha} c_{\alpha} + \sum_{|\beta| = k} D^{\beta} f(z_1 \circ z') z^{\beta} c_{\beta}
\end{align*}
for some constants $c_{\alpha}, c_{\beta} > 0$, and $z'$ such that $z_1 \circ z' \in Q_r(z_0)$ with $d_{\ell}(z', 0) \le d_{\ell}(z,0)$. Here, we are using the notation $z^{\beta} := t^{\beta_t} x_1^{\beta_{x_1}} \dots x_n^{\beta_{x_n}} v_1^{\beta_{v_1}} \dots v_n^{\beta_{v_n}}$ and observe that $|z^{\beta}| \le d_{\ell}(z,0)^{k}$. Hence, denoting
\begin{align*}
p_0(z) :=  \sum_{0 \le |\alpha| \le k } D^{\alpha}f(z_1)z^{\alpha} c_{\alpha} \in \cP_k, 
\end{align*}
we obtain 
\begin{align*}
|f(z_1 \circ z) - p_0(z)| &= \left| \sum_{|\beta| = k} ( D^{\beta} f(z_1 \circ z') - D^{\beta} f(z_1)) z^{\beta} c_{\beta} \right| \\
&\le C d_{\ell}(z,0)^k \sum_{|\beta| = k}  |D^{\beta} f(z_1 \circ z') - D^{\beta} f(z_1)| \le C d_{\ell}(z,0)^{k+\eps} \sum_{|\beta| = k}[D^{\beta} f]_{C^{\eps}_{\ell}(Q_r(z_0))}.
\end{align*}
In the last step, we used that $D^{\beta}f \in C^{\eps}_{\ell}(Q_r(z_0))$ and that $d_{\ell}(z_1 \circ z' , z_1) = d_{\ell}(z' , 0) \le d_{\ell}(z,0)$ by \eqref{eq:distance-invariance}.
Thus, letting $p_{z_1} \in \cP_k$ such that $p_0(z) = p_{z_1}(z_1 \circ z)$, and using again \eqref{eq:distance-invariance}, we obtain for any $\tilde{z} = z_1 \circ z \in Q_r(z_0)$:
\begin{align*}
|f(\tilde{z}) - p_{z_1}(\tilde{z})| \le C d_{\ell}(\tilde{z},z_1)^{k+\eps} \sum_{|\beta| = k}[D^{\beta} f]_{C^{\eps}_{\ell}(Q_r(z_0))}.
\end{align*}
This implies the desired result, since $\tilde{z},z_1 \in Q_r(z_0)$ were arbitrary.
\end{proof}

\subsection{Generalized kinetic polynomials}

We need several auxiliary lemmas for kinetic polynomials, respectively for finite dimensional subspaces of the form
\begin{align*}
\cP_I = \Big\{ \sum_{q \in I} \alpha^{(q)} q(t,x,v) : \alpha^{(q)} \in \R \Big\},
\end{align*}
where $I$ is a finite set of homogeneous functions (monomials), i.e. for every $q \in I$ there is $\lambda_q \ge 0$ such that $q(r^2t,r^3x,rv) = r^{\lambda_q}q(t,x,v)$.
Note that the spaces $\cP_I$ are generalizations of the spaces of kinetic polynomials, and $\cP_{I_k} = \cP_k$ when 
\begin{align}
\label{eq:Ik}
I_k = \Big\{ t^{\beta} x_1^{\beta_{x_1}} \cdot \dots \cdot x_n^{\beta_{x_n}} v_1^{\beta_{v_1}} \cdot \dots \cdot v_n^{\beta_{v_n}} : |\beta| = 2 \beta_t + 3 |\beta_x| + |\beta_v| \le k \Big\}.
\end{align}

\begin{lemma}
\label{lemma:L43}
Let $k \in \N$ and $\eps \in (0,1)$. Let $I$ be a finite set of homogeneous functions $q$ of degree $\lambda_q \le k$. Let $D \subset \R^{1+2n}$ with $0 \in \partial D$, and $u \in L^{\infty}(H_1(0))$. Let $\delta > 0$. If for any $\rho \in (0,\delta)$ there is $p_{\rho} \in \cP_I$ such that for some $c_0 > 0$
\begin{align*}
\Vert u - p_{\rho} \Vert_{L^{\infty}(H_{\rho}(0))} \le c_0 \rho^{k+\eps},
\end{align*}
then there exists $p_0 \in \cP_I$ such that for any $\rho \in (0,\delta)$
\begin{align*}
\Vert u - p_{0} \Vert_{L^{\infty}(H_{\rho}(0))} &\le C c_0 \rho^{k+\eps}
\end{align*}
for some constant $C$, depending only on $n,k,\eps$.
\end{lemma}

\begin{proof}
For any $\rho \in (0,\frac{\delta}{2})$ we have by assumption
\begin{align*}
\Vert p_{\rho} - p_{2\rho} \Vert_{L^{\infty}(H_{\rho}(0))} \le \Vert p_{\rho} - u \Vert_{L^{\infty}(H_{\rho}(0))} +\Vert u - p_{2\rho} \Vert_{L^{\infty}(H_{2\rho}(0))} \le c_0 (1 + 2^{k+\eps}) \rho^{k+\eps}.
\end{align*}
Hence, for any $q \in I$ we observe that
\begin{align}
\label{eq:Cauchy-coeff}
|\alpha_{\rho}^{(q)} - \alpha_{2\rho}^{(q)}| \le C \rho^{-\lambda_q} \Vert p_{\rho} - p_{2\rho} \Vert_{L^{\infty}(\partial Q_{\rho}(0))}  \le C c_0 \rho^{k + \eps - \lambda_q}.
\end{align}

Here, we are using that $\sum_{q \in I} |\alpha^{(q)}|$ is comparable to $\Vert p \Vert_{L^{\infty}(\partial Q_1(0))}$ for any $p \in \cP_I$ since they are both norms on the finite dimensional space $\cP_I$, as well as the homogeneity of the $q$ for any $q \in I$, and scaling (see also \cite[Lemma 2.8]{ImSi21}).

Moreover, we have
\begin{align*}
\Vert p_{\delta} \Vert_{L^{\infty}(H_{\delta}(0))} \le \Vert p_{\delta} - u \Vert_{L^{\infty}(H_{\delta}(0))} + \Vert u \Vert_{L^{\infty}(H_{\delta}(0))} \le c_0 (\delta/2)^{k+\eps} + \Vert u \Vert_{L^{\infty}(H_{\delta}(0))},
\end{align*}
and therefore
\begin{align*}
|\alpha_{\delta}^{(q)}| \le C \delta^{-\lambda_q} \big(c_0 \delta^{k+\eps} + \Vert u \Vert_{L^{\infty}(H_{\delta}(0))} \big)< \infty.
\end{align*}
Hence, since $k+\eps - \lambda_q > 0$ for every $q \in I$, we deduce that 
\begin{align*}
\alpha_0^{(q)} = \lim_{\rho \searrow 0} \alpha_{\rho}^{(q)}
\end{align*}
exists for any $q \in I$. Thus, due to \eqref{eq:Cauchy-coeff}:
\begin{align*}
|\alpha_0^{(q)} - \alpha_{\rho}^{(q)}| \le \sum_{j = 0}^{\infty} |\alpha_{2^{-j}\rho}^{(q)} - \alpha_{2^{-j-1}\rho}^{(q)}| \le C c_0  \sum_{j = 0}^{\infty} (2^{-j} \rho)^{k+\eps - \lambda_q} \le C c_0 \rho^{k+\eps - \lambda_q},
\end{align*}
which yields by defining $p_0 = \sum_{q \in I} \alpha_0^{(q)} q$
\begin{align*}
\Vert p_{0} - p_{\rho} \Vert_{L^{\infty}(H_{\rho}(0))} \le C \sum_{q \in I} |\alpha_0^{(q)} - \alpha_{\rho}^{(q)}| \rho^{\lambda_q} \le C c_0 \rho^{k+\eps}.
\end{align*}
Altogether, we have 
\begin{align*}
\Vert p_{0} - u \Vert_{L^{\infty}(H_{\rho}(0))} \le \Vert u - p_{\rho} \Vert_{L^{\infty}(H_{\rho}(0))} + \Vert p_0 - p_{\rho} \Vert_{L^{\infty}(H_{\rho}(0))} \le C c_0 r^{k+\eps}.
\end{align*}

\end{proof}

\begin{lemma}
\label{lemma:orth-proj-prop}
Let $k \in \N$ and $\eps \in (0,1)$. Let $I$ be a finite set of homogeneous functions $q$ of degree $\lambda_q \le k$. Let $D \subset \R^{1+2n}$ with $0 \in \partial D$, and $u \in L^{\infty}(H_{\delta}(0))$. Let $\delta > 0$. Assume that for any $r \in (0,\delta]$ there is 
\begin{align*}
p_{r}(z) = \sum_{q \in I} \alpha^{(q)}_r q(z) \in \cP_I
\end{align*}
such that for some function $\theta : (0,\delta] \to (0,\infty)$ with $\theta(\rho) \nearrow \infty$ as $\rho \to 0$,
\begin{align*}
\Vert u - p_{\rho} \Vert_{L^{\infty}(H_{\rho}(0))} \le \theta(\rho) \rho^{k+\eps}.
\end{align*}
Then, it holds for any $R \ge 1$ with $Rr \le \delta$
\begin{align*}
\frac{\Vert u - p_{r} \Vert_{L^{\infty}(H_{Rr}(0))}}{r^{k+\eps} \theta(r)} &\le C R^{k+\eps},\\
 \frac{|\alpha_r^{(q)}|}{\theta(r)} &\le C \left( [\delta^{k+\eps-\lambda_q}(\theta(\delta) + C(K)) + \delta^{-\lambda_q}\Vert u \Vert_{L^{\infty}(H_{\delta}(0))}] \theta(r)^{-1} + (2^{-K} \delta)^{k+\eps - \lambda_q} \right) ~~ \forall K \in \N
\end{align*}
for some $C$, depending only on $n,k,\eps$, where $C(K) := \sum_{i = 1}^K \frac{\theta(2^{-i-1} \delta )}{\theta(r)} (2^{-i})^{k+\eps - \lambda_q}$.
\end{lemma}

Note that the second statement proves that $|\alpha_r^{(q)}| \theta(r)^{-1} \to 0$ as $r \to 0$, and that the modulus of convergence does not depend on $|\alpha_r^{(q)}|$.

\begin{proof}
We have by construction for any $q \in I$
\begin{align*}
r^{\lambda_q} |\alpha^{(q)}_{r} - \alpha^{(q)}_{2r}| &\le c \Vert p_{r} - p_{2r} \Vert_{L^{\infty}(H_r(0))} \\
&\le c \Vert u - p_{r}\Vert_{L^{\infty}(H_r(0))}  + c \Vert u - p_{2r}\Vert_{L^{\infty}(H_{2r}(0))} \\
&\le c \theta(r) r^{k+\eps} + c \theta(2r) (2r)^{k+\eps} \le c \theta(r) r^{k+\eps}.
\end{align*}
Iterating this inequality, we get for any $N \in \N$, as long as $2^Nr \le \delta$:
\begin{align}
\label{eq:coeff-comp}
\begin{split}
|\alpha^{(q)}_{r} - \alpha^{(q)}_{2^Nr}| &\le \sum_{i = 0}^{N-1} |\alpha^{(q)}_{2^i r} - \alpha^{(q)}_{2^{i+1}r}| \le c \sum_{i = 0}^{N-1} \theta(2^i r) (2^i r)^{k+\eps - \lambda_q} \\
& \le c \theta(r) r^{k+\eps - \lambda_q}\sum_{i = 0}^{N-1} \frac{\theta(2^i r)}{\theta(r)} 2^{i(k+\eps - \lambda_q)} \le c \theta(r) (2^N r)^{k +\eps - \lambda_q},
\end{split}
\end{align}
which implies that for any $R \ge 1$ with $Rr \le \delta$
\begin{align*}
\Vert p_{r} - p_{Rr} \Vert_{L^{\infty}(H_{Rr}(0))} \le c \sum_{q \in I} (Rr)^{\lambda_q} |\alpha^{(q)}_{r} - \alpha^{(q)}_{Rr}| \le  c \sum_{q \in I} (Rr)^{\lambda_q} \theta(r) (Rr)^{k +\eps - \lambda_q} \le c \theta(r) (Rr)^{k+\eps}.
\end{align*}
As a consequence, we establish the first claim. In fact, for any $R \ge 1$  with $Rr \le \delta$:
\begin{align*}
\frac{\Vert u - p_{r} \Vert_{L^{\infty}(H_{R r}(0))} }{r^{k+\eps} \theta(r)} &\le \frac{\Vert u - p_{R r} \Vert_{L^{\infty}(H_{R r}(0))} }{r^{k+\eps} \theta(r)} +  \frac{\Vert p_{R r} - p_{r} \Vert_{L^{\infty}(H_{R r}(0))} }{r^{k+\eps} \theta(r)} \\
&\le \frac{\theta(R r) (R r)^{k+\eps}}{r^{k+\eps} \theta(r)} + c \frac{\theta(r) (R r)^{k+\eps}}{r^{k+\eps} \theta(r)} \le c R^{k+\eps}.
\end{align*}

The second claim follows by choosing for any $r > 0$ a number $N \in \N$ such that $2^N r \in [\delta/2,\delta]$ and deducing from \eqref{eq:coeff-comp} and the monotonicity of $r \mapsto \theta(r)$ that for any $K \in \N$:
\begin{align*}
\frac{|\alpha^{(q)}_{r} - \alpha^{(q)}_{2^N r}|}{\theta(r)} &\le c \sum_{i = 1}^N \frac{\theta(2^{N-i}r)}{\theta(r)} (2^{N-i}r)^{k+\eps - \lambda_q}  \\
&\le c \delta^{k+\eps - \lambda_q} \sum_{i = 1}^K \frac{\theta(2^{-i-1} \delta )}{\theta(r)} (2^{-i})^{k+\eps - \lambda_q}  + c \delta^{k+\eps - \lambda_q} \sum_{i = K+1}^N \ (2^{-i})^{k+\eps - \lambda_q} \\
&=: C(K) \delta^{k+\eps - \lambda_q} \theta(r)^{-1} + c \delta^{k+\eps - \lambda_q} 2^{-K(k+\eps - \lambda_q)}.
\end{align*}
and thus since for any $\rho \in [\delta/2,\delta]$
\begin{align*}
|\alpha^{(q)}_{\rho}| \le c \delta^{-\lambda_q} \Vert p_{\rho} \Vert_{L^{\infty}(H_{\rho}(0))} \le c \delta^{k+\eps - \lambda_q} \theta(\delta) + c \delta^{-\lambda_q} \Vert u \Vert_{L^{\infty}(H_{\delta}(0))} =: c_{\delta},
\end{align*}
we have
\begin{align*}
\frac{|\alpha^{(q)}_{r}|}{\theta(r)} &\le \frac{|\alpha^{(q)}_{2^N r}|}{\theta(r)}  + \frac{|\alpha^{(q)}_{r} - \alpha^{(q)}_{2^N r}|}{\theta(r)} \\
&\le (c_{\delta} + C(K) \delta^{k+\eps-\lambda_q})\theta(r)^{-1} + c\delta^{k+\eps - \lambda_q} 2^{-K(k+\eps - \lambda_q)} ,
\end{align*}
exactly as we claimed in \eqref{eq:coefficients-vanish}.
\end{proof}

\subsection{Flattening the boundary}

As is common in the study of the boundary regularity for solutions to second order partial differential equations, it is convenient to study domains with flat boundaries first. The goal of this section is to construct a diffeomorphism $\Phi$ which allows to transform kinetic equations in general domains to equations in a domain with flat boundary. One of its central properties is that $\Phi$ preserves the specular reflection (and also the in-flow) boundary condition.

Such a diffeomorphism was already introduced in \cite{GHJO20,DGY22} in three dimensions and it can be generalized to higher dimensions in a straightforward way. Nonetheless, in this paper we will work with a slightly modified version of the diffeomorphism in \cite{GHJO20,DGY22}, which has the advantage that it preserves the regularity of the domain in the $x$-variable, compared to the diffeomorphism in \cite{GHJO20,DGY22}, which has one derivative less than the regularity of the domain.

Given a $C^{\beta}$ domain $\Omega \subset \R^n$ with $\beta \ge 2$ and $\beta \not\in \N$, and given $x_0 \in \partial \Omega$ we can assume without loss of generality that $n_{x_0} = e_n$.

We construct the flattening diffeomorphism as follows: First, let us consider a regularized distance function $d \in C^{\beta}(\overline{\Omega}) \cap C^{\infty}_{loc}(\Omega)$ satisfying
\begin{align}
\label{eq:reg-dist}
\nabla d(x) = -n_x ~~ \forall x \in \partial \Omega, \qquad \Vert d \Vert_{C^{\gamma}(B_r(z))} \le C r^{\beta - \gamma} ~~ \forall \gamma > \beta, ~~ B_{2r}(z) \subset \Omega,
\end{align}
where $n_x \in \mathbb{S}^{n-1}$ denotes the outward unit normal normal vector of $\Omega$ at $x \in \partial \Omega$. We refer to \cite{Lie85} for a construction of such a regularized distance function.

Then, there is a $C^{\beta}$ diffeomorphism $\psi : \{ x_n > 0 \} \cap B_1 \to \Omega \cap B_1(x_0)$ with $\psi(\{ x_n = 0 \} \cap B_1) \subset \partial \Omega \cap B_1(x_0)$, $\psi(\{ x_n > 0 \} \cap B_1) = \Omega \cap B_1(x_0)$, and (see \cite[Lemma A.3]{AbRo20})
\begin{align}
\label{eq:psi-dist}
d(\psi(y)) = (y_n)_+ , ~~ \text{ and therefore } ~~ \delta_{j,n} = (\nabla d)(\psi(y)) \partial_j \psi(y)  ~~ \forall y \in \{ y_n > 0 \} \cap B_1.
\end{align}
Note that we can assume without loss of generality that $\psi^{-1}(x_0) = 0$ and $D \psi^{-1}(x_0) = Id_n$.

Then, we define
\begin{align*}
\phi^{-1}(y) = \psi(y',0) + y_n (\nabla d)(\psi(y)).
\end{align*}
Note that $\phi \in C^{\beta}(\overline{\Omega} \cap B_1)$ since for any $z \in \{z_n > 0 \}$ with $r = \frac{z_n}{2}$, it holds
\begin{align}
\label{eq:phi-reg}
\begin{split}
\Vert \phi^{-1} \Vert_{C^{\beta}(B_r(z))} &\le \Vert \psi \Vert_{C^{\beta}(B_r(z))} + C \Vert \nabla d \Vert_{C^{\beta-1}(\psi(B_r(z)))} \Vert \psi \Vert_{C^{\beta-1}(B_r(z))} \\
&\quad + C r \Vert \nabla d \Vert_{C^{\beta}(\psi(B_r(z)))} \Vert \psi \Vert_{C^{\beta}(B_r(z))} \\
&\le C \Vert \psi \Vert_{C^{\beta}(B_r(z))} < \infty,
\end{split}
\end{align}
where we used the second property in \eqref{eq:reg-dist} and that $d(\psi(B_r(z))) \ge r$ by the first property in \eqref{eq:psi-dist}.

Moreover, note that we clearly have by the first property in \eqref{eq:reg-dist}
\begin{align}
\label{eq:phi-n}
\partial_n \phi^{-1}(y',0) = (\nabla d)(\psi(y',0)) = -n_{\psi(y',0)} =  -n_{\phi^{-1}(y',0)},
\end{align}
and, for $j \in \{1,\dots,n-1\}$, by the second property in \eqref{eq:psi-dist}
\begin{align}
\label{eq:phi-j}
\partial_j \phi^{-1}(y',0) = \partial_j \psi(y',0) \perp n_{\psi(y',0)} = n_{\phi^{-1}(y',0)}.
\end{align}

In particular, since 
\begin{align*}
D \phi^{-1}(y',0) = \Big( \partial_1 \phi^{-1}(y',0) , \dots , \partial_n \phi^{-1}(y',0) \Big),
\end{align*}
a combination of \eqref{eq:phi-n} and \eqref{eq:phi-j} implies that
\begin{align}
\label{eq:trafo-new-properties-2}
(D \phi^{-1}(y',0))^T (\nabla d)(\phi^{-1}(y',0)) = e_n.
\end{align}

Moreover, we define
\begin{align}
\label{eq:Phi-def}
\Phi(t,x,v) = (t,\phi(x) , D \phi(x) \cdot v), \qquad \Phi^{-1}(t,y,w) = (t,\phi^{-1}(y) , D \phi^{-1}(y) \cdot w).
\end{align}

Since it holds by construction that $\phi(x_0) = (x_0',0)$ and $D \phi(x_0) = \mathrm{Id}_n$,  we have that $\Phi(t_0,x_0,v_0) = (t_0,x_0',0,v_0)$. For $z_0 = (t_0,x_0,v_0)$, we also define $\tilde{z}_0 = \Phi(t_0,x_0,v_0) = (t_0,x_0',0,v_0)$. Moreover, $\Phi$ satisfies $\Phi((-1,1) \times (\Omega \cap B_{\delta_1}(x_0)) \times \R^n) \subset ((-1,1) \times ( \{ x_n > 0 \} \cap B_{\delta_2}(x_0',0)) \times \R^n$ for some $\delta_1,\delta_2 \in (0,1]$ by construction, depending only on $\Omega$.

The following properties are elementary but crucial to us:

\begin{lemma}
\label{lemma:trafo-sr}
Let $\Omega$ and $\Phi$ be as above, and $x_0 \in \partial \Omega$. Then, the following hold true:
\begin{itemize}
\item[(1)] $\Phi$ preserves the boundary regions, i.e. if $x \in \partial \Omega \cap B_{\delta_1}(x_0)$, then $(t,x,v) \in \gamma_{\pm}$ if and only if $\Phi(t,x,v) \in \gamma_{\pm}$, i.e. $\pm (D \phi(x) \cdot v)_n < 0$.
\item[(2)] $\Phi$ commutes with the reflection operator, i.e. for any $(t,y,w) \in \Phi((-1,1) \times (\{ x_n \ge 0 \} \cap B_{\delta_1}(x_0)) \times \R^n)$ with $y_n = 0$ it holds
\begin{align*}
D \phi^{-1}(y) \cdot \mathcal{R}_{y} w = \mathcal{R}_{\phi^{-1}(y)} (D \phi^{-1}(y) \cdot w).
\end{align*}
\end{itemize}
\end{lemma}

\begin{proof}
Let us denote $x = \phi^{-1}(y) = \phi^{-1}(y',0) \in \partial \Omega$. Then, since $\nabla d(\phi^{-1}(y)) = -n_{\phi^{-1}(y)}$ by the first property in \eqref{eq:reg-dist}, and by \eqref{eq:trafo-new-properties-2} we have for any $w \in \R^n$:
\begin{align*}
[D \phi^{-1}(y) \cdot w] \cdot n_{\phi^{-1}(y)} = [(D \phi^{-1}(y))^T \cdot n_{\phi^{-1}(y)}] \cdot w = - w_n.
\end{align*}
Hence, $[D \phi^{-1}(y) \cdot w] \cdot n_{\phi^{-1}(y)}$ has the same sign as $-w_n$, the outward normal vector to $\{ x_n > 0 \},$ which implies the first claim.

For the second claim, we observe that
\begin{align*}
\mathcal{R}_y w = (w', -w_n) = w - 2w_n e_n, \qquad \mathcal{R}_{\phi^{-1}(y)} w = w - 2 (w \cdot n_{\phi^{-1}(y)}) n_{\phi^{-1}(y)},
\end{align*}
and therefore
\begin{align*}
D \phi^{-1}(y) \cdot \mathcal{R}_{y} w = D \phi^{-1}(y) \cdot w - 2 D \phi^{-1}(y) \cdot w_n e_n,
\end{align*}
as well as, using \eqref{eq:trafo-new-properties-2} one more time:
\begin{align*}
\mathcal{R}_{\phi^{-1}(y)} (D \phi^{-1}(y) \cdot w) &= D \phi^{-1}(y) \cdot w - 2 \Big( [D \phi^{-1}(y) \cdot w] \cdot n_{\phi^{-1}(y)} \Big) n_{\phi^{-1}(y)} \\
&= D \phi^{-1}(y) \cdot w - 2 \Big( [(D \phi^{-1}(y))^T \cdot n_{\phi^{-1}(y)}] \cdot w \Big) n_{\phi^{-1}(y)} \\
&= D \phi^{-1}(y) \cdot w + 2 w_n n_{\phi^{-1}(y)}.
\end{align*}
Hence, the second property holds true if and only if
\begin{align*}
-n_{\phi^{-1}(y)} = D \phi^{-1}(y) e_n = \partial_n \phi^{-1}(y),
\end{align*}
which is indeed satisfied by construction, see the second property in \eqref{eq:phi-n}. This completes the proof of the second claim.
\end{proof}

The following lemma investigates the kinetic equation that is obtained by transforming $f$ to $\tilde{f} = f \circ \Phi^{-1}$, where $\Phi$ denotes the flattening transformation constructed in \eqref{eq:Phi-def}. This lemma will become crucial in Subsection \ref{subsec:gen-dom}, where we will establish regularity of kinetic equations in general domains.

\begin{lemma}
\label{lemma:SR-preserved}
Let $\Omega$ and $\Phi$ be as above, and let $z_0 = (t_0,x_0,v_0) \in \gamma$. Let $f : \R^{1+2n} \to \R$ and define $\tilde{f}(t,y,w) = f(\Phi^{-1}(t,y,w))$. Let $D \subset (-1,1) \times (\Omega \cap B_{\delta_1}(x_0)) \times \R^n$.

Then, the following hold true:
\begin{itemize}
\item[(i)] We have
\begin{align*}
f(t,x,v) = f(t,x,\mathcal{R}_x v) ~~ \forall (t,x,v) \in \gamma_- \cap D,
\end{align*} 
if and only if
\begin{align*}
\tilde{f}(t,y,w) = \tilde{f}(t,y,w - 2w_n) = \tilde{f}(t,y,\mathcal{R}_y w) ~~ \forall (t,y,w) \in \Phi(D).
\end{align*}

\item[(ii)] Assume that $f$ is a weak solution to
\begin{align*}
\partial_t f + v \cdot \nabla_x f + (- a^{i,j}\partial_{v_i,v_j}) f = - b \cdot \nabla_v f - c f + h ~~ \text{ in } D,
\end{align*}
where $a^{i,j} \in C^{0,1}_{\ell}(D)$ and $b,c,h \in L^{\infty}(D)$.
Then, $\tilde{f}$ is a weak solution to
\begin{align*}
\partial_t \tilde{f} + v \cdot \nabla_x \tilde{f} + (- \tilde{a}^{i,j}\partial_{v_i,v_j}) \tilde{f} = - \tilde{b} \cdot \nabla_v \tilde{f} - \tilde{c} \tilde{f} + \tilde{h} ~~ \text{ in } \Phi(D),
\end{align*}
where we denote
\begin{align*}
A(y) &= (D \phi)( \phi^{-1}(y) ),\\
%H_{\rho_2}(\tilde{z}_0) &= \big( (-1,1) \times \{ y_n > 0 \} \times \R^n \big) \cap Q_{\delta_2}(\tilde{z}_0),\\
\tilde{a}^{i,j}(t,y,w) &= A^{i,r}(y) A^{j,s}(y) a^{r,s}(\Phi^{-1}(t,y,w)),\\
\tilde{b}^i(t,y,w) &= A^{i,j}(y) b^j(\Phi^{-1}(t,y,w)) + a^{j,l}(\Phi^{-1}(t,y,w)) \partial_{x_j,x_l} \phi^{i}(\phi^{-1}(y)) \\
&\quad + (A^{-1})^{j,s}(y)w_s (A^{-1})^{l,r}(y)w_r \partial_{x_j,x_l} \phi^i(\phi^{-1}(y)),\\
\tilde{c}(t,y,w) &= c(\Phi^{-1}(t,y,w)),\\
\tilde{h}(t,y,w) &= h(\Phi^{-1}(t,y,w)).
\end{align*}

\item[(iii)] Let $k \in \N \cup \{ 0 \}$ and $\eps \in (0,1]$. Then, if $\partial \Omega \in C^{\beta}$ for some $\beta \ge \frac{k + 4 + \eps}{2}$ and it holds $D \subset Q_{\delta}(z_0)$ for some $\delta \le c_0 \min \{1 ,|v_0|^{-1} \}$:
\begin{align*}
\Vert \tilde{a}^{i,j} \Vert_{C^{k+\eps}_{\ell}(\Phi(D))} &\le C (1+|v_0|^{\frac{k+5}{2}})\Vert a^{i,j} \Vert_{C^{k+\eps}_{\ell}(D)},\\
\Vert \tilde{b}^i \Vert_{C^{k+\eps}_{\ell}(\Phi(D))} &\le C (1+|v_0|^{\frac{k+6}{2}}) \Vert b^i \Vert_{C^{k+\eps}_{\ell}(D)} + C(1 + |v_0|)^{2},\\
\Vert \tilde{c} \Vert_{C^{k+\eps}_{\ell}(\Phi(D))}  &\le C (1+|v_0|^{\frac{k+4}{2}})\Vert c \Vert_{C^{k+\eps}_{\ell}(D)},  \\
\Vert \tilde{h} \Vert_{C^{k+\eps}_{\ell}(\Phi(D))} &\le C (1+|v_0|^{\frac{k+4}{2}})\Vert h \Vert_{C^{k+\eps}_{\ell}(D)}, \\
\Vert \tilde{f} \Vert_{C^{k+\eps}_{\ell}(\Phi(D))} &\le C (1+|v_0|^{\frac{k+4}{2}}) \Vert f \Vert_{C^{k+\eps}_{\ell}(D)},\\
\Vert f \Vert_{C^{k+\eps}_{\ell}(D)} &\le C (1+|v_0|^{\frac{k+4}{2}}) \Vert \tilde{f} \Vert_{C^{k+\eps}_{\ell}(\Phi(D))},
\end{align*}
where $C$ depends only on $n,k,\eps,\Omega,c_0$. 

Moreover, if $\delta \le c_0 \min\{1 , |v_0|^{-\frac{k+6}{2(\eps - \eps')}} \}$ for some $\eps' \in (0,\eps)$, then
\begin{align*}
\Vert \tilde{a}^{i,j} \Vert_{C^{k+\eps'}_{\ell}(\Phi(D))} &\le C \Vert a^{i,j} \Vert_{C^{k+\eps}_{\ell}(D)}, \\
\Vert \tilde{b}^i \Vert_{C^{k+\eps'}_{\ell}(\Phi(D))} &\le C \Vert b^i \Vert_{C^{k+\eps}_{\ell}(D)} + C(1 + |v_0|)^{2},\\
\Vert \tilde{c} \Vert_{C^{k+\eps'}_{\ell}(\Phi(D))}  &\le C \Vert c \Vert_{C^{k+\eps}_{\ell}(D)},  \\
\Vert \tilde{h} \Vert_{C^{k+\eps'}_{\ell}(\Phi(D))} &\le C \Vert h \Vert_{C^{k+\eps}_{\ell}(D)},
\end{align*}
where $C$ depends only on $n,k,\eps,\Omega,c_0$.
\end{itemize}

\end{lemma}

\begin{proof}
Note that \autoref{lemma:trafo-sr} immediately implies (i). Claim (ii) follows by a straightforward computation and we refer to \cite[p.8]{Sil22} and \cite[Section 7.1.3]{GHJO20} for more details. Note that due to our assumptions on the coefficients, the equation at hand is in divergence form and in non-divergence form at the same time. Hence, the proof in our setting is exactly the same, up to computing the transformed coefficients, which is an easy exercise. 

To see (iii), we first prove the following two claims for any $z_1,z_2 \in \R^{1+2n}$ such that $d_{\ell}(z_1,z_2) \le c \min\{ 1 , \max\{ |v_1|^{-1} , |v_2|^{-1} \}$:
\begin{align}
\label{eq:claim-first-distance}
d_{\ell}(\Phi^{-1}(z_1) , \Phi^{-1}(z_2)) &\le C  d_{\ell}(z_1,z_2) + C \min\{ |v_1| , |v_2| \}^{2/3} d_{\ell}(z_1,z_2)^2 \le C d_{\ell}(z_1,z_2),\\
\label{eq:claim-second-distance}
|x_1-x_2| &\le C  d_{\ell}(z_1,z_2)^3 + C \min\{ |v_1| , |v_2| \} d_{\ell}(z_1,z_2)^2 \le C d_{\ell}(z_1,z_2).
\end{align}

To see it, we assume that $|v_1| \le |v_2|$ and compute by Taylor's formula, using \eqref{eq:dist-comp} and that $\beta \ge 2$:
\begin{align*}
|\phi^{-1}(x_1) - \phi^{-1}(x_2) & - (t_1 - t_2)D \phi^{-1}(x_1) \cdot v_1|^{1/3} \\
 &\le \big|D\phi^{-1}(x_1) \cdot (x_1 - x_2) - (t_1 - t_2)D \phi^{-1}(x_1) \cdot v_1 \big|^{1/3} + C |t_1 - t_2|^{1/3} |x_1 - x_2|^{2/3} \\
&\le C d_{\ell}(z_1,z_2) + C d_{\ell}(z_1,z_2)^{8/3} + C d_{\ell}(z_1,z_2)^2  |v_1|^{2/3}.
\end{align*}
Moreover, note that 
\begin{align*}
|x_1-x_2| \le |x_1 - x_2 - (t_1 - t_2)v_1| + |t_1 - t_2| |v_1| \le d_{\ell}(z_1,z_2)^3 + d_{\ell}(z_1,z_2)^{2} |v_1|,
\end{align*}
which proves both claims. 

%In particular, from \eqref{eq:claim-second-distance}, \autoref{lemma:reverse-der-Holder}, and the assumptions on $D$, we deduce that for $\beta \ge k + \eps +2$, 
%\begin{align}
%\label{eq:A-est}
% \Vert A^{-1} \Vert_{C^{k+\eps}_{\ell}(\Phi(D))} + \Vert A \Vert_{C^{k+\eps}_{\ell}(\Phi(D))} + \Vert \partial_{x_j,x_l} \phi^i (\phi^{-1} (\cdot)) \Vert_{C^{k+\eps}_{\ell}(\Phi(D))} \le C.
%\end{align}

Next, we claim that if $\partial \Omega \in C^{\frac{k+2+\eps}{2}}$ and a function $g \in C^{k+\eps}_{\ell}(D)$, where $D \subset Q_{\delta}(z_0)$ with $\delta \le c_0 \min \{1 , |v_0|^{-1} \}$ is as in (iii), it holds
\begin{align}
\label{eq:claim-chain-rule}
\Vert g(\Phi^{-1} (\cdot)) \Vert_{C^{k+\eps}_{\ell}(\Phi(D))} \le C (1 + |v_0|)^{\frac{k+4}{2}} \Vert g \Vert_{C^{k+\eps}_{\ell}(D)}.
\end{align}
Moreover, if $\eps' \in (0,\eps)$ and $\delta \le c_0 \min \{ 1 , |v_0|^{-\frac{k+4}{2(\eps - \eps')}} \}$, then
\begin{align}
\label{eq:claim-chain-rule-improved}
\Vert g(\Phi^{-1} (\cdot)) \Vert_{C^{k+\eps'}_{\ell}(\Phi(D))} \le C  \Vert g \Vert_{C^{k+\eps}_{\ell}(D)}.
\end{align}
Note that the claims in (iii) for $\tilde{a}^{i,j}, \tilde{b}^i, \tilde{c}, \tilde{h}$ follow immediately from \eqref{eq:claim-chain-rule} together with the fact that $\phi$ is as regular as $\partial \Omega$ by \eqref{eq:phi-reg} and the product rule. In particular, the regularity of $\partial \Omega \in C^{\frac{k+\eps + 4}{2}}$ in (iii) and (iv) cannot be improved, since in order to prove $\tilde{b}^i \in C^{k+\eps}_{\ell}$, we need to take $g$ to be a second derivative of $\phi$ in \eqref{eq:claim-chain-rule}. Then, \eqref{eq:claim-chain-rule} only gives a $C^{k+\eps}_{\ell}$ estimate if $\phi \in C^{\frac{(k+2)+\eps+2}{2}}$. Hence, we require $\partial\Omega \in C^{\frac{k+\eps+4}{2}}$.

It remains to prove \eqref{eq:claim-chain-rule} and \eqref{eq:claim-chain-rule-improved}. To see it, note that for any $\beta \in (\N \cup \{ 0 \})^{2n+1}$ with $|\beta| \le k$,
\begin{align}
\label{eq:Linfty-Dphi}
\Vert D^{\beta} \Phi^{-1} \Vert_{L^{\infty}(\Phi(D))} \le C(1 + |v_0|)^{\frac{|\beta|+2}{2}},
\end{align}
where $D^{\beta} \Phi^{-1}$ is defined as in \eqref{eq:Dbeta-f}. This follows immediately from the definition of $\Phi^{-1}$ (see \eqref{eq:Phi-def}), observing that the term $(\partial_t + v \cdot \nabla_x )^{\beta_t} (D\phi^{-1}(y) \cdot v)$ is the contribution that leads to the highest power of $|v_0|$, which is of order $\beta_t+1$, and becomes maximal if $\beta_t = \frac{|\beta|}{2}$.
%Moreover, note that since $\partial \Omega \in C^{\frac{k + 4}{3}}$ by assumption, we have for any $l \le k+1$ and $3|\beta_x| = l \in \N \cup \{0\}$, 
%and $z_1,z_2 \in \R^{1+2n}$ with $d_{\ell}(z_1,z_2) \le c_0 \min\{ 1 , \max \{ |v_1|^{-2} , |v_2|^{-2}\} \}$, using also \eqref{eq:claim-second-distance}
%\begin{align*}
%\big|\partial^{\beta_x}_x \phi(x_1) v_1 - \partial^{\beta_x}_x \phi(x_2) v_2 \big| &\le  \big|\partial^{\beta_x}_x \phi(x_1) - \partial^{\beta_x}_x \phi(x_2) \big| |v_2| + \big|\partial^{\beta_x}_x \phi(x_2) \big| |v_1 - v_2| \\
%&\le C |x_1-x_2| |v_2| + C |v_1 - v_2| \\
%&\le C (d_{\ell}(z_1,z_2)^{3} + |v_2| d_{\ell}(z_1,z_2)^2)|v_2| + C d_{\ell}(z_1,z_2) \le C d_{\ell}(z_1,z_2),
%\end{align*}
%since $\phi \in C^{\frac{k+1}{3}}_x$. 
Moreover, if $\partial \Omega \in C^{\frac{k+2+\eps}{2}}$ and therefore $\phi \in C_x^{\frac{k+2 + \eps}{2}}$ by construction of the flattening diffeomorphism, see \eqref{eq:phi-reg}, we have for any $\beta_t \in \N \cup \{ 0 \}$ with $2\beta_t \le k$, and $z_1,z_2 \in \R^{1+2n}$ with $d_{\ell}(z_1,z_2) \le c_0 \min\{ 1 , \max \{ |v_1|^{-1} , |v_2|^{-1}\} \}$, using also \eqref{eq:claim-second-distance}
\begin{align*}
& |(\partial_t + v_1 \cdot \nabla_x)^{\beta_t} D \phi(x_1) v_1 - (\partial_t + v_2 \cdot \nabla_x)^{\beta_t} D \phi(x_2) v_2| \\
&\qquad \le C |(\nabla_x)^{\beta_t} D \phi(x_1) - (\nabla_x)^{\beta_t} D \phi(x_2)||v_1|^{\beta_t+1} + C|v_1- v_2| |v_1|^{\beta_t+1} \\
&\qquad \le C|x_1-x_2|^{\eps/2}|v_2|^{\beta_t+1} + C|v_1-v_2| |v_1|^{\beta_t + 1}\\
&\qquad \le C d_{\ell}(z_1,z_2)^{3\eps/2} |v_1|^{\beta_t+1} + C d_{\ell}(z_1,z_2)^{\eps} |v_1|^{\beta_t+2} + C d_{\ell}(z_1,z_2) |v_1|^{\beta_t + 1} \\
&\qquad \le Cd_{\ell}(z_1,z_2)^{\eps} (1 + |v_0|)^{\frac{|\beta|+4}{2}}.
\end{align*}

Therefore, and using similar arguments for the other entries of $\Phi$ and other derivatives $D^{\beta}$ of the form $\partial_v^{\beta_v}$, as well as mixed derivatives, (but observing that they all lead to lower powers of $|v_0|$ and require less regularity of $\partial \Omega$) we deduce that whenever $d_{\ell}(z_1,z_2) \le c_0 \min\{ 1 , \max \{ |v_1|^{-1} , |v_2|^{-1}\} \}$, then for any $\beta \in (\N \cup \{ 0 \})^{2n+1}$ with $|\beta| \le k$:
\begin{align}
\label{eq:Ceps-Dphi}
[  D^{\beta} \Phi^{-1} ]_{C^{\eps}_{\ell}(\Phi(D))} \le C (1 + |v_0|)^{\frac{|\beta|+4}{2}}.
\end{align}
Moreover, by \eqref{eq:claim-first-distance} we have
\begin{align}
\label{eq:Ceps-Dg}
\Vert (D^{\beta} g)(\Phi^{-1}(\cdot)) \Vert_{L^{\infty}(\Phi(D))} +  [ (D^{\beta} g)(\Phi^{-1}(\cdot)) ]_{C^{\eps}_{\ell}(\Phi(D))} \le C.
\end{align}
In particular, we have by \eqref{eq:Linfty-Dphi}, \eqref{eq:Ceps-Dphi}, \eqref{eq:Ceps-Dg}, and the product rule that for any $\gamma,\beta \in (\N \cup \{0\})^{2n+1}$ with $|\gamma|, |\beta| \le k$:
\begin{align}
\label{eq:before-chain-rule}
\begin{split}
\big[(D^{\gamma} g)(\Phi^{-1}(\cdot))D^{\beta} \Phi^{-1} \big]_{C^{\eps}_{\ell}(\Phi(D))} &\le 2 \big[(D^{\gamma} g)(\Phi^{-1}(\cdot)) \big]_{C^{\eps}_{\ell}(\Phi(D))} \Vert D^{\beta} \Phi^{-1} \Vert_{L^{\infty}(\Phi(D))} \\
&\quad + 2\Vert (D^{\gamma} g)(\Phi^{-1}(\cdot))\Vert_{L^{\infty}(\Phi(D))} \big[ D^{\beta} \Phi^{-1} \big]_{C^{\eps}(\Phi(D))} \\
&\le C (1 + |v_0|)^{\frac{|\beta|+4}{2}}.
\end{split}
\end{align}
By the chain rule and \autoref{lemma:reverse-der-Holder}, this implies \eqref{eq:claim-chain-rule}. To see \eqref{eq:claim-chain-rule-improved}, we observe that if $\delta \le c_0 \min \{ 1 , |v_0|^{-\frac{k+4}{2(\eps - \eps')}} \}$, then \eqref{eq:Ceps-Dphi} remains valid with $\eps'$ instead of $\eps$, and instead of the second estimate in \eqref{eq:Ceps-Dg} we have
\begin{align*}
\big[ (D^{\beta} g)(\Phi^{-1}(\cdot)) \big]_{C^{\eps'}_{\ell}(\Phi(D))} \le C \delta^{\eps - \eps'} \le C \min\{ 1 , |v_0|^{- \frac{ k + 4}{2}} \},
\end{align*}
and instead of \eqref{eq:Ceps-Dphi} we have
\begin{align*}
[  D^{\beta} \Phi^{-1} ]_{C^{\eps'}_{\ell}(\Phi(D))} \le C \delta^{\eps-\eps'} (1 + |v_0|)^{\frac{|\beta|+4}{2}} \le C,
\end{align*}
which yields the following variant of \eqref{eq:before-chain-rule}
\begin{align*}
\big[(D^{\gamma} g)(\Phi^{-1}(\cdot))D^{\beta} \Phi^{-1} \big]_{C^{\eps'}_{\ell}(\Phi(D))} &\le 2 \big[(D^{\gamma} g)(\Phi^{-1}(\cdot)) \big]_{C^{\eps'}_{\ell}(\Phi(D))} \Vert D^{\beta} \Phi^{-1} \Vert_{L^{\infty}(\Phi(D))} \\
&\quad + 2\Vert (D^{\gamma} g)(\Phi^{-1}(\cdot))\Vert_{L^{\infty}(\Phi(D))} \big[ D^{\beta} \Phi^{-1} \big]_{C^{\eps'}(\Phi(D))} \\
&\le C (1 + |v_0|)^{\frac{|\beta|+4}{2}} (1 + |v_0|)^{-\frac{|\beta|+4}{2}}  \le C,
\end{align*}
and hence, we deduce \eqref{eq:claim-chain-rule-improved}.

Finally, to prove the last claim in (iii), we observe that \eqref{eq:claim-first-distance} and \eqref{eq:claim-chain-rule} remain true when $\Phi^{-1}$ is replaced by $\Phi$, using exactly the same arguments.

%To see Claim (iii), we recall that in \cite[p.10]{ImSi21} it was observed that for some $C_1,C_2 > 0$
%\begin{align*}
%d_{\ell}(z_1,z_2) \le C_1 |z_1 - z_2|^{\frac{1}{1+2n}}, \qquad d_{\ell}(z_1,z_2)^{\frac{1}{1+2n}} \ge C_2 |z_1 - z_2|,
%\end{align*}
%and therefore for any $\alpha > 0$ and $D \subset \R^{1+2n}$
%\begin{align*}
%\Vert f \Vert_{C^{\frac{\alpha}{1+2n}}_{\ell}(D)} \le C_3 \Vert f \Vert_{C^{\alpha}(D)} \qquad \Vert f \Vert_{C^{\frac{\alpha}{1+2n}}(D)} \le C_4 \Vert f \Vert_{C^{\alpha}_{\ell}(D)},
%\end{align*}
%where $C_3,C_4$ only depend on $C_1,C_2,\alpha,n$.
%
%Next, note that by construction $\phi \in C^{\beta-1}$ and $\Phi \in C^{\beta-2}$. Therefore, it must be
%\begin{align*}
%\Vert A \Vert_{C^{\frac{\beta}{1+2n}}_{\ell}(\Phi(D))} \le \Vert A \Vert_{C^{\beta}(\Phi(D))} \le c, \qquad \Vert A^{-1} \Vert_{C^{\frac{\beta}{1+2n}}_{\ell}(D)} \le \Vert A^{-1} \Vert_{C^{\beta}(D)} \le c,
%\end{align*}
%where $c > 0$ only depends on $\beta,n,\phi$. From here, the desired result follows immediately from the definitions of $\tilde{a}^{i,j}, \tilde{b}^i,\tilde{c},\tilde{h}$ and \autoref{lemma:product-rule}.
\end{proof}

%\begin{remark}
%Note that in case $g \in C^{k+\eps}_{\ell}(D)$ is independent of $x,v$, then we have whenever $D \subset Q_{\delta}(z_0)$ for some $\delta \le c_0 \min\{ 1 , |v_0|^{-2} \}$ in the setting of \autoref{lemma:SR-preserved} (instead of \eqref{eq:claim-chain-rule})
%\begin{align}
%\label{eq:claim-chain-rule-improved-2}
%\Vert g(\Phi^{-1} (\cdot)) \Vert_{C^{k+\eps}_{\ell}(\Phi(D))} \le C \Vert g \Vert_{C^{k+\eps}_{\ell}(D)}.
%\end{align}
%\end{remark}

\subsection{Reflecting kinetic equations}

In this subsection we see that solutions to kinetic equations which satisfy the specular reflection condition can be extended across the boundary by mirror reflection. For this purpose, we need to define the reflected domain $\mathcal{R}(D)$ associated with any domain $D \subset (-1,1) \times \Omega \times \R^n$. For flat domains, i.e. when $\Omega = \{ x_n > 0 \}$, we set
\begin{align}
\label{eq:reflected-domain-def}
\mathcal{R}(D) = \{ (t,x,v', -v_n) : (t,x,v) \in D \}, \qquad \mathcal{S}(D) = D \cup \mathcal{R}(D).
\end{align}
Note that if $z_0 \in \gamma_0$, then $\mathcal{R}(H_r(z_0)) = H_r(z_0)$.
In non-flat spatial domains $\Omega$, we let $x^{\ast} \in \partial \Omega$ be the projection of $x \in \Omega$ to the boundary and define
\begin{align*}
\mathcal{R}(D) = \{ (t,x,R_{x^{\ast}}v) : (t,x,v) \in D \}, \qquad \mathcal{S}(D) = D \cup \mathcal{R}(D).
\end{align*}

We state the lemma only for equations in divergence form since this is the setting in which we will apply it in the proof of \autoref{lemma:boundary-reg-specular}.

\begin{lemma}
\label{lemma:mirror}
Let $\Omega = \{ x_n > 0 \}$ and $z_0 \in \R^{1+2n}$ with $x_0 \in \{ x_n = 0 \}$. Let $r \in (0,1]$, and $a^{i,j},b,c,h \in L^{\infty}(H_{2r}(z_0))$ and $f$ be a weak solution to
\begin{equation*}
\left\{\begin{array}{rcl}
\partial_t f + v \cdot \nabla_x f -\partial_{v_j}(a^{i,j}\partial_{v_i} f) &=& - b \cdot \nabla_v f - c f + h  ~~  \text{ in }  \mathcal{S}(H_{2r}(z_0)),\\
f(t,x, v) &=& f(t,x, \mathcal{R}_x v) ~~ \qquad\quad \text{ on } \gamma_- \cap \mathcal{S}(H_{2r}(z_0) ).
\end{array}\right.
\end{equation*}
Then, the mirror reflection 
\begin{align*}
\bar{f}(t,x,v) = 
\begin{cases}
f(t,x,v) &~~ \text{ if } (t,x,v) \in \mathcal{S}(H_{2r}(z_0)),\\
f(t,x' , -x_n , v', - v_n) &~~ \text{ if } (t,x',-x_n,v',-v_n) \in \mathcal{S}(H_{2r}(z_0))
\end{cases}
\end{align*}
of $f$ is a weak solution to
\begin{align*}
\partial_t \bar{f} + v \cdot \nabla_x \bar{f} + (- \bar{a}^{i,j}\partial_{v_i,v_j}) \bar{f} &= - \bar{b} \cdot \nabla_v \bar{f} - \bar{c} \bar{f} + \bar{h}  ~~ \text{ in } \mathcal{S}(Q_{2r}(z_0)),
\end{align*}
where $\bar{a}^{i,j} = a^{i,j}$, $\bar{b}^i = b^i$, $\bar{c} = c$, and $\bar{h} = h$ in $H_{2r}(z_0)$, and
\begin{align*}
\bar{a}^{i,j}(z) &= (-1)^{\delta^{i,n} + \delta^{j,n}} a^{i,j}(t,x', - x_n, v', -v_n) ,\\
\bar{b}^{i}(z) &= (-1)^{\delta^{i,n}} b^i(t,x', - x_n, v', -v_n),\\
\bar{c}(z) &= c(t,x', - x_n, v', -v_n),\\
\bar{h}(z) &= h(t,x', - x_n, v', -v_n),
\end{align*}
whenever $z \in \mathcal{R}(H_{2r}(z_0))$. Here, $\delta^{i,j} = \1_{\{i = j\}}$ denotes  Kronecker's delta.
\end{lemma}

\begin{proof}
This result was proved in \cite[Section 7.2]{GHJO20} in three dimensions. The proof in our more general case goes in the exact same way. The main observation is that $\bar{f}$ is continuous across the boundary since it satisfies the specular reflection condition, and that the equation for $\bar{f}$ in $\{ x_n < 0 \}$ is inherited from the equation for $f$ by construction.
\end{proof}

\begin{remark}
\label{remark:weak-implies-strong}
The interior regularity results from the following section will imply that the mirror extended function is actually a strong solution, once $a^{i,j},b,c,h$ are smooth enough and $a^{i,j}(t,x',0,v) = 0$ whenever $\delta^{i,n} \not= \delta^{j,n}$, and $b^i(t,x',0,v) = c(t,x',0,v) = 0$ and $h(t,x',0,v',v_n) = h(t,x',0,v',-v_n)$.
\end{remark}

The following lemma looks quite innocent. However, note that the proof heavily uses the kinetic geometry and that, for a general diffeomorphism $\Phi \in C^{\infty}$, it is not true that $f(\Phi) \in C^{k,\eps}_{\ell}$, whenever $f \in C^{k,\eps}_{\ell}$.

\begin{lemma}
\label{lemma:Holder-spaces-gamma_-}
Let $\Omega = \{ x_n > 0 \}$ and $z_0 \in \R^{1+2n}$ with $x_0 \in \{ x_n = 0 \}$. Let $k \in \N \cup \{ 0 \}$, $\eps \in (0,1]$, $r \in (0,1]$, and $f \in C^{k,\eps}_{\ell}(H_r(z_0))$. Then, it holds
\begin{align*}
C^{-1} [f]_{C^{k,\eps}_{\ell}(\mathcal{R}(\gamma_- \cap Q_r(z_0)))} \le [f(t,x,\mathcal{R}_x v)]_{C^{k,\eps}_{\ell}(\gamma_- \cap Q_r(z_0))} \le  C [f]_{C^{k,\eps}_{\ell}(\mathcal{R}(\gamma_- \cap Q_r(z_0)))},
\end{align*}
where $C$ depends only on $n,k,\eps$. Moreover, we have
\begin{align*}
\mathcal{R}(\gamma_- \cap Q_r(z_0)) = \gamma_+ \cap Q_r(\mathcal{R}(z_0)).  
\end{align*}
\end{lemma}

\begin{proof}
Given any $z,z_0 \in \gamma_- \cap Q_r(z_0)$, i.e. $x_n = (x_0)_n = 0$ and $v_n , (v_0)_n > 0$ it holds $R_x(v) = (v' , - v_n)$ and $R_{x_0}(v_0) = (v_0' , - (v_0)_n)$ and therefore by \eqref{eq:dist-comp} 
\begin{align*}
d_{\ell}(z,z_0) &\le C \left( |t - t_0|^{1/2} + | x - x_0 - (t-t_0)v_0|^{1/3} + |v - v_0| \right) \\
&\le C \left( |t - t_0|^{1/2} + |( x - x_0 - (t-t_0)\mathcal{R}_x(v_0) )'|^{1/3} + |(t-t_0)\mathcal{R}_{x_0}(v_0)_n|^{1/3} + |\mathcal{R}_x(v) - \mathcal{R}_{x_0}(v_0)| \right) \\
&\le C d_{\ell}(\mathcal{R}_x(v) , \mathcal{R}_{x_0}(v_0)).
\end{align*}
Analogously, one shows
\begin{align*}
d_{\ell}(\mathcal{R}_x(v) , \mathcal{R}_{x_0}(v_0)) \le C d_{\ell}(z,z_0).
\end{align*}
From here, the proof follows immediately by the definition of kinetic H\"older spaces, and observing that $p \in \cP_k$, if and only if $\tilde{p}(t,x,v) = p(t,x,\mathcal{R}_x v) \in \cP_k$. The proof of the the second claim follows immediately by writing down the definitions of the corresponding sets. 
\begin{align*}
Q_r(z_0) &= \{ (t,x) : t \in B_{r^2}(t_0) , x \in B_{r^3}((x_0 + (t-t_0)v_0)',(t-t_0)(v_0)_n )\} \times B_r(v_0) , \\
 Q_r(\mathcal{R}(z_0)) &= \{ (t,x) : t \in B_{r^2}(t_0) , x \in B_{r^3}((x_0 + (t-t_0)v_0)' ,  - (t-t_0)(v_0)_n )\} \times B_r(v_0' , -(v_0)_n),\\
 \mathcal{R}(Q_r(z_0)) &= \{ (t,x) : t \in B_{r^2}(t_0) , x \in B_{r^3}((x_0 + (t-t_0)v_0)',(t-t_0)(v_0)_n ) \} \times B_r(v_0' , -(v_0)_n),
\end{align*}
and observing that the intersections of the latter two sets with $\R \times \{x_n = 0 \} \times \R^n$ coincide.
\end{proof}

\subsection{Definition of weak solutions}

Note that since we are interested in higher order regularity estimates in this article, our solutions will satisfy the equation in a classical sense inside the solution domain by interior regularity estimates (see \autoref{lemma:interior-reg}). Still, in order to prove regularity estimates up to the boundary, we heavily rely on the boundary $C^{\alpha}$ regularity established in \cite{Sil22} for weak solutions $f$ that are in particular assumed to satisfy $\nabla_v f \in L^2(D)$. It is easy to see that (see also the comment after Definition 5.1 in \cite{Sil22}) any classical solution $f \in L^2(D)$ with $\nabla_v f \in L^2(D)$ is a weak solution. Moreover, by \autoref{lemma:boundary-reg} and \autoref{lemma:boundary-reg-specular}, any weak solution attains its boundary values on $\gamma$ in a locally uniform H\"older continuous way.

Following \cite[Definitions 3.1 and 5.1]{Sil22}, we define the following weak solution concepts:

\begin{definition}
Let $D \subset \R \times \Omega \times \R^n$, $a^{i,j},b,c,h \in L^{\infty}(D)$. We say that $f \in L^2(D)$ with $\nabla_v f \in L^2(D)$ is a weak solution to 
\begin{align}
\label{eq:def-eq}
\partial_t f + v \cdot \nabla_x f - \partial_{v_j}(a^{i,j}\partial_{v_i} f) = - b \cdot \nabla_v f - c f + h  ~~ \text{ in } D,
\end{align}
if for any $\varphi \in C_c^1(\R^{1+2n})$ with $\supp(\varphi) \cap \gamma \subset \gamma_- \cup \gamma_0$ it holds
\begin{align*}
\int_D -f (\partial_t \varphi + v \cdot \nabla_x \varphi) + a^{i,j} \partial_{v_i} f \partial_{v_j} \varphi + b_i \partial_{v_i} f \varphi + cf\varphi - h \varphi \d z = - \int_{\gamma_-} f \varphi (v \cdot n) \d \gamma.
\end{align*}
If, in addition, $D$ is such that for any $z \in \gamma \cap \partial D$ then also $(t,x,\mathcal{R}_x v) \in \gamma \cap \partial D$, then we say that $f$ is a weak solution to \eqref{eq:def-eq} with
\begin{align*}
f(t,x,v) = f(t,x,\mathcal{R}_x v) ~~ \text{ on } \gamma_- \cap D,
\end{align*}
if for any $\varphi \in C_c^1(\R^{1+2n})$ with $\supp(\varphi) \cap \partial D \subset \gamma$ it holds
\begin{align*}
\int_D -f (\partial_t \varphi + v \cdot \nabla_x \varphi) + a^{i,j} \partial_{v_i} f \partial_{v_j} \varphi + b_i \partial_{v_i} f \varphi + cf\varphi - h \varphi \d z = 0.
\end{align*}
\end{definition}

Note that we will also use this definition for equations that are written in non-divergence form, whenever $a \in C^{0,1}(D)$. In that case, weak solutions have to be understood with respect to the rewritten operator in divergence form.

Moreover, we have the following energy estimate.

\begin{lemma}
\label{lemma:energy-est}
Let $\Omega \subset \R^n$, and $z_0 \in \R^{1+2n}$ with $x \in \overline{\Omega}$, $r \in (0,1]$, and $a^{i,j}\in C^{0,1}_{\ell}(H_{2r}(z_0))$, $b,c,h \in L^{\infty}(H_{2r}(z_0))$, and assume that $a^{i,j}$ satisfies \eqref{eq:unif-ell}. Let $f$ be a weak solution to
\begin{align*}
\partial_t f + v \cdot \nabla_x f + (- a^{i,j}\partial_{v_i,v_j}) f = - b \cdot \nabla_v f - c f + h  ~~ \text{ in } H_{2r}(z_0).
\end{align*}
\begin{itemize}
\item[(i)] If $Q_{2r}(z_0) \cap \gamma \subset \gamma_-$, then 
\begin{align*}
\Vert \nabla_v f \Vert_{L^2(H_r(z_0))}^2 \le C r^{-2}\left(1 + \frac{|v_0|}{r}\right) \left(\Vert f \Vert_{L^2(H_{2r}(z_0))}^2 + \Vert f \Vert_{L^2(\gamma_- \cap Q_{2r}(z_0))}^2 \right).
\end{align*}
\item[(ii)] If $f$ also satisfies
\begin{align*}
f(t,x,v) = f(t,x,\mathcal{R}_x v) ~~ \text{ on } \gamma_- \cap Q_{2r}(z_0),
\end{align*}
and $Q_{2r}(z_0) \cap \gamma \subset \gamma$, then
\begin{align*}
\Vert \nabla_v f \Vert_{L^2(H_r(z_0))}^2 \le C r^{-2}\left(1 + \frac{|v_0|}{r}\right) \Vert f \Vert_{L^2(H_{2r}(z_0))}^2.
\end{align*}
\end{itemize}
The constant $C$ depends only on $n,\lambda,\Lambda,\Vert b \Vert_{L^{\infty}(Q_{2r}(z_0))}, \Vert c \Vert_{L^{\infty}(Q_{2r}(z_0))}$, and $\Vert h \Vert_{L^{\infty}(Q_{2r}(z_0))}$.
\end{lemma}

\begin{proof}
The proof follows along the lines of \cite[Proposition 9]{GuMo22}, testing the weak formulation for $f$ with $f \varphi^2$, where $\varphi \in C^{1}_c(Q_{2r}(z_0))$ is an appropriate cut-off function. This choice of test function can be justified by an approximation argument (see \cite{GuMo22}) which works in the same way at the boundary.
\end{proof}

\subsection{Interior regularity estimates}

In this subsection, we state a higher order interior Schauder estimate for kinetic equations. It follows from the results in \cite{HeSn20,ImMo21,Loh23} (see also \cite{KLN25} for interior estimates for kinetic equations in divergence form).

\begin{lemma}[Interior estimate]
\label{lemma:interior-reg}
Let $z_0 \in \R^{1+2n}$, $k \in \N$ with $k \ge 2$, $\eps \in (0,1)$, and $r \in (0,1]$. Let $a^{i,j},b,c,h \in C^{k-2+\eps}_{\ell}(Q_{2r}(z_0))$ and assume that $a^{i,j}$ satisfies \eqref{eq:unif-ell}. Let $f$ be a solution to
\begin{align*}
\partial_t f + v \cdot \nabla_x f + (- a^{i,j}\partial_{v_i,v_j}) f = - b \cdot \nabla_v f - c f + h  ~~ \text{ in } Q_{2r}(z_0).
\end{align*}
Then, $f \in C^{k+\eps}_{\ell}(Q_{r}(z_0))$ and
\begin{align}
\label{eq:interior-reg}
[f]_{C^{k+\eps}_{\ell}(Q_r(z_0))} \le C r^{-k-\eps}\big(\Vert f \Vert_{L^{\infty}(Q_{2r}(z_0))} + r^2 \Vert h \Vert_{C^{k-2+\eps}_{\ell}(Q_{2r}(z_0))} \big),
\end{align}
where $C$ depends only on $n,k,\eps,\lambda, \Lambda$, $\Vert a^{i,j} \Vert_{C^{k-2+\eps}_{\ell}(Q_{2r}(z_0))}, \Vert b \Vert_{C^{k-2+\eps}_{\ell}(Q_{2r}(z_0))}$, and $\Vert c \Vert_{C^{k-2+\eps}_{\ell}(Q_{2r}(z_0))}$.

Moreover, if $a^{i,j}$ is constant and $b=c=0$, then \eqref{eq:interior-reg} holds for any $r > 0$, where $C > 0$ depends only on $n,k,\eps,\lambda,\Lambda$.
\end{lemma}

Note that we did not specify the notion of solution in \autoref{lemma:interior-reg}. In fact, the result applies to solutions $f$ which might be classical solutions or weak solutions. In both cases, \autoref{lemma:interior-reg} implies that $f \in C^{k+\eps}_{\ell}$ for some $k \ge 2$, and therefore a posteriori $f$ is always a classical solution.

\begin{proof}
This result was established in \cite{Loh23} for $z_0 = 0$ and $r = \frac{1}{2}$ in one-sided cylinders. 
The general case follows by standard scaling arguments. In fact, we deduce from \autoref{lemma:scaling} that $f_{z_0,2r}$ solves 
\begin{align*}
\partial_t f_{z_0,2r} + v \cdot \nabla_x f_{z_0,2r} + (- \tilde{a}^{i,j}\partial_{v_i,v_j}) f_{z_0,2r} = - \tilde{b} \cdot \nabla_v f_{2r,z_0} - \tilde{c} f_{z_0,2r} + \tilde{h} ~~ \text{ in } Q_{1}(0),
\end{align*}
where $\tilde{a}^{i,j}, \tilde{b},\tilde{c} \in C^{k-2+\eps}_{\ell}(Q_1(0))$, and for any $r \in (0,1]$
\begin{align}
\label{eq:Holder-norms-scaling}
\begin{split}
\Vert \tilde{a}^{i,j} \Vert_{C^{k-2+\eps}_{\ell}(Q_1(0))} & + \Vert \tilde{b} \Vert_{C^{k-2+\eps}_{\ell}(Q_1(0))} + \Vert \tilde{c} \Vert_{C^{k-2+\eps}_{\ell}(Q_1(0))} \\
& \le \Vert a^{i,j} \Vert_{C^{k-2+\eps}_{\ell}(Q_{2r}(z_0))} + \Vert b \Vert_{C^{k-2+\eps}_{\ell}(Q_{2r}(z_0))} + \Vert c \Vert_{C^{k-2+\eps}_{\ell}(Q_{2r}(z_0))} \le \Lambda,
\end{split}
\end{align}
by \autoref{lemma:Holder-scaling} and satisfy \eqref{eq:unif-ell} with $\lambda,\Lambda$.
Then, we apply the interior H\"older estimate with $r = 1/2$ and $z_0 = 0$ to $f_{z_0,2r}$ and obtain
\begin{align*}
[f_{2r}]_{C^{k+\eps}(\tilde{Q}_{1/2}(0))} \le C \big(\Vert f_{z_0,2r} \Vert_{L^{\infty}(\tilde{Q}_{1}(0))} + \Vert \tilde{h} \Vert_{C^{k-2+\eps}_{\ell}(\tilde{Q}_{1}(0))} \big).
\end{align*}
Then, upon observing that 
\begin{align*}
\Vert f_{z_0,2r} \Vert_{L^{\infty}(\tilde{Q}_{1}(0))} = \Vert f \Vert_{L^{\infty}(\tilde{Q}_{2r}(z_0))}, \qquad \Vert \tilde{h} \Vert_{L^{\infty}(\tilde{Q}_{1}(0))} = 4 r^2   \Vert h \Vert_{L^{\infty}(\tilde{Q}_{2r}(z_0))},
\end{align*}
and since by \autoref{lemma:Holder-scaling} it holds
\begin{align*}
[f_{z_0,2r}]_{C^{k+\eps}_{\ell}(\tilde{Q}_{1/2}(0))} = (2r)^{k+\eps} [f]_{C^{k+\eps}_{\ell}(\tilde{Q}_r(z_0))},
\end{align*}
we deduce
\begin{align}
\label{eq:one-sided-int}
[f]_{C^{k+\eps}_{\ell}(Q_r(z_0))} \le C r^{-k - \eps} \big(\Vert f \Vert_{L^{\infty}(\tilde{Q}_{2r}(z_0))} + r^2 \Vert h \Vert_{C^{k+\eps}_{\ell}(\tilde{Q}_{2r}(z_0))} \big).
\end{align}
This is the first desired result in one-sided cylinders. To deduce the result in two-sided cylinders, we recall that we can write
\begin{align*}
Q_r(z_0) = \tilde{Q}_r(z_0) \cup \tilde{Q}_r(z_0 + (r^2/2 , r^2/2 v_0, 0)) \cup \tilde{Q}_r(z_0 + (r^2 , r^2 v_0, 0))
\end{align*}
as the sum of one-sided cylinders which have an overlap of order $r$. By similar transformations as above, we obtain \eqref{eq:one-sided-int} for $\tilde{Q}_r(z_0 + (r^2/2 , r^2/2 v_0, 0))$ and $\tilde{Q}_r(z_0 + (r^2 , r^2 v_0, 0))$ and deduce the first claim by summing up the estimates resulting from \eqref{eq:one-sided-int} and using the standard additivity properties of the kinetic H\"older (semi)-norms. The proof of the second claim goes in the exact same way, but holds for any $r > 0$, since \eqref{eq:Holder-norms-scaling} is always satisfied.
\end{proof}

For completeness, let us also mention the following interior estimate of De Giorgi--Nash--Moser type for kinetic equations in divergence form with bounded measurable coefficients. It was established in \cite{PaPo04,WaZh09,GIMV19,GuMo22}, and here we state a rescaled version.

Analogous interior regularity estimates for kinetic equations in non-divergence form are still unknown and establishing them is a major open problem in the field.

\begin{lemma}
\label{lemma:interior-reg-DGNM}
Let $z_0 \in \R^{1+2n}$,  and $r \in (0,1]$. Let $a^{i,j},b,c,h \in L^{\infty}(Q_{2r}(z_0))$ and assume that $a^{i,j}$ satisfies \eqref{eq:unif-ell}. Let $f$ be a weak solution to
\begin{align*}
\partial_t f + v \cdot \nabla_x f - \partial_{v_j}(a^{i,j}\partial_{v_i} f) = - b \cdot \nabla_v f - c f + h  ~~ \text{ in } Q_{2r}(z_0).
\end{align*}
Then, $f \in C^{\alpha}_{\ell}(Q_{r}(z_0))$ and
\begin{align}
\label{eq:interior-Calpha}
\begin{split}
[f]_{C^{\alpha}_{\ell}(Q_r(z_0))} &\le C r^{-\alpha}\big(\Vert f \Vert_{L^{\infty}(Q_{2r}(z_0))} + r^2 \Vert h \Vert_{L^{\infty}(Q_{2r}(z_0))} \big),\\
\Vert f \Vert_{L^{\infty}(Q_r(z_0))} &\le C \left( r^{-(2 + 4n)} \Vert f \Vert_{L^1(Q_{2r}(z_0))} + r^2 \Vert h \Vert_{L^{\infty}(Q_{2r}(z_0))} \right),
\end{split}
\end{align}
where $\alpha \in (0,1)$ and $C$ depend only on $n,\lambda$, $\Lambda$, $\Vert b \Vert_{L^{\infty}(Q_{2r}(z_0))}$, and $\Vert c \Vert_{L^{\infty}(Q_{2r}(z_0))}$.

Moreover, if $a^{i,j}$ is constant and $b=c=0$, then \eqref{eq:interior-Calpha} holds for any $r > 0$, where $C > 0$ depends only on $n,k,\eps,\lambda,\Lambda$.
\end{lemma}

\begin{proof}
The result was proved on scale one in \cite{PaPo04,WaZh09,GIMV19,GuMo22}.   By \autoref{lemma:scaling}, we deduce the results in general cylinders, following the scaling argument within the proof of \autoref{lemma:interior-reg}. 
\end{proof}

\subsection{Boundary H\"older estimates for kinetic equations}

In this subsection, we collect and establish $C^{\alpha}$ regularity estimates for solutions to kinetic equations at the boundary. For equations with bounded measurable coefficients that are in divergence form, such estimates have been established in \cite{Sil22,Zhu22}. The proof goes by reducing the boundary regularity to interior regularity results for kinetic equations of De Giorgi--Nash--Moser type, such as \autoref{lemma:interior-reg-DGNM}.

Note that the results in \cite{Sil22} are only stated on scale one and that a mere scaling argument as in the proof of \autoref{lemma:interior-reg} does not directly imply the result on general scales. This is because the kinetic scaling leads to a shift in the $x$ variable, which makes the domains time-dependent (see \autoref{lemma:scaling}).
Hence, we give a short explanation of how to modify the proof in order to get the result on scale $r$.

\begin{lemma}[Boundary in-flow estimate]
\label{lemma:boundary-reg}
Let $\Omega \subset \R^n$ be a convex $C^{1,1}$ domain, and $z_0 \in \R^{1+2n}$ with $x \in \overline{\Omega}$, $r \in (0,1]$, and $a^{i,j}\in C^{0,1}_{\ell}(H_{2r}(z_0))$, $b,c,h \in L^{\infty}(H_{2r}(z_0))$, and assume that $a^{i,j}$ satisfies \eqref{eq:unif-ell}. Let $f$ be a weak solution to
\begin{align*}
\partial_t f + v \cdot \nabla_x f + (- a^{i,j}\partial_{v_i,v_j}) f = - b \cdot \nabla_v f - c f + h  ~~ \text{ in } H_{2r}(z_0).
\end{align*}
Then, $f \in C^{\alpha}_{\ell}(H_{r}(z_0))$ and
\begin{align}
\label{eq:boundary-reg}
\begin{split}
[ f ]_{C^{\alpha}_{\ell}(H_r(z_0))} &\le C r^{-\alpha} \big( \Vert f \Vert_{L^{\infty}(H_{2r}(z_0))} + r^{\alpha} [ f ]_{C^{\alpha}_{\ell}(\gamma_- \cap Q_{2r}(z_0))} + r^2 \Vert h \Vert_{L^{\infty}(H_{2r}(z_0))} \big),\\
\Vert f \Vert_{L^{\infty}(H_r(z_0))} &\le C \left( r^{-(2+4n)}  \Vert f \Vert_{L^{1}(H_{2r}(z_0))} +  \Vert f \Vert_{L^{\infty}(\gamma_- \cap Q_{2r}(z_0))} + r^2 \Vert h \Vert_{L^{\infty}(H_{2r}(z_0))} \right),
\end{split}
\end{align}
where $\alpha \in (0,1)$ and $C$ depend only on $n,\lambda,\Lambda$,  $\Vert b \Vert_{L^{\infty}(H_{2r}(z_0))}$, and $\Vert c \Vert_{L^{\infty}(H_{2r}(z_0))}$, and $\Omega$.

Moreover, if $a^{i,j}$ is constant and $b=c=0$, then \eqref{eq:boundary-reg} holds for any $r > 0$, where $\alpha \in (0,1)$ and $C$ depend only on $n,\lambda,\Lambda, \Omega$.
\end{lemma}

Note that the constants in \cite{Sil22} might depend on $z_0$ if $\Omega$ is a general smooth domain. Since we need estimates that are uniform in $z_0$, we only state the result in convex domains.

\begin{proof}
First, we observe that by the regularity assumptions on $a$, we can write the equation for $f$ as an equation in divergence form, due to the regularity assumption on $a$. In fact, it holds
\begin{align*}
\partial_t f + v \cdot \nabla_x f - \partial_{v_i} (a^{i,j} \partial_{v_j} f) &= - (b+d) \cdot \nabla_v f - c f + h  ~~ \text{ in } H_{2r}(z_0),
\end{align*}
where $d_j = \sum_i \partial_{v_i} a^{i,j}$, and we can assume that $f$ solves  this equation in the weak sense.
We prove the desired result for solutions to kinetic equations in divergence form. 

The proof goes in the same way as in \cite{Sil22} if $r \not= 1$ and in two-sided cylinders by making straightforward adaptations in \cite[Sections 6.8]{Sil22}. Let us first explain how to prove the result in one-sided cylinders $\tilde{Q}_r(z_0)$.

First, we observe that the interior growth lemma in \cite[Lemma 6.2]{Sil22} also holds with $\tilde{Q}_1(z_0)$ replaced by $Q_r(z_0)$, $\tilde{Q}^{-}(z_0)$ replaced by $\tilde{Q}^{-}_r(z_0)$, where we define
\begin{align*}
\tilde{Q}^{-}_r := ( -3r^2/4 , - r^2/2 ] \times B_{(r/2)^3} \times B_{r/2} = \tilde{Q}_{r/2}(-r^2/2,0,0), \qquad \tilde{Q}^{-}_r(z_0) = z_0 \circ \tilde{Q}^{-}_r,
\end{align*}
and the second and third assumption becoming
\begin{align*}
|\{f \le 0 \} \cap \tilde{Q}^-_r(z_0)| \ge \mu |\tilde{Q}^-_r(z_0)|, \qquad r^2 \Vert G \Vert_{L^{\infty}(\tilde{Q}_r(z_0))} \le \eps_0.
\end{align*} 
This follows immediately by a scaling argument similar to the proof of \autoref{lemma:interior-reg}.

From here, \cite[Lemma 8.1]{Sil22} holds true with $\tilde{H}_1(z_0)$ replaced by $\tilde{H}_r(z_0)$ and the bound $\Vert G \Vert_{L^{\infty}} \le \eps_0$ replaced by $r^2 \Vert G \Vert_{L^{\infty}} \le \eps_0$. The proof goes in the exact same way. Consequently, \cite[Corollary~8.2]{Sil22}, which treats the boundary regularity away from $\gamma_-$, also holds true on scale $r$.\\
Next, we observe that by a straightforward adaptation of the arguments of the proof of \cite[Lemma~8.3]{Sil22}, a similar result also holds on scale $r$. Indeed, since $\Omega$ is convex and $\gamma_- \cap \tilde{Q}_{r/8}(z_0) \not= \emptyset$, then there is a $\mu > 0$ depending only on $n$ such that
\begin{align*}
|\tilde{Q}^{-}_r(z_0) \cap \{ x \not \in \Omega \}| > \mu |\tilde{Q}^{-}_r(z_0)|.
\end{align*}
Hence, also \cite[Lemma 8.4]{Sil22} holds with $\tilde{H}_1(z_0)$ replaced by $\tilde{H}_r(z_0)$ and the assumption $\Vert G \Vert_{L^{\infty}} \le \eps_0$ replaced by $r^2 \Vert G \Vert_{L^{\infty}} \le \eps_0$. Finally, by combination of the aforementioned lemmas, one easily adapts the proof of \cite[Theorem 1.2]{Sil22} to scale $r$. First, we divide the function $f$ by $\Vert f \Vert_{L^{\infty}(\tilde{H}_r(z_0))}/\eps_0 + [ f ]_{C^{\alpha}(\gamma_- \cap \tilde{Q}_r(z_0))} / (\eps_0 r^{\alpha}) + \Vert h \Vert_{L^{\infty}(\tilde{H}_r(z_0))}/(\eps_0 r^2)$ so that we can assume
\begin{align*}
\Vert f \Vert_{L^{\infty}(\tilde{H}_r(z_0))} \le \eps_0, \quad [ f ]_{C^{\alpha}(\gamma_- \cap \tilde{Q}_r(z_0))} \le \eps_0 r^{-\alpha}, \quad \Vert h \Vert_{L^{\infty}(\tilde{H}_r(z_0))} \le \eps_0 r^{-2}.
\end{align*}
We have
\begin{align*}
\osc_{\tilde{H}_{r/2}(z_0)} f \le C_0 := 2 \Vert f \Vert_{L^{\infty}(\tilde{H}_{r}(z_0))} \le 2 \eps_0,
\end{align*}
and as in \cite{Sil22}, we take $\alpha > 0$ such that $2^{-\alpha} = (1 - \theta/2)$ and claim that as long as $\gamma_- \cap \tilde{Q}_{2^{-k-3}r}(z_0) \not=\emptyset$, it holds
\begin{align*}
\osc_{\tilde{H}_{2^{-k}r}(z_0)} f \le C_0 (1- \theta/2)^{k-1}.
\end{align*}
The proof goes by induction in the exact same way as in \cite{Sil22}, using \cite[Lemma 8.4]{Sil22} and the fact that $m_0 = \inf_{\gamma_- \cap \tilde{Q}_{2^{-k}r}(z_0)} f$ and $m_0 = \sup_{\gamma_- \cap \tilde{Q}_{2^{-k}r}(z_0)} f$ satisfy 
\begin{align*}
0 \le m_1 - m_0 \le \osc_{\gamma_- \cap \tilde{Q}_{2^{-k}r}(z_0)} f \le (2^{-k}r)^{\alpha} [f]_{C^{\alpha}(\gamma_- \cap \tilde{Q}_{r}(z_0))} \le \eps_0 2^{-k\alpha}.
\end{align*}
Once there is $k = k_0$ so that $\gamma_- \cap \tilde{Q}_{2^{-k_0 - 3}} = \emptyset$, we iterate the rescaled version of \cite[Lemma 8.1]{Sil22} and conclude the proof in the same way as in \cite{Sil22}. This proves
\begin{align*}
[ f ]_{C^{\alpha}_{\ell}(\tilde{H}_r(z_0))} \le C r^{-\alpha} \big(\Vert f \Vert_{L^{\infty}(\tilde{H}_{2r}(z_0))} + r^{\alpha} [ f ]_{C^{\alpha}_{\ell}(\gamma_- \cap \tilde{Q}_{2r}(z_0))} + r^2 \Vert h \Vert_{L^{\infty}(\tilde{H}_{2r}(z_0))} \big),
\end{align*}
whenever $f$ solves the equation in $\tilde{H}_{2r}(z_0)$ with $x_0 \in \overline{\Omega}$. The proof of the local boundedness estimate at the boundary goes in the exact same way as in \cite[Theorem 1.1]{Sil22}. The result in two-sided cylinders follows by using the estimate in one-sided cylinders together with a covering argument. Note that it is a little more involved than in the proof of the interior estimates (see \autoref{lemma:interior-reg}) since we need to make sure that all the centers $z_0$ satisfy $x_0 \in \overline{\Omega}$.
\end{proof}

Moreover, in \cite{Sil22} it was proved that the $C^{\alpha}_{\ell}$ regularity holds true at the boundary when a specular reflection boundary conditions is imposed. Note that the constants in \cite{Sil22} might depend on $z_0$ if $\Omega$ is a general smooth domain. Since we require estimates that are uniform in $z_0$, we only state the result in the half-space.

\begin{lemma}[Boundary specular reflection estimate]
\label{lemma:boundary-reg-specular}
Let $z_0 \in \R^{1+2n}$, $\Omega = \{ x \cdot e > 0 \}$ for some $e \in \mathbb{S}^{n-1}$, $x_0 \in \{ x \cdot e = 0 \}$, $r \in (0,1]$, and $a^{i,j}\in C^{0,1}_{\ell}(\mathcal{S}(H_{2r}(z_0)))$, $b,c,h \in L^{\infty}(\mathcal{S}(H_{2r}(z_0)))$, and assume that $a^{i,j}$ satisfies \eqref{eq:unif-ell}. Let $f$ be a weak solution to
\begin{equation*}
\left\{\begin{array}{rcl}
\partial_t f + v \cdot \nabla_x f + (- a^{i,j}\partial_{v_i,v_j}) f &=& - b \cdot \nabla_v f - c f + h  ~~  \text{ in }  \mathcal{S}(H_{2r}(z_0)),\\
f(t,x, v) &=& f(t,x, \mathcal{R}_x v) ~~ \qquad\quad \text{ on } \gamma_- \cap \mathcal{S}(H_{2r}(z_0) ).
\end{array}\right.
\end{equation*}
Then, $f \in C^{\alpha}_{\ell}(H_{r}(z_0))$ and
\begin{align}
\label{eq:boundary-reg-specular}
[ f ]_{C^{\alpha}_{\ell}(H_r(z_0))} &\le C r^{-\alpha} \big(\Vert f \Vert_{L^{\infty}(\mathcal{S}(H_{2r}(z_0)))} + r^2 \Vert h \Vert_{L^{\infty}(\mathcal{S}(H_{2r}(z_0)))} \big),\\
\label{eq:boundary-reg-specular-Linfty-L1}
\Vert f \Vert_{L^{\infty}(H_r(z_0))} &\le C \left( r^{-(2+4n)}  \Vert f \Vert_{L^{1}(\mathcal{S}(H_{2r}(z_0)))} + r^2 \Vert h \Vert_{L^{\infty}(\mathcal{S}(H_{2r}(z_0)))} \right),
\end{align}
where $\alpha \in (0,1)$ and $C$ depend only on $n, \lambda$, $\Lambda$, $\Vert b \Vert_{L^{\infty}( \mathcal{S}(H_{2r}(z_0)))}$, and $\Vert c \Vert_{L^{\infty}(\mathcal{S}(H_{2r}(z_0)))}$.

Moreover, if $a^{i,j}$ is constant and $b=c=0$, then \eqref{eq:boundary-reg-specular} holds for any $r > 0$, where $\alpha \in (0,1)$ and $C$ depend only on $n,\lambda,\Lambda$.
\end{lemma}

\begin{proof}
First, we observe that we can write the equation for $f$ as an equation in divergence form, as in the proof of \autoref{lemma:boundary-reg}.
We consider the mirror extension $\bar{f}$ of $f$ and apply \autoref{lemma:mirror} to deduce that $\bar{f}$ is a weak solution to an equation with bounded measurable coefficients in $Q_{2r}(z_0)$. Hence, by the interior regularity estimates from \autoref{lemma:interior-reg-DGNM} 
%(or \autoref{lemma:interior-reg} in case $a^{i,j}$ is constant and $b=c=0$),
we obtain \eqref{eq:boundary-reg-specular} and \eqref{eq:boundary-reg-specular-Linfty-L1}, which concludes the proof. 
\end{proof}

\subsection{Replacing norms by seminorms}

In our proofs of the boundary regularity by blow-up, we need to replace the norm $\Vert h \Vert_{L^{\infty}(H_r(z_0))}$ by a H\"older seminorm of $h$ in order to match the scaling order of the corresponding regularity estimate. This is always possible due to the following lemma.

%For instance, instead of the estimate in \autoref{lemma:boundary-reg}
%\begin{align*}
%[ f ]_{C^{\alpha}(\tilde{H}_r(z_0))} \le C r^{-\alpha} \big(\Vert f \Vert_{L^{\infty}(\tilde{H}_{2r}(z_0))} + r^{\alpha} [ f ]_{C^{\alpha}(\gamma_- \cap \tilde{Q}_{2r}(z_0))} + r^2 \Vert h \Vert_{L^{\infty}(\tilde{H}_{2r}(z_0))} \big),
%\end{align*}
%we would like to have the following boundary Dirichlet estimate
%\begin{align*}
%[ f ]_{C^{\alpha}(\tilde{H}_r(z_0))} \le C r^{-\alpha} \big(\Vert f \Vert_{L^{\infty}(\tilde{H}_{2r}(z_0))} + r^{\alpha} [ f ]_{C^{\alpha}(\gamma_- \cap \tilde{Q}_{2r}(z_0))} + r^2 \osc_{\tilde{H}_{2r}(z_0)} h \big).
%\end{align*}
%Note that when applying this modified estimate we will often use that $\osc_{\tilde{H}_{2r}(z_0)} h \le r^{\gamma} [h]_{C^{\gamma}_{\ell}(\tilde{H}_{2r}(z_0))}$.
%To obtain the improved estimate, we simply need to combine the respective (boundary) regularity estimate with the following lemma:

\begin{lemma}
\label{lemma:osc-results}
Let $D \subset \R^{1+2n}$ be a domain and $z_0 \in \overline{D}$, $r > 0$, and $a^{i,j},b,c \in C^{\eps}_{\ell}(H_{r}(z_0))$ and $h \in C^{k+\eps}_{\ell}(H_{r}(z_0))$ for some $k \in \N \cup \{0\}$, $\eps \in (0,1)$, and assume that $a^{i,j}$ satisfies \eqref{eq:unif-ell}. Let $f$ be a  weak solution to
\begin{align*}
\partial_t f + v \cdot \nabla_x f + (- a^{i,j}\partial_{v_i,v_j}) f = - b \cdot \nabla_v f - c f + h  ~~ \text{ in } H_{r}(z_0).
\end{align*}
Then, it holds
\begin{align*}
\Vert h \Vert_{L^{\infty}(H_r(z_0))} &\le C \Vert f \Vert_{L^{\infty}(H_r(z_0))} + C \inf_{p \in \cP_k} \left(\Vert h - p \Vert_{C^{\eps}_{\ell}(H_r(z_0))} + r^{\eps} [ h - p ]_{C^{\eps}_{\ell}(H_r(z_0))}  \right) \\
&\le C \big( \Vert f \Vert_{L^{\infty}(H_r(z_0))} + r^{k+\eps} [ h]_{C^{k+\eps}_{\ell}(H_r(z_0))} \big),
\end{align*}
where $C$ depends only on $n,\lambda,\Lambda, k, D$, $\Vert a^{i,j} \Vert_{C^{\eps}_{\ell}(H_r(z_0))}, \Vert b \Vert_{C^{\eps}_{\ell}(H_r(z_0))}$, and $\Vert c \Vert_{C^{\eps}_{\ell}(H_r(z_0))}$.
\end{lemma}

Note that the $C^{\eps}_{\ell}$ norm can also be replaced by the $L^{\infty}$ norm. We do not prove this more generalized result, since it would require a more refined stability result for kinetic equations.

\begin{proof}
Note that we can assume without loss of generality that $z_0 = 0$ and $r = 1$ by applying the transformation $z_0 \circ S_r \cdot$ to the equation for $f$ and using \autoref{lemma:Holder-scaling}.
Let us fix $p \in \cP_k$ and write $h = (h - p) + p =: g + p$. Note that the claim follows, once we show that
\begin{align}
\label{eq:p-ass}
\Vert p\Vert_{L^{\infty}(H_1)} \le C \big( \Vert f \Vert_{L^{\infty}(H_1)} + \Vert g \Vert_{C^{\eps}_{\ell}(H_1)} \big).
\end{align}
Indeed, since $h \in C_{\ell}^{k+\eps}(H_r(z_0))$, it holds  by the definition of the kinetic H\"older spaces
\begin{align*}
\inf_{p \in \cP_k} \Vert h - p \Vert_{C^{\eps}_{\ell}(H_1(z_0))} \le c [ h ]_{C_{\ell}^{k+\eps}(H_1(z_0))}.
\end{align*}
Therefore, the estimate \eqref{eq:p-ass} implies
\begin{align*}
\Vert h \Vert_{L^{\infty}(H_1(z_0))} \le C \big( \Vert f \Vert_{L^{\infty}(H_1(z_0))} + \inf_{p \in \cP_k} \Vert h - p \Vert_{C^{\eps}_{\ell}(H_1(z_0))} \big) \le  C \big( \Vert f \Vert_{L^{\infty}(H_1(z_0))} + [ h ]_{C_{\ell}^{k+\eps}(H_1(z_0))} \big),
\end{align*}
as desired.

We will prove \eqref{eq:p-ass} by contradiction. Assume that it does not hold true, i.e. that there exist sequences of solutions $(f_l)$ to
\begin{align*}
\partial_t f_l + v \cdot \nabla_x f_l + (- a_l^{i,j}\partial_{v_i,v_j}) f_l + b_l \cdot \nabla_v f_l + c_l f_l + h_l = g_l + p_l ~~ \text{ in }  H_1,
\end{align*}
where $(g_l) \subset C^{\eps}_{\ell}(H_1)$, $(p_l) \subset \cP_k$, and $a^{i,j}$ satisfies \eqref{eq:unif-ell} and 
\begin{align*}
\Vert a^{i,j} \Vert_{C^{\eps}_{\ell}(H_r(z_0))} + \Vert b^{i} \Vert_{C^{\eps}_{\ell}(H_r(z_0))} + \Vert c \Vert_{C^{\eps}_{\ell}(H_r(z_0))} \le \Lambda,
\end{align*}
but \eqref{eq:p-ass} is violated. Note that by taking $p_l/|p_l|$ instead of $p_l$ and normalizing $f_l$ and $g_l$ accordingly, we can assume that
\begin{align*}
\Vert f_l \Vert_{L^{\infty}(H_1)} + \Vert g_l \Vert_{C^{\eps}_{\ell}(H_1)} \to 0, \qquad \Vert p_l \Vert_{L^{\infty}(H_1)} = 1.
\end{align*}
In particular, by application of the interior regularity estimates in \autoref{lemma:interior-reg} we deduce that for any $\rho > 0$ and $z \in D$ with $Q_{2\rho}(z) \subset H_1$, it holds 
\begin{align*}
\Vert f_l \Vert_{C^{2+\eps}_{\ell}(Q_{\rho}(z))} \le C(\rho,z) (\Vert f_l \Vert_{L^{\infty}(H_1)} + \Vert g_l \Vert_{C^{\eps}_{\ell}(H_1)} + 1 ) \le C(\rho,z)
\end{align*} 
for a constant $C(\rho,z) > 0$ independent of $l$, and therefore by the Arzel\`a-Ascoli theorem we have $f_l \to f_0$ locally uniformly (and in $C^{\alpha}_{\ell}$) in $H_1$, where $f_0 \equiv 0$. Moreover, we have $a_l^{i,j} \to a^{i,j}$, $b_l \to b$, $c_l \to c$ by assumption and again by the Arzel\`a-Ascoli theorem, which implies that also
\begin{align*}
\partial_t f_l + v \cdot \nabla_x f_l + (- a_l^{i,j}\partial_{v_i,v_j}) f_l + b_l \cdot \nabla_v f_l + c_l f_l \to \partial_t f_0 + v \cdot \nabla_x f_0 + (- a^{i,j}\partial_{v_i,v_j}) f_0 + b \cdot \nabla_v f_0 + c f_0
\end{align*}
locally uniformly. Also, it holds $p_l \to p_0 \in \cP_k$ with $\Vert p_0 \Vert_{L^{\infty}(H_1)} = 1$. Altogether, we know
\begin{align*}
\partial_t f_0 + v \cdot \nabla_x f_0 + (- a^{i,j}\partial_{v_i,v_j}) f_0 + b \cdot \nabla_v f_0 + c f_0 = p_0  ~~ \text{ in } H_1.
\end{align*}
However, this contradicts the fact that $f_0 \equiv 0$. Hence, \eqref{eq:p-ass} follows, and the proof is complete.
\end{proof}

The following lemma is similar to the previous one but it deals with boundary data instead of source terms.

\begin{lemma}
\label{lemma:osc-bdry-data}
Let $\Omega \subset \R^{n}$ be a convex $C^{k,\eps}$ domain and $z_0 \in (-1,1) \times \overline{\Omega} \times \R^n$, $r > 0$, and $a^{i,j} \in C^{0,1}_{\ell}(H_r(z_0))$, $b,c,h \in C^{\eps}_{\ell}(H_{r}(z_0))$ and $F \in C^{k+\eps}_{\ell}(H_{r}(z_0) \cap \gamma_-)$ for some $\eps > 0$ and $k \in \N \cup \{ 0 \}$, and assume that $a^{i,j}$ satisfies \eqref{eq:unif-ell}. Let $f$ be a weak solution to
\begin{equation*}
\left\{\begin{array}{rcl}
\partial_t f + v \cdot \nabla_x f + (- a^{i,j}\partial_{v_i,v_j}) f &=& - b \cdot \nabla_v f - c f + h  ~~  \text{ in } H_{r}(z_0),\\
f &=& F ~~ \quad\qquad\qquad\qquad \text{ on } H_{r}(z_0) \cap \gamma_-.
\end{array}\right.
\end{equation*}
Then, it holds
\begin{align*}
\Vert F \Vert_{L^{\infty}(H_r(z_0) \cap \gamma_-)} \le C \big( \Vert f \Vert_{L^{\infty}(H_r(z_0))} + \Vert h \Vert_{C^{\eps}_{\ell}(H_r(z_0))} + r^{k+\eps} [ F ]_{C^{k+\eps}_{\ell}(H_r(z_0) \cap \gamma_-)} \big),
\end{align*}
where $C$ depends only on $n,\lambda,\Lambda, k$, $\Vert a^{i,j} \Vert_{C^{0,1}_{\ell}(H_r(z_0))}, \Vert b \Vert_{C^{\eps}_{\ell}(H_r(z_0))}$, and $\Vert c \Vert_{C^{\eps}_{\ell}(H_r(z_0))}$.
\end{lemma}

\begin{proof}
Since the proof goes in the same way as the proof of \autoref{lemma:osc-results}, we only give a brief sketch. Again, we can assume without loss of generality that $z_0 = 1$ and $r = 1$. We fix $p \in \cP_k$ and write $F = (F-p) + p := G + p$ and observe that it suffices to prove
\begin{align}
\label{eq:p-ass-bdry}
\Vert p\Vert_{L^{\infty}(H_1 \cap \gamma_-)} \le C \big( \Vert f \Vert_{L^{\infty}(H_1)} +  \Vert G \Vert_{C^{\eps}_{\ell}(H_1 \cap \gamma_-)} + \Vert h \Vert_{L^{\infty}(H_1)}  \big).
\end{align}
By contradiction, we assume that there exist solutions $(f_l)$ with boundary data $G_l + p_l$, as well as coefficients $(a_l^{i,j})$, $b_l,c_l$, $h_l$ such that \eqref{eq:p-ass-bdry} fails. By a normalization, we can assume
\begin{align*}
\Vert f_l \Vert_{L^{\infty}(H_1)} + \Vert G_l \Vert_{C^{\eps}_{\ell}(H_1 \cap \gamma_-)} +  \Vert h_l \Vert_{C^{\eps}(H_1)} \to 0, \qquad \Vert p_l \Vert_{L^{\infty}(H_1 \cap \gamma_-)} = 1.
\end{align*}
Then, by the boundary regularity estimate in \autoref{lemma:boundary-reg}, we deduce that $f_l$ is bounded in $C_{\ell}^{\alpha}(H_{\rho})$ for any $\rho \in (0,1)$, which implies that $f_l \to f_0 = 0$  uniformly in $H_{\rho}$ for any $\rho \in (0,1)$. Moreover, note that the space of all polynomials restricted to $H_1 \cap \gamma_-$ is finite dimensional, and therefore we must have $p_l \to p_0$ with $\Vert p_0 \Vert_{L^{\infty}(H_1 \cap \gamma_-)} = 1$. Moreover, it must hold $f_0 = p_0$ on $H_1 \cap \gamma_-$, but since $f_0 = 0$ and $|p_0| = 1$, we obtain a contradiction. Therefore \eqref{eq:p-ass-bdry} must hold true and the proof is complete.
\end{proof}

\section{Kinetic Liouville theorems}
\label{sec:Liouville}

In this section we establish various Liouville theorems for kinetic equations in the full space and in the half-space with various boundary conditions.

\subsection{Liouville theorems in the full space}

The main results of this subsection are \autoref{prop:Liouville-higher-neg-times} and \autoref{prop:Liouville-higher-pos-times}. These results are Liouville theorems of any order for solutions in the full space, but either for negative or for positive times. They will be used in the proof of boundary regularity near points $z_0 \in \gamma_{\pm}$.

\begin{proposition}
\label{prop:Liouville-higher-neg-times}
Let $k \in \N \cup \{ 0 \}$, $p \in \cP_{k-2}$, $\eps \in (0,1)$, $(a^{i,j})$ be a constant matrix satisfying \eqref{eq:unif-ell}, and $f$ be a solution to
\begin{align*}
\partial_t f + v \cdot \nabla_x f + (-a^{i,j} \partial_{v_i,v_j}) f &= p ~~ \text{ in } (-\infty,0) \times \R^n \times \R^n
\end{align*}
with $\Vert f \Vert_{L^{\infty}(\tilde{Q}_R(0))} \le C (1 + R)^{k+\eps}$ for any $R > 0$. Then $f \in \cP_{k}$.
\end{proposition}

%From now on, whenever $f$ solves an equation of the aforementioned form with a polynomial of kinetic degree $k-1$ on the right-hand side, we write
%\begin{align*}
%\partial_t f + v \cdot \nabla_x f + (-\Delta_v) f =^{k} 0.
%\end{align*}

In the proof of the kinetic Liouville theorem, we will use the following Liouville theorem for parabolic equations (see \cite[Proposition A.1]{Kuk22} for a similar result), where we treat parabolic polynomials in the same way as kinetic polynomials that are constant in $x$.

\begin{lemma}
\label{lemma:parabolic-Liouville}
Let $k \in \N \cup \{ 0 \}$, $p \in \cP_{k-2}$, $\eps \in (0,1)$, $(a^{i,j})$ be a constant matrix satisfying \eqref{eq:unif-ell}, and $f$ be a solution to
\begin{align*}
\partial_t f + (-a^{i,j} \partial_{v_i,v_j}) f &= p ~~ \text{ in } (-\infty,0) \times \R^n
\end{align*}
with $\Vert f \Vert_{L^{\infty}(\tilde{Q}_R(0))} \le C (1 + R)^{k+\eps}$ for any $R > 0$. Then $f \in \cP_{k}$.
\end{lemma}

We give a sketch of the proof of \autoref{lemma:parabolic-Liouville} below.
We are now in a position to prove the kinetic Liouville theorem.

\begin{proof}[Proof of \autoref{prop:Liouville-higher-neg-times}]
We split the proof into two steps.

\textbf{Step 1:} Note that for any $y \in \R^n$ we have that
\begin{align*}
g_1(t,x,v) = f(t,x+y,v) - f(t,x,v)
\end{align*}
is a solution to 
\begin{align*}
\partial_t g_1 + v \cdot \nabla_x g_1 + (-a^{i,j} \partial_{v_i,v_j}) g_1 = h_1 ~~ \text{ in } (-\infty,0) \times \R^n \times \R^n,
\end{align*}
where $h_1(t,x,v) = p(t,x+y,v) - p(t,x,v).$
This follows from the invariance of the equation under the operation $\circ$ and the observation that $(t,x+y,v) = (0,y,0) \circ (t,x,v)$. Hence, applying the interior regularity from \autoref{lemma:interior-reg} (namely \eqref{eq:one-sided-int}) to $f$ and using the growth condition on $f$, 
we get that for any $R > 1$:
\begin{align}
\label{eq:g1-growth}
\begin{split}
\Vert g_1 \Vert_{L^{\infty}(\tilde{Q}_R(0))} &\le |y|^{\alpha/3}[f]_{C^{\alpha}_{\ell}(\tilde{Q}_R(0))} \\
&\le c |y|^{\alpha/3} R^{-\alpha} (\Vert f \Vert_{L^{\infty}(\tilde{Q}_{2R}(0))} + R^2 \Vert h \Vert_{L^{\infty}(\tilde{Q}_{2R}(0))} ) \le c |y|^{\alpha/3} R^{k+\eps-\alpha}.
\end{split}
\end{align}
Note that here, we can choose $\alpha \in (0,1]$ arbitrarily, however the precise choice is not important to us. Next, we define 
\begin{align*}
g_2(t,x,v) = g_1(t,x+y,v) - g_1(t,x,v)
\end{align*}
and apply the interior regularity from \autoref{lemma:interior-reg}  to $g_1$, using \eqref{eq:g1-growth}, and the fact that $\Vert h_1 \Vert_{L^{\infty}(\tilde{Q}_R(0))} \le C R^{k-3,}$, to obtain
\begin{align*}
\Vert g_2 \Vert_{L^{\infty}(\tilde{Q}_R(0))} &\le |y|^{\alpha/3}[g_1]_{C^{\alpha}_{\ell}(\tilde{Q}_R(0))} \\
&\le c |y|^{\alpha/3} R^{-\alpha} (\Vert g_1 \Vert_{L^{\infty}(\tilde{Q}_{2R}(0))} + R^2 \Vert h_1 \Vert_{L^{\infty}(\tilde{Q}_{2R}(0))} ) \le c^2 |y|^{2\alpha/3} R^{k+\eps-2\alpha}.
\end{align*}
Iterating this procedure, i.e. taking
\begin{align*}
g_m(t,x,v) = g_{m-1}(t,x+y,v) - g_{m-1}(t,x,v), \qquad h_m(t,x,v) = h_{m-1}(t,x+y,v) - h_{m-1}(t,x,v),
\end{align*}
and observing that $\Vert h_m \Vert_{L^{\infty}(\tilde{Q}_R(0))} \le C R^{k-2-m}$, we obtain
\begin{align*}
\Vert g_m \Vert_{L^{\infty}(\tilde{Q}_R(0))} \le c^m |y|^{m\alpha/3} R^{k+\eps - m \alpha}.
\end{align*}
Once $m \in \N$ is so big that $k+\eps-m\alpha < 0$, we deduce that $g_m \equiv 0$ by taking $R \to \infty$ and therefore $g_{m-1}$ is constant in the direction of $y$. Since  we can also define $g_m(t,x,v) = g_{m-1}(t,x+y',v) - g_{m-1}(t,x,v)$ for any $y'\in \R^n$, we obtain by the same argument that $g_{m-1}$ is a function depending only on $t,v$. Unraveling the higher order increments by a similar procedure, we deduce that $f$ must be a polynomial in $x$ whose coefficients are functions in $t,v$. Hence, by the growth assumption on $f$, it must be
\begin{align}
\label{eq:poly-sum}
f(t,x,v) = \sum_{\substack{\beta = (0,\beta_x,0),\\ |\beta| \le k}} f_{\beta}(t,v) x_1^{\beta_{x_1}} \cdot \dots \cdot x_n^{\beta_{x_n}},
\end{align}
where $f_{\beta} : (-\infty,0) \times \R^n \to \R$ satisfy the growth condition
\begin{align}
\label{eq:growth-f-beta}
\Vert f_{\beta} \Vert_{L^{\infty}(\tilde{Q}_R(0))} \le C (1 + R)^{k + \eps - |\beta|} ~~ \forall R > 0.
\end{align}

Note that the desired result follows once we show that all the $f_{\beta}$ are polynomials of kinetic degree $k - |\beta|$, since this implies $f \in \cP_{k}$ by the formula \eqref{eq:poly-sum}.

\textbf{Step 2:} We will prove that $f_{\beta} \in \cP_{k - |\beta|}$ for any $\beta = (0,\beta_x,0)$ with $|\beta| \le k$ by induction over the degree $|\beta|$, going from the highest value $|\beta| = k$ to the lowest $|\beta| = 0$.

Let us first take any $\beta = (0,\beta_x,0)$ with $|\beta| = k$. Then, we observe that by \eqref{eq:poly-sum} it holds $f_{\beta} = c_{\beta} \partial_{\beta}f$ for some $c_{\beta} \in \R$. Hence we obtain by differentiating the kinetic equation for $f$ in $x$ with respect to $\partial_{\beta}$:
\begin{align*}
\partial_t f_{\beta} + (-a^{i,j} \partial_{v_i,v_j}) f_{\beta} = 0 ~~ \text{ in } (-\infty,0) \times \R^n.
\end{align*} 
Here we dropped the term $v \cdot \nabla_x f_{\beta}$ since $f_{\beta}$ does not depend on $x$ and used that $\partial_{\beta} p = 0$ since $p \in \cP_{k-2}$.
Moreover, by \eqref{eq:growth-f-beta} we have $\Vert f_{\beta} \Vert_{L^{\infty}(\tilde{Q}_R(0))} \le C (1 + R)^{\eps}$. Hence, we can apply \autoref{lemma:parabolic-Liouville} to deduce that $f_{\beta} \equiv c \in \cP_0 = \cP_{k - (k-1) - 1}$, as desired.

Now, let us assume that for some $l \in \{0,\dots,k\}$ we have already shown that $f_{\beta} \in \cP_{k - |\beta|}$, whenever $\beta = (0,\beta_x,0)$ with $l < |\beta| \le k$. It remains to prove that $f_{\beta} \in \cP_{k -l}$ whenever $\beta = (0,\beta_x,0)$ with $|\beta| = l$. 
Let us take $\beta = (0,\beta_x,0)$ such that $|\beta| = l$. First, we observe that
\begin{align*}
\partial_{\beta} f(t,x,v) = c_{\beta} f_{\beta}(t,v) + \sum_{\substack{\gamma = (0,\gamma_x,0), \\ |\gamma| \le k, \\ \gamma > \beta}} c_{\gamma,\beta} f_{\gamma}(t,v) x_1^{\gamma_{x_1} - \beta_{x_1}} \cdot \dots \cdot x_n^{\gamma_{x_n} - \beta_{x_n}},
\end{align*}
where $c_{\gamma,\beta}, c_{\beta} \in \R$. 
By the notation $\gamma > \beta$ we mean that $\gamma_{x_i} \ge \beta_{x_i}$ for every $i \in \{1,\dots,n\}$ and $\gamma \not=\beta$. \\
By the induction assumption we have that $f_{\gamma} \in \cP_{k - |\gamma|}$ for any $\gamma$ in the sum, and therefore
\begin{align*}
\partial_{\beta} f - c_{\beta} f_{\beta} \in \cP_{k - |\beta|}.
\end{align*}
It is easy to see that for any $P \in \cP_{k - |\beta|}$ it holds
\begin{align*}
\partial_t P + v \cdot \nabla_x P + (-a^{i,j} \partial_{v_i,v_j}) P \in \cP_{k - |\beta| - 2}.
\end{align*}
Therefore, differentiating the equation for $f$ with respect to $\partial_{\beta}$, we deduce
\begin{align*}
\partial_t f_{\beta} + (-a^{i,j} \partial_{v_i,v_j}) f_{\beta} = p_{\beta} ~~ \text{ in } (-\infty,0) \times \R^n
\end{align*}
for some $p_{\beta} \in \cP_{k-|\beta|-2}$, where we again dropped the term $v \cdot \nabla_x f_{\beta}$ since $f_{\beta}$ does not depend on $x$, and we used that $\partial_{\beta} p \in \cP_{k - |\beta| - 2}$ since $p \in \cP_{k-2}$. Recalling the growth assumption on $f_{\beta}$ from \eqref{eq:growth-f-beta}, we can apply the parabolic Liouville theorem with growth (see \autoref{lemma:parabolic-Liouville}) with $k := k-|\beta|$ to $f_{\beta}$ and deduce that $f_{\beta} \in \cP_{k - |\beta|}$, as desired. This concludes the proof.
\end{proof}

For the sake of completeness, let us provide a sketch of the proof of the parabolic Liouville theorem in \autoref{lemma:parabolic-Liouville}.

\begin{proof}[Proof of \autoref{lemma:parabolic-Liouville}]
The procedure is very similar to the one from the proof of \autoref{prop:Liouville-higher-neg-times}. Indeed, setting for $h \in \R^n$ and $\tau < 0$
\begin{align*}
g_1(t,v) = f(t+\tau,v+h) - f(t,v), \qquad g_k(t,v) = g_{k-1}(t+\tau,v+h) - g_{k-1}(t,v) ~~ k \in \N,
\end{align*}
applying the interior regularity for parabolic equations on large scales to $f,g_1,g_2,\dots$, and iterating this procedure in the same way as in Step 1 of the proof of \autoref{prop:Liouville-higher-neg-times}, we obtain that $g_m$ is constant for some $m \in \N$. Then, $f$ must be a polynomial, and the growth condition implies that $f \in \cP_k$.
\end{proof}

Next, we prove a kinetic Liouville theorem for solutions in the full space for positive times. The proof goes by extending the solution to all times and repeating the proof of the Liouville theorem for negative times. In order to extend the solutions, we require additional information at time zero.

\begin{proposition}
\label{prop:Liouville-higher-pos-times}
Let $k \in \N \cup \{ 0 \}$, $\eps \in (0,1)$, $(a^{i,j})$ be a constant matrix satisfying \eqref{eq:unif-ell}, $p \in \cP_{k-2}$, $\bar{p} \in \cP_k$, and $f$ with $f, \nabla_v f \in L^2_{loc}([0,\infty) \times \R^n \times \R^n)$ be a solution to
\begin{equation*}
\left\{\begin{array}{rcl}
\partial_t f + v \cdot \nabla_x f + (-a^{i,j} \partial_{v_i,v_j}) f &=& p ~~ \text{ in } (0,\infty) \times \R^n \times \R^n , \\
f(0,\cdot,\cdot) &=& \bar{p} ~~ \text{ in } \R^{n} \times \R^n
\end{array}\right.
\end{equation*}
with $\Vert f \Vert_{L^{\infty}(\tilde{Q}_R(R^2,0,0))} \le C (1 + R)^{k+\eps}$ for any $R > 0$. Then $f \in \cP_{k}$.
\end{proposition}

The initial condition $f(0,\cdot,\cdot) = \bar{p}$ in $\R^n \times \R^n$ has to be understood in the $L^2_{loc}$-sense, i.e.
\begin{align*}
\Vert f(t,\cdot,\cdot) - \bar{p} \Vert_{L^{2}(K)} \to 0 ~~ \text{ as } t \to 0 ~~ \forall K \subset \R^n \times \R^n ~~ \text{ compact}.
\end{align*}

%As for the proof of \autoref{prop:Liouville-higher-neg-times}, we require a parabolic Liouville theorem, however this time it suffices to have it for solutions in $\R \times \R^n$.
%
%\begin{lemma}
%\label{lemma:parabolic-Liouville-two-sided}
%Let $k \in \N \cup \{ 0 \}$, $p \in \cP_{k-2}$, $\eps \in (0,1)$, and $f$ be a solution to
%\begin{align*}
%\partial_t f + (-\Delta_v) f &= p ~~ \text{ in } \R \times \R^n
%\end{align*}
%with $\Vert f \Vert_{L^{\infty}(Q_R(0))} \le C (1 + R)^{k+\eps}$ for any $R > 0$. Then $f \in \cP_{k}$.
%\end{lemma}

We are now in a position to prove the kinetic Liouville theorem of higher order.

\begin{proof}
First, we extend $\bar{p} \in \cP_{k-2}$ to a polynomial of degree $k-2$ in $[0,\infty) \times \R^n \times \R^n$ and observe that 
\begin{align*}
(\partial_t + v \cdot \nabla_x + (-a^{i,j} \partial_{v_i,v_j})) \bar{p} =: \tilde{p} \in \cP_{k-2}.
\end{align*}
Then, we consider $f := f - \bar{p}$, which satisfies $f(0,\cdot,\cdot) = 0$ in $\R^n \times \R^n$, solves the equation with right-hand side $p:= p - \tilde{p} \in \cP_{k-2}$, and still satisfies the growth condition. Hence, we can assume without loss of generality that $\bar{p} = 0$.

Next, we extend $f$ by zero to negative times. This way, we obtain that $f$ is a weak solution to
\begin{align}
\label{eq:weak-sol-claim}
\partial_t f + v \cdot \nabla_x f + (-a^{i,j} \partial_{v_i,v_j}) f &= p ~~ \text{ in } \R \times \R^n \times \R^n,
\end{align}
where $p \equiv 0$ in $(-\infty,0) \times \R^n \times \R^n$. Indeed, since as a strong solution, $f$ is in particular a weak solution in $(0,\infty) \times \R^n \times \R^n$, and since $f(0,\cdot,\cdot) = 0$, we obtain that for any test-function $\phi \in C_c^{\infty}([0,\infty) \times \R^n \times \R^n)$, and for any $0 < \sigma < \tau$
\begin{align*}
\int_{\sigma}^{\tau} & \iint_{\R^n \times \R^n} (-f \partial_t \phi + (v \cdot \nabla_x \phi) f  + a^{i,j}\partial_{v_i} f \partial_{v_j} \phi) \d v \d x \d t \\
&= \iint_{\R^n \times \R^n} f(\sigma) \phi(\sigma) \d v \d x - \iint_{\R^n \times \R^n} f(\tau) \phi(\tau) \d v \d x + \int_{\sigma}^{\tau} \iint_{\R^n \times \R^n} p \phi \d v \d x \d t.
\end{align*}
By taking $\sigma \to 0$, $\tau \to \infty$, and using that $f(0) = 0$, we deduce
\begin{align*}
\int_0^{\infty} \iint_{\R^n \times \R^n} (-f \partial_t \phi + (v \cdot \nabla_x \phi) f  + a^{i,j}\partial_{v_i} f \partial_{v_j} \phi) \d v \d x \d t = \int_0^{\infty} \iint_{\R^n \times \R^n} p \phi \d v \d x \d t.
\end{align*}
Moreover, since $f \equiv p \equiv 0$ in $(-\infty,0] \times \R^n \times \R^n$, we have for any $\phi \in C_c^{\infty}(\R \times \R^n \times \R^n )$:
\begin{align*}
\int_{-\infty}^{0} \iint_{\R^n \times \R^n} (- f \partial_t \phi + (v \cdot \nabla_x \phi) f  + a^{i,j}\partial_{v_i} f \partial_{v_j} \phi) \d v \d x \d t = \int_{-\infty}^0 \iint_{\R^n \times \R^n} p \phi \d v \d x \d t. 
\end{align*}
Altogether, by summing up the previous two identities, we have for any $\phi \in C_c^{\infty}(\R \times \R^n \times \R^n)$
\begin{align*}
\int_{\R} \iint_{\R^n \times \R^n} (- f \partial_t \phi + (v \cdot \nabla_x \phi) f  + a^{i,j}\partial_{v_i} f \partial_{v_j} \phi) \d v \d x \d t = \int_{\R} \iint_{\R^n \times \R^n} p \phi \d v \d x \d t,
\end{align*}
and hence $f$ is a weak solution in $\R^{1+2n}$, as claimed in \eqref{eq:weak-sol-claim}.

Now we are in a position to start the actual proof of the Liouville theorem. By the same procedure as in Step 1 of the proof of \autoref{prop:Liouville-higher-neg-times}, namely by applying the interior regularity from \autoref{lemma:interior-reg-DGNM} in $Q_R(0)$, we deduce
\begin{align}
\label{eq:poly-sum-pos-times}
f(t,x,v) = \sum_{\substack{\beta = (0,\beta_x,0),\\ |\beta| \le k}} f_{\beta}(t,v) x_1^{\beta_{x_1}} \cdot \dots \cdot x_n^{\beta_{x_n}},
\end{align}
where $f_{\beta} : \R \times \R^n \to \R$ satisfy the growth condition
\begin{align}
\label{eq:growth-f-beta-pos-times}
\Vert f_{\beta} \Vert_{L^{\infty}(Q_R(0))} \le C (1 + R)^{k + \eps - |\beta|} ~~ \forall R > 0.
\end{align}
Note that since $f$ is only a weak solution, we had to apply \autoref{lemma:interior-reg-DGNM} instead of \autoref{lemma:interior-reg}.
Another difference compared to Step 1 in the proof of \autoref{prop:Liouville-higher-neg-times} is that $p$ is not a polynomial (in $t$) in our case. However, since it is a polynomial in $x$, it still holds $\Vert h_m \Vert_{L^{\infty}(Q_R(0))} \le C R^{k-2-m}$ for any $y \in \R^n$ (same notation as in the proof of \autoref{prop:Liouville-higher-neg-times}).

It remains to show that the $f_{\beta}$ are polynomials of kinetic degree $k - |\beta|$. We will first prove that they are polynomials in $v$ for any $t$. Similar to the procedure in Step 2 of the proof of \autoref{prop:Liouville-higher-neg-times}, we prove it by induction over $|\beta|$. For $\beta = (0,\beta_x,0)$ with $|\beta| = k$, we observe that $f_{\beta} = c_{\beta} D_{\beta} f$, where $D_{\beta}$ denotes any incremental quotient approximating the partial derivative $\partial_{\beta}$. Hence, it holds in the weak sense
\begin{align}
\label{eq:f-beta-PDE-2}
\partial_t f_{\beta} + (-a^{i,j} \partial_{v_i,v_j}) f_{\beta} = c_{\beta} D_{\beta} p ~~ \text{ in } \R \times \R^n,
\end{align}
where $c_{\beta} D_{\beta} p \in \cP_{k-2-|\beta|}$ in $(0,\infty) \times \R^n$ and $c_{\beta} D_{\beta} p \equiv 0$ in $(-\infty,0) \times \R^n$. Proceeding again similar to Step 1 in the proof of \autoref{prop:Liouville-higher-neg-times}, i.e. by applying interior (parabolic) regularity estimates to increments in the $v$-variable of $D_{\beta} f$, and using that increments of $c_{\beta} D_{\beta} p$ in the $v$-variable have the correct growth, we deduce that $f_{\beta}$ is of the form
\begin{align*}
f_{\beta}(t,v) = \sum_{\substack{\gamma = (0,0,\gamma_v)\\ |\gamma| \le k - |\beta|}} f_{\beta,\gamma}(t) v_1^{\gamma_{v_1}} \cdot \dots \cdot v_n^{\gamma_{v_n}}.
\end{align*}
Using this information in the same way as in Step 2 of the proof of \autoref{prop:Liouville-higher-neg-times}, we can show inductively that  any $f_{\beta}$ satisfies \eqref{eq:f-beta-PDE-2} and hence is of the aforementioned form for any $\beta = (0,\beta_x,0)$ with $|\beta| \le k$. Therefore, using also \eqref{eq:growth-f-beta-pos-times}
\begin{align}
\label{eq:poly-sum-pos-times-2}
f(t,x,v) = \sum_{\substack{\beta = (0,\beta_x,\beta_v),\\ |\beta| \le k}} f^{(\beta)}(t) x_1^{\beta_{x_1}} \cdot \dots \cdot x_n^{\beta_{x_n}}v_1^{\beta_{v_1}} \cdot \dots \cdot v_n^{\beta_{v_n}},
\end{align}
where $f^{(\beta)} : \R \times \R^n \to \R$ satisfy the growth condition
\begin{align}
\label{eq:growth-f-beta-pos-times-2}
\Vert f^{(\beta)} \Vert_{L^{\infty}(Q_R(0))} \le C (1 + R)^{k + \eps - |\beta|} ~~ \forall R > 0.
\end{align}
Now, it remains to prove that $f^{(\beta)} \in \cP_{k-|\beta|}$ for any $\beta = (0,\beta_x,\beta_v)$ with $|\beta| \le k$. Again, we prove it by induction over $|\beta|$. When $\beta = (0,\beta_x,\beta_v)$ with $|\beta| = k$, then for any incremental quotient $D_{\beta}$ approximating $\partial_{\beta}$, we have that $f^{(\beta)} = c_{\beta} D_{\beta} f$, and hence, it holds in the weak sense
\begin{align*}
\partial_t f^{(\beta)} = \partial_t f^{(\beta)} + v \cdot \nabla_x f^{(\beta)} + (-a^{i,j} \partial_{v_i,v_j}) f^{(\beta)} = c_{\beta} D_{\beta} p ~~ \text{ in } (0,\infty) \times \R^n \times \R^n.
\end{align*}
Using that $(t \mapsto c_{\beta} D_{\beta} p(t,x,v)) \in \cP_{k - 2 - |\beta|}$ in $(0,\infty)$ by assumption, we deduce that $f^{(\beta)}$ must also be a polynomial in $(0,\infty)$, and by \eqref{eq:growth-f-beta-pos-times-2}, it must be $f^{(\beta)} \in \cP_{k-|\beta|}$ in $(0,\infty)$. From here, as in Step 2 of the proof of \autoref{prop:Liouville-higher-neg-times}, we can conclude inductively that $f^{(\beta)} \in \cP_{k - |\beta|}$ for every $\beta = (0,\beta_x,\beta_v)$ with $|\beta| \le k$ and thus the proof is complete.
\end{proof}

 \subsection{Liouville theorems in the half-space}

In this subsection we provide Liouville theorems in the half-space for solutions satisfying the specular reflection. In stark contrast to the Liouville theorems in the full space, in this case, solutions are no longer polynomials, unless their growth is slow enough. For the in-flow boundary condition, this behavior is expected already for zero source terms due to the counterexample in \cite{HJV14}. However, for solutions satisfying the specular reflection condition we encounter an interesting phenomenon: When the source term is zero, then solutions are polynomials, however, when the source term is a polynomial of degree at least $3$, then there exist solutions $f$ that are not polynomials. Instead, there exist so-called \textit{Tricomi solutions}, which are not better than $C^{4,1}_{\ell}$. 

Given $A > 0$, the Tricomi solution in dimension $n=1$ is defined as follows (see also \eqref{eq:Tricomi-sol-def}):
\begin{align*}
\mathcal{T}_{A}(x,v) := \mathcal{T}_{A,3}(x,v) := A^{-\frac{5}{2}} v^{5} - 2 \cdot 9^5 A^{-\frac{5}{6}} x^{\frac{5}{3}} U \left( - \frac{5}{3} ; \frac{2}{3} ; \frac{-v^3}{9Ax} \right), \quad x,v \in \R,
\end{align*}
where $U$ denotes the Tricomi confluent hypergeometric function.

Before we state and prove the Liouville theorem in the half-space, we list several properties of $\mathcal{T}_A$, which we will prove later in this section.

\begin{proposition}
\label{prop:Tricomi}
Let $A > 0$. The Tricomi solution satisfies the following properties:
\begin{itemize}
\item[(i)] $\mathcal{T}_A$ is homogeneous of degree $5$ and it satisfies for $p(x,v)= -40A v^3 \in \cP_3$:

\begin{equation}
\label{eq:Tricomi-PDE}
\left\{\begin{array}{rcl}
v \partial_x \mathcal{T}_A + (-A \partial_{v,v})\mathcal{T}_A &=& p ~~\qquad\quad \text{ in } \{ x > 0 \} \times \R,\\
\mathcal{T}_A(0,v) &=& \mathcal{T}_A(0,-v) ~~ \forall v \in \R.
\end{array}\right.
\end{equation}
\item[(ii)] It holds $\mathcal{T}_A \in C^{\infty}_{\text{loc}}(\{x > 0 \} \times \R)$ and moreover
\begin{align*}
\mathcal{T}_{A} \not\in C^5_{\ell}((\R \times \{ x \ge 0 \} \times \R) \cap Q_1(0)).
\end{align*}
\item[(iii)]
It holds $\mathcal{T}_{A} \in C^{4,1}_{\ell}((\R \times \{ x \ge 0 \} \times \R) \cap Q_R(0))$ for every $R > 0$ and we have
\begin{align*}
\Vert \mathcal{T}_A \Vert_{C^{4,1}_{\ell}((\R \times \{ x \ge 0 \} \times \R) \cap Q_1(0))} \le C,
\end{align*}
where $C$ only depends on an upper and lower bound of $A$.
\end{itemize}
\end{proposition}

We have the following Liouville theorem for solutions in the half-space which satisfy the specular reflection condition and do not grow faster than $(1+R)^{6-\eps}$. In particular, this result implies \autoref{thm2}.

\begin{theorem}
\label{thm:Liouville-higher-half-space-SR}
Let $k \in \{ 0 , 1 , 2 , 3 , 4, 5 \}$, $p \in \cP_{k-2}$, $\eps \in (0,1)$, $(a^{i,j})$ be a constant matrix satisfying \eqref{eq:unif-ell}, and $f$ be a weak solution to

\begin{equation*}
\left\{\begin{array}{rcl}
\partial_t f + v \cdot \nabla_x f + (-a^{i,j} \partial_{v_i,v_j}) f &=& p ~~ \qquad\qquad~~ \text{ in } \R \times \{ x_n > 0 \} \times \R^n,\\
f(t,x,v) &=& f(t,x,\mathcal{R}_x v) ~~  \text{ in } \R \times \{ x_n = 0 \} \times \R^n
\end{array}\right.
\end{equation*}
with $\Vert f \Vert_{L^{\infty}(Q_R(0))} \le C (1 + R)^{k+\eps}$ for any $R > 0$. Then, the following holds true:
\begin{itemize}
\item If $k \le 4$, or $k = 5$ and there exist $q \in \cP_3$, and $c \in \R$ such that
\begin{align*}
p(t,x,v) = q(t,x',v') + c(v_n^3 - 2a^{n,n} x_n),
\end{align*}
then $f \in \cP_{k}$.
\item Otherwise, there exist $P \in \cP_5$ satisfying the specular reflection condition $P(t,x',0,v',v_n) = P(t,x',0,v',-v_n)$, and $m \in \R \setminus \{0\}$ such that 
\begin{align}
\label{eq:Tricomi}
f(t,x,v) = P(t,x,v) + m \mathcal{T}_{a^{n,n}}(x_n,v_n),
\end{align}
where $\mathcal{T}$ denotes the Tricomi solution.
\end{itemize}
\end{theorem}

For the purpose of this article, it is sufficient to state the Liouville theorem up to $k = 5$. It is possible to give a formula for $f$ even in case $k > 5$ by repeating the arguments of our proof. However, we do not pursue this generalization here. See also \autoref{prop:Liouville-higher-half-space-SR-1D} and \autoref{lemma:Liouville-higher-half-space-SR-1D-hom} for higher order Liouville theorems in one dimension. 

We point out that under certain structural assumptions on $p$, solutions will always be polynomials also for $k \ge 5$, as the following result implies.

\begin{lemma}
\label{lemma:Liouville-flip}
Let $k \in \N \cup \{ 0 \}$, $p \in \cP_{k-2}$, $\eps \in (0,1)$, $(a^{i,j})$ be a constant matrix satisfying \eqref{eq:unif-ell} and $a^{i,n} = a^{n,i} = 0$, whenever $i \not= n$, and $f$ be a weak solution to

\begin{equation*}
\left\{\begin{array}{rcl}
\partial_t f + v \cdot \nabla_x f + (-a^{i,j} \partial_{v_i,v_j}) f &= p ~~~ \qquad\qquad \text{ in } \R \times \{ x_n > 0 \} \times \R^n,\\
f(t,x,v) &= f(t,x,\mathcal{R}_x v) ~~ \text{ in } \R \times \{ x_n = 0 \} \times \R^n
\end{array}\right.
\end{equation*}
with $\Vert f \Vert_{L^{\infty}(Q_R(0))} \le C (1 + R)^{k+\eps}$ for any $R > 0$. Moreover, assume that
\begin{align}
\label{eq:p-symmetry}
p(t,x',x_n,v',v_n) = p(t,x',-x_n,v',-v_n) ~~ \forall (t,x,v) \in \R \times \R^n \times \R^n.
\end{align}
Then, $f \in \cP_{k}$.
\end{lemma}

\begin{proof}
By \autoref{lemma:mirror}, the extended function $\bar{f}$ given by $\bar{f} = f$ in $\R \times \{ x_n > 0 \} \times \R^n$ and $\bar{f}(t,x,v) = f(t,x',-x_n,v',-v_n)$ in $\R \times \{ x_n \le 0 \} \times \R^n$ is a weak solution to
\begin{align*}
\partial_t f + v \cdot \nabla_x f + (-a^{i,j} \partial_{v_i,v_j}) f &= p ~~ \text{ in } \R \times \R^n \times \R^n.
\end{align*}
Since $p$ is a polynomial and $a^{i,j}$ is constant, we deduce that $\bar{f}$ is in particular a classical solution (see \autoref{lemma:interior-reg}).
Hence, we can apply \autoref{prop:Liouville-higher-neg-times} to deduce that $f \in \cP_k$, as desired.
\end{proof}

Note that the assumption \eqref{eq:p-symmetry} is not preserved under flattening the boundary.

To prove the higher order Liouville theorem in the half-space in \autoref{thm:Liouville-higher-half-space-SR}, let us first prove the following stationary version in 1D, assuming in addition that $f$ is homogeneous.

\begin{lemma}
\label{lemma:Liouville-higher-half-space-SR-1D-hom}
Let $\lambda \in \N \cup \{ 0 \}$, $p \in \cP_{\lambda}$ be a homogeneous polynomial of degree $\lambda$, $A > 0$, and $f$ be a weak solution to

\begin{equation*}
\left\{\begin{array}{rcl}
v \partial_x f - A\partial_{vv} f &=& p ~~\quad\qquad \text{ in } \{ x > 0 \} \times \R,\\
f(0,v) &=& f(0,-v) ~~ \text{ in } \R
\end{array}\right.
\end{equation*}
with $f(r^3x,rv) = r^{\lambda+2} f(x,v)$, i.e. $f$ is ($\lambda+2$)-homogeneous, and $\Vert f \Vert_{L^{\infty}(\tilde{Q}_1(0))} \le C$ for some $C > 0$. Then, the following holds true:
\begin{itemize}
\item If $\lambda \not= 6k+3$ for some $k \in \N$, then $f \in \cP_{\lambda+2}$.
\item If $\lambda = 6k+3$ for some $k \in \N$ and $p$ is of the form
\begin{align}
\label{eq:p-special}
p(x,v) = \sum_{l = 1}^{2k+1} c_l \big( l A^{-1} v^{6(k+1) - 3l} x^{l-1} - (6k+5 - 3l)(6k+4 - 3l) v^{6(k + 1)- 3(l+1)}x^l \big)
\end{align}
for some $c_l \in \R$, $l \in \{1 , \dots , 2k+1 \}$, then $f \in \cP_{\lambda+2}$.
\item If $\lambda = 6k+3$ for some $k \in \N$ and $p$ is not of the form \eqref{eq:p-special}, then
\begin{align}
\label{eq:Tricomi-sol}
f(x,v) = P(x,v) + m \mathcal{T}_{A,\lambda}(x,v)
\end{align}
for some ($\lambda+2$)-homogeneous $P \in \cP_{\lambda + 2}$ satisfying $P(0,v) = 0 = P(0,-v)$ and some $m \in \R \setminus \{ 0 \}$, where $\mathcal{T}_{A,\lambda}$ is given by
\begin{align}
\label{eq:Tricomi-sol-def}
\mathcal{T}_{A,\lambda}(x,v) = A^{-\frac{\lambda+2}{2}} v^{\lambda+2} - 2 \cdot 9^{\lambda + 2} A^{-\frac{\lambda+2}{6}} x^{\frac{\lambda + 2}{3}} U \left( - \frac{\lambda + 2}{3} ; \frac{2}{3} ; \frac{-v^3}{9Ax} \right),
\end{align}
and satisfies $\mathcal{T}_{A,\lambda}(0,v) = \mathcal{T}_{A,\lambda}(0,-v)$.
\end{itemize}
Moreover, for any $\lambda,p$, there exists a solution $f$ that is ($\lambda+2$)-homogeneous.
\end{lemma}

%\begin{remark}
%The homogeneity $\lambda = 3$ is the lowest for which not all homogeneous solutions are given by polynomials in \autoref{lemma:Liouville-higher-half-space-SR-1D-hom}. Hence, it is of particular relevance for the proof of our main result. In this case, the exceptional polynomials $p$ in \eqref{eq:p-special} are of the form
%\begin{align*}
%p(x,v) = c(A^{-1}v^3 - 2x), ~~ c \in \R.
%\end{align*}
%%Moreover, by repeating the proof (for $A=1$) we find that all $3$-homogeneous polynomials $p \in \cP_3$ are of the form $p(x,v) = ax + bv^3$ and that 
%%\begin{align*}
%%P(x,v) = -\frac{a}{2} xv^2 - \frac{1}{20}\left(b + \frac{a}{2}\right) v^5
%%\end{align*}
%%solves \eqref{eq:inhom-1D-PDE-poly}. Then, $P(0,v) = - \left(\frac{b}{20} + \frac{a}{40}\right) v^5$ and therefore, $m = \left(\frac{b}{10} + \frac{a}{20}\right)$ by the proof of \autoref{lemma:Liouville-higher-half-space-SR-1D-hom}. If $A \not= 1$ we deduce by scaling that for $p(x,v) = ax + b v^3$ we have the following Tricomi-type solution in \eqref{eq:Tricomi-sol}
%%\begin{align*}
%%f(x,v) = -\frac{a}{2A} xv^2 - \left( \frac{b}{20 A} + \frac{a}{2A^2} \right) v^5 + A^{\frac{5}{6}}\left( \frac{b A^{\frac{3}{2}}}{10} + \frac{a A^{\frac{1}{2}}}{20} \right) x^{\frac{5}{3}} U \left(-\frac{5}{3} ; \frac{2}{3} ; \frac{-v^3}{9Ax} \right).
%%\end{align*}
%\end{remark}

Note that Step 3 of the proof is inspired by a computation in \cite{HJV14} (see also \cite{Zhu22}).

\begin{proof}[Proof of \autoref{lemma:Liouville-higher-half-space-SR-1D-hom}]
The proof is split into several steps.

\textbf{Step 1:} First, note that it is sufficient to prove the desired result in case $A = 1$. Indeed, assume we have already proved the result for $A = 1$. Then, define $\tilde{f}(x,v) = f(A^{1/2}x,A^{1/2}v)$, and observe that it satisfies
\begin{align*}
[v \partial_x - \partial_{vv}]\tilde{f}(x,v)= [A^{1/2}v \partial_x f - A\partial_{vv}f](A^{1/2}x,A^{1/2}v) = p(A^{1/2}x,A^{1/2}v).
\end{align*}
Hence, the desired result for general $A \not= 1$ can be deduced from the special case $A = 1$.

\textbf{Step 2:} We claim that for any $p \in \cP_{\lambda}$ that is homogeneous of degree $\lambda$, there exists a $(\lambda+2)$-homogeneous polynomial $P \in \cP_{\lambda + 2}$ solving
\begin{align}
\label{eq:inhom-1D-PDE-poly}
v \partial_x f - \partial_{vv} f = p ~~ \text{ in } (0,\infty) \times \R.
\end{align} 
Note that here we are disregarding the boundary condition at $x = 0$.
To prove this claim, it suffices to find solutions for any monomial of degree $\lambda$, i.e. it suffices to prove that for any $\lambda_1,\lambda_2 \in \N \cup \{ 0 \}$ with $\lambda_1 + 3 \lambda = \lambda$ and any $a \in \R$, there exists a ($\lambda+2$)-homogeneous polynomial $P \in \cP_{\lambda +2}$ solving
\begin{align}
\label{eq:inhom-1D-PDE}
v \partial_x f - \partial_{vv} f = a x^{\lambda_1} v^{\lambda_2} ~~ \text{ in } (0,\infty) \times \R.
\end{align}
We will prove this claim by induction over $\lambda_1$. First, if $\lambda_1 = 0$, then we can choose 
\begin{align*}
P_{0,\lambda_2,a} = - \frac{a}{(\lambda_2+2)(\lambda_2 + 1)} v^{\lambda_2 + 2},
\end{align*}
and obtain that $P_{0,\lambda_2,a} $ solves 
\begin{align*}
v \partial_x P_{0,\lambda_2,a}  - \partial_{vv} P_{0,\lambda_2,a}  = a v^{\lambda_2} ~~ \text{ in } (0,\infty) \times \R.
\end{align*}
This proves the claim for any $\lambda_2 \in \N \cup \{ 0 \}$ and $\lambda_1 = 0$.
Next, let us assume that we have proved the claim for some $\lambda_1 \in \N \cup \{0\}$, i.e. that for any $\lambda_2 \in \N \cup \{ 0 \}$ and $a \in \R$, we have found a homogeneous $P_{\lambda_1,\lambda_2,a} \in \cP_{\lambda+2}$ solving 
\begin{align*}
v \partial_x P_{\lambda_1,\lambda_2,a} - \partial_{vv} P_{\lambda_1,\lambda_2,a} = a x^{\lambda_1} v^{\lambda_2} ~~ \text{ in } (0,\infty) \times \R.
\end{align*} 
Then, let us fix $a \in \R$ and $\lambda_2 \in \N \cup \{ 0 \}$ and observe that
\begin{align*}
(v \partial_x - \partial_{vv} ) x^{\lambda_1+1} P_{0,\lambda_2,a} = (\lambda_1+1) x^{\lambda_1} v P_{0,\lambda_2,a} + a x^{\lambda_1+1} v^{\lambda_2} = a x^{\lambda_1+1} v^{\lambda_2} - \frac{a (\lambda_1+1)}{(\lambda_2+2)(\lambda_2 + 1)} x^{\lambda_1} v^{\lambda_2 + 3}.
\end{align*}
Then, by the induction assumption, for $b =  \frac{a (\lambda_1+1)}{(\lambda_2+2)(\lambda_2 + 1)}$ we can find a homogeneous $P_{b,\lambda_1,\lambda_2+3} \in \cP_{3\lambda_1+\lambda_2+5}$ such that
\begin{align*}
(v \partial_x - \partial_{vv}) P_{b,\lambda_1,\lambda_2+3} = b x^{\lambda_1} v^{\lambda_2+3} ~~ \text{ in } (0,\infty) \times \R.
\end{align*}
Hence, we set
\begin{align*}
P_{a,\lambda_1 + 1 , \lambda_2} = x^{\lambda_1+1} P_{0,\lambda_2,a} + P_{b,\lambda_1,\lambda_2+3} \in \cP_{3 \lambda_1 + \lambda_2 + 5} = \cP_{3(\lambda_1 + 1) + \lambda_2 + 2},
\end{align*}
and conclude the proof of the claim.

\textbf{Step 3:} The goal of this step is to find all functions $h$ that are homogeneous of degree $\lambda + 2$ and solve the corresponding homogeneous equation (still disregarding the boundary condition)
\begin{align}
\label{eq:hom-1D-PDE}
v \partial_x h - \partial_{vv} h = 0 ~~ \text{ in } (0,\infty) \times \R.
\end{align}
We will find all ($\lambda+2$)-homogeneous solution $h$ to \eqref{eq:hom-1D-PDE} by writing
\begin{align*}
h(x,v) = x^{\frac{\lambda+2}{3}} \Psi(\tau), \qquad \tau := \frac{-v^3}{9 x},
\end{align*}
and identifying the corresponding ODE solved by $\Psi$.
We compute
\begin{align*}
\partial_x \left( x^{\frac{\lambda+2}{3}} \Psi(\tau) \right) &= \frac{\lambda + 2}{3x} x^{\frac{\lambda + 2}{3}} \Psi(\tau) + \frac{v^3}{9x^2} x^{\frac{\lambda+2}{3}} \Psi'(\tau),\\
\partial_{vv} \left( x^{\frac{\lambda+2}{3}} \Psi(\tau) \right) &= x^{\frac{\lambda + 2}{3}}  \partial_v \left( \Psi'(\tau) \left( - \frac{v^2}{3 x} \right) \right) = x^{\frac{\lambda + 2}{3}} \left(\Psi''(\tau) \frac{v^4}{9 x^2} - \Psi'(\tau) \frac{2v}{3x} \right).
\end{align*}
Hence, we obtain
\begin{align*}
(v \partial_x - \partial_{vv})h(x,v) &= \frac{v x^{\frac{\lambda+2}{3}}}{3 x} \left( (\lambda+2) \Psi(\tau) + 2 \Psi'(\tau) + \frac{v^3}{3x} \Psi'(\tau)  - \frac{v^3}{3x} \Psi''(\tau) \right) \\
&= \frac{v x^{\frac{\lambda+2}{3}}}{x} \left( \frac{\lambda+2}{3} \Psi(\tau) + \left(\frac{2}{3} - \tau \right) \Psi'(\tau) + \tau \Psi''(\tau) \right).
\end{align*}
This implies that $\Psi$ solves the so-called ``Kummer's equation''
\begin{align}
\label{eq:hom-Kummer}
\frac{\lambda+2}{3} \Psi(\tau) + \left(\frac{2}{3} - \tau \right) \Psi'(\tau) + \tau \Psi''(\tau) = 0 .
%a (-9\tau)^{\frac{\lambda_2 - 1}{3}}.
\end{align}

Note that any solution $\Psi$ is of the form $\Psi = C_1\Psi_1 + C_2 \Psi_2$ for some $C_1, C_2 \in \R$, where $\Psi_1,\Psi_2$ are given by
\begin{align*}
\Psi_1(\tau) =  {}_1F_1\left(-\frac{\lambda+2}{3}; \frac{2}{3} ;\tau\right) , \qquad \Psi_2(\tau) = \tau^{1/3} ~ {}_1F_1\left(-\frac{\lambda+1}{3}; \frac{4}{3} ;\tau\right),
\end{align*}
and ${}_1 F_1$ denote hypergeometric functions of the first kind.
%This is 
%due to the linearity of the inhomogeneous ODE \eqref{eq:inhom-Kummer}, and the corresponding homogeneous equation
%\begin{align}
%\label{eq:hom-Kummer}
%\frac{\lambda+2}{3} \Psi_1(\tau) + \left(\frac{2}{3} - \tau \right) \Psi_1'(\tau) + \tau \Psi_1''(\tau) = 0
%\end{align}
%is of the form $C_1 \Psi_1 + C_2 \Psi_2$. 
%The latter is a well-known fact about \eqref{eq:hom-Kummer}, which is also known as ``Kummer's equation''. 

Hence, we have that $h$ is of the form
\begin{align}
\label{eq:hom-ODE-sol}
h(x,v) = x^{\frac{\lambda+2}{3}} \Psi(\tau) = C_1 x^{\frac{\lambda+2}{3}}  {}_1F_1\left(-\frac{\lambda+2}{3}; \frac{2}{3} ;\tau\right) + C_2 v x^{\frac{\lambda+1}{3}} {}_1F_1\left(-\frac{\lambda+1}{3}; \frac{4}{3} ;\tau\right).
\end{align}

\textbf{Step 4:} By combination of the first and second step, any solution $f$ to \eqref{eq:inhom-1D-PDE-poly} is of the form $f = P + h$, where $P \in \cP_{\lambda+2}$ is the ($\lambda+2$)-homogeneous polynomial from Step 2 and $h$ is a ($\lambda+2$)-homogeneous function of the form \eqref{eq:hom-ODE-sol} (for some $C_1, C_2 \in \R$).  We will now classify all solutions $f = P + h$ that do not grow faster than a polynomial and satisfy the specular reflection condition (as in the statement of the result). 

To do so, let us first recall that if $-a \in \N$, then ${}_1F_1(a,b;\tau)$ is a polynomial of degree $-a$. Otherwise, it holds 
\begin{align}
\label{eq:as-exp-1F1}
{}_1F_1(a;b;z) \sim \Gamma(b) \left( \frac{e^z z^{a-b}}{\Gamma(a)} + \frac{(-z)^{-a}}{\Gamma(b-a)} \right) \qquad \text{ as } |z| \to \infty.
\end{align}
Here, by the notation $f \sim g$, we mean that $f/g \to 1$.

In particular, we deduce from \eqref{eq:hom-ODE-sol} that $h(1,v)$ grows exponentially as $v \to - \infty$, whenever $C_1 > 0$ and $\frac{\lambda+1}{3} \in \N$ or $C_2 > 0$ and $\frac{\lambda+2}{3} \in \N$. Thus, in  case $\frac{\lambda+1}{3} \in \N$, we must have $C_1 = 0$, and in case $\frac{\lambda+2}{3} \in \N$ it must be $C_2 = 0$ for $f$ to satisfy the growth condition. Hence, in both of these cases we have $h \in \cP_{\lambda + 2}$, and therefore also $f = P + h \in \cP_{\lambda+2}$. Finally, note that in both of these cases, there indeed exists a solution $f$ satisfying in addition the specular reflection boundary condition. Indeed, if $P$ already satisfies the specular reflection condition, then we can simply take $C_1 = 0$ in the definition of $h$ and have thus $f = P$. Otherwise, it must be $P(0,v) = c_0 v^{\lambda + 2}$ for some $c_0 \in \R \setminus \{0\}$ and $\lambda \in \N$ must be odd. Clearly, since $\Psi(\tau)$ was also homogeneous of degree $\frac{\lambda + 2}{3}$, it  holds $h(0,v) = C_1 v^{\lambda+2}$ where we can still choose $C_1 \in \R$. By taking $C_1 = - c_0$, we obtain that $f(0,v) = 0$ and therefore $f$ satisfies the specular reflection boundary condition.

It remains to treat the case $\frac{\lambda}{3} \in \N$. By the asymptotic expansion \eqref{eq:as-exp-1F1} for ${}_1 F_1$, the function $h$ grows exponentially also in this case, unless $C_1,C_2 \in \R$ are chosen such that for some  $m \in \R$:
\begin{align*}
C_1 = m\frac{\Gamma\left(-\frac{1}{3}\right)}{\Gamma\left(-\frac{\lambda+1}{3}\right)}, \qquad C_2 = m\frac{\Gamma\left(\frac{1}{3}\right)}{\Gamma\left(-\frac{\lambda+2}{3}\right)}.
\end{align*}
In this case, the exponential functions in the expansions cancel out and $h$ becomes a multiple of the so-called Tricomi solution (also known as Tricomi confluent hypergeometric function), i.e.
\begin{align*}
h(x,v) = m x^{\frac{\lambda+2}{3}} U\left(-\frac{\lambda+2}{3};\frac{2}{3};\tau \right).
\end{align*}
First, let us assume that $\lambda = 6k$ for some $k \in \N$, i.e. that $\frac{\lambda}{3} \in \N$ and $\lambda$ is also even. Then, $P \in \cP_{\lambda + 2} = \cP_{6k + 2}$ is even and therefore it already  satisfies the specular reflection condition $P(0,v) = P(0,-v)$. In particular, this guarantees the existence of a solution also in this case. Moreover, it follows that $f$ also satisfies the specular reflection condition if and only if $h(0,v) = h(0,-v)$. However, since $h$ solves the homogeneous equation \eqref{eq:hom-1D-PDE}, if it satisfies the specular reflection condition, by \autoref{lemma:Liouville-flip} it must be $h \in \cP_{\lambda + 2}$. This proves the result in case $\lambda = 6k$ for some $k \in \N$.\\
It remains to consider the case when $\lambda = 6k+3$ for some $k \in \N$, i.e. when $\frac{\lambda}{3} \in \N$ and $\lambda$ is odd. Since $P \in \cP_{\lambda+2} = \cP_{6k+5}$ is ($6k+5$)-homogeneous we have that either $P(0,v) \equiv 0$ or $P(0,v) = c v^{6k+5}$ for some $c \in \R \setminus \{ 0 \}$. In the first case, $P$ already satisfies the specular reflection condition and hence by the same arguments as before, we deduce from \autoref{lemma:Liouville-flip} that $h \in \cP_{\lambda + 2}$. This case occurs if and only if $p$ is of the form \eqref{eq:p-special}, as can be seen from the fact that any homogeneous polynomial of degree $\lambda + 2$ which satisfies the specular reflection condition is of the form $\sum_{l = 1}^{2k+1} c_l v^{6k + 5 - 3l} x^l$ for some $c_l \in \R$, and then
\begin{align*}
p(x,v) &= (v \partial_x - \partial_{vv}) \left( \sum_{l = 1}^{2k+1} c_l v^{6k + 5 - 3l} x^l \right) \\
&= \sum_{l = 1}^{2k+1} c_l( l v^{6(k+1) - 3l}x^{l-1} - (6k+5 - 3l)(6k+4 - 3l)v^{6(k + 1)- 3(l+1)}x^l).
\end{align*}

Finally, in case $P(0,v) = c v^{6k+5}$, we can write $P(x,v) = P_1(x,v) + c v^{6k+5}$, where $P_1 \in \cP_{\lambda+2}$ satisfies $P_1(0,v) = 0 = P_1(0,-v)$, and recall from \cite[(3.35), (3.36)]{HJV14} that 
\begin{align}
\label{eq:Tricomi-asymp}
U\left(-a,\frac{2}{3},-\tau^3\right) \sim 
\begin{cases}
K|\tau|^{3a}, ~~ \text{ as } \tau \to + \infty,\\
|\tau|^{3a}, ~~ \text{ as } \tau \to -\infty,
\end{cases}
\quad \text{ where } \quad  K = 2 \cos\left(\pi \left(a + \frac{1}{3} \right)\right).
\end{align}
In particular, since for $a = \frac{\lambda+2}{3} = \frac{6k + 5}{3}$ we have $K = 2$, this means
\begin{align*}
h(0,v) \sim 9^{-(6k+5)}m
\begin{cases}
2 v^{6k + 5} ~~ \text{ for } v > 0,\\
- v^{6k + 5} ~~ \text{ for } v \le 0.
\end{cases}
\end{align*}

In particular, the function $f = P + h$ satisfies the specular reflection condition if and only if 
\begin{align*}
c + 2 \cdot 9^{-(6k+5)}m = - \left(c - 9^{-(6k+5)}m \right) \quad \Leftrightarrow m = -2\cdot 9^{6k+5} c.
\end{align*}
In this case,
\begin{align*}
f(0,v) = 3c|v|^{6k+5} = f(0,-v) ~~ \forall v \in \R.
\end{align*}
Moreover, by construction
\begin{align*}
f(x,v) = P(x,v) + h(x,v) = P_1(x,v) + c \left[ v^{\lambda+2}- 2 \cdot 9^{\lambda+2} x^{\frac{5}{3}}  U\left(-\frac{\lambda+2}{3};\frac{2}{3}; \frac{-v^3}{9x} \right) \right]
\end{align*}
is a solution to \eqref{eq:inhom-1D-PDE}. Altogether, this proves the desired result.
\end{proof}

The previous proof puts us in a position to show the first two claims in \autoref{prop:Tricomi}.

\begin{proof}[Proof of \autoref{prop:Tricomi}(i),(ii)]
The claim (i) follows immediately from the proof of \autoref{lemma:Liouville-higher-half-space-SR-1D-hom}.
The first claim in (ii) follows since $U$ is analytic.
It is easy to see that Tricomi solution $\mathcal{T}_A$ is not smooth at $(x,v) = (0,0)$, but only $C^{5/3}_x$ in $x$ by taking suitable sequences $v_i \to 0$ and $x_i \to 0$ such that $-\frac{v_i^3}{9 Ax_i} \to c_0 \in \R$. Indeed, then it holds 
\begin{align*}
\mathcal{T}_A(x_i,v_i) - A^{-\frac{5}{2}}v_i^5 = -2\cdot 9^5 A^{-\frac{5}{6}} x_i^{\frac{5}{3}} U \left( -\frac{5}{3};\frac{2}{3};\frac{-v_i^3}{9 A x_i} \right) \sim -2 \cdot 9^5 A^{-\frac{5}{6}} x_i^{\frac{5}{3}}U \left( -\frac{5}{3};\frac{2}{3};c_0 \right) ~~ \text{ as }  i \to \infty.
\end{align*}
Here, by the notation $f_i \sim g_i$, we mean that $f_i/g_i \to 1$ as $i \to \infty$. This observation immediately implies that $\mathcal{T}_A \not\in C^{5}_{\ell}$. We will prove (iii), namely that $\mathcal{T}_A \in C^{4,1}_{\ell}$, directly after \autoref{prop:regularity-SR}.
%Finally, note that when $x_i,v_i$ are chosen such that $-\frac{v_i^3}{9A x_i} \to \pm \infty$, then by \eqref{eq:Tricomi-asymp}
%\begin{align*}
%\mathcal{T}_A(x_i,v_i) \sim 3 v_i^5 ~~ \text{ as } i \to \infty,
%\end{align*}
%which is why $\mathcal{T}_A(0,v) = \mathcal{T}_A(0,-v)$, i.e. $\mathcal{T}_A$ satisfies the specular reflection condition.
\end{proof}

The following lemma is a higher order Liouville theorem in the half-space for the stationary equation in 1D. It generalizes \autoref{lemma:Liouville-higher-half-space-SR-1D-hom} in the sense that $f$ is not assumed to be homogeneous any more.

\begin{proposition}
\label{prop:Liouville-higher-half-space-SR-1D}
Let $k \in \N \cup \{ 0 \}$, $p \in \cP_{k-2}$, $\eps \in (0,1)$, $A > 0$, and $f$ be a weak solution to
\begin{equation*}
\left\{\begin{array}{rcl}
v \partial_x f - A\partial_{vv} f &=& p ~~ \qquad\quad \text{ in } \{ x > 0 \} \times \R,\\
f(0,v) &=& f(0,-v) ~~ \text{ in } \R
\end{array}\right.
\end{equation*}
with $\Vert f \Vert_{L^{\infty}(\tilde{Q}_R(0))} \le C (1 + R)^{k+\eps}$ for any $R > 0$. Then the following holds true:
\begin{itemize}
\item If $p$ does not contain monomials of degree $6l+3$ for any $l \in \N \cup \{ 0 \}$, or if they are all of the form \eqref{eq:p-special}, then $f \in \cP_k$.
\item Otherwise, there are $P \in \cP_k$ such that $P(0,v) = P(0,-v)$, and $m_{l} \in \R$ such that 
\begin{align*}
f(x,v) = P(x,v) + \sum_{l = 0}^{\lfloor\frac{k-3}{6} \rfloor} m_l \mathcal{T}_{A,6l+3}(x,v).
\end{align*}
\end{itemize}
\end{proposition}

The idea of the proof is to show that $f$ can be written as the sum of a polynomial and a sum of homogeneous solutions to which we can apply \autoref{lemma:Liouville-higher-half-space-SR-1D-hom}.

\begin{proof}[Proof of \autoref{prop:Liouville-higher-half-space-SR-1D}]
First, we recall from the proof of \autoref{lemma:Liouville-higher-half-space-SR-1D-hom} that it suffices to assume $A =1$. 

\textbf{Step 1:} Note that it is sufficient to prove the result under the assumption that $p \in \cP_{\lambda}$ is homogeneous of degree $\lambda \in \{ 0 , 1, \dots , k \}$. Indeed, by the linearity of both, the equation and the boundary condition, any solution with respect to a general $p = \sum_{\lambda = 0}^k p_{\lambda} \in \cP_k$ where $p_{\lambda} \in \cP_{\lambda}$ are homogeneous of degree $\lambda$, can be written as the sum of solutions $f_{\lambda}$, which solve the aforementioned equation with respect to $p_{\lambda}$. This can be proved by induction over $k$, observing that the existence of the solutions $f_{\lambda}$ is guaranteed by \autoref{lemma:Liouville-higher-half-space-SR-1D-hom}.

We are now in a position to give the actual proof of the result, which is split into several steps.

\textbf{Step 2:} We start by introducing the space 
\begin{align*}
\tilde{\cP}_k &= \{ p \in \cP_k : p = p_1 + x p_2, ~~ p_1,p_2 ~\text{ are  polynomials of even degree in $x,v \in \R$} \} \\
&=  \{ p \in \cP_k : p(0,v) = p(0,-v) \},
\end{align*}
and by claiming that there exists a polynomial $q \in \tilde{\cP}_k$ such that for the function 
\begin{align*}
g(x,v) := 3x \partial_x f(x,v) + v \partial_v f(x,v) - (\lambda+2) f(x,v) = \frac{d}{dr} \frac{f(r^3 x , r v)}{r^{\lambda+2}} \Bigg\vert_{r = 1}
\end{align*}
it holds that $g = q$.
To see this, note that for any $r \in (0,\infty)$ we have that
\begin{align*}
g_r(x,v) = (r-1)^{-1}\left[\frac{f(r^3x,rv)}{r^{\lambda+2}} - f(x,v)\right]
\end{align*}
solves 
\begin{equation}
\label{eq:diff-Liouville}
\left\{\begin{array}{rcl}
v \partial_x g_r - \partial_{vv} g_r &=& 0 ~~\qquad\quad \text{ in } \{ x > 0 \} \times \R,\\
g_r(0,v) &=& g_r(0,-v) ~~ \text{ in } \R,
\end{array}\right.
\end{equation}
since $p$ is homogeneous of degree $\lambda$. Note that by \autoref{lemma:Liouville-flip} it must hold $g_r \in \cP_{k}$ in $\{ x> 0\} \times \R$, and since the $g_r$ all satisfy the specular reflection condition, we deduce that they must belong to the subspace $\tilde{\cP}_k$ for any $r > 0$.

Moreover, we have $g \in C^{\infty}_{loc}((0,\infty) \times \R)$ by interior regularity for $f$. Hence, it holds that $g_r \to g$ locally uniformly in $(0,\infty) \times \R$ (up to a subsequence) as $r \to 1$, and therefore, it must also be $g \in \tilde{\cP}_k$ (since $\tilde{\cP}_k$ is a finite dimensional subspace). This proves
\begin{align}
\label{eq:hom-poly-help-1}
3x \partial_x f(x,v) + v \partial_v f(x,v) - (\lambda+2) f(x,v) = q(x,v) ~~ \forall (x,v) \in \{ x > 0 \} \times \R
\end{align}
for some $q \in \tilde{\cP}_k$, as desired.

\textbf{Step 3:} Note that if $q(x,v) = b_{\mu} x^{\mu_1} v^{\mu_2}$ is a monomial of degree $|\mu| = 3\mu_1 + \mu_2 \le k$, then 
\begin{align}
\label{eq:tilde-q}
\tilde{q}_{\mu}(x,v) =  
\begin{cases}
\frac{b_{\mu}}{|\mu| - (\lambda + 2)} x^{\mu_1} v^{\mu_2}, ~~ \quad \qquad \text{ if }  |\mu| \not= \lambda + 2,\\
 \frac{1}{3}b_{\mu} x^{\mu_1} v^{\mu_2} \log(x), ~~~ \qquad \text{ if } |\mu| = \lambda + 2, ~~ \mu_2 \not= \lambda + 2 ,\\
 \sgn(v) b_{\mu} x^{\mu_1} v^{\mu_2} \log(|v|) ~~ \text{ if } |\mu| = \lambda + 2, ~~ \mu_2 = \lambda + 2
\end{cases}
\end{align}
is a solution to \eqref{eq:hom-poly-help-1}.

Moreover, note that any solution $f$ of the equation in \eqref{eq:hom-poly-help-1} is of the form $f = f_1 + \tilde{q}$, where $f_1$ solves the corresponding homogeneous equation
\begin{align*}
\frac{d}{dr} \frac{f_1(r^3 x , r v)}{r^{\lambda +2}} \Bigg\vert_{r = 1} = 3x \partial_x f_1(x,v) + v \partial_v f_1(x,v) - (\lambda + 2) f_1(x,v) = 0,
\end{align*}
which means in particular that $f_1$ is homogeneous of degree $\lambda + 2$.

Hence, if we consider now a general polynomial $q = \sum_{\mu} b_{\mu} x^{\mu_1} v^{\mu_2} \in \tilde{\cP}_k$ in \eqref{eq:hom-poly-help-1}, then we have
\begin{align*}
f = f_1 + f_2,
\end{align*}
where $f_1$ is homogeneous of degree $\lambda + 2$, and $f_2 = \sum_{\mu} \tilde{q}_{\mu}$, where all of the $\tilde{q}_{\mu}$ are as in \eqref{eq:tilde-q}. A straightforward, explicit computation reveals that for any $\mu$, it holds that the function
\begin{align*}
v \partial_x \tilde{q}_{\mu} - \partial_{vv} \tilde{q}_{\mu}
\end{align*}
is not homogeneous of degree $\lambda$, as opposed to the function
\begin{align*}
v \partial_x f_1 - \partial_{vv} f_1,
\end{align*}
which is homogeneous of degree $\lambda$.
Recalling that $v \partial_x f - \partial_{vv} f = p$, where $p$ is $\lambda$-homogeneous, we deduce that 
\begin{align*}
v \partial_x f_2 - \partial_{vv} f_2 = 0.
\end{align*}
By combination of this information with \eqref{eq:tilde-q}, we see that it must be $b_{\mu} = 0$ whenever $|\mu| = \lambda + 2$, and therefore in particular, $f_2 \in \cP_k$. 

Since we already know that $q \in \tilde{\cP}_k$, by the definition of $\tilde{q}_{\mu}$ it must be also $f_2 \in \tilde{\cP}_k$.
As a consequence, $f_2(0,v) = f_2(0,-v)$ satisfies the specular reflection condition. Thus, altogether, we have deduced that $f_1$ satisfies
\begin{equation*}
\left\{\begin{array}{rcl}
v \partial_x f_1 - \partial_{vv} f_1 &=& p ~~~\qquad\quad \text{ in } \{ x > 0 \} \times \R,\\
f_1(0,v) &=& f_1(0,-v) ~~ \text{ in } \R.
\end{array}\right.
\end{equation*}

Since we know that $f_1$ is homogeneous of degree $\lambda + 2$, the desired result follows by application of \autoref{lemma:Liouville-higher-half-space-SR-1D-hom}. 
\end{proof}

We are now finally in a position to give a proof of the general Liouville theorem in the half-space for solutions satisfying the specular reflection condition.

\begin{proof}[Proof of \autoref{thm:Liouville-higher-half-space-SR}]
We write $x = (x' ,x_n)$ and $v = (v',v_n)$. Then, by the same procedure as in Step 1 of the proof of \autoref{prop:Liouville-higher-neg-times}, namely by applying the interior regularity from \autoref{lemma:interior-reg} in $Q_R(0)$ to increments of $f$ in the $x_i$ direction for $i \in \{1 , \dots, n-1\}$ and in $t$, we deduce that they are independent of $t,x'$. In particular, this allows us to write
\begin{align}
\label{eq:poly-sum-pos-times-halfspace}
f(t,x,v) = \sum_{\substack{\beta = (\beta_t,\beta_{x'},0),\\ |\beta| \le k}} f_{\beta}(x_n,v) t^{\beta_t} x_1^{\beta_{x_1}} \cdot \dots \cdot x_{n-1}^{\beta_{x_{n-1}}} ,
\end{align}
where the $f_{\beta} : \R^n \times \R \to \R$ satisfy the growth condition
\begin{align*}
\Vert f_{\beta} \Vert_{L^{\infty}(Q_R(0))} \le C (1 + R)^{k + \eps - |\beta|} ~~ \forall R > 0,
\end{align*}
as well as the specular reflection condition
\begin{align*}
f_{\beta}(0,v',v_n) = f_{\beta}(0,v', -v_n) ~~ \forall v \in \R^n.
\end{align*}
Then, let $\beta = (\beta_t,\beta_{x'},0)$ with $|\beta| = k$ and observe that $f_{\beta} = c_{\beta} \partial_{\beta} f$ for some $c_{\beta} \in \R$. Hence, by differentiating the equation for $f$ with respect to $\partial_{\beta}$, we obtain
\begin{align*}
v_n \partial_{x_n} f_{\beta} + (-a^{i,j} \partial_{v_i,v_j}) f_{\beta} = 0 ~~ \text{ in } (0,\infty) \times \R^n.
\end{align*}
Here, we used that $f_{\beta}$ does not depend on $t,x'$ and that $\partial_{\beta} p = 0$ since $p \in \cP_{k-2}$. In particular, due to the specific structure of the equation for $f_{\beta}$, we can take increments of $f_{\beta}$ in $v'$ and deduce from the interior regularity  (see \autoref{lemma:interior-reg}) and the growth condition that $f_{\beta}(v,x_n) = f_{\beta}(v_n,x_n)$ is independent of $v'$, and therefore it satisfies
\begin{align*}
v_n \partial_{x_n} f_{\beta} - a^{n,n}\partial_{v_n,v_n} f_{\beta} = 0 ~~ \text{ in } (0,\infty) \times \R.
\end{align*}
We can now apply the Liouville theorem in 1D (see \autoref{prop:Liouville-higher-half-space-SR-1D}) to deduce that $f_{\beta} \in \cP_{k - |\beta|} = \cP_0$. \\
We can repeat this procedure with $|\beta| < k$ in the exact same way as in the proof of \autoref{prop:Liouville-higher-neg-times}, as long as $|\beta| > k - 5$ by using that $\partial_{\beta} f - c_{\beta} f_{\beta} \in \cP_{k - |\beta|}$ since for these values of $\beta$, we will always apply the 1D Liouville theorem from \autoref{prop:Liouville-higher-half-space-SR-1D} with $k < 5$. Indeed, for $|\beta| < k$, we first obtain by taking incremental quotients in $v'$ that $f_{\beta}$ is of the form
\begin{align}
\label{eq:rep-f-beta}
f_{\beta}(x_n,v) = \sum_{|\tilde{\beta}| \le k - |\beta|} f_{\beta,\tilde{\beta}}(v_n,x_n) v_1^{\tilde{\beta}_{v_1}} \cdot \dots \cdot v_{n-1}^{\tilde{\beta}_{v_{n-1}}},
\end{align}
and then by applying inductively the 1D Liouville theorem in the half-space (\autoref{prop:Liouville-higher-half-space-SR-1D} with $k := k - |\beta| - |\tilde{\beta}| < 5$) and deriving the equations for $f_{\beta,\tilde{\beta}}$ by differentiation, we deduce that $f_{\beta,\tilde{\beta}} \in \cP_{k - |\beta| - |\tilde{\beta}|}$, which yields that $f_{\beta} \in \cP_{k - |\beta|}$. This proves the desired result in case $k \le 4$. 

In case $k = 5$, we can proceed in the same way, until we reach $|\beta| = 0$, i.e. $f_{\beta} =: f_0$. Indeed, by proceeding as before and applying \autoref{prop:Liouville-higher-half-space-SR-1D} (with $k < 5$), we are in the situation that $f_{\gamma} \in \cP_{5 - |\gamma|}$ for any $0 < |\gamma| \le 4$ and also that $f_0$ is of the form \eqref{eq:rep-f-beta}, with $f_{0,\tilde{\gamma}} \in \cP_{5 - |\tilde{\gamma}|}$ for any $0 < |\tilde{\gamma}| \le 5 - |\beta| = 5$. Therefore, it remains to consider $|\tilde{\beta}| = 0$ and $f_{0,\tilde{\beta}} =: f_{0,0}$. For $f_{0,0}$ it holds
\begin{align*}
v_n \partial_{x_n} f_{0,0} - \partial_{v_n,v_n} f_{0,0} = p_{0,0} ~~ \text{ in } (0,\infty) \times \R
\end{align*}
for some $p_{0,0} \in \cP_{3}$. Moreover $f_{0,0}$ satisfies the specular reflection condition and $\Vert f_{0,0} \Vert_{L^{\infty}(\tilde{Q}_R(0))} \le C (1 + R)^{5 + \eps}$. Now, we have to apply \autoref{prop:Liouville-higher-half-space-SR-1D} for the first time with $k = 5$. This allows us to deduce that $f_0$ is of the form
\begin{align*}
f_0(x_n,v_n) = P(x_n,v_n) + m \mathcal{T}_{a^{n,n},3}(x_n,v_n)
\end{align*}
for some $P \in \cP_5$ satisfying the specular reflection condition $P(x',0,v',v_n) = P(x',0,v',-v_n)$, and $m \in \R$, which yields the desired result. 

Finally, note that if $p = q$, for some $q \in \cP_3$ depending only on $t,x',v'$, then we can immediately deduce that $f \in \cP_5$ by \autoref{lemma:Liouville-flip}. Moreover, if $p = c(v_n^3 - 2a^{n,n} x_n)$ for some $c \in \R$, then it must be $p_{0,0} = p + q$, for some polynomial $q$ that is of degree strictly less than $3$ in $x_n,v_n$ due to \eqref{eq:poly-sum-pos-times-halfspace} and \eqref{eq:rep-f-beta}. Hence, in this case it still holds $p_{0,0} \in \cP_3$ by \autoref{prop:Liouville-higher-half-space-SR-1D}.
\end{proof}

\begin{proof}[Proof of \autoref{thm2}]
It follows immediately from \autoref{thm:Liouville-higher-half-space-SR}.
\end{proof}

\section{Boundary regularity outside the grazing set}
\label{sec:reg-outside}

Recall that the kinetic boundary $\gamma := (-1,1) \times \partial \Omega \times \R^n$ is split into the following three parts $\gamma = \gamma_+ \cup \gamma_- \cup \gamma_0$, where
\begin{align*}
\gamma_{\pm} = \{ (t,x,v) \in (-1,1) \times \partial \Omega \times \R^n : \pm n_x \cdot v > 0 \}, ~~ \gamma_0 = \big( (-1,1) \times \partial \Omega \times \R^n \big) \setminus (\gamma_+ \cup \gamma_-),
\end{align*}
and $n_x \in \mathbb{S}^{n-1}$ denotes the outward unit normal vector of $\Omega$ at $x \in \partial \Omega$. Moreover, let us recall the notation $H_r(z_0) = \big((-1,1) \times \Omega \times \R^n \big) \cap Q_r(z_0)$, given some $r > 0$ and $z_0 \in \R \times \R^n \times \R^n$.

In this section we establish the boundary regularity for solutions to kinetic equations away from the grazing set $\gamma_{0}$. We will consider equations subject to in-flow boundary conditions in \autoref{prop:regularity-inflow} and to specular regulation boundary conditions in \autoref{prop:regularity-SR}. 

\subsection{Expansion at $\gamma_+$}

By the interior regularity results from \autoref{lemma:interior-reg}, in order to prove boundary regularity, it suffices to establish estimates on the asymptotic expansions of solutions near the boundary. In this section, we will prove such estimates near boundary points $z_0 \in \gamma_+$. The proof goes by a contradiction compactness argument, which heavily relies on the higher order Liouville theorem for negative times (see \autoref{prop:Liouville-higher-neg-times}). The interior and boundary regularity result in \autoref{lemma:interior-reg} and \autoref{lemma:boundary-reg} will be crucial in order to extract converging subsequences.

We refer the interested reader to Subsection \ref{subsec:strategy} for a more detailed description of the contradiction compactness argument that is used in order to prove the result. 

Note that the proof for $\gamma_+$ is independent of the boundary condition.

\begin{lemma}
\label{lemma:higher-reg-gamma_+}
Let $\Omega \subset \R^n$ be a convex $C^{1,1}$ domain and $z_0 \in \gamma_+$. Let $R \in (0,1]$ be such that $Q_{2R}(z_0) \cap \gamma_0 = \emptyset$. Let $k \in \N$ with $k \ge 3$, $\eps \in (0,1)$, and $a^{i,j},b,c,h \in C^{k+\eps-2}_{\ell}(H_R(z_0))$, and assume that $a^{i,j}$ satisfies \eqref{eq:unif-ell}.
Let $f$ be a weak solution to 
\begin{align*}
\partial_t f + v \cdot \nabla_x f + (-a^{i,j} \partial_{v_i,v_j})f &= - b \cdot \nabla_v f - c f + h ~~ \text{ in } H_R(z_0).
\end{align*}
Then, there exists $p_{z_0} \in \cP_k$ such that for any $r \in (0,\frac{R}{2}]$, and any $z \in H_{r}(z_0)$ it holds:
\begin{align}
\label{eq:higher-reg-gamma_+_simplified}
|f(z) - p_{z_0}(z)| \le C \left(\frac{r}{R}\right)^{k+\eps} \big( \Vert f \Vert_{L^{\infty}(H_R(z_0))} + R^{k+\eps} [ h ]_{C^{k+\eps-2}_{\ell}(H_R(z_0))} \big).
\end{align}
The constant $C$ depends only on $n,k,\eps,\lambda,\Lambda, \Omega$, $\Vert a^{i,j} \Vert_{C^{k+\eps-2}_{\ell}(H_R(z_0))}$, $\Vert b \Vert_{C^{k+\eps-2}_{\ell}(H_R(z_0))}$, and \\$\Vert c \Vert_{C^{k+\eps-2}_{\ell}(H_R(z_0))}$, but not on $z_0$ and $r,R$.
\end{lemma}

A crucial feature of the result is that the constants are independent of $z_0$. For this it is important to have convexity of $\Omega$ (by \autoref{lemma:boundary-reg}). When convexity is not assumed, we can still obtain \eqref{eq:higher-reg-gamma_+_simplified} by a very similar proof (fixing $z_0$) with a constant $C$, depending on $z_0$.

Note that we only assume $k \ge 3$ in order to write the operator in divergence form so that we can apply \autoref{lemma:boundary-reg}.

\begin{proof}

We prove the claim \eqref{eq:higher-reg-gamma_+_simplified} in several steps.

\textbf{Step 1:} We assume by contradiction that \eqref{eq:higher-reg-gamma_+_simplified} does not hold. Then, there exist $R_l \in (0,1]$, $f_l$, $h_l$, $a^{i,j}_{l}$, $b_l$, $c_l$, $z^0_{l}$ such that for any $l \in \N$ (after a normalization)
\begin{align}
\label{eq:blow-up-normalization}
\begin{split}
R_l^{-k-\eps} \left(\Vert f_l \Vert_{L^{\infty}(H_{R_l}(z^0_{l}))} + R_l^{k+\eps} [ h_l ]_{C^{k+\eps-2}_{\ell}(H_{R_l}(z^0_{l}))}\right) &\le 1, \\
 \Vert a^{i,j}_l \Vert_{C^{k+\eps-2}_{\ell}(H_{R_l}(z^0_{l}))} + \Vert b_l \Vert_{C^{k+\eps-2}_{\ell}(H_{R_l}(z^0_{l}))} + \Vert c_l \Vert_{C^{k+\eps-2}_{\ell}(H_{R_l}(z^0_{l}))} &\le \Lambda,
\end{split}
\end{align}
and also \eqref{eq:unif-ell} holds true with $\lambda$ for any $l \in \N$, such that the $f_l$ are solutions to 
\begin{align*}
\partial_t f_l + v \cdot \nabla_x f_l + (-a^{i,j}_l \partial_{v_i,v_j})f_l &= - b_l \cdot \nabla_v f_l - c_l f_l  + h_l ~~ \text{ in } H_{R_l}(z^0_{l}),
\end{align*}
but it holds
\begin{align}
\label{eq:contradiction-assumption}
\sup_{l \in \N} \inf_{p \in \cP_k} \sup_{ r \in (0,\frac{R_l}{2}] }  \frac{\Vert f_l - p \Vert_{L^{\infty}(H_{r}(z^0_{l}))}}{r^{k+\eps}} = \infty.
\end{align}

For any $r \in (0,1]$ we consider the $L^2(H_r(z^0_{l}))$ projections of $f_l$ over $\cP_k$ and denote them by $p_{l,r} \in \cP_k$. They satisfy the following properties
\begin{align}
\label{eq:orth}
\Vert f_l - p_{l,r} \Vert_{L^2(H_r(z^0_{l}))} &\le \Vert f_l - p \Vert_{L^2(H_r(z^0_{l}))} ~~ \forall p \in \cP_k,\\
\label{orth-2}
\int_{H_r(z^0_{l})} (f_l(z) - p_{l,r}(z)) p(z) \d z &= 0 ~ \qquad\qquad\qquad\qquad \forall p \in \cP_k.
\end{align}
Moreover, we introduce the following quantity for $r \in (0,\frac{1}{4}]$:
\begin{align*}
\theta(r) = \sup_{l \in \N} \sup_{\rho \in [r,\frac{R_l}{2}]} \rho^{-k-\eps} \Vert f_l - p_{l,\rho} \Vert_{L^{\infty}(H_{\rho}(z^0_{l}))}.
\end{align*}
Note that we have $\theta(r) = -\infty$ if $r \ge r_0 := \sup R_l/2$, but $\theta(r) > 0$ once $r < r_0$. We deduce from \autoref{lemma:L43} that
\begin{align}
\label{eq:theta-increasing}
\theta(r) \nearrow \infty \qquad \text{ as } \qquad r \searrow 0.
\end{align}
Indeed, if we had $\theta(r) \le C$ for some $C > 0$ as $r \searrow 0$, then for any $l \in \N$, and $\rho \in (0,\delta_l]$,
\begin{align*}
\Vert f_l - p_{l,\rho} \Vert_{L^{\infty}(H_{\rho}(z^0_{l}))} \le C \rho^{k+\eps},
\end{align*}
and by defining $\bar{f}_l(z) := f_l(z_l^0 \circ z)$ and $\bar{p}_{l,\rho}(z) := p_{l,\rho}(z_l^0 \circ z) \in \cP_k$, \eqref{eq:trafo-cylinders} would imply 
\begin{align*}
\Vert \bar{f}_l - \bar{p}_{l,\rho} \Vert_{L^{\infty}(H_{\rho}(0))} \le C \rho^{k+\eps},
\end{align*}
but by \autoref{lemma:L43} and applying \eqref{eq:trafo-cylinders} again, we deduce that for any $l \in \N$ there exists $p_{l,0}$ such that for any $\rho \in (0,\delta_l]$
\begin{align*}
\Vert f_l - p_{l,0} \Vert_{L^{\infty}(H_{\rho}(z^0_{l}))} \le C \rho^{k+\eps},
\end{align*}
which contradicts \eqref{eq:contradiction-assumption}.

Thus, \eqref{eq:theta-increasing} holds true and we can extract further subsequences $(r_m)_m$ and $(l_m)_m$ with $r_m \searrow 0$ and $r_m \le R_m/2$, where $R_m := R_{l_m}$, such that
\begin{align}
\label{eq:theta-claim}
\frac{\Vert f_{l_m} - p_{l_m,r_m} \Vert_{L^{\infty}(H_{r_m}(z_{l_m}^0))}}{r_m^{k+\eps} \theta(r_m)} \ge \frac{1}{2} ~~ \forall m \in \N, \qquad R_m r_m^{-1} \to \infty ~~ \text{ as } m \to \infty.
\end{align} 

While the existence of subsequences satisfying the first property is an immediate consequence of the definition of $\theta$, the verification of the second property relies on \eqref{eq:theta-increasing} and requires a little more work. To prove it, let us assume by contradiction that there is $N > 0$ such that $R_m r_m^{-1} \le N$. Then for any $l \in \N$ and any $\rho \in [r,R_l]$ it holds $\rho \ge R_l/N$ and therefore 
\begin{align*}
\rho^{-k-\eps} \Vert f_l - p_{l,\rho} \Vert_{L^{\infty}(H_{\rho}(z_l^0))} \le c(N) R_l^{-k-\eps} \Vert f_l \Vert_{L^{\infty}(H_{R_l}(z_l^0))} + c(N) R_l^{-k-\eps} \Vert p_{l,\rho} \Vert_{L^{\infty}(H_{R_l}(z_l^0))}.
\end{align*}
Note that the (averaged) $L^2$- and $L^{\infty}$-norm of polynomials are comparable, and therefore by H\"older's inequality and \eqref{eq:orth} (applied with $p \equiv 0$),
\begin{align*}
\Vert p_{l,\rho} \Vert_{L^{\infty}(H_{R_l}(z_l^0))} &\le R_l^{-n/2}\Vert p_{l,\rho} \Vert_{L^{2}(H_{R_l}(z_l^0))} \\
 &\le c R_l^{-n/2} \Vert f_l - p_{l,\rho} \Vert_{L^2(H_{R_l}(z_l^0))} + c  \Vert f_l \Vert_{L^{\infty}(H_{R_l}(z_l^0))} \le c \Vert f_l \Vert_{L^{\infty}(H_{R_l}(z_l^0))}.
\end{align*}
Thus, by the normalization condition \eqref{eq:blow-up-normalization}
\begin{align*}
\rho^{-k-\eps} \Vert f_l - p_{l,\rho} \Vert_{L^{\infty}(H_{\rho}(z_l^0))} \le c(N)R_l^{-k-\eps} \Vert f_l \Vert_{L^{\infty}(H_{R_l}(z_l^0))} \le c(N).
\end{align*}
But since $l \in \N$ and $\rho \in [r,R_l/2]$ were arbitrary, this implies
\begin{align*}
\theta(r) \le c(N),
\end{align*}
contradicting \eqref{eq:theta-increasing}. Thus, there must be subsequences such that $R_mr_m^{-1} \to \infty$, as claimed in \eqref{eq:theta-claim}.

Let us define for any $R > 0$ and $m \in \N$ the rescaled domains  $H_R^{(m)} := (T_{z_m^0,r_m} \times \R^n) \cap Q_{R}(0)$, where $T_{z_m^0,r_m}$ is defined as in \eqref{eq:Tr-def}, and consider the functions
\begin{align*}
g_m(z) : = 
%\frac{f_{j_m}(t_0 + r_m^2 t , x_0 + r_m^3 x + r_m^2 t v_0 , v_0 + r_mv) - p_{j_m,r_m}(t_0 + r_m^2 t , x_0 + r_m^3 x + r_m^2 t v_0 , v_0 + r_mv)}{r_m^{k+\eps} \theta(r_m)} \\
\frac{f_{l_m}(z_m^0 \circ S_{r_m} z ) - p_{l_m,r_m}(z_m^0 \circ S_{r_m} z)}{r_m^{k+\eps} \theta(r_m)},
\end{align*}
and set $h_m = h_{l_m}$ and $z^0_m = z^0_{l_m}$.
By construction, they satisfy for any $m \in \N$
\begin{align}
\label{eq:gm-prop-1}
\int_{H^{(m)}_1} g_m(z) p(z_m^0 \circ S_{r_m} z) \d z = 0 ~~ \forall p \in \cP_k.
\end{align}
Moreover, we claim that 
\begin{align}
\label{eq:gm-prop-2}
\Vert g_m \Vert_{L^{\infty}(H^{(m)}_1)} \ge \frac{1}{2}, \qquad \Vert g_m \Vert_{L^{\infty}(H^{(m)}_R)} \le c R^{k+\eps} ~~ \forall R \in \left[ 1, \frac{R_m}{2r_m} \right], ~~ \forall m \in \N.
\end{align}

The first property is immediate by construction. 
To see the second property, let us write
\begin{align*}
p_{l,r}(z_l^0 \circ z) =: \bar{p}_{l,r}(z) = \sum_{|\beta| \le k} \alpha^{(\beta)}_{l,r} t^{\beta_t} x_1^{\beta_{x_1}} \cdot \dots \cdot x_n^{\beta_{x_n}} v_1^{\beta_{v_1}} \cdot \dots \cdot v_n^{\beta_{v_n}} \in \cP_k
\end{align*}
for multi-indeces $\beta \in (\N \cup \{0\})^{1+2n}$ with $\beta = (\beta_t,\beta_x,\beta_v)$, and $\alpha^{(\beta)}_{l,r} \in \R$. Then, we have by the definition of $\theta$ for $r \le R_l/2$:
\begin{align*}
\Vert f_{l}(z_{l}^0 \circ \cdot) - \bar{p}_{l,r} \Vert_{L^{\infty}(H_r(0))} = \Vert f_{l} - p_{l,r} \Vert_{L^{\infty}(H_r(z_l^0))} \le \theta(r) r^{k+\eps},
\end{align*}
and therefore, \autoref{lemma:orth-proj-prop} implies that for any $R \le \frac{R_l}{2r}$:
\begin{align*}
\frac{\Vert f_{l}(z_{l}^0 \circ \cdot) - \bar{p}_{l,r} \Vert_{L^{\infty}(H_{Rr}(0))}}{r^{k+\eps} \theta(r)} \le c R^{k+\eps},
\end{align*}
which yields in particular that for any $R \le \frac{R_m}{2r_m}$:
\begin{align*}
\Vert g_m \Vert_{L^{\infty}(H_R^{(m)})} &= \frac{\Vert f_{l_m} - p_{l_m,r_m} \Vert_{L^{\infty}(H_{R r_m}(z_m^0))} }{r_m^{k+\eps} \theta(r_m)} = \frac{\Vert f_{l_m}(z_{m}^0 \circ \cdot) - \bar{p}_{l_m,r_m} \Vert_{L^{\infty}(H_{R r_m}(0))}}{r_m^{k+\eps} \theta(r_m)} \le c R^{k+\eps},
\end{align*}
and hence implies the second property in \eqref{eq:gm-prop-2}.

Note that \autoref{lemma:orth-proj-prop}, using also \eqref{eq:blow-up-normalization}, as well as the fact that $\frac{\theta(R_m)R_m^{\eps}}{\theta(r_m)} \to \infty$ (since $\theta(r_m) \to 0$, $r_m \le R_m/2$ and $r_m^{-1} R_m \to \infty$), also implies the following property for the coefficients of $\bar{p}_{l,r}$:
\begin{align}
\label{eq:coefficients-vanish}
 \frac{|\alpha^{(\beta)}_{l_m,r_m}|}{\theta(r_m)} \to 0 ~~ \text{ as } m \to \infty, ~~ \forall |\beta| \le k.
\end{align}

\textbf{Step 2:}
Next, we investigate the equation that is satisfied by $g_m$. Let us first introduce
\begin{align*}
\tilde{p}_{l_m,r_m} = (\partial_t + v \cdot \nabla_x) p_{l_m,r_m} \in \cP_{k-2}, \quad \tilde{p}_{l_m,r_m}^{(i,j)} = \partial_{v_i,v_j} p_{l_m,r_m} \in \cP_{k-2}, \quad \tilde{p}_{l_m,r_m}^{(i)}  = \partial_{v_i} p_{l_m,r_m} \in \cP_{k-1}.
\end{align*}
Then, it holds
\begin{align}
\label{eq:PDE-polynomial}
\begin{split}
\partial_t p_{l_m,r_m} & + v \cdot \nabla_x p_{l_m,r_m} + (-a^{i,j}_m \partial_{v_i,v_j}) p_{l_m,r_m} + b^{i}_m \partial_{v_i} p_{l_m,r_m} + c_m p_{l_m,r_m} \\
&= \tilde{p}_{l_m,r_m} + a^{i,j}_m \tilde{p}_{l_m,r_m}^{(i,j)} + b^i_m\tilde{p}_{l_m,r_m}^{(i)} + c_m p_{l_m,r_m}.
\end{split}
\end{align}

Note that the coefficients of these polynomials are all multiples of the coefficients of $p_{l_m,r_m}$, and in particular their $L^{\infty}$ norms are all comparable. 

By the equations that hold for $f_{l_m}, p_{l_m,r_m}$ as well as by \autoref{lemma:scaling} we deduce

\begin{align*}
\partial_t g_m + v \cdot \nabla_x g_m + \tilde{a}^{i,j}_m \partial_{v_i,v_j} g_m + \tilde{b}^{i}_m \partial_{v_i} g_m + \tilde{c}_m g_m = \tilde{h}_m ~~ \text { in } H^{(m)}_{R_mr_m^{-1}},
\end{align*}
where
\begin{align*}
\tilde{a}^{i,j}_m(z) = a^{i,j}_m(z_m^0 \circ S_{r_m} z), ~~ \tilde{b}^i_m(z) = r_m b^i_m(z_m^0 \circ S_{r_m} z), ~~ \tilde{c}_m(z) = r_m^2 c_m(z_m^0 \circ S_{r_m} z),
\end{align*}
and
\begin{align*}
\tilde{h}_m(z) = \frac{h_m(z_m^0 \circ S_{r_m} z ) - (\tilde{p}_{l_m,r_m} + a^{i,j}_m \tilde{p}_{l_m,r_m}^{(i,j)} + b^i_m\tilde{p}_{l_m,r_m}^{(i)} + c_m p_{l_m,r_m})(z_m^0 \circ S_{r_m} z )}{r_m^{k+\eps-2} \theta(r_m)}.
\end{align*}

Our next goal is to take the limit $m \to \infty$ in the equation for $g_m$. To do so, we first claim that
\begin{align}
\label{eq:hm-vanishes-Holder}
[\tilde{h}_m]_{C_{\ell}^{k+\eps-2}(H^{(m)}_{R_m r_m^{-1}})} \to 0 ~~ \text{ as } m \to \infty.
\end{align}

To see it, note that by the assumption on $h_m$ in \eqref{eq:blow-up-normalization} and the scaling properties of the kinetic H\"older norms (see \autoref{lemma:Holder-scaling}) we have
\begin{align*}
\frac{[h_m(z_m^0 \circ S_{r_m} \cdot)]_{C_{\ell}^{k+\eps-2}(H^{(m)}_{R_m r_m^{-1}})}}{r_m^{k+\eps-2} \theta(r_m)} = \frac{[h_m]_{C_{\ell}^{k+\eps-2}(H_{R_m}(z_m^0))}}{\theta(r_m)} \le \theta(r_m)^{-1} \to 0.
\end{align*}
Moreover, since $\tilde{p}_{l_m,r_m} \in \cP_{k-2}$:
\begin{align*}
\frac{[\tilde{p}_{l_m,r_m}(z_m^0 \circ S_{r_m} \cdot)]_{C_{\ell}^{k+\eps-2}(H^{(m)}_{R_m r_m^{-1}})}}{r_m^{k+\eps-2} \theta(r_m)} = \frac{[\tilde{p}_{l_m,r_m}]_{C_{\ell}^{k+\eps-2}(H_{R_m}(z_m^0))}}{\theta(r_m)} = 0.
\end{align*}
For the remaining terms $a^{i,j}_m \tilde{p}^{(i,j)}_{l_m,r_m}, b^i_m \tilde{p}^{(i)}_{l_m,r_m}, c_m p_{l_m,r_m}$, we proceed in the following way, which we explain in detail only for $a^{i,j}_m \tilde{p}^{(i,j)}_{l_m,r_m}$. We apply the product rule (see \autoref{lemma:product-rule}),  \autoref{lemma:der-Holder}, \eqref{eq:blow-up-normalization}, $R_m \le \frac{1}{4}$, and \eqref{eq:coefficients-vanish} to deduce
\begin{align*}
\frac{[a^{i,j}_m \tilde{p}^{(i,j)}_{l_m,r_m}(z_m^0 \circ S_{r_m} \cdot)]_{C_{\ell}^{k+\eps-2}(H^{(m)}_{R_m r_m^{-1}})}}{r_m^{k+\eps-2} \theta(r_m)} &= \frac{[a^{i,j}_m \tilde{p}^{(i,j)}_{l_m,r_m}]_{C_{\ell}^{k+\eps-2}(H_{R_m}(z_m^0))}}{\theta(r_m)} \\
%&\le \sum_{|\beta| = k-2} \sum_{\gamma \le \beta} \frac{[\partial_{\gamma} a^{i,j}_m \partial_{\beta - \gamma} \tilde{p}^{(i,j)}_{l_m,r_m} ]_{C_{\ell}^{\eps}(H_{R_m}(z_m^0))} }{\theta(r_m)} \\
%&\le \sum_{|\beta| = k-2} \sum_{\gamma \le \beta} \frac{[\partial_{\gamma} a^{i,j}_m]_{C_{\ell}^{\eps}(H_{R_m}(z_m^0))} \Vert  \partial_{\beta - \gamma}  \tilde{p}^{(i,j)}_{l_m,r_m} \Vert_{L^{\infty}(H_{R_m}(z_m^0))} }{\theta(r_m)} \\
%&\quad + \sum_{|\beta| = k-2} \sum_{\gamma \le \beta} \frac{[ \partial_{\beta - \gamma}  \tilde{p}^{(i,j)}_{l_m,r_m}]_{C_{\ell}^{\eps}(H_{R_m}(z_m^0))} \Vert \partial_{\gamma} a^{i,j}_m  \Vert_{L^{\infty}(H_{R_m}(z_m^0))} }{\theta(r_m)}  \\
&\le C \frac{ \Vert a^{i,j}_m \Vert_{C^{k+\eps-2}_{\ell}(H_{R_m}(z_m^0))}  \Vert \tilde{p}^{(i,j)}_{l_m,r_m} \Vert_{C^{k+\eps-2}_{\ell}(H_{R_m}(z_m^0))}}{\theta(r_m)} \\
&\le C \frac{\Vert p_{l_m,r_m} \Vert_{C^{k+\eps}_{\ell}(H_{R_m}(z_m^0))}}{\theta(r_m)} \\
&\le C \frac{\Vert \bar{p}_{l_m,r_m} \Vert_{L^{\infty}(H_{R_m}(0))}}{\theta(r_m)} \\
& \le C \sum_{|\beta| \le k} \frac{|\alpha_{l_m,r_m}^{(\beta)}|}{\theta(r_m)} \to 0.
\end{align*}

By analogous arguments for the terms coming from $\tilde{p}_{l_m,r_m}^{(i)}$, and $p_{l_m,r_m}$, we eventually deduce \eqref{eq:hm-vanishes-Holder}. 

Due to \eqref{eq:hm-vanishes-Holder} we can rewrite the equation for $g_m$ as follows:
\begin{align}
\label{eq:PDE-gm}
\partial_t g_m + v \cdot \nabla_x g_m + (-\tilde{a}^{i,j}_m \partial_{v_i,v_j}) g_m + \tilde{b}^{i}_m \partial_{v_i} g_m + \tilde{c}_m g_m = \tilde{h}_m = P_m + e_m ~~ \text { in } H^{(m)}_{R_m r_m^{-1}},
\end{align}
where $P_m \in \cP_{k-2}$, such that we have for any $R \le R_m r_m^{-1}$
\begin{align*}
\Vert e_m \Vert_{L^{\infty}(H^{(m)}_{R})} = \Vert \tilde{h}_m - P_m \Vert_{L^{\infty}(H^{(m)}_{R})} \le c R^{k+\eps - 2} [\tilde{h}_m]_{C_{\ell}^{k+\eps-2}(H^{(m)}_{R})} \to 0 ~~ \text{ as } m \to \infty.
\end{align*}
Moreover, note that by \autoref{lemma:osc-results}, \eqref{eq:gm-prop-2}, and \eqref{eq:hm-vanishes-Holder}, we have for any $R \le R_m r_m^{-1}$:
\begin{align}
\label{eq:hm-vanishes}
\Vert \tilde{h}_m \Vert_{L^{\infty}(H^{(m)}_{R})}  \le C \left( \Vert g_m \Vert_{L^{\infty}(H^{(m)}_{R})} + R^{k+\eps-2} [\tilde{h}_m]_{C^{k+\eps-2}_{\ell}(H^{(m)}_{R})} \right) \le C(R),
\end{align}
and therefore
\begin{align*}
\Vert P_m \Vert_{L^{\infty}(H^{(m)}_{R})} \le \Vert \tilde{h}_m - P_m \Vert_{L^{\infty}(H^{(m)}_{R})} + \Vert \tilde{h}_m \Vert_{L^{\infty}(H^{(m)}_{R})} \le C R^{k+\eps - 2}[\tilde{h}_m]_{C^{k+\eps-2}_{\ell}(H^{(m)}_{R})} + C(R) \le C(R).
\end{align*}
Note that $C(R)$ is independent of $m$, since we can apply \autoref{lemma:osc-results} to the rescaled function $\bar{h}_m(z) = \tilde{h}_m(z_m^0 \circ S_{r_m} z)$ first, and use \autoref{lemma:Holder-scaling} to deduce the estimate for $\tilde{h}_m$. 

Hence, we conclude that, up to a subsequence, $P_m \to p_0 \in \cP_k$, and therefore
\begin{align}
\label{eq:hm-convergence}
\Vert \tilde{h}_m - p_0 \Vert_{L^{\infty}(H^{(m)}_{R})} \to 0
\end{align}
as $m \to \infty$ for any $R > 0$. Moreover, by \eqref{eq:blow-up-normalization} and the definitions of $\tilde{a}^{i,j}_m$, $\tilde{b}^i_m$, and $\tilde{c}_m$, we have (up to a subsequence)
\begin{align*}
\Vert \tilde{a}^{i,j}_m - a^{i,j} \Vert_{L^{\infty}(H^{(m)}_R)} + \Vert \tilde{b}^i_m \Vert_{L^{\infty}(H^{(m)}_R)} + \Vert \tilde{c}_m \Vert_{L^{\infty}(H^{(m)}_R)} \to 0
\end{align*}
for some constant matrix $(a^{i,j})$ satisfying \eqref{eq:unif-ell}, and for any $R > 0$.  While the convergence of the lower order coefficients is obvious, note that for the convergence of $\tilde{a}^{i,j}_m \to a^{i,j}$, we are using that
\begin{align*}
\Vert \tilde{a}^{i,j}_m - \tilde{a}^{i,j}_m(0) \Vert_{L^{\infty}(H^{(m)}_R)} \to 0 ~~ \forall R > 0
\end{align*}
by \eqref{eq:blow-up-normalization}, and that the sequence of constant matrices $(\tilde{a}^{i,j}_m(0))_m$ is bounded by \eqref{eq:blow-up-normalization}, and therefore must also converge (up to a subsequence) to another constant matrix $(a^{i,j})$.

Moreover, note that $T_{z_m^0,r_m} \times \R^n \to (-\infty,0) \times \R^n \times \R^n$ by \autoref{lemma:domains-convergence}. Note that we are able to apply \autoref{lemma:domains-convergence} to $z_m^0 \in \gamma_+$ since by assumption it holds $Q_{2R_m}(z_m^0) \cap \gamma_0 = \emptyset$, which yields $|v_m^0 \cdot n_{x_m^0}| \ge 2 R_m$, and therefore, by \eqref{eq:theta-claim}, it follows $|v_m^0 \cdot n_{x_m^0}| r_m^{-1} \ge 2 R_m r_m^{-1} \to \infty$.

Thus, by the interior regularity from \autoref{lemma:interior-reg} and using again \autoref{lemma:osc-results}, we have for any kinetic cylinder $Q_1(z_1) \subset (-\infty,0] \times \R^n \times \R^n$ (for $m$ large enough) using also \eqref{eq:hm-vanishes-Holder}
\begin{align*}
\Vert g_m \Vert_{C^{k+\eps}(Q_{1/2}(z_1))} \le C \big( \Vert g_m \Vert_{L^{\infty}(Q_1(z_1))} + [ \tilde{h}_m ]_{C^{k+\eps-2}_{\ell}(Q_1(z_1))}  \big) \le C,
\end{align*}
and therefore by the Arzel\`a-Ascoli theorem it holds that $g_m \to g$ locally uniformly in $(-\infty,0) \times \R^n \times \R^n$. Moreover, by the aforementioned estimate, the $g_m$ are classical (i.e. pointwise) solutions to \eqref{eq:PDE-gm} and it holds
\begin{align*}
\partial_t g_m + v \cdot \nabla_x g_m + (-\tilde{a}^{i,j}_m \partial_{v_i,v_j}) g_m + \tilde{b}^{i}_m \partial_{v_i} g_m + \tilde{c}_m g_m  \to \partial_t g + v \cdot \nabla_x g + (-a^{i,j} \partial_{v_i,v_j}) g
\end{align*}
locally uniformly in $(-\infty,0) \times \R^n \times \R^n$. Thus, by combining \eqref{eq:PDE-gm} and \eqref{eq:hm-convergence} we obtain that $g$ solves the following limiting equation in the classical sense
\begin{align}
\label{eq:g-equation}
\partial_t g + v \cdot \nabla_x g + (-a^{i,j} \partial_{v_i,v_j})g = p_0 ~~ \text{ in } (-\infty , 0) \times \R^n \times \R^n
\end{align}
for some polynomial $p_0 \in \cP_{k-2}$ and some constant matrix $(a^{i,j})$.

Moreover, by taking the second property in \eqref{eq:gm-prop-2} to the limit, and using that $R_m r_m^{-1} \to \infty$, we obtain 
\begin{align}
\label{eq:g-growth}
\Vert g \Vert_{L^{\infty}(Q_R)} \le c R^{k+\eps} ~~ \forall R \ge 1.
\end{align}

\textbf{Step 3:} 
Next, we claim that it also holds 
\begin{align}
\label{eq:bdry-conv-higher-+}
\Vert g_m - g \Vert_{L^{\infty}(H^{(m)}_1)} \to 0 ~~ \text{ as } m \to \infty.
\end{align}

To prove it, we recall the equation that is satisfied by $f_{l_m}$, as well as \eqref{eq:PDE-polynomial}, which yields
\begin{align*}
\partial_t (f_{l_m} - p_{l_m,r_m}) & + v \cdot \nabla_x (f_{l_m} - p_{l_m,r_m}) + (-a^{i,j}_m \partial_{v_i,v_j}) (f_{l_m} - p_{l_m,r_m}) \\
&= - b_m^i \partial_{v_i} (f_{l_m} - p_{l_m,r_m}) - c_m (f_{l_m} - p_{l_m,r_m}) \\
&\quad  +  h_{l_m} - (\tilde{p}_{l_m,r_m} + a_m^{i,j} \tilde{p}^{(i,j)}_{l_m,r_m} + b_m^i \tilde{p}^{(i)}_{l_m,r_m} + c_m p_{l_m,r_m} ) ~~ \text{ in } H_{R_m}(z_m^0).
\end{align*}
Next, we recall the definition of $g_m$ and apply the boundary regularity for kinetic equations from \autoref{lemma:boundary-reg} to $f_{l_m} - p_{l_m,r_m}$, using that $\gamma_0 \cap Q_{2 R_m}(z_m^0) = \emptyset$ by construction (at least for $m \in \N$ large enough). This yields for some $\alpha \in (0,1)$ by \autoref{lemma:Holder-scaling}, \autoref{lemma:boundary-reg}, and recalling also the definitions of $g_m$ and $\tilde{h}_m$ and using the growth bound in \eqref{eq:gm-prop-2}, as well as \eqref{eq:hm-vanishes}:
\begin{align*}
[ g_m ]_{C^{\alpha}(H^{(m)}_{1})} &= r_m^{\alpha} r_m^{-k-\eps} \theta(r_m)^{-1} \left[ f_{l_m} - p_{l_m,r_m} \right ]_{C^{\alpha}(H_{r_m}(z_m^0))} \\
&\le C r_m^{-k-\eps} \theta(r_m)^{-1} \Big( \Vert f_{l_m} - p_{l_m,r_m} \Vert_{L^{\infty}(H_{2r_m}(z_m^0))} \\
&\qquad\qquad + r_m^2 \Vert h_{l_m} - (\tilde{p}_{l_m,r_m} + a_m^{i,j} \tilde{p}^{(i,j)}_{l_m,r_m} + b_m^i \tilde{p}^{(i)}_{l_m,r_m} + c_m p_{l_m,r_m} ) \Vert_{L^{\infty}(H_{2 r_m}(z_m^0))} \Big) \\
&\le C \Vert g_m \Vert_{L^{\infty}(H^{(m)}_{2})} + C \Vert \tilde{h}_m \Vert_{L^{\infty}(H_{2}^{(m)})} \le C.
\end{align*}

Hence, we deduce that
\begin{align*}
\Vert g_m \Vert_{C^{\alpha}(H_1^{(m)})} \le C,
\end{align*}
which by the Arzel\`a-Ascoli theorem implies \eqref{eq:bdry-conv-higher-+}, as desired.

%For $I_1$, using the definition of $g_m$ and the growth bound in \eqref{eq:gm-prop-2}, we obtain
%\begin{align*}
%I_1 \le \Vert g_m \Vert_{L^{\infty}(H^{(m)}_2)} \le C.
%\end{align*}
%For $I_2$ we use \eqref{eq:blow-up-normalization}, \eqref{eq:coeff-poly-vanish}, and the fact that $\tilde{p}_{l_m,r_m} \in \cP_{k-2}$ to deduce
%\begin{align*}
%I_2 \le C \theta(r_m)^{-1}  [h_{l_m} - \tilde{p}_{l_m,r_m}]_{C_{\ell}^{k+\eps-2}(H_{2r_m}(z_m^0))} =  C \theta(r_m)^{-1}[h_{l_m}]_{C_{\ell}^{k+\eps-2}(H_{2r_m}(z_m^0))} \le C \theta(r_m)^{-1} \le C.
%\end{align*}
%
%For $I_2$ we recall the definition of $\tilde{h}_m$ and use \autoref{lemma:Holder-scaling} as well as \eqref{eq:hm-vanishes} to deduce
%\begin{align*}
%I_2 \le \Vert \tilde{h}_m \Vert_{L^{\infty}(H_2^{(m)})} \le C.
%\end{align*}

\textbf{Step 4:} 
By \eqref{eq:g-equation} and \eqref{eq:g-growth}, we can apply the higher order Liouville theorem (see \autoref{prop:Liouville-higher-neg-times}) to $g$ and obtain that $g \in \cP_k$. 
Moreover, by \eqref{eq:bdry-conv-higher-+}, we can take the limit in \eqref{eq:gm-prop-2} and deduce that $g$ also satisfies the following property
\begin{align}
\label{eq:g-prop-2}
\Vert g \Vert_{L^{\infty}(H_1)} \ge \frac{1}{2}.
\end{align}
Moreover, using again \eqref{eq:bdry-conv-higher-+} it follows upon choosing 
\begin{align*}
p(t,x,v) = g(r_m^{-2}(t-t_m^0), r_m^{-3}(x - x_m^0 - (t-t_m^0)v_m^0 ) , r_m^{-1}(v-v_m^0)) \in \cP_k,
\end{align*}
i.e. choosing $p \in \cP_k$ such that
in \eqref{eq:gm-prop-1} that $p(z_0 \circ S_{r_m} z) = g_m(z)$,
\begin{align*}
0 = \int_{H^{(m)}_1} g_m(z) p(z_m^0 \circ S_{r_m} z) \d z = \int_{H^{(m)}_1} g_m(z) g(z) \d z \to \int_{H_1} g^2(z) \d z,
\end{align*}
where we used again \eqref{eq:bdry-conv-higher-+}. Hence, it must be $g \equiv 0$ in $\tilde{Q}_1$, which contradicts \eqref{eq:g-prop-2}. Hence, we have shown \eqref{eq:higher-reg-gamma_+_simplified}, and thus the proof is complete.
\end{proof}

\subsection{Expansion at $\gamma_-$}

The next step is to deduce an expansion at boundary points $z_0 \in \gamma_-$. We will first prove it for solutions satisfying an in-flow boundary condition (see \autoref{lemma:higher-reg-gamma_--Dirichlet}) by proceeding in a similar way as in the proof of \autoref{lemma:higher-reg-gamma_+} and then derive the expansion for solutions satisfying a specular reflection boundary condition in a second step (see \autoref{lemma:higher-reg-gamma_-SR}). 

An important difference in the proof of \autoref{lemma:higher-reg-gamma_--Dirichlet} is the use of the Liouville theorem for positive times (see \autoref{prop:Liouville-higher-pos-times}) instead of its counterpart for negative times.

The proof of the expansion at $\gamma_-$ is much more subtle (compared to the one at $\gamma_+$) since we need to track the boundary condition throughout the blow-up procedure. In general domains, this leads to an explicit dependence of the constant on powers of $|v_0|$. We are also able to show that the constant is independent of $|v_0|$ if $\Omega = \{ x_n = 0 \}$ (see \autoref{lemma:higher-reg-gamma_--Dirichlet-flat}).

\begin{lemma}
\label{lemma:higher-reg-gamma_--Dirichlet}
Let $k \in \N$ with $k \ge 3$, $\eps \in (0,1)$ be such that $\frac{k+\eps}{3} \not \in \N$. Let $\Omega \subset \R^n$ be a convex $C^{\frac{k+\eps}{3}}$ domain and $z_0 \in \gamma_-$. Let $R \in (0,1]$ be such that $Q_{2R}(z_0) \cap \gamma_0 = \emptyset$, and $a^{i,j},b,c,h \in C^{k+\eps-2}_{\ell}(H_R(z_0))$, and assume that $a^{i,j}$ satisfies \eqref{eq:unif-ell}.
Let $F \in C^{k+\eps}_{\ell}(H_R(z_0))$ and $f$ be a weak solution to 

\begin{equation*}
\left\{\begin{array}{rcl}
\partial_t f + v \cdot \nabla_x f + (-a^{i,j} \partial_{v_i,v_j})f &=& - b \cdot \nabla_v f - c f + h ~~ \text{ in } H_R(z_0), \\
f &=& F ~~~~ \qquad\qquad\qquad \text{ on } \gamma_- \cap Q_R(z_0).
\end{array}\right.
\end{equation*}
Then, there exists $p_{z_0} \in \cP_k$ such that for any $r \in (0,\frac{R}{2}]$ and any $z \in H_{r}(z_0)$ it holds:
\begin{align}
\label{eq:higher-reg-gamma_-_simplified}
\begin{split}
|f(z) - p_{z_0}(z)| \le C (1 + |v_0|)^{\frac{k+\eps}{2}}  \left( \frac{r}{R} \right)^{k+\eps} & \Big( \Vert f \Vert_{L^{\infty}(H_R(z_0))} + R^{k+\eps} [ F ]_{C_{\ell}^{k+\eps}(H_R(z_0))}  \\
&\qquad\qquad\quad\quad ~~ + R^{k+\eps} [ h ]_{C^{k+\eps-2}_{\ell}(H_R(z_0))} \Big).
\end{split}
\end{align}
The constant $C$ depends only on $n,k,\eps,\lambda,\Lambda, \Omega$, $\Vert a^{i,j} \Vert_{C^{k+\eps-2}_{\ell}(H_R(z_0))}, \Vert b \Vert_{C^{k+\eps-2}_{\ell}(H_R(z_0))}$, and \\$\Vert c \Vert_{C^{k+\eps-2}_{\ell}(H_R(z_0))}$, but not on $z_0$ and $r,R$.
\end{lemma}

\begin{proof}

We prove the claim \eqref{eq:higher-reg-gamma_-_simplified} in several steps. The proof follows the overall strategy of the proof of \autoref{lemma:higher-reg-gamma_+}, however there are several additional difficulties arising from the boundary condition. 

First, we observe that we can assume without loss of generality that $F = 0$. Indeed, if $F$ is not zero, then we can consider $\bar{f} := f-F$ instead of $f$, which solves an equation of the form
\begin{equation*}
\left\{\begin{array}{rcl}
\partial_t \bar{f} + v \cdot \nabla_x \bar{f} + (-a^{i,j} \partial_{v_i,v_j})\bar{f} &=& - b \cdot \nabla_v \bar{f} - c \bar{f} + \bar{h} ~~ \text{ in } H_R(z_0), \\
\bar{f} &=& 0 ~~~~~~ \qquad\qquad\qquad \text{ on } \gamma_- \cap Q_R(z_0),
\end{array}\right.
\end{equation*}
where 
\begin{align*}
\bar{h} := h - \partial_t F - v \cdot \nabla_x F - (-a^{i,j} \partial_{v_i,v_j})F + b \cdot \nabla_v F + c F,
\end{align*}
and by \autoref{lemma:product-rule} and \autoref{lemma:Holder-interpol}, it holds
\begin{align*}
[\bar{h}]_{C^{k+\eps-2}_{\ell}(H_R(z_0))} \le C (1+|v_0|) \left( \Vert f \Vert_{L^{\infty}(H_R(z_0))} + R^{k+\eps}[F]_{C^{k+\eps-2}_{\ell}(H_R(z_0))} \right).
\end{align*}

\textbf{Step 1:} 
We assume by contradiction that \eqref{eq:higher-reg-gamma_-_simplified} does not hold. Then, there are sequences $R_l \in (0,1]$ and $f_l$, $h_l$, $a^{i,j}_{l}$, $b_l$, $c_l$, $z^0_{l}$ such that for any $l \in \N$ \eqref{eq:unif-ell} holds true with $\lambda$ for any $l \in \N$, as well as
\begin{align}
\label{eq:blow-up-normalization_-}
R_l^{-k-\eps} \Vert f_l \Vert_{L^{\infty}(H_{R_l}(z_l^0))} + [ h_l ]_{C^{k+\eps-2}_{\ell}(H_{R_l}(z_l^0))} \le 1,
\end{align}
such that the $f_l$ are solutions to
 
\begin{equation}
\label{eq:fl-sol-gamma_-}
\left\{\begin{array}{rcl}
\partial_t f_l + v \cdot \nabla_x f_l + (-a^{i,j}_l \partial_{v_i,v_j})f_l &=& - b_l \cdot \nabla_v f_l - c_l f_l  + h_l ~~ \text{ in } H_1(z^0_{l}), \\
f_l &=& 0 ~~~\qquad\qquad\qquad\qquad ~ \text{ on } \gamma_- \cap H_{R_l}(z^0_l)
\end{array}\right.
\end{equation}
but it holds with $\kappa = \frac{k+\eps}{2}$
\begin{align}
\label{eq:blow-up-assumption_-}
\sup_{l \in \N} \inf_{p \in \cP_k} \sup_{ r \in (0,\frac{R_l}{2}] }  \frac{\Vert f_l - p \Vert_{L^{\infty}(H_{r}(z^0_{l}))}}{(1 + |v_l^0|)^{\kappa} r^{k+\eps}} = \infty.
\end{align}
For $r \in (0,1]$ we consider the $L^2(H_r(z^0_{l}))$ projections of $f_l$ over $\cP_k$ and denote them by $p_{l,r} \in \cP_k$.
Moreover, we introduce the quantity
\begin{align*}
\theta(r) = \sup_{l \in \N} \sup_{\rho \in [r,\frac{R_l}{2}]} (1 + |v_l^0|)^{-\kappa} \rho^{-k-\eps} \Vert f_l - p_{l,\rho} \Vert_{L^{\infty}(H_{\rho}(z^0_{l}))},
\end{align*}
and we deduce from \autoref{lemma:L43} that $\theta(r) \nearrow \infty$ as $r \searrow 0$ in the same way as in the proof of \autoref{lemma:higher-reg-gamma_+}. Thus, in analogy to \eqref{eq:theta-claim}, we can extract further subsequences $(r_m)_m$ and $(l_m)_m$ with $r_m \searrow 0$, such that
\begin{align*}
\frac{\Vert f_{l_m} - p_{l_m,r_m} \Vert_{L^{\infty}(H_{r_m}(z_{l_m}^0))}}{r_m^{k+\eps} (1 + |v_{l_m}^0|)^{\kappa} \theta(r_m)} \ge \frac{1}{2} ~~ \forall m \in \N, \qquad (1 + |v_{l_m}^0|)^{-\bar{\kappa}} R_m r_m^{-1} \to \infty ~~ \text{ as } m \to \infty
\end{align*} 
for any $0 < \bar{\kappa} \le \kappa (k+\eps)^{-1}$. In fact, to see the second property, note that if $(1 + |v_{l_m}^0|)^{-\bar{\kappa}} R_m r_m^{-1} \le N$, then for any $l \in \N$ and $\rho \in [r_l,R_l]$ it holds $\rho \ge R_l (1+|v_l^0|)^{-\bar{\kappa}} /N$, and therefore, using also \eqref{eq:blow-up-normalization_-}
\begin{align*}
\rho^{-k-\eps} (1 + |v_l^0|)^{-\kappa} \Vert f_l - p_{l,\rho} \Vert_{L^{\infty}(H_{\rho}(z_l^0))} \le c(N) R_l^{-k-\eps} (1 + |v_l^0|)^{\bar{\kappa}(k+\eps)-\kappa} \Vert f_l \Vert_{L^{\infty}(H_{R_l}(z_l^0))} \le C(N).
\end{align*}
This would imply boundedness of $\theta(r)$, a contradiction.
Note that, in particular, we obtain
\begin{align}
\label{eq:rm-properties}
R_m r_m^{-1} \to \infty, \qquad |v^0_{l_m}| r_m^{2} \to 0,
\end{align}
and we can choose $1/\bar{\kappa} = 2$, since $\kappa = (k+\eps)/2$.

Moreover, we define for any $R > 0$ and $m \in \N$ the rescaled domains  $H_R^{(m)} := (T_{z_m^0,r_m} \times \R^n) \cap Q_{R}(0)$ and consider the functions
\begin{align*}
g_m(z) := \frac{f_{l_m}(z_m^0 \circ S_{r_m} z ) - p_{l_m,r_m}(z_m^0 \circ S_{r_m} z)}{r_m^{k+\eps} (1 + |v_m^0|)^{\kappa} \theta(r_m)}, \qquad
 P_m(z) = -\frac{p_{l_m,r_m}(z_m^0 \circ S_{r_m} z)}{r_m^{k+\eps}(1 + |v_m^0|)^{\kappa} \theta(r_m)}.
\end{align*}
By construction, and using the same arguments as in the proof of \autoref{lemma:higher-reg-gamma_+}, as well as \autoref{lemma:orth-proj-prop}, we have for any $m \in \N$
\begin{align}
\label{eq:gm-prop-1_-}
\int_{H^{(m)}_1} g_m(z) p(z_m^0 \circ S_{r_m} z) \d z &= 0 ~~ \forall p \in \cP_k\\
\label{eq:gm-prop-2_-}
\Vert g_m \Vert_{L^{\infty}(H^{(m)}_1)}& \ge \frac{1}{2}, \qquad \Vert g_m \Vert_{L^{\infty}(H^{(m)}_R)} \le c R^{k+\eps} ~~ \forall R \in \left[1 , \frac{R_m}{2 r_m}\right], ~~ \forall m \in \N.
\end{align}

\textbf{Step 2:} By the equations that hold for $f_{l_m}, p_{l_m,r_m}$ and \autoref{lemma:scaling} we deduce
\begin{equation}
\label{eq:gm-PDE_-}
\left\{\begin{array}{rcl}
\partial_t g_m + v \cdot \nabla_x g_m + (-\tilde{a}^{i,j}_m \partial_{v_i,v_j}) g_m + \tilde{b}^{i}_m \partial_{v_i} g_m + \tilde{c}_m g_m &=& \tilde{h}_m ~~ \text { in } H^{(m)}_{R_m r_m^{-1}},\\
g_m &=& P_m ~~ \text{ on } \gamma_-^{(m)} \cap Q_{R_m r_m^{-1}}
\end{array}\right.
\end{equation}
where we defined $\gamma_-^{(m)} = \{ (t,x,v) : t \in (r_m^{-2}(-1-t_m^0 ) , r_m^{-2}(1-t_m^0)], x \in \partial \Omega_{z_m^{0},r_m}(t), (z_m^0 \circ S_{r_m} z) \in \gamma_- \}$, and $\tilde{a}^{i,j}_m, \tilde{b}^{i}_m, \tilde{c}_m, \tilde{h}_m$ are defined as in the proof of \autoref{lemma:higher-reg-gamma_+} with the exception that $\tilde{h}_m$ also has to be divided by $(1 + |v_m^0|)^{\kappa}$. Moreover, by the same arguments as in the proof of \autoref{lemma:higher-reg-gamma_+}, we have that $\tilde{h}_m$ satisfies \eqref{eq:hm-vanishes-Holder}, \eqref{eq:hm-vanishes}, and \eqref{eq:hm-convergence} for some $p_0 \in \cP_{k-2}$.

In addition, note that $T_{z_m^0,r_m} \to (0,\infty) \times \R^n$ by \autoref{lemma:domains-convergence} (here we used in particular that by $Q_{2R_m}(z_m^0) \cap \gamma_0 = \emptyset$ and since $R_mr_m^{-1} \to \infty$ by \eqref{eq:rm-properties}, it follows $|v_m^0 \cdot n_{x_m^0}| r_m^{-1} \ge 2 R_m r_m^{-1} \to \infty$,  which allows us to apply \autoref{lemma:domains-convergence}).

Thus, by the interior regularity from \autoref{lemma:interior-reg} and using again \autoref{lemma:osc-results}, we have (for $m$ large enough), using also \eqref{eq:hm-vanishes-Holder}, for any kinetic cylinder $Q_1(z_1) \subset (0,\infty) \times \R^n \times \R^n$ 
\begin{align*}
\Vert g_m \Vert_{C^{k+\eps}_{\ell}(Q_{1/2}(z_1))} \le C \big( \Vert g_m \Vert_{L^{\infty}(Q_1(z_1))} + [ \tilde{h}_m ]_{C^{k+\eps-2}_{\ell}(Q_1(z_1))}  \big) \le C.
\end{align*}
Therefore by the Arzel\`a-Ascoli theorem it holds that $g_m \to g$ and
\begin{align*}
\partial_t g_m + v \cdot \nabla_x g_m + (-\tilde{a}^{i,j}_m \partial_{v_i,v_j}) g_m + \tilde{b}^{i}_m \partial_{v_i} g_m + \tilde{c}_m g_m \to \partial_t g + v \cdot \nabla_x g + (- a^{i,j} \partial_{v_i,v_j}) g 
\end{align*}
 locally uniformly in $(0,\infty) \times \R^n \times \R^n$. Hence, combining \eqref{eq:PDE-gm} and \eqref{eq:hm-convergence}, we obtain that $g$ solves the following limiting equation (in the classical sense)
\begin{align}
\label{eq:g-equation_-}
\partial_t g + v \cdot \nabla_x g + (-a^{i,j} \partial_{v_i,v_j})g = p_0 ~~ \text{ in } (0,\infty) \times \R^n \times \R^n
\end{align}
for some polynomial $p_0 \in \cP_{k-2}$ and some constant matrix $(a^{i,j})$. Taking the second property in \eqref{eq:gm-prop-2_-} to the limit, using that $R_m r_m^{-1} \to \infty$, we also obtain 
\begin{align*}
\Vert g \Vert_{L^{\infty}(Q_R)} \le c R^{k+\eps} ~~ \forall R \ge 1.
\end{align*}

\textbf{Step 3:} We claim that for any $z_1 = (0,x_1,v_1) \in \R^{1+2n}$ it holds
\begin{align}
\label{eq:bdry-conv-higher--}
%\Vert \tilde{F}_m - P_0 \Vert_{L^{\infty}(H_1^{(m)}(z_1) \cap \gamma_-^{(m)})} \to 0, \qquad 
\Vert g_m - g \Vert_{L^{\infty}(H^{(m)}_1(z_1))} \to 0 ~~ \text{ as } m \to \infty
\end{align} 
%for some $P_0 \in \cP_k$. 
We will use this property in Step 5 in order to deduce that $g(0,\cdot,\cdot) \in \cP_k$, which is needed to apply the Liouville theorem in \autoref{prop:Liouville-higher-pos-times}.

To prove it, we recall from \eqref{eq:trafo-cylinders-2} and \autoref{lemma:kinetic-inclusion} that for any $z_1 = (0,x_1,v_1)$ and $R \ge 1$ it holds
\begin{align}
\label{eq:inclusion-gamma_-_higher}
z \in Q_1(z_1) \Leftrightarrow (z_m^0 \circ S_{r_m} z) \in Q_{r_m}(z_m^0 + S_{r_m} z_1), \qquad Q_{Rr_m}(z_m^0 + S_{r_m} z_1) \subset Q_{2Rr_m(1 + |z_1|)}(z_m^0).
\end{align}
%Next, note that for the second property in 
To show \eqref{eq:bdry-conv-higher--}, by the Arzel\`a-Ascoli theorem, it suffices to prove
\begin{align}
\label{eq:bdry-conv-higher---help}
\Vert g_m \Vert_{C^{\alpha}(\tilde{H}_1^{(m)}(z_1))} \le C
\end{align}
for a constant $C$ and $\alpha \in (0,1)$ that are independent of $m$, (at least when $m$ is large), but might depend on $z_1$.
To do so, recalling the definition of $g_m$, and using also the $C^{\alpha}$ boundary regularity from \autoref{lemma:boundary-reg} on $f_{l_m} - p_{l_m,r_m}$ in $H_{2r_m(1 + |z_1|)}(z_m^0)$ for $m$ large enough, we obtain
\begin{align*}
[g_m]_{C^{\alpha}(H^{(m)}_{1}(z_1))} &= r_m^{\alpha}  \left[ \frac{f_{l_m} - p_{l_m,r_m}}{r_m^{k+\eps} (1 + |v_m^0|)^{\kappa} \theta(r_m)} \right ]_{C^{\alpha}(H_{r_m}(z_m^0 + S_{r_m} z_1))} \\
&\le r_m^{\alpha-k-\eps} (1 + |v_m^0|)^{-\kappa} \theta(r_m)^{-1}  [f_{l_m} - p_{l_m,r_m}]_{C^{\alpha}(H_{2r_m(1 + |z_1|)}(z_m^0))} \\
&\le C r_m^{-k-\eps} (1 + |v_m^0|)^{-\kappa} \theta(r_m)^{-1} \big( \Vert f_{l_m} - p_{l_m,r_m} \Vert_{L^{\infty}(H_{4r_m(1 + |z_1|)}(z_m^0))} \\
& \qquad \qquad +  r_m^{\alpha} [ p_{l_m,r_m} ]_{C^{\alpha}(\gamma_- \cap Q_{4r_m(1 + |z_1|)}(z_m^0))} \\
&\qquad \qquad + r_m^2 \Vert h_{l_m} - (\tilde{p}_{l_m,r_m} + a_m^{i,j} \tilde{p}^{(i,j)}_{l_m,r_m} + b_m^i \tilde{p}^{(i)}_{l_m,r_m} + c_m p_{l_m,r_m} ) \Vert_{L^{\infty}(H_{4 r_m (1 + |z_1|)}(z_m^0))} \big) \\
&= C(I_1 + I_2 + I_3).
\end{align*}
For $I_1$ we compute
\begin{align*}
I_1 \le C \frac{\Vert f_{l_m} - p_{l_m,r_m} \Vert_{L^{\infty}(H_{4r_m(1+|z_1|)}(z_m^0))}}{r_m^{k+\eps} (1 + |v_m^0|)^{\kappa} \theta(r_m)} \le C \Vert g_m \Vert_{L^{\infty}(H^{(m)}_{4(1+|z_1|)})} \le C,
\end{align*}
where we used the second property in \eqref{eq:gm-prop-2_-} and that by \eqref{eq:trafo-cylinders} it holds for any $R > 0$:
\begin{align*}
z \in Q_{r_mR}(0) ~~ \Leftrightarrow z_0 \circ S_{r_m} z \in Q_{R}(z_m^0).
\end{align*}

For $I_2$, we observe first that by \autoref{lemma:osc-bdry-data}, the second property in \eqref{eq:gm-prop-2_-}, and \eqref{eq:hm-vanishes-Holder}, \eqref{eq:hm-vanishes}, and since $P_m \in \cP_k$, we deduce for any $R \ge 1$, if $m$ is large enough,
\begin{align}
\label{eq:Pm-bounded}
\Vert P_m \Vert_{L^{\infty}(\gamma_-^{(m)} \cap Q_R(0))} 
%&\le \Vert \tilde{F}_m - P_m \Vert_{L^{\infty}(H_R^{(m)} \cap \gamma_-^{(m)})} + \Vert \tilde{F}_m \Vert_{L^{\infty}(H_R^{(m)} \cap \gamma_-^{(m)})} \\
&\le C [ P_m ]_{C_{\ell}^{k+\eps}(H_R^{(m)})} + C \Vert g_m \Vert_{L^{\infty}(H_R^{(m)})} + C \Vert \tilde{h}_m \Vert_{C^{\eps}_{\ell}(H_R^{(m)})} \le C.
\end{align}
Note that $C$ is independent of $m$, since we can apply \autoref{lemma:osc-bdry-data} to the rescaled function $\bar{P}_m(z) = P_m(z_m^0 \circ S_{r_m} z)$ first, and then deduce the estimate for $P_m$. 
Hence, by H\"older interpolation (see \autoref{lemma:Holder-interpol}), \autoref{lemma:scaling}, and \eqref{eq:Pm-bounded}, we get
\begin{align*}
I_2 &:= r_m^{-k-\eps + \alpha} (1 + |v_m^0|)^{-\kappa} \theta(r_m)^{-1} [ p_{l_m,r_m} ]_{C^{\alpha}(\gamma_- \cap Q_{4r_m(1+|z_1|)}(z_m^0))} \\
&\le C \Vert P_{m} \Vert_{L^{\infty}( \gamma_-^{(m)} \cap Q_{4(1+|z_1|)})} + C  \theta(r_m)^{-1}[ p_{l_m,r_m} ]_{C^{k+\eps}(\gamma_- \cap Q_{4r_m(1 + |z_1|)}(z_m^0))} \\
&\le C \Vert g_m \Vert_{L^{\infty}(H^{(m)}_{4(1+|z_1|)})} + C \Vert \tilde{h}_m \Vert_{C^{\eps}(H^{(m)}_{4(1+|z_1|)})} +  C [ P_m ]_{C_{\ell}^{k+\eps}(H^{(m)}_{4(1+|z_1|)})} \le C,
\end{align*}
where we also used \eqref{eq:hm-vanishes-Holder}, \eqref{eq:hm-vanishes}, and \eqref{eq:gm-prop-2_-}.

For $I_3$ we have by \eqref{eq:hm-vanishes} that for $m$ large enough
\begin{align*}
I_3 \le C \Vert \tilde{h}_m \Vert_{L^{\infty}(H^{(m)}_{4 (1 + |z_1|)})} \le C.
\end{align*}
Altogether, we also deduce the second property in \eqref{eq:bdry-conv-higher--}, as desired.

Moreover, by the energy estimate in \autoref{lemma:energy-est} applied to $f_{l_m} - p_{l_m,r_m}$ in $H_{4Rr_m(1+|z_1|)}(z_m^0)$, and using the uniform bounds \eqref{eq:gm-prop-2_-}, \eqref{eq:hm-vanishes}, and \eqref{eq:Pm-bounded} in the same way as in the proof of \eqref{eq:bdry-conv-higher---help}, we get 
\begin{align*}
\Vert \nabla_v g_m \Vert_{L^{2}(H_1^{(m)}(z_1))} \le C r_m^{2-k-\eps} (1 + |v_m^0|)^{-\kappa} \theta(r_m)^{-1} \Vert \nabla_v (f_{l_m} - p_{l_m,r_m}) \Vert_{L^{2}(H_{2r_m(1 + |z_1|)}(z_m^0))} \le C.
\end{align*}
Hence, recalling the interior estimates from Step 1, it must also hold $\nabla_v g \in L^2_{loc}([0,\infty) \times \R^n \times \R^n)$.

\textbf{Step 4:} We claim that for any $z_1 = (0,x_1,v_1) \in \R^{1+2n}$ it holds
\begin{align}
\label{eq:bdry-conv-higher--_2}
\Vert P_m - P_0 \Vert_{L^{\infty}(\gamma_-^{(m)} \cap Q_1(z_1))} \to 0 ~~ \text{ as } m \to \infty
\end{align} 
for some function $P_0$, such that its restriction to $\gamma_-^{(m)}$ is in $\cP_k$. Also this property will be used in the next step.

To prove \eqref{eq:bdry-conv-higher--_2}, note that (up to a rotation) we can write $\partial \Omega$ locally in $Q_{r_m}(z_m^0)$ as follows using that $|v_m^0|r_m^{2} \to 0$ by \eqref{eq:rm-properties}), in order to guarantee that the kinetic shift $x_m^0 - (t-t_m^0)v_m^0$ is negligible for large enough $m$:
\begin{align*}
\partial \Omega \cap Q_{r_m}(z_m^0) = \{ x \in \R^n : x_n = \Gamma(x') \} \cap Q_{r_m}(z_m^0)
\end{align*}
for some function $\Gamma \in C^{\frac{k+\eps}{3}}(Q_{r_m}(z_m^0))$.
In particular, we can write
\begin{align*}
(z_m^0 \circ \gamma_-) \cap Q_{R r_m}(0) &:= \{z \in Q_{R r_m}(0) : z_m^0 \circ z \in \gamma_- \} \\
&= \Big\{ z \in Q_{R r_m}(0) : x_n = \Gamma\big(x' + (x_m^0)' + t (v_m^0)'\big) - (x_m^0)_n - t (v_m^0)_n =: \Gamma_m(z') \Big\},
\end{align*}
where we denote $z'= (t,x',v)$.

Let us observe that by \eqref{eq:inclusion-gamma_-_higher}, denoting $p_m := p_{l_m,r_m}$, we have that \eqref{eq:bdry-conv-higher--_2} holds true if we can prove that for any $R \ge 1$ there is $P_0$ as above, such that
\begin{align}
\label{eq:bdry-data-conv-help-0}
\left\Vert \frac{r_m^{-k-\eps}}{(1 + |v_m^0|)^{\kappa} \theta(r_m)} p_{m} - P_0 \right\Vert_{L^{\infty}(\gamma_- \cap H_{R r_m}(z_m^0))} \to 0,
\end{align}
and therefore, by writing, as in the proof of \autoref{lemma:higher-reg-gamma_+}, $\bar{p}_{m}(z) = p_{m}(z_{m}^0 \circ z)$, it suffices to find $P_0$,
\begin{align*}
\left\Vert \frac{r_m^{-k-\eps}}{(1 + |v_m^0|)^{\kappa}\theta(r_m)} \bar{p}_{m} - P_0 \right\Vert_{L^{\infty}((z_m^0 \circ \gamma_-) \cap Q_{R r_m}(0))} \to 0.
\end{align*}

Hence, we need to find $P_0 \in \cP_k$ such that
\begin{align}
\label{eq:bdry-data-conv-help-1}
\sup_{z \in Q_{Rr_m}(0)} \left| \frac{r_m^{-k-\eps}}{(1 + |v_m^0|)^{\kappa}\theta(r_m)}  \bar{p}_m(t,x',\Gamma_m(z'),v) - P_0(z') \right| \to 0,
\end{align}
in order to deduce \eqref{eq:bdry-conv-higher--_2}.

Let us now prove \eqref{eq:bdry-data-conv-help-1}. First, by the definition of $\bar{p}_m$, we can write
\begin{align*}
\bar{p}_m(z) = \sum_{|\beta| \le k} \alpha^{(\beta)}_{m} t^{\beta_t} x_1^{\beta_{x_1}} \cdot \dots \cdot x_n^{\beta_{x_n}} v_1^{\beta_{v_1}} \cdot \dots \cdot v_n^{\beta_{v_n}} = \sum_{|\beta| \le k} \alpha^{(\beta)}_{m} (z)^{\beta},
\end{align*}
where $\beta = (\beta_t,\beta_x,\beta_v) \in (\N \cup \{0\})^{1+2n}$ are multi-indeces with $|\beta| = 2|\beta_t| + 3 |\beta_x| + |\beta_v|$. Recalling \autoref{lemma:orth-proj-prop}, we have by the same arguments as in the proof of \eqref{eq:coefficients-vanish}, 
\begin{align}
\label{eq:coefficients-vanish-2}
\frac{|\alpha^{(\beta)}_{m}|}{(1 + |v_m^0|)^{\kappa} \theta(r_m)} \to 0 ~~ \text{ as } m \to \infty, ~~ \forall |\beta| \le k.
\end{align}
Moreover, we can write, denoting by $\xi \in (\N \cup \{ 0 \})^{n-1}$ multi-indeces with $|\xi| = \xi_1 + \dots + \xi_{n-1}$, 
\begin{align*}
\Gamma(x' + (x_m^0)' + t (v_m^0)') = \sum_{|\xi| \le \lfloor\frac{k+\eps}{3} \rfloor} \omega^{(\xi)} (x')^{\xi} + E_{\Gamma}(x'),
\end{align*}
where, since $|t (v_m^0)'| \le C$, due to the fact that $|t| \le (Rr_m)^2$ and $|(v_m^0)'| \le C r_m^{-2}$ by \eqref{eq:rm-properties}, by Taylor's formula, there exist $\omega^{(\xi)}, E_{\Gamma}(x') \in \R$ such that
\begin{align}
\label{eq:Gamma-bound}
|\omega^{(\xi)}| \le |\partial^{\xi} \Gamma ((x_m^0)' + t (v_m^0)')| \le C, \qquad |E_{\Gamma}(x')| \le C |x'|^{\frac{k+\eps}{3}}.
\end{align}

Therefore, we obtain for $z \in Q_{Rr_m}(0)$, denoting $\beta' = (\beta_t,\beta_{x'},\beta_v)$,
\begin{align*}
\bar{p}_m(t,x',\Gamma_m(z'),v) &= \sum_{|\beta| \le k} \alpha^{(\beta)}_{m} (z')^{\beta'} \big(\Gamma_m(z')\big)^{\beta_{x_n}} \\
&=\sum_{\substack{|\beta| \le k,\\ \beta_{x_n} = 0 }} \alpha^{(\beta)}_{m} (z')^{\beta'}  + \sum_{\substack{|\beta| \le k, \\ \beta_{x_n} \ge 1}} \sum_{|\xi| \le \lfloor\frac{k+\eps}{3} \rfloor} \alpha^{(\beta)}_{m} (z')^{\beta'} \big( \omega^{(\xi)} (x')^{\xi}\big)^{\beta_{x_n}} \\
&\quad + \sum_{\substack{|\beta| \le k, \\ \beta_{x_n} \ge 1}} \alpha^{(\beta)}_{m} (z')^{\beta'} \sum_{l = 1}^{\beta_{x_n}} \binom{\beta_{x_n}}{l} \big(E_{\Gamma}(x') \big)^{l} \Big[  \sum_{|\xi| \le \lfloor\frac{k+\eps}{3} \rfloor} \omega^{(\xi)} (x')^{\xi} \Big]^{\beta_{x_n} - l} \\
&\quad - \sum_{\substack{|\beta| \le k, \\ \beta_{x_n} \ge 1}} \alpha^{(\beta)}_{m} (z')^{\beta'} \big( (x_m^0)_n \big)^{\beta_{x_n}} - \sum_{\substack{|\beta| \le k, \\ \beta_{x_n} \ge 1}} \alpha^{(\beta)}_{m} (z')^{\beta'} \big( t (v_m^0)_n \big)^{\beta_{x_n}}.
\end{align*}
Clearly, we have, 
\begin{align*}
\tilde{P}_m(z') &:= \frac{r_m^{-k-\eps}}{(1 + |v_m^0|)^{\kappa}\theta(r_m)}  \Bigg( \sum_{\substack{|\beta| \le k, \\ \beta_{x_n} = 0}} \alpha^{(\beta)}_{m} (z')^{\beta'} + \sum_{\substack{|\beta| \le k, \\ \beta_{x_n} \ge 1}} \sum_{\substack{|\xi| \le \lfloor\frac{k+\eps}{3} \rfloor, \\ |\beta'| + 3|\xi|\beta_{x_n} \le k }} \alpha^{(\beta)}_{m} (z')^{\beta'} \big( \omega^{(\xi)} (x')^{\xi}\big)^{\beta_{x_n}} \\
&\qquad\qquad\qquad  - \sum_{\substack{|\beta| \le k, \\ \beta_{x_n} \ge 1}} \alpha^{(\beta)}_{m} (z')^{\beta'} \big( (x_m^0)_n \big)^{\beta_{x_n}} - \sum_{\substack{|\beta| \le k, \\ \beta_{x_n} \ge 1,\\ |\beta'| + 2\beta_{x_n} \le k}} \alpha^{(\beta)}_{m} (z')^{\beta'} \big( t (v_m^0)_n \big)^{\beta_{x_n}} \Bigg) \in \cP_k.
\end{align*}
Moreover, by \eqref{eq:coefficients-vanish-2} and \eqref{eq:Gamma-bound}, we have for any $z \in Q_{Rr_m}(0)$:
\begin{align*}
\frac{r_m^{-k-\eps}}{(1 + |v_m^0|)^{\kappa}\theta(r_m)} & \Big| \sum_{\substack{|\beta| \le k, \\ \beta_{x_n} \ge 1}} \sum_{\substack{|\xi| \le \lfloor\frac{k+\eps}{3} \rfloor, \\ |\beta'| + 3|\xi|\beta_{x_n} \ge k +1}}\alpha^{(\beta)}_{m} (z')^{\beta'} \big( \omega^{(\xi)} (x')^{\xi}\big)^{\beta_{x_n}} \Big| \\
 &\le r_m^{-k-\eps} \sum_{|\beta| \le k} \sum_{\substack{|\xi| \le \lfloor\frac{k+\eps}{3} \rfloor, \\ |\beta'| + 3|\xi|\beta_{x_n} \ge k +1}} \frac{|\alpha^{(\beta)}_{m}|}{(1 + |v_m^0|)^{\kappa}\theta(r_m)} \big| (z')^{\beta'} \big| \big| (x')^{\xi} \big|^{\beta_{x_n}}  \\
 &\le C (R)  r_m^{-k-\eps} \sum_{|\beta| \le k} \sum_{\substack{|\xi| \le \lfloor\frac{k+\eps}{3} \rfloor, \\ |\beta'| + 3|\xi|\beta_{x_n} \ge k +1}} \frac{|\alpha^{(\beta)}_{m}|}{(1 + |v_m^0|)^{\kappa} \theta(r_m)} r_m^{|\beta'| + 3\beta_{x_n} |\xi|} \le C(R) r_m^{1-\eps} \to 0.
\end{align*}
Analogously, we compute
\begin{align*}
\frac{r_m^{-k-\eps}}{(1 + |v_m^0|)^{\kappa}\theta(r_m)} & \Big| \sum_{\substack{|\beta| \le k, \\ \beta_{x_n} \ge 1, \\ |\beta'| + 2\beta_{x_n} \ge k+1}} \alpha^{(\beta)}_{m} (z')^{\beta'} \big( t (v_m^0)_n \big)^{\beta_{x_n}} \Big| \\
&\le C(R) r_m^{-k-\eps} \sum_{\substack{|\beta| \le k, \\ |\beta'| + 2\beta_{x_n} \ge k+1}} \frac{|\alpha^{(\beta)}_{m}|}{(1 + |v_m^0|)^{\kappa} \theta(r_m)} r_m^{|\beta'| + 2 \beta_{x_n}} |(v_m^0)_n|^{\beta_{x_n}} \\
&\le C(R) r_m^{1-\eps} |(v_m^0)_n|^{\frac{k+1}{2}} \le C(R).
\end{align*}
Here, we used that $|v_m^0|r_m^{\frac{2(1-\eps)}{k+1}} \le |v_m^0|r_m \to 0$ by \eqref{eq:rm-properties}.

Moreover, we have the following estimate for the error terms coming from the Taylor expansion of $\Gamma$:

\begin{align*}
& \frac{r_m^{-k-\eps}}{(1 + |v_m^0|)^{\kappa}\theta(r_m)} \Big| \sum_{\substack{|\beta| \le k, \\ \beta_{x_n} \ge 1}} \alpha^{(\beta)}_{m} (z')^{\beta'} \sum_{l = 1}^{\beta_{x_n}} \binom{\beta_{x_n}}{l} \big(E_{\Gamma}(x') \big)^{l} \Big[  \sum_{|\xi| \le \lfloor\frac{k+\eps}{3} \rfloor} \omega^{(\xi)} (x')^{\xi} \Big]^{\beta_{x_n} - l} \Big| \\
&\qquad \le C r_m^{-k-\eps} \sum_{\substack{|\beta| \le k, \\ \beta_{x_n} \ge 1}} \frac{|\alpha^{(\beta)}_{m} |}{(1 + |v_m^0|)^{\kappa} \theta(r_m)} \big| (z')^{\beta'}  \big| \sum_{l = 1}^{\beta_{x_n}} \sum_{|\xi| \le \lfloor \frac{k+\eps}{3} \rfloor} |x'|^{\frac{k+\eps}{3}l}  \big| (x')^{\xi} \big|^{\beta_{x_n} - l} \\
&\qquad \le C(R) r_m^{-k-\eps} \sum_{\substack{|\beta| \le k, \\ \beta_{x_n} \ge 1}} \frac{|\alpha^{(\beta)}_{m} |}{(1 + |v_m^0|)^{\kappa} \theta(r_m)} r_m^{|\beta'|} \sum_{l = 1}^{\beta_{x_n}}  \sum_{|\xi| \le \lfloor \frac{k+\eps}{3} \rfloor} r_m^{(k+\eps)l + 3|| (\beta_{x_n} - l)} \\
&\qquad \le C(R) \sup_{|\beta| \le k} \frac{|\alpha^{(\beta)}_{m} |}{(1 + |v_m^0|)^{\kappa} \theta(r_m)} \to 0.
\end{align*}

Therefore, altogether, we have shown that
\begin{align}
\label{eq:bdry-data-conv-help-2}
\sup_{z \in Q_{Rr_m}(0)} & \left| \frac{r_m^{-k-\eps}}{(1 + |v_m^0|)^{\kappa} \theta(r_m)}  \bar{p}_m(t,x',\Gamma_m(z'),v) - \tilde{P}_m(z') \right| \to 0.
\end{align}
Hence, in order to deduce \eqref{eq:bdry-data-conv-help-1}, it remains to show that the sequence of polynomials $(\tilde{P}_m)_m$ converges to a polynomial $P_0 \in \cP_k$. Since $\cP_k$ is a finite dimensional space, to do so, it suffices to prove that $\Vert \tilde{P}_m \Vert_{L^{\infty}(Q_{Rr_m}(0))}$ is bounded. To see this, note that by construction of $\tilde{P}_m$
\begin{align*}
\frac{r_m^{-k-\eps}}{(1 + |v_m^0|)^{\kappa} \theta(r_m)}  \sup_{z \in Q_{R r_m}(0)}  \big|\bar{p}_m(t,x',\Gamma_m(z'),v) \big| &= \frac{r_m^{-k-\eps}}{(1 + |v_m^0|)^{\kappa}}\theta(r_m) \sup_{z \in (z_m^0 \circ \gamma_-^{(m)}) \cap Q_{Rr_m}(0)} \big| \bar{p}_m(z) \big| \\
&= \Vert P_m \Vert_{L^{\infty}(\gamma_-^{(m)} \cap Q_R(0))},
\end{align*}
and therefore,
\eqref{eq:bdry-data-conv-help-2} and \eqref{eq:Pm-bounded}, imply
\begin{align*}
\Vert \tilde{P}_m \Vert_{L^{\infty}(Q_{Rr_m}(0))} &\le \sup_{z \in Q_{Rr_m}(0)} \left| \frac{r_m^{-k-\eps}}{(1 + |v_m^0|)^{\kappa}\theta(r_m)}  \bar{p}_m(t,x',\Gamma_m(z'),v) - \tilde{P}_m(z') \right| +  \Vert P_m \Vert_{L^{\infty}(\gamma_-^{(m)} \cap Q_R(0))} \\
&\le C.
\end{align*}
Therefore, we conclude that $\tilde{P}_m \to P_0 \in \cP_k$. In particular we obtain \eqref{eq:bdry-data-conv-help-1}, which implies \eqref{eq:bdry-conv-higher--_2}, as desired.

\textbf{Step 5:} As in the proof of \autoref{lemma:higher-reg-gamma_+} we would now like to apply the Liouville theorem for positive times from \autoref{prop:Liouville-higher-pos-times} in order to deduce that $g \in \cP_k$. In order to do so, we need to prove that 
\begin{align}
\label{eq:g-equiv-poly}
g(0,\cdot,\cdot) \in \cP_k ~~ \text{ in } \R^n \times \R^n.
\end{align}

By \eqref{eq:bdry-conv-higher--} and \eqref{eq:bdry-conv-higher--_2}, since $g_m = P_m$ on $\gamma_-^{(m)} \cap Q_{r_m^{-1}}$, we deduce that 
\begin{align}
\label{eq:bdry-initial-poly}
\Vert g - P_0 \Vert_{L^{\infty}(\gamma_-^{(m)} \cap Q_{1}(z_1))} \le \Vert g - g_m \Vert_{L^{\infty}(\gamma_-^{(m)} \cap Q_{1}(z_1))} + \Vert P_m - P_0 \Vert_{L^{\infty}(\gamma_-^{(m)} \cap Q_{1}(z_1))}  \to 0.
\end{align}
In particular, it holds $g=P_0$ in $(\{ 0 \} \times \R^{2n}) \cap Q_1(z_1) = \lim_{m \to \infty} \gamma_-^{(m)} \cap Q_1(z_1)$, and since $P_0 \in \cP_k$ in $\{ 0 \} \times \R^{2n}$, we deduce that \eqref{eq:g-equiv-poly} holds true. Therefore, $g \in \cP_k$ by \autoref{prop:Liouville-higher-pos-times}. Moreover, note that $g$ attains its initial datum in a locally uniformly H\"older continuous way, since $g \in C^{\alpha'}_{loc}([0,\infty) \times \R^n \times \R^n)$ for any $\alpha'< \alpha$ by the Arzel\`a-Ascoli theorem, \eqref{eq:bdry-conv-higher---help} and the interior estimate in Step 2. 

 Hence, we can use \eqref{eq:gm-prop-1_-} and \eqref{eq:gm-prop-2_-} in the same way as in Step 4 of the proof of \autoref{lemma:higher-reg-gamma_+} to get a contradiction. Thus, the proof is complete.
\end{proof}

We end this subsection by stating a variant of \autoref{lemma:higher-reg-gamma_--Dirichlet} in the special case $\Omega = \{ x_n > 0\}$, where we can establish the expansion without any dependence of the constant on $|v_0|$.

\begin{lemma}
\label{lemma:higher-reg-gamma_--Dirichlet-flat}
Let $\Omega = \{ x_n > 0 \}$ and $z_0 \in \gamma_-$. Let $R \in (0,1]$ be such that $Q_{2R}(z_0) \cap \gamma_0 = \emptyset$. Let $k \in \N$ with $k \ge 3$, $\eps \in (0,1)$, and $a^{i,j},b,c,h \in C^{k+\eps-2}_{\ell}(H_R(z_0))$, and assume that $a^{i,j}$ satisfies \eqref{eq:unif-ell}.
Let $F \in C^{k+\eps}_{\ell}(H_R(z_0))$ and $f$ be a weak solution to 

\begin{equation*}
\left\{\begin{array}{rcl}
\partial_t f + v \cdot \nabla_x f + (-a^{i,j} \partial_{v_i,v_j})f &=& - b \cdot \nabla_v f - c f + h ~~ \text{ in } H_R(z_0), \\
f &=& F ~~~~ \qquad\qquad\qquad \text{ on } \gamma_- \cap Q_R(z_0).
\end{array}\right.
\end{equation*}
Then, there exists $p_{z_0} \in \cP_k$ such that for any $r \in (0,\frac{R}{2}]$ and any $z \in H_{r}(z_0)$ it holds:
\begin{align*}
\begin{split}
|f(z) - p_{z_0}(z)| \le C \left( \frac{r}{R} \right)^{k+\eps} & \Big( \Vert f \Vert_{L^{\infty}(H_R(z_0))} + R^{k+\eps} [ F ]_{C_{\ell}^{k+\eps}(H_R(z_0))}  \\
&\qquad\qquad\quad\quad ~~ + R^{k+\eps} [ h ]_{C^{k+\eps-2}_{\ell}(H_R(z_0))} \Big).
\end{split}
\end{align*}
The constant $C$ depends only on $n,k,\eps,\lambda,\Lambda$, $\Vert a^{i,j} \Vert_{C^{k+\eps-2}_{\ell}(H_R(z_0))}, \Vert b \Vert_{C^{k+\eps-2}_{\ell}(H_R(z_0))}$, and \\$\Vert c \Vert_{C^{k+\eps-2}_{\ell}(H_R(z_0))}$, but not on $z_0$ and $r,R$.
\end{lemma}

\begin{proof}
We proceed as in the proof of \autoref{lemma:higher-reg-gamma_--Dirichlet}, but do not subtract $F$ at the beginning. Instead, we assume by contradiction that the claim does not hold. Then, there are sequences $R_l \in (0,1]$ and $f_l$, $F_l$, $h_l$, $a^{i,j}_{l}$, $b_l$, $c_l$, $z^0_{l}$ such that for any $l \in \N$ \eqref{eq:unif-ell} holds true with $\lambda$ for any $l \in \N$, as well as
\begin{align}
\label{eq:blow-up-normalization_--flat}
R_l^{-k-\eps} \Vert f_l \Vert_{L^{\infty}(H_{R_l}(z_l^0))} + [ F_l ]_{C^{k+\eps}_{\ell}(H_{R_l}(z_l^0))} + [ h_l ]_{C^{k+\eps-2}_{\ell}(H_{R_l}(z_l^0))} \le 1,
\end{align}
such that the $f_l$ are solutions to \eqref{eq:fl-sol-gamma_-} with $f_l = F_l$ on $\gamma_- \cap H_{R_l}(z^0_l)$, but it holds \eqref{eq:blow-up-assumption_-} with $\kappa = 0$. Now, we closely follow the proof of \autoref{lemma:higher-reg-gamma_--Dirichlet} (setting $\kappa = \bar{\kappa} = 0$ everywhere. The only change is that the functions $g_m$ solve \eqref{eq:gm-PDE_-} with $g_m = \tilde{F}_m$ on $\gamma_-^{(m)} \cap Q_{R_m r_m^{-1}}$, where
\begin{align*}
\tilde{F}_m(z) = \frac{F_{l_m}(z_m^0 \circ S_{r_m} z ) - p_{l_m,r_m}(z_m^0 \circ S_{r_m} z)}{r_m^{k+\eps} \theta(r_m)}.
\end{align*}
Apart from this change, the proofs of Steps 1 and 2 remain exactly the same.
In Step 3, we still claim \eqref{eq:bdry-conv-higher--} and prove it in the same way, with the only change in the estimate for $I_2$. In fact,  we compute using \eqref{eq:inclusion-gamma_-_higher}, \autoref{lemma:Holder-scaling}, $R_m r_m^{-1} \to \infty$, $p_{l_m,r_m} \in \cP_k$, and \eqref{eq:blow-up-normalization_--flat}
\begin{align}
\label{eq:Fm-holder}
\begin{split}
[ \tilde{F}_m ]_{C_{\ell}^{k+\eps}(H^{(m)}_{R}(z_1))} &\le C \frac{ (1 + |z_1|)^{k+\eps} r_m^{k+\eps} [ F_{l_m} - p_{l_m,r_m} ]_{C^{k+\eps}_{\ell}(H_{2Rr_m(1 + |z_1|)}(z_m^0))} }{r_m^{k+\eps} \theta(r_m)} \\
&\le C \frac{[ F_{l_m} ]_{C^{k+\eps}_{\ell}(H_{2Rr_m(1 + |z_1|)}(z_m^0))} }{\theta(r_m)} \le C \theta(r_m)^{-1} \to 0 ~~ \text{ as } m \to \infty,
\end{split}
\end{align}
where $C$ depend on $|z_1|$ and $R \ge 1$ but not on $m$. Hence, by H\"older interpolation (see \autoref{lemma:Holder-interpol}), \autoref{lemma:scaling}, and \autoref{lemma:osc-bdry-data}, we get
\begin{align*}
I_2 &:= r_m^{-k-\eps+\alpha} \theta(r_m)^{-1} [ F_{l_m} - p_{l_m,r_m} ]_{C^{\alpha}(\gamma_- \cap Q_{4r_m(1 + |z_1|)}(z_m^0))} \\
&\le C \Vert \tilde{F}_{l_m} \Vert_{L^{\infty}( \gamma_-^{(m)} \cap Q_{4(1+|z_1|)})} + C  \theta(r_m)^{-1}[ F_{l_m} - p_{l_m,r_m} ]_{C^{k+\eps}(\gamma_- \cap Q_{4r_m(1 + |z_1|)}(z_m^0))} \\
&\le C \Vert g_m \Vert_{L^{\infty}(H^{(m)}_{4(1+|z_1|)})} + C \Vert \tilde{h}_m \Vert_{C^{\eps}(H^{(m)}_{4(1+|z_1|)})} +  C [ \tilde{F}_m ]_{C_{\ell}^{k+\eps}(H^{(m)}_{4(1+|z_1|)})} \le C,
\end{align*}
where we also used \eqref{eq:hm-vanishes-Holder}, \eqref{eq:hm-vanishes}, and \eqref{eq:gm-prop-2_-}. This finishes Step 3 and establishes \eqref{eq:bdry-conv-higher--}.

Next, let us take a look at Step 4, which simplifies dramatically, in comparison to \autoref{lemma:higher-reg-gamma_--Dirichlet}.

In fact, by \eqref{eq:Fm-holder} we can write $\tilde{F}_m = e_m + P_m$ for $P_m \in \cP_{k}$ such  that 
\begin{align*}
\Vert e_m \Vert_{L^{\infty}(H_1^{(m)}(z_1))} = \Vert \tilde{F}_m - P_m \Vert_{L^{\infty}(H_1^{(m)}(z_1))} \le c [\tilde{F}_m]_{C_{\ell}^{k+\eps}(H_1^{(m)}(z_1))} \to 0 ~~ \text{ as } m \to \infty,
\end{align*}
and by \autoref{lemma:osc-bdry-data}, the second property in \eqref{eq:gm-prop-2_-}, and \eqref{eq:hm-vanishes-Holder}, \eqref{eq:hm-vanishes} we deduce
\begin{align}
\label{eq:Pm-bd-flat}
\begin{split}
\Vert P_m \Vert_{L^{\infty}(H_1^{(m)}(z_1) \cap \gamma_-^{(m)})} &\le \Vert \tilde{F}_m - P_m \Vert_{L^{\infty}(H_1^{(m)}(z_1) \cap \gamma_-^{(m)})} + \Vert \tilde{F}_m \Vert_{L^{\infty}(H_1^{(m)}(z_1) \cap \gamma_-^{(m)})} \\
&\le C [ \tilde{F}_m ]_{C_{\ell}^{k+\eps}(H_1^{(m)}(z_1))} + C \Vert g_m \Vert_{L^{\infty}(H_1^{(m)}(z_1))} + C \Vert \tilde{h}_m \Vert_{C^{\eps}_{\ell}(H_1^{(m)}(z_1))} \le C.
\end{split}
\end{align}
Since $\Omega = \{ x_n > 0 \}$, we have that the graph of $\gamma_-^{(m)}$ is linear for any $m$. Therefore, the restriction of $P_m$ to $\gamma_-^{(m)}$ is still in $\cP_k$. In particular, \eqref{eq:Pm-bd-flat} implies that there exists a function $P_0$ such that its restriction to $\gamma_-^{(m)}$ is still in $\cP_k$. In particular, we have
\begin{align}
\label{eq:P0-flat}
\Vert \tilde{F}_m - P_0 \Vert_{L^{\infty}(H_1^{(m)}(z_1) \cap \gamma_-^{(m)})} \to 0 ~~ \text{ as } m \to \infty.
\end{align}
This concludes Step 4, and in particular, by combination of \eqref{eq:P0-flat} and Step 3, we deduce \eqref{eq:bdry-initial-poly}, as in the proof of \autoref{lemma:higher-reg-gamma_--Dirichlet}. This concludes the proof, as before.
\end{proof}

\subsection{In-flow: Boundary regularity outside the grazing set}

We prove that expansions at boundary points together with interior regularity estimates imply regularity estimates up to the boundary.

\begin{lemma}
\label{lemma:regularity-inflow}
Let $k \in \N$ with $k \ge 3$, $\eps \in (0,1)$ be such that $\frac{k+\eps}{3} \not \in \N$. Let $\Omega \subset \R^n$ be a convex $C^{\frac{k+\eps}{3}}$ domain,  $z_0 \in \gamma \setminus \gamma_0$ and $R \in (0,1]$ be such that $Q_{4R}(z_0) \cap \gamma_0 = \emptyset$. Let $a^{i,j},b,c,h \in C^{k+\eps-2}_{\ell}(H_R(z_0))$ and assume that $a^{i,j}$ satisfies \eqref{eq:unif-ell}. Let $F \in C^{k+\eps}_{\ell}(H_R(z_0))$ and let $f$ be a weak solution to 

\begin{equation*}
\left\{\begin{array}{rcl}
\partial_t f + v \cdot \nabla_x f + (-a^{i,j} \partial_{v_i,v_j})f &=& - b \cdot \nabla_v f - c f + h ~~ \text{ in } H_R(z_0), \\
f &=& F ~~~~\qquad\qquad\qquad \text{ on } \gamma_- \cap Q_R(z_0).
\end{array}\right.
\end{equation*}
Then, it holds:
\begin{align*}
[ f ]_{C^{k+\eps}_{\ell}(H_{R/16}(z_0))} \le C (1 + |v_0|)^{\frac{k+\eps}{2}} R^{-k-\eps} \big( \Vert f \Vert_{L^{\infty}(H_R(z_0))} + R^{k+\eps} [ F ]_{C_{\ell}^{k+\eps}(H_R(z_0))} + R^{k+\eps} [ h ]_{C^{k+\eps-2}_{\ell}(H_R(z_0))} \big).
\end{align*}
The constant $C$ depends only on $n,k,\eps,\lambda,\Lambda, \Omega$, $\Vert a^{i,j} \Vert_{C^{k+\eps-2}_{\ell}(H_R(z_0))}, \Vert b \Vert_{C^{k+\eps-2}_{\ell}(H_R(z_0))}$, and \\$\Vert c \Vert_{C^{k+\eps-2}_{\ell}(H_R(z_0))}$, but not on $z_0$ and $R$.
\end{lemma}

\begin{proof}
Let us assume without loss of generality that
\begin{align}
\label{eq:normalization-reg}
(1 + |v_0|)^{\frac{k+\eps}{2}} R^{-k-\eps} \left(\Vert f \Vert_{L^{\infty}(H_R(z_0))} + R^{k+\eps}[ F ]_{C_{\ell}^{k+\eps}(H_R(z_0))} + R^{k+\eps}[ h ]_{C^{k+\eps-2}_{\ell}(H_R(z_0))} \right) \le 1.
\end{align}
We claim that for any $z^{\ast} \in H_{R/16}(z_0)$ it holds 
\begin{align}
\label{eq:claim-reg}
[f]_{C^{k+\eps}_{\ell}(Q_{r/2}(z^{\ast}))} \le C,
\end{align}
where we define $r := d_{\ell}(z^{\ast}) := \sup \{ r > 0 : Q_r(z^{\ast}) \subset (-1,1) \times \Omega \times \R^n \}$. 

To prove it, we note that for any $z^{\ast}$ we can find $z^{\ast}_0 \in \gamma \cap Q_{R/8}(z_0)$ such that $z^{\ast} \in Q_r(z^{\ast}_0)$, where we used \autoref{lemma:kinetic-balls}. Note that by \autoref{lemma:kinetic-balls}  and by construction  we have 
\begin{align}
\label{eq:inclusion-regularity}
Q_{r}(z^{\ast}) \subset Q_{4r}(z_0^{\ast}), \qquad Q_{2R}(z_0^{\ast}) \cap \gamma_0 = \emptyset, \qquad H_{R/4}(z_0^{\ast}) \subset H_R(z_0).
\end{align}
In case $r \ge \frac{R}{20}$ we have that \eqref{eq:claim-reg} follows immediately from interior estimates (see \autoref{lemma:interior-reg}).
 
Hence, we can assume from now on that $r < \frac{R}{20}$. In this case, let us denote by $p_{z^{\ast}_0} \in \cP_k$ the polynomial of the higher order expansion at $z_0^{\ast}$ from \autoref{lemma:higher-reg-gamma_+} (if $z_0^{\ast} \in \gamma_+$) or from \autoref{lemma:higher-reg-gamma_--Dirichlet} (if $z_0^{\ast} \in \gamma_-$). Note that the function $f - p_{z_0^{\ast}}$ solves
\begin{align*}
[\partial_t + v \cdot \nabla_x + (-a^{i,j} \partial_{v_i,v_j}) + b \cdot \nabla_v  + c ](f - p_{z_0^{\ast}}) = h + \tilde{h} ~~ \text{ in } H_R(z_0),
\end{align*}
where by \eqref{eq:inclusion-regularity}
\begin{align}
\label{eq:tilde-h-estimate}
\begin{split}
[ \tilde{h} ]_{C^{k+\eps-2}_{\ell}(Q_{r}(z^{\ast}))} &\le [ \partial_t p_{z_0^{\ast}} - v \cdot \nabla_x p_{z_0^{\ast}} - (-a^{i,j} \partial_{v_i,v_j})p_{z_0^{\ast}}  - b \cdot \nabla_v  p_{z_0^{\ast}} - c p_{z_0^{\ast}} ]_{C^{k+\eps-2}_{\ell}(H_{R/5}(z_0^{\ast}))} \\
&\le [ a^{i,j} \partial_{v_i,v_j} p_{z_0^{\ast}} ]_{C^{k+\eps-2}_{\ell}(H_{R/5}(z_0^{\ast}))}  + [ b \cdot \nabla_v  p_{z_0^{\ast}} ]_{C^{k+\eps-2}_{\ell}(H_{R/5}(z_0^{\ast}))} + [ c p_{z_0^{\ast}} ]_{C^{k+\eps-2}_{\ell}(H_{R/5}(z_0^{\ast}))} \\
&\le C \Vert p_{z^{\ast}_0} \Vert_{C^{k+\eps}_{\ell}(H_{R/5}(z_0))} \le C R^{-k-\eps} \Vert p_{z_0^{\ast}} \Vert_{L^{\infty}(H_{R/5}(z^{\ast}_0))}.
\end{split}
\end{align}
Here, we used the product rule (see \autoref{lemma:product-rule}), and \autoref{lemma:der-Holder} in the second to last estimate. The last estimate follows since $p_{z_0^{\ast}} \in \cP_k$, which implies by H\"older interpolation (see \autoref{lemma:Holder-interpol})
\begin{align*}
\Vert p_{z^{\ast}_0} \Vert_{C^{k+\eps}_{\ell}(H_{R/5}(z^{\ast}_0))} \le C R^{-k} \Vert p_{z^{\ast}_0} \Vert_{L^{\infty}(H_{R/5}(z^{\ast}_0))} + C [p_{z^{\ast}_0}]_{C^{k+\eps}_{\ell}(H_{R/5}(z^{\ast}_0))} = C R^{-k} \Vert p_{z^{\ast}_0} \Vert_{L^{\infty}(H_{R/5}(z^{\ast}_0))}.
\end{align*}

Then, by interior Schauder estimates (see \autoref{lemma:interior-reg}) applied to the function $f - p_{z_0^{\ast}}$, \autoref{lemma:osc-results}, \eqref{eq:inclusion-regularity}, and \eqref{eq:tilde-h-estimate} we obtain
\begin{align*}
[f]_{C^{k+\eps}_{\ell}(Q_{r/2}(z^{\ast}))} &= [f - p_{z^{\ast}_0}]_{C^{k+\eps}_{\ell}(Q_{r/2}(z^{\ast}))} \\
&\le C r^{-k-\eps} \Vert f - p_{z^{\ast}_0} \Vert_{L^{\infty}(Q_{r}(z^{\ast}))} + C  [h + \tilde{h}]_{C^{k+\eps-2}_{\ell}(Q_r(z^{\ast}))} \\
&\le C r^{-k-\eps} \Vert f - p_{z^{\ast}_0} \Vert_{L^{\infty}(H_{4r}(z^{\ast}_0))}  + C [h]_{C^{k+\eps-2}_{\ell}(Q_r(z^{\ast}))}\\
&\qquad + C R^{-k-\eps} \left( \Vert f - p_{z^{\ast}_0} \Vert_{L^{\infty}(H_{R/5}(z^{\ast}_0))}  +  \Vert f \Vert_{L^{\infty}(H_{R/5}(z^{\ast}_0))} \right).
\end{align*}
Again by \eqref{eq:inclusion-regularity}, we are in a position to apply the expansion from either \autoref{lemma:higher-reg-gamma_+} or \autoref{lemma:higher-reg-gamma_--Dirichlet} at the boundary point $z_0^{\ast}$, which yields for any $\rho \in (0,\frac{R}{5}]$:
\begin{align}
\label{eq:appl-exp}
\begin{split}
\Vert f - p_{z^{\ast}_0} \Vert_{L^{\infty}(H_{\rho}(z^{\ast}_0))} &\le C \left(\frac{\rho}{R}\right)^{k+\eps} (1 + |v_0|)^{\frac{k+\eps}{2}} \Big( \Vert f \Vert_{L^{\infty}(H_{R/4}(z^{\ast}_0))} + R^{k+\eps} [ F ]_{C_{\ell}^{k+\eps}(\gamma_- \cap H_{R/4}(z^{\ast}_0))} \\
&\qquad\qquad\qquad + R^{k+\eps} [ h ]_{C^{k+\eps-2}_{\ell}(H_{R/4}(z^{\ast}_0))} \Big).
\end{split}
\end{align}
Then, using \eqref{eq:inclusion-regularity}, we can apply \eqref{eq:appl-exp} with $\rho = 4r \le \frac{R}{5}$ and $\rho = \frac{R}{5}$, and together with the normalization condition \eqref{eq:normalization-reg} this implies 
\begin{align*}
[f]_{C^{k+\eps}_{\ell}(Q_{r/2}(z^{\ast}))} &\le  C r^{-k-\eps} \Vert f - p_{z^{\ast}_0} \Vert_{L^{\infty}(H_{4r}(z^{\ast}_0))}  + C [ h ]_{C^{k+\eps-2}_{\ell}(H_r(z^{\ast}))} \\
&\quad + C R^{-k-\eps} \left( \Vert f - p_{z^{\ast}_0} \Vert_{L^{\infty}(H_{R/4}(z^{\ast}_0))} +  \Vert f \Vert_{L^{\infty}(H_{R/4}(z^{\ast}_0))} \right) \\
&\le C (1 + |v_0|)^{\frac{k+\eps}{2}}  R^{-k-\eps} \left( \Vert f \Vert_{L^{\infty}(H_R(z_0))} + R^{k+\eps} [ F ]_{C_{\ell}^{k+\eps}(H_R(z_0))} + R^{k+\eps} [ h ]_{C^{k+\eps-2}_{\ell}(H_R(z_0))} \right) \\
&\le C,
\end{align*}
which yields \eqref{eq:claim-reg}, as desired. From here, we conclude the proof of the desired regularity estimate by \autoref{lemma:covering}.
\end{proof}

By combination of \autoref{lemma:regularity-inflow} and \autoref{lemma:interior-reg} (applied on small enough radii), together with a standard covering argument we obtain the following immediate corollary, stating a regularity estimate in general cylinders away from $\gamma_0$, which do not need to be centered at a boundary point:

\begin{proposition}
\label{prop:regularity-inflow}
Let $k \in \N$ with $k \ge 3$, $\eps \in (0,1)$ be such that $\frac{k+\eps}{3} \not \in \N$. Let $\Omega \subset \R^n$ be a convex $C^{\frac{k+\eps}{3}}$ domain, $z_0 \in (-1,1) \times \overline{\Omega} \times \R^n$ and $R \in (0,1]$ be such that $Q_{2R}(z_0) \cap \gamma_0 = \emptyset$. Let $a^{i,j},b,c,h \in C^{k+\eps-2}_{\ell}(H_R(z_0))$ and assume that $a^{i,j}$ satisfies \eqref{eq:unif-ell}. Let $F \in C^{k+\eps}_{\ell}(H_R(z_0))$ and let $f$ be a weak solution to 

\begin{equation*}
\left\{\begin{array}{rcl}
\partial_t f + v \cdot \nabla_x f + (-a^{i,j} \partial_{v_i,v_j})f &=& - b \cdot \nabla_v f - c f + h ~~ \text{ in } H_R(z_0), \\
f &=& F ~~~~\qquad\qquad\qquad \text{ on } \gamma_- \cap Q_R(z_0).
\end{array}\right.
\end{equation*}
Then, $f \in C^{k+\eps}_{\ell}(H_{R/2}(z_0))$ and it holds:
\begin{align*}
[ f ]_{C^{k+\eps}_{\ell}(H_{R/2}(z_0))} \le C (1 + |v_0|)^{\frac{k+\eps}{2}} R^{-k-\eps} \big( \Vert f \Vert_{L^{\infty}(H_R(z_0))} + R^{k+\eps} [ F ]_{C_{\ell}^{k+\eps}(H_R(z_0))} + R^{k+\eps} [ h ]_{C^{k+\eps-2}_{\ell}(H_R(z_0))} \big).
\end{align*}
The constant $C$ depends only on $n,k,\eps,\lambda,\Lambda, \Omega$, $\Vert a^{i,j} \Vert_{C^{k+\eps-2}_{\ell}(H_R(z_0))}, \Vert b \Vert_{C^{k+\eps-2}_{\ell}(H_R(z_0))}$, and \\$\Vert c \Vert_{C^{k+\eps-2}_{\ell}(H_R(z_0))}$, but not on $z_0$ and $R$.
\end{proposition}

Let us make two remarks on the previous result:

\begin{remark}
\begin{itemize}
\item[(i)] By application of the local boundedness estimate in \autoref{lemma:boundary-reg}, we can replace the $L^{\infty}$ norm of $f$ by the $L^1$ norm.
\item[(ii)] The $C^{k+\eps}_{\ell}(H_R(z_0))$ norm of $F$ can be replaced by the $C^{k+\eps}_{\ell}(\gamma_- \cap Q_R(z_0))$ norm whenever $\Omega$ is smooth enough so that a trace operator can be defined, which preserves the $C^{k+\eps}_{\ell}$ regularity. This is clearly the case when $\Omega = \{ x_n > 0 \}$. We do not discuss this issue for more general domains in the current paper and refer to \cite{AAMN24,Sil22} for a discussion on kinetic traces for Sobolev spaces. 
\end{itemize}

\end{remark}

\subsection{Specular-reflection: Boundary regularity outside the grazing set}

Now, we are in a position to establish the boundary regularity in $\gamma_-$ for solutions satisfying the specular reflection condition. We will only prove this result in case $\Omega = \{ x_n > 0 \}$. Recall that in this case, $\mathcal{R}_x v = (v',-v_n)$. Moreover, recall the definition of the domains $\mathcal{R}(H_R(z_0))$ and $\mathcal{S}(H_R(z_0))$ from \eqref{eq:reflected-domain-def}. 

We start by establishing the following expansion at boundary points in $\gamma_-$.
 
\begin{lemma}
\label{lemma:higher-reg-gamma_-SR}
Let $\Omega = \{ x_n > 0 \}$, $z_0 \in \gamma_-$ and $R \in (0,1]$ be such that $Q_{2R}(z_0) \cap \gamma_0 = \emptyset$. Let $k \in \N$ with $k \ge 3$, $\eps \in (0,1)$, and $a^{i,j},b,c,h \in C^{k+\eps-2}_{\ell}( \mathcal{S}(H_R(z_0)))$ and assume that $a^{i,j}$ satisfies \eqref{eq:unif-ell}.
Let $f$ be a weak solution to 
\begin{equation*}
\left\{\begin{array}{rcl}
\partial_t f + v \cdot \nabla_x f + (-a^{i,j} \partial_{v_i,v_j})f &=& - b \cdot \nabla_v f - c f + h ~~ \text{ in } \mathcal{S}(H_R(z_0)), \\
f(t,x,v) &=& f(t,x,\mathcal{R}_x v) ~~ \qquad\quad \text{ on } \gamma_- \cap \mathcal{S}(Q_R(z_0)).
\end{array}\right.
\end{equation*}
Then, there exists $p_{z_0} \in \cP_k$ such that for any $r \in (0,\frac{R}{2}]$ and any $z \in H_{r}(z_0)$ it holds:
\begin{align}
\label{eq:higher-reg-gamma_-_simplified-SR}
|f(z) - p_{z_0}(z)| \le C \left( \frac{r}{R} \right)^{k+\eps} \Big(  \Vert f \Vert_{L^{\infty}(\mathcal{S}(H_R(z_0)))} + R^{k+\eps} [ h ]_{C^{k+\eps-2}_{\ell}(\mathcal{S}(H_R(z_0))} \Big).
\end{align}
The constant $C$ depends only on $n,k,\eps,\lambda,\Lambda$, $\Vert a^{i,j} \Vert_{C^{k+\eps-2}_{\ell}(\mathcal{S}(H_R(z_0)))}$, $\Vert b \Vert_{C^{k+\eps-2}_{\ell}(\mathcal{S}(H_R(z_0)))}$, and \\
$\Vert c \Vert_{C^{k+\eps-2}_{\ell}(\mathcal{S}(H_R(z_0)))}$, but not on $z_0$ and $r,R$.
\end{lemma}

\begin{proof}
By \autoref{lemma:Holder-spaces-gamma_-} and \autoref{prop:regularity-inflow}, which is applicable since $\mathcal{R}(\gamma_- \cap Q_{R/2}(z_0)) = \gamma_+ \cap Q_{R/2}(\mathcal{R}(z_0))$ and $\mathcal{R}(z_0) \in \gamma_+$ with $Q_{R}(\mathcal{R}(z_0)) \cap \gamma_0 = \emptyset$ by assumption, it holds
\begin{align}
\label{eq:gamma_--claim}
\begin{split}
[ f(\mathcal{R} \cdot) ]_{C^{k+\eps}_{\ell}(\gamma_- \cap H_{R/2}(z_0))}  & \le C [f]_{C^{k+\eps}_{\ell}(\mathcal{R}(\gamma_- \cap H_{R/2}(z_0)))} \\
& \le C  \left( R^{-k-\eps}\Vert f \Vert_{L^{\infty}(\mathcal{R}(H_R(z_0)))} + [ h ]_{C^{k+\eps - 2}_{\ell}(\mathcal{R}(H_{R}(z_0)))} \right).
\end{split}
\end{align}
Note that we can extend the function $z \mapsto f(\mathcal{R} z)$ from $\gamma_- \cap H_R(z_0)$ to $H_R(z_0)$ in a constant way, i.e. by setting $F(t,x',x_n,v) := f(\mathcal{R}(t,x',0,v))$, we have $F = f(\mathcal{R} \cdot)$ on $\gamma_- \cap H_R(z_0)$ and 
\begin{align*}
[F]_{C^{k+\eps}_{\ell}(H_{R/2}(z_0))} \le C [ f(\mathcal{R} \cdot) ]_{C^{k+\eps}_{\ell}(\gamma_- \cap H_{R/2}(z_0))}.
\end{align*}
By the specular reflection condition, it holds $f = F$ on $\gamma_- \cap Q_R(z_0)$. Therefore, by application of \autoref{lemma:higher-reg-gamma_--Dirichlet-flat}, we obtain for any $r \in (0,\frac{R}{2}]$ and $z \in H_r(z_0)$:
\begin{align}
\label{eq:gamma_-appl1}
\begin{split}
|f(z) - p_{z_0}(z)| &\le C  \left(\frac{r}{R}\right)^{k+\eps} \big( \Vert f \Vert_{L^{\infty}(H_{R/2}(z_0))} + R^{k+\eps} [ f(\mathcal{R} \cdot) ]_{C^{k+\eps}_{\ell}(\gamma_- \cap H_{R/2}(z_0))} \\
&\qquad\qquad\quad\qquad\qquad\qquad\qquad~ + R^{k+\eps}[ h ]_{C^{k+\eps-2}_{\ell}(H_{R/2}(z_0)) } \big).
\end{split}
\end{align}

Hence, the desired result follows by combination of \eqref{eq:gamma_--claim} and \eqref{eq:gamma_-appl1}.
\end{proof}

As before, we deduce boundary regularity estimates away from the grazing set.

\begin{proposition}
\label{prop:regularity-SR}
Let $\Omega = \{x_n > 0 \}$, $z_0 \in (-1,1) \times \overline{\Omega} \times \R^n$, and $R \in (0,1]$ be such that $Q_{2R}(z_0) \cap \gamma_0 = \emptyset$. Let $k \in \N$ with $k \ge 3$, $\eps \in (0,1)$, and $a^{i,j},b,c,h \in C^{k+\eps-2}_{\ell}(\mathcal{S}(H_R(z_0)))$ and assume that $a^{i,j}$ satisfies \eqref{eq:unif-ell}.
Let $f$ be a bounded weak solution to 

\begin{equation*}
\left\{\begin{array}{rcl}
\partial_t f + v \cdot \nabla_x f + (-a^{i,j} \partial_{v_i,v_j})f &=& - b \cdot \nabla_v f - c f + h ~~ \text{ in }  \mathcal{S}(H_R(z_0)), \\
f(t,x,v) &=& f(t,x,\mathcal{R}_x v) ~\qquad\quad \text{ on } \gamma_- \cap \mathcal{S}(Q_R(z_0)).
\end{array}\right.
\end{equation*}
Then, $f \in C^{k+\eps}_{\ell}(H_{R/2}(z_0))$ and it holds:
\begin{align*}
[ f ]_{C^{k+\eps}_{\ell}(H_{R/2}(z_0))} \le C R^{-k-\eps} \big( \Vert f \Vert_{L^{\infty}( \mathcal{S}(H_R(z_0)))} + R^{k+\eps} [ h ]_{C^{k+\eps-2}_{\ell}(\mathcal{S}(H_R(z_0)))} \big).
\end{align*}
The constant $C$ depends only on $n,k,\eps,\lambda,\Lambda$, $\Vert a^{i,j} \Vert_{C^{k+\eps-2}_{\ell}(\mathcal{S}(H_R(z_0)))}, \Vert b \Vert_{C^{k+\eps-2}_{\ell}(\mathcal{S}(H_R(z_0)))}$, and \\
$\Vert c \Vert_{C^{k+\eps-2}_{\ell}(\mathcal{S}(H_R(z_0)))}$, but not on $z_0$ and $R$.
\end{proposition}

Note that by application of the local boundedness estimate in \eqref{eq:boundary-reg-specular-Linfty-L1}, we can also replace the $L^{\infty}$ norm of $f$ by the $L^1$ norm.

\begin{proof}
The proof is exactly the same as the proof of \autoref{lemma:regularity-inflow}. In fact, for $z_0 \in \gamma \setminus \gamma_0$ we claim again that for any $z^{\ast} \in H_{R/16}(z_0)$ it holds 
\begin{align}
\label{eq:claim-reg-2}
[f]_{C^{k+\eps}_{\ell}(Q_{r/2}(z^{\ast}))} \le C R^{-k-\eps} \left(\Vert f \Vert_{L^{\infty}(\mathcal{S}(H_R(z_0)))} + R^{k+\eps}[ h ]_{C^{k+\eps-2}_{\ell}(\mathcal{S}(H_R(z_0)))} \right),
\end{align}
where we define $r := d_{\ell}(z^{\ast}) := \sup \{ r > 0 : \tilde{Q}_r(z^{\ast}) \subset (-1,1) \times \Omega \times \R^n \}$. The proof goes exactly as in \autoref{lemma:regularity-inflow} but uses \autoref{lemma:higher-reg-gamma_-SR} instead of \autoref{lemma:higher-reg-gamma_--Dirichlet}. Once \eqref{eq:claim-reg-2} is proved, we deduce the desired result for $z_0  \in \gamma \setminus \gamma_0$ by \autoref{lemma:covering}. The result for general $z_0 \not\in \gamma$ then follows by a combination with interior estimates and a covering argument in the same way as \autoref{prop:regularity-inflow}.
\end{proof}

We are now in a position to prove that the Tricomi solution satisfies $\mathcal{T}_A \in C^{4,1}$.

\begin{proof}[Proof of \autoref{prop:Tricomi}(iii)]
Let $z^{\ast} \in H_1(0) = ((-1,1) \times \{ x  > 0 \} \times \R) \cap Q_1(0)$ in 1D, and define
\begin{align*}
r &= d_{\ell}(z^{\ast} , \gamma) = \sup\{ r > 0 : Q_r(z^{\ast}) \subset (-1,1) \times \{ x  > 0 \} \times \R \}, \\
d &= d_{\ell}(z^{\ast} , \gamma_0) = \sup\{ r > 0 : Q_r(z^{\ast}) \cap  \gamma_0 = \emptyset \} \ge r,
\end{align*}
where $\gamma_0 = (-1,1) \times \{ 0 \} \times \{ 0 \}$. 
Note that by \autoref{prop:Tricomi}(i) we have that $\mathcal{T}_A$ solves \eqref{eq:Tricomi-PDE}. 

We observe that $r \le d$ and $Q_{d}(z^{\ast}) \cap \gamma_0 = \emptyset$. Hence we can apply \autoref{prop:regularity-SR} (and \autoref{lemma:Holder-interpol}) to deduce for  $\eps \in (0,1)$:
\begin{align*}
[\mathcal{T}_A]_{C^{4,1}(H_{r/4}(z^{\ast}))} &\le [\mathcal{T}_A]_{C^{4,1}(H_{d/4}(z^{\ast}))} \\
&\le C d^{\eps} [\mathcal{T}_A]_{C^{4,1}(H_{d/4}(z^{\ast}))} + C d^{-5} \Vert \mathcal{T}_A \Vert_{L^{\infty}(H_{d/4}(z^{\ast}))} \\
&\le Cd^{-5} \Vert \mathcal{T}_A \Vert_{L^{\infty}(H_{d/2}(z^{\ast}))} + C[p]_{C^{3+\eps}(H_{d/2}(z^{\ast}))} \le d^{-5} \Vert \mathcal{T}_A \Vert_{L^{\infty}(H_{d/2}(z^{\ast}))}, 
\end{align*}
where we also used that $p \in \cP_3$. Note that by the homogeneity of $\mathcal{T}_{A}$ we have in particular
\begin{align}
\label{eq:Tricomi-hom}
\Vert \mathcal{T}_A \Vert_{L^{\infty}(Q_R(0))} \le C R^5 ~~ \forall R > 0.
\end{align}
Thus, by \eqref{eq:Tricomi-hom} and since $H_{d/2}(z^{\ast}) \subset H_{2d}(0)$,  we deduce
\begin{align*}
[\mathcal{T}_A]_{C^{4,1}(H_{r/4}(z^{\ast}))} \le C d^{-5} \Vert \mathcal{T}_A \Vert_{L^{\infty}(H_{2d}(0)))} \le C.
\end{align*}
Since $z^{\ast}$ was arbitrary, we deduce the desired result from \autoref{lemma:covering}.
The dependence of $C$ only on an upper and lower bound of $A$ follows from the respective dependence of constants in \autoref{prop:regularity-SR} and \eqref{eq:Tricomi-hom}.
\end{proof}

\section{Global boundary regularity with specular reflection condition}
\label{sec:reg-gamma0}

In this section we prove the optimal regularity for solutions to kinetic equations subject to the specular reflection boundary condition.

\subsection{Expansion at $\gamma_0$}

In the previous section, we established expansions at boundary points in $\gamma_{\pm}$. It remains to establish an expansion at boundary points which belong to the grazing set $\gamma_0$. We will achieve it in this subsection. This is the most involved part of the theory and it turns out that solutions are in general not smooth near these points. The reason for this behavior is the existence of Tricomi-type solutions in the Liouville-type theorem in the half-space (see \autoref{thm:Liouville-higher-half-space-SR}).

We will establish an expansion of order $5 +\eps$ in the half-space, which is sufficient to deduce optimal regularity estimates in general domains via the transformation $\Phi$ in \eqref{lemma:trafo-sr}. Clearly, expansions of order $k+\eps$ for $k \in \{3,4\}$ can be established in the same way (under weaker assumptions on the coefficients).

We introduce the space $\tilde{\cP}_k$ of polynomials satisfying the specular reflection boundary condition in the half-space
\begin{align*}
\tilde{\cP}_k = \{p \in \cP_k : p(t,x',0,v',v_n) = p(t , x' , 0 , v', -v_n) \},
\end{align*}
and given $a > 0$ we also define the space $\tilde{\cP}_{a,5}$ which contains $\tilde{\cP}_5$ but also the Tricomi solution $\mathcal{T}_a(t,x,v) := \mathcal{T}_{a,3}(x_n,v_n)$ from \autoref{thm:Liouville-higher-half-space-SR}, namely
\begin{align*}
\tilde{\cP}_{a,5} = \{P =  p + \tau \mathcal{T}_a : p \in \tilde{\cP}_5 ,~ \tau \in \R \}.
\end{align*}
Recall that $\mathcal{T}_a$ is homogeneous of degree $5$, and hence the space $\tilde{\cP}_{a,5}$ is finite dimensional and all its basis functions are homogeneous of degree less or equal $5$. In particular the space $\tilde{\cP}_{a,5}$ satisfies the assumptions of \autoref{lemma:L43} and \autoref{lemma:orth-proj-prop}, since it is of the form $\tilde{\cP}_{a,5} = \cP_{I_{a,5}}$, where
\begin{align*}
I_{a,5} := \tilde{I}_5 \cup \{ \mathcal{T}_a \} := \{ q \in I_5 : q(t,x',0,v',v_n) = q(t,x',0,v',-v_n) \} \cup \{ \mathcal{T}_{a} \}.
\end{align*}
The definition of $I_5$ can be found in \eqref{eq:Ik}.
Note that all elements in $\tilde{\cP}_{a,5}$ satisfy the specular reflection boundary condition in the half-space.

We are now in a position to establish the expansion at points in $\gamma_0$. Recall that since $\Omega = \{ x_n > 0 \}$, it holds $\mathcal{R}_x v = (v',-v_n)$.

\begin{lemma}
\label{lemma:higher-reg-gamma_0-SR-halfspace}
Let $\Omega = \{ x_n > 0 \}$ and $z_0 \in \gamma_0$, i.e. $(x_0)_n = (v_0)_n = 0$. Let $R \in (0,1]$, $\eps \in (0,1)$, and $a^{i,j},b,c,h \in C^{3+\eps}_{\ell}(H_R(z_0))$ and assume that $a^{i,j}$ satisfies \eqref{eq:unif-ell}.
Let $f$ be a weak solution to 

\begin{equation*}
\left\{\begin{array}{rcl}
\partial_t f + v \cdot \nabla_x f + (-a^{i,j} \partial_{v_i,v_j})f &=& - b \cdot \nabla_v f - c f + h ~~ \text{ in } H_R(z_0), \\
f(t,x,v) &=& f(t,x,\mathcal{R}_xv) ~\qquad\quad \text{ on } \gamma_- \cap Q_R(z_0).
\end{array}\right.
\end{equation*}
Then, there exists $P_{z_0} \in \tilde{\cP}_{a,5}$, where $a := a^{n,n}(z_0)$, such that for any $r \in (0,\frac{R}{2}]$ and any $z \in H_{r}(z_0)$ it holds:
\begin{align}
\label{eq:higher-reg-gamma_0_simplified_halfspace}
|f(z) - P_{z_0}(z)| \le C \left(\frac{r}{R} \right)^{5+\eps} \big( \Vert f \Vert_{L^{\infty}(H_R(z_0))} + R^{5+\eps} [ h ]_{C^{3+\eps}_{\ell}(H_R(z_0))} \big).
\end{align}
The constant $C$ depends only on $n,\eps,\lambda,\Lambda$, $\Vert a^{i,j} \Vert_{C^{3+\eps}_{\ell}(H_R(z_0))}, \Vert b \Vert_{C^{3+\eps}_{\ell}(H_R(z_0))}$, and $\Vert c \Vert_{C^{3+\eps}_{\ell}(H_R(z_0))}$, but not on $z_0$ and $r,R$.
\end{lemma}

Note that for $z_0 \in \gamma_0$ it holds $Q_R(z_0) = \mathcal{R}(Q_R(z_0))$.

\begin{proof}
We prove the claim \eqref{eq:higher-reg-gamma_0_simplified_halfspace} in several steps, following the overall strategy from \autoref{lemma:higher-reg-gamma_+} and \autoref{lemma:higher-reg-gamma_--Dirichlet}.

\textbf{Step 1:} 
We assume by contradiction that \eqref{eq:higher-reg-gamma_0_simplified_halfspace} does not hold. Then, there exist $R_l \in (0,1]$, $f_l$, $h_l$, $a^{i,j}_{l}$, $b_l$, $c_l$, $z^0_{l}$ such that for any $l \in \N$
\begin{align}
\label{eq:blow-up-normalization-0}
\begin{split}
R_l^{-5-\eps} \left(\Vert f_l \Vert_{L^{\infty}(H_{R_l}(z^0_{l}))} + R_l^{5+\eps} [ h_l ]_{C^{3+\eps}_{\ell}(H_{R_l}(z^0_{l}))}\right) &\le 1, \\
 \Vert a^{i,j}_l \Vert_{C^{3+\eps}_{\ell}(H_{R_l}(z^0_{l}))} + \Vert b_l \Vert_{C^{3+\eps}_{\ell}(H_{R_l}(z^0_{l}))} + \Vert c_l \Vert_{C^{3+\eps}_{\ell}(H_{R_l}(z^0_{l}))} &\le \Lambda,
\end{split}
\end{align}
and also \eqref{eq:unif-ell} holds true with $\lambda$ for any $l \in \N$, such that the $f_l$ are solutions to 

\begin{equation*}
\left\{\begin{array}{rcl}
\partial_t f_l + v \cdot \nabla_x f_l + (-a^{i,j}_l \partial_{v_i,v_j})f_l &=& - b_l \cdot \nabla_v f_l - c_l f_l  + h_l ~~ \text{ in } H_{R_l}(z^0_{l}),\\
f_l(t,x',0,v',v_n) &=& f_l(t,x',0,v',-v_n) ~\qquad \text{ on } \gamma_- \cap Q_{R_l}(z^0_l).
\end{array}\right.
\end{equation*}
but it holds
\begin{align}
\label{eq:contradiction-assumption-0}
\sup_{l \in \N} \inf_{P \in \tilde{\cP}_{a_l,5}} \sup_{ r \in (0,\frac{R_l}{2}] }  \frac{\Vert f_l - P \Vert_{L^{\infty}(H_{r}(z^0_{l}))}}{r^{5+\eps}} = \infty,
\end{align}
where we set $a_l = a_l^{n,n}(z_l^0)$. Note that $a_l \ge \lambda > 0$ by \eqref{eq:unif-ell}.

For $r \in (0,1]$ we consider the $L^2(H_r(z^0_{l}))$ projections of $f_l$ over $\tilde{\cP}_{a_l,5}$ and denote them by $P_{l,r} \in \tilde{\cP}_{a_l,5}$.
Moreover, we introduce the quantity
\begin{align*}
\theta(r) = \sup_{l \in \N} \sup_{\rho \in [r,\frac{R_l}{2}]} \rho^{-5-\eps} \Vert f_l - P_{l,\rho} \Vert_{L^{\infty}(H_{\rho}(z^0_{l}))},
\end{align*}
and we deduce from \autoref{lemma:L43} that $\theta(r) \nearrow \infty$ as $r \searrow 0$ in the same way as in the proof of \autoref{lemma:higher-reg-gamma_+}. Thus, we can extract further subsequences $(r_m)_m$ and $(l_m)_m$ with $r_m \searrow 0$, such that
\begin{align*}
\frac{\Vert f_{l_m} - P_{l_m,r_m} \Vert_{L^{\infty}(H_{r_m}(z_{l_m}^0))}}{r_m^{5+\eps} \theta(r_m)} \ge \frac{1}{2} ~~ \forall m \in \N, \qquad R_m r_m^{-1} \to \infty ~~ \text{ as } m \to \infty.
\end{align*} 
The second property follows in the exact same way as in the proof of \autoref{lemma:higher-reg-gamma_+}.

Moreover, we define for any $R > 0$ and $m \in \N$ the rescaled domains  $H_R^{(m)} := (T_{z_m^0,r_m} \times \R^n) \cap Q_{R}(0)$ and consider the function
\begin{align*}
g_m(z) := \frac{f_{l_m}(z_{l_m}^0 \circ S_{r_m} z ) - P_{l_m,r_m}(z_{l_m}^0 \circ S_{r_m} z)}{r_m^{5+\eps} \theta(r_m)}.
\end{align*}
By construction, and using the same arguments as in the proof of \autoref{lemma:higher-reg-gamma_+} and \autoref{lemma:orth-proj-prop}, we have for any $m \in \N$
\begin{align}
\label{eq:gm-prop-1_0}
\int_{H^{(m)}_1} g_m(z) p(z_m^0 \circ S_{r_m} z) \d z &= 0 ~~ \forall p \in \tilde{\cP}_{a_m,5},\\
\label{eq:gm-prop-2_0}
\Vert g_m \Vert_{L^{\infty}(H^{(m)}_1)} \ge \frac{1}{2}, \qquad \Vert g_m \Vert_{L^{\infty}(H^{(m)}_R)} &\le c R^{5+\eps} ~~ \forall R \in \left[1,\frac{R_m}{2r_m} \right], ~~ \forall m \in \N,
\end{align}
where we set $z^0_m = z^0_{l_m}$ and $a_m = a_{l_m}$. In fact, by writing for some $p_{l,r} \in \tilde{\cP}_5$ and $\tau_{l,r}, \alpha_{l,r}^{(q)} \in \R$
\begin{align*}
P_{l,r} = p_{l,r} + \tau_{l,r} \mathcal{T}_{a_l}, \qquad \text{ where } \qquad p_{l,r}(z_l^0 \circ z) =: \bar{p}_{l,r}(z) = \sum_{q \in \tilde{I}_{5}} \alpha^{(q)}_{l,r} q(z),
\end{align*}
and observing that $\bar{P}_{l,r} = \bar{p}_{l,r} + \tau_{l,r} \mathcal{T}_{a_l} \in \tilde{\cP}_{a_l,5}$ since the Tricomi solution satisfies $\mathcal{T}_{a_l}(z) = \mathcal{T}_{a_l}(z_l^0 \circ z)$ since $z_l^0 \in \gamma_0$, we can apply \autoref{lemma:orth-proj-prop} and deduce the second property in \eqref{eq:gm-prop-2_0}. Moreover, \autoref{lemma:orth-proj-prop} also implies
\begin{align}
\label{eq:coefficients-vanish-0}
\sup_{l \in \N} \frac{|\tau_{l,r}|}{\theta(r)} + \sup_{l \in \N} \frac{|\alpha^{(q)}_{l,r}|}{\theta(r)} \to 0 ~~ \text{ as } r \to 0, ~~ \forall q \in \tilde{I}_{5}.
\end{align}

\textbf{Step 2:} As in the proofs of \autoref{lemma:higher-reg-gamma_+} and \autoref{lemma:higher-reg-gamma_--Dirichlet} we investigate the equation that is satisfied by $g_m$ and take the limit $m \to \infty$. The fundamental difference compared to the previous proofs is that the functions $P_{l_m,r_m} \in \tilde{\cP}_{a_{l_m},5}$ are not polynomials but are of the form $P_{l_m,r_m} = p_{l_m,r_m} + \tau_{l_m,r_m} \mathcal{T}_{a_m}$ for some polynomial $p_{l_m,r_m} \in \tilde{\cP}_5$. Hence, we need to repeat the arguments from \autoref{lemma:higher-reg-gamma_+} in order to treat the polynomial part of $P_{l_m,r_m}$ and derive new tools to deal with the Tricomi remainder. To do so, we recall the definitions of $\tilde{p}_{l_m,r_m}, \tilde{p}_{l_m,r_m}^{(i,j)} \in \cP_{k-2}$, $\tilde{p}_{l_m,r_m}^{(i)} \in \cP_{k-1}$ and $p_{l_m,r_m} \in \cP_k$ from \eqref{eq:PDE-polynomial} and define
\begin{align*}
\tilde{h}_m^{(p)} (z) = \frac{h_m(z_m^0 \circ S_{r_m} z ) - (\tilde{p}_{l_m,r_m} + a^{i,j}_m \tilde{p}_{l_m,r_m}^{(i,j)} + b^i_m\tilde{p}_{l_m,r_m}^{(i)} + c_m p_{l_m,r_m})(z_m^0 \circ S_{r_m} z )}{r_m^{3+\eps} \theta(r_m)},
\end{align*}
and
\begin{align*}
\tilde{h}_m^{(\mathcal{T})}(z) = \tau_{l_m,r_m}\frac{[\partial_t + v \cdot \nabla_x + \tilde{a}_m^{i,j} \partial_{v_i,v_j} + \tilde{b}_m^i \partial_{v_i} + \tilde{c}_m ] [\mathcal{T}_{a_m}(z_m^0 \circ S_{r_m} \cdot)](z)}{r_m^{5+\eps} \theta(r_m)}.
\end{align*}
Then, it holds

\begin{equation}
\label{eq:PDE-gm-0}
\left\{\begin{array}{rcl}
\partial_t g_m + v \cdot \nabla_x g_m + (-\tilde{a}^{i,j}_m \partial_{v_i,v_j}) g_m + \tilde{b}^{i}_m \partial_{v_i} g_m + \tilde{c}_m g_m &=& \tilde{h}_m^{(p)} + \tilde{h}_m^{(\mathcal{T})} =: \tilde{h}_m  ~~~ \text { in } H^{(m)}_{R_m r_m^{-1}},\\
g_m(t,x',0,v',v_n) &=& g_m(t,x',0,v',-v_n) ~~ \text{ in } Q^{(m)}_{R_m r_m^{-1}},
\end{array}\right.
\end{equation}
where
\begin{align*}
\tilde{a}^{i,j}_m(z) = a^{i,j}_m(z_m^0 \circ S_{r_m} z), ~~ \tilde{b}^i_m(z) = r_m b^i_m(z_m^0 \circ S_{r_m} z), ~~ \tilde{c}_m(z) = r_m^2 c_m(z_m^0 \circ S_{r_m} z).
\end{align*}

Let us first argue why $g_m$ satisfies the specular reflection condition. Note that clearly, the domain $\Omega$ does not change under the transformation $z_m^0 \circ S_{r_m} \cdot$, since $(x_m^0)_n = (v_m^0)_n = 0$.

The boundary condition for $g_m$ follows from the fact that for any $z = (t,x',0,v) \in Q^{(m)}_{R_m r_m^{-1}}$ we have $z_0 \circ S_{r_m} z \in \gamma \cap \tilde{Q}_{R_m}(z_0)$, i.e. $0 = (x_m^0 + r_m^3 x + r_m^2 t v_m^0)_n = r_m^3 (x_n)$, and moreover,
\begin{align*}
\mathcal{R}_{x_m^0 + r_m^3 x + r_m^2 t v_m^0} (v_m^0 + r_m v) &= (v_m^0 + r_m v) - 2 (v_m^0 + r_m v)_n \\
&= (v_m^0 + r_m v) - 2 r_m v \cdot e_n = (v_m^0 + r_m \mathcal{R}_{x_m^0 + r_m^3 x + r_m^2 t v_m^0} v).
\end{align*}
Hence, since $f_{l_m}$, $p_{l_m,r_m}$ and $\mathcal{T}_{a_m}$ all satisfy the specular reflection condition with respect to $\{ x_n > 0 \}$, we have
\begin{align*}
f_{l_m}(z_m^0 \circ S_{r_m} (t,x,v)) &= f_{l_m}(t_m^0 + r_m^2 t , x_m^0 + r_m^3 x + r_m^2 t v_m^0 , \mathcal{R}_{x_m^0 + r_m^3 x + r_m^2 t v_m^0} (v_m^0 + r_m v)) \\
&= f_{l_m}(t_m^0 + r_m^2 t , x_m^0 + r_m^3 x + r_m^2 t v_m^0 , v_m^0 + r_m \mathcal{R}_{x_m^0 + r_m^3 x + r_m^2 t v_m^0}(v))\\
&= f_{l_m}(z_m^0 \circ S_{r_m} (t,x,\mathcal{R}_x v)),
\end{align*}
and analogously for $p_{l_m,r_m}$, and $\mathcal{T}_{a_m}$.

Our next goal is to take the limit $m \to \infty$ in the equation for $g_m$. To do so, we recall that by \eqref{eq:hm-vanishes-Holder}
\begin{align}
\label{eq:hm-vanishes-Holder-0}
[\tilde{h}_m^{(p)}]_{C_{\ell}^{3+\eps}(H^{(m)}_{R_m r_m^{-1}})} \to 0 ~~ \text{ as } m \to \infty.
\end{align}
For $\tilde{h}_m^{(\mathcal{T})}$ we argue as follows: First, we compute using that $\mathcal{T}_{a_m}$ solves  
\begin{align*}
\frac{\tau_{l_m,r_m}[\partial_t + v \cdot \nabla_x + a^{i,j}_m(z_m^0) \partial_{v_i,v_j}]\mathcal{T}_{a_m}}{r_m^{3+\eps} \theta(r_m)} =:p_{m}^{(\mathcal{T})}  ~~ \text{ in } \R \times \{ x_n > 0 \} \times \R^n
\end{align*}
for some $p_{m}^{(\mathcal{T})} \in \cP_3$ by \autoref{prop:Tricomi}(i) (see also \autoref{lemma:Liouville-higher-half-space-SR-1D-hom}), and thus
\begin{align*}
\tilde{h}_m^{(\mathcal{T})}(z) &=  \tau_{l_m,r_m} \frac{[\partial_t + v \cdot \nabla_x + \tilde{a}^{i,j}_m \partial_{v_i,v_j} + r_m^{-1} \tilde{b}_m^i \partial_{v_i} + r_m^{-2} \tilde{c}_m ] \mathcal{T}_{a_m}(z_m^0 \circ S_{r_m} z)}{r_m^{3+\eps} \theta(r_m)} \\
&= p_m^{(\mathcal{T})}(z_m^0 \circ S_{r_m} z) + \frac{\tau_{l_m,r_m}}{r_m^{3+\eps} \theta(r_m)}(\tilde{a}^{i,j}_m(z) - a^{i,j}_m(z_m^0) ) \partial_{v_i,v_j}  \mathcal{T}_{a_m}(z_m^0 \circ S_{r_m} z)\\
&\quad +  \frac{\tau_{l_m,r_m}}{r_m^{4+\eps} \theta(r_m)} \tilde{b}^i_m(z) \partial_{v_i}  \mathcal{T}_{a_m}(z_m^0 \circ S_{r_m} z) +  \frac{\tau_{l_m,r_m}}{r_m^{5+\eps} \theta(r_m)} \tilde{c}_m(z) \mathcal{T}_{a_m}(z_m^0 \circ S_{r_m} z) \\
&= p_m^{(\mathcal{T})}(z_m^0 \circ S_{r_m} z) + \frac{\tau_{l_m,r_m}}{r_m^{\eps} \theta(r_m)}(a^{n,n}_m(z_m^0 \circ S_{r_m}z) - a^{n,n}_m(z_m^0) ) \partial_{v_n,v_n}  \mathcal{T}_{a_m}(x_n,v_n)\\
&\quad + \frac{\tau_{l_m,r_m} r_m^{1-\eps}}{ \theta(r_m)} b^n_m(z_0^m \circ S_{r_m} z) \partial_{v_n}  \mathcal{T}_{a_m}(x_n,v_n) +  \frac{\tau_{l_m,r_m} r_m^{2-\eps}}{ \theta(r_m)} c_m(z_0^m \circ S_{r_m} z) \mathcal{T}_{a_m}(x_n,v_n) \\
&=: p_m^{(\mathcal{T})}(z_m^0 \circ S_{r_m} z) + \tilde{h}_m^{(\mathcal{T},a)}(z) + \tilde{h}_m^{(\mathcal{T},b)}(z) + \tilde{h}_m^{(\mathcal{T},c)}(z).
\end{align*}
In the last step we used the definitions of $\tilde{a}^{i,j}_m,\tilde{b}^i_m, \tilde{c}_m$. Moreover, we used that since $z^0_m \in \gamma_0$ it holds $\mathcal{T}_{a_m} (z_m^0 \circ S_{r_m} z) = \mathcal{T}_{a_m}(r_m^3 x_n, r_m v_n)$, and that since $\mathcal{T}_{a_m}$ is homogeneous of degree $5$ and since it is locally $C^{4,1}_{\ell}$ by \autoref{prop:Tricomi}, we have that $\partial_{v_i, v_j} \mathcal{T}_{a_m}$ exists in the classical sense and is homogeneous of degree $3$ (unless $i \not= n$ or $j \not= n$, in which case it is zero), i.e.
\begin{align*}
\partial_{v_n} \mathcal{T}_{a_m}(z_m^0 \circ S_{r_m} z) &= \partial_{v_n} \mathcal{T}_{a_m}(r_m^3 x_n, r_m v_n) = r_m^4 \partial_{v_n}\mathcal{T}_{a_m}(x_n,v_n), \\
\partial_{v_n,v_n} \mathcal{T}_{a_m}(z_m^0 \circ S_{r_m} z) &= \partial_{v_n,v_n} \mathcal{T}_{a_m}(r_m^3 x_n, r_m v_n) = r_m^3 \partial_{v_n,v_n}\mathcal{T}_{a_m}(x_n,v_n).
\end{align*}

We will investigate the terms $\tilde{h}_m^{(\mathcal{T},a)},\tilde{h}_m^{(\mathcal{T},b)},\tilde{h}_m^{(\mathcal{T},c)}$, separately, and prove that they tend to zero as $m \to \infty$ in $C^{\eps}_{\ell}$.

To prove it, let us first observe that by rescaling \autoref{prop:Tricomi}(iii) and using also \autoref{lemma:der-Holder}, we can deduce that for any $R \ge 1$ there is a constant $C(R)$, independent of $m$, such that
\begin{align}
\label{eq:Tricomi-derivatives}
\Vert \partial_{v_n,v_n} \mathcal{T}_{a_m} \Vert_{C^{\eps}_{\ell}(Q_R)} + \Vert \partial_{v_n} \mathcal{T}_{a_m} \Vert_{C^{\eps}_{\ell}(Q_R)} \le C\Vert \mathcal{T}_{a_m} \Vert_{C^{4,1}_{\ell}(Q_R)} \le C(R).
\end{align}
Moreover, note that since $\Vert a_m^{n,n} \Vert_{C^{\eps}_{\ell}(H_{R_m}(z_m^0))} \le 1$ by \eqref{eq:blow-up-normalization-0}, and using also \autoref{lemma:Holder-scaling}, for any $R \le r_m^{-1}$:
\begin{align*}
\Vert a_m^{n,n}(z_m^0 \circ S_{r_m} \cdot) - a_m^{n,n}(z_m^0) \Vert_{L^{\infty}(H_{R}^{(m)})} + [ a_m^{n,n}(z_m^0 \circ S_{r_m} \cdot) ]_{C^{\eps}_{\ell}(H_{R}^{(m)})} \le C(R) r_m^{\eps}[ a_m^{n,n} ]_{C^{\eps}_{\ell}(H_{R r_m}(z_m^0))} \le 2 r_m^{\eps}.
\end{align*}

Consequently, we have for any $R \ge 1$ by \autoref{lemma:product-rule},  and \eqref{eq:Tricomi-derivatives}
\begin{align*}
\Vert \tilde{h}_m^{(\mathcal{T},a)} \Vert_{C^{\eps}_{\ell}(H^{(m)}_R)} &\le \frac{|\tau_{l_m,r_m}|}{r_m^{\eps} \theta(r_m)} \Vert a^{n,n}_m(z_m^0 \circ S_{r_m} \cdot) - a^{n,n}_m(z_m^0) \Vert_{L^{\infty}(H_{R}^{(m)})} \Vert \partial_{v_n,v_n} \mathcal{T}_{a_m} \Vert_{C^{\eps}_{\ell}(H^{(m)}_R)} \\
&\le C(R) \frac{|\tau_{l_m,r_m}|}{\theta(r_m)} \to 0 ~~ \text { as } m \to \infty,
\end{align*}
where we used that $r_m \to 0$ and \eqref{eq:coefficients-vanish-0} in the last step. 

For the drift and the zero order term, we obtain by similar considerations, \autoref{lemma:product-rule}, \eqref{eq:Tricomi-derivatives}, as well as locally uniform bound of $b_m, c_m$ in $C^{\eps}_{\ell}$ by \eqref{eq:blow-up-normalization-0} that for any $R \ge 1$
\begin{align*}
\Vert \tilde{h}_m^{(\mathcal{T},b)} \Vert_{C^{\eps}_{\ell}(H^{(m)}_R)} &\le C \Vert b_m^n \Vert_{C^{\eps}_{\ell}(H_{Rr_m}(z_m^0))} \frac{|\tau_{l_m,r_m}|}{\theta(r_m)} r_m^{1-\eps} \Vert \partial_{v_n} \mathcal{T}_{a_m} \Vert_{C^{\eps}_{\ell}(Q_R)} \le C(R) \frac{|\tau_{l_m,r_m}|}{\theta(r_m)} r_m^{1-\eps}  \to 0,
\end{align*}
as well as
\begin{align*}
\Vert \tilde{h}_m^{(\mathcal{T},c)} \Vert_{C^{\eps}_{\ell}(H^{(m)}_R)} &\le C \Vert c_m \Vert_{C^{\eps}_{\ell}(H_{Rr_m}(z_m^0))} \frac{|\tau_{l_m,r_m}|}{\theta(r_m)} r_m^{2-\eps} \Vert \mathcal{T}_{a_m} \Vert_{C^{\eps}_{\ell}(Q_R)} \le C(R) \frac{|\tau_{l_m,r_m}|}{\theta(r_m)} r_m^{2-\eps}  \to 0,
\end{align*}
as $m \to \infty$. Altogether, we have shown that
\begin{align*}
\Vert \tilde{h}_m^{(\mathcal{T})} - p_m^{(\mathcal{T})} \Vert_{C^{\eps}_{\ell}(H^{(m)}_R)} &\le \Vert \tilde{h}_m^{(\mathcal{T},a)} \Vert_{C^{\eps}_{\ell}(H_R^{(m)})} + \Vert \tilde{h}_m^{(\mathcal{T},b)} \Vert_{C^{\eps}_{\ell}(H_R^{(m)})} + \Vert \tilde{h}_m^{(\mathcal{T},c)} \Vert_{C^{\eps}_{\ell}(H_R^{(m)})} \to 0 ~~ \text{ as } m \to \infty.
\end{align*}

Hence, by \eqref{eq:hm-vanishes-Holder-0} there exists $p_m^{(p)} \in \cP_3$ such that for $p_m = p_m^{(p)} + p_m^{(\mathcal{T})}$ and every $R \ge 1$ it holds
\begin{align}
\label{eq:hm-conv-prelim-0}
\begin{split}
\Vert \tilde{h}_m - p_m \Vert_{C^{\eps}_{\ell}(H^{(m)}_R)} &\le \Vert \tilde{h}_m^{(p)} - p_m^{(p)} \Vert_{C^{\eps}_{\ell}(H^{(m)}_R)} + \Vert \tilde{h}_m^{(\mathcal{T})} - p_m^{(\mathcal{T})} \Vert_{C^{\eps}_{\ell}(H^{(m)}_R)} \\
&\le CR^3 [ \tilde{h}_m^{(p)} ]_{C^{3+\eps}_{\ell}(H^{(m)}_R)} + \Vert \tilde{h}_m^{(\mathcal{T})} - p_m^{(\mathcal{T})} \Vert_{C^{\eps}_{\ell}(H^{(m)}_R)} \to 0 ~~ \text{ as } m \to \infty.
\end{split}
\end{align}
Therefore, using also \autoref{lemma:osc-results} and \eqref{eq:gm-prop-2_0}, we have for any $R \le \frac{R_m}{2r_m}$
\begin{align*}
\Vert \tilde{h}_m \Vert_{L^{\infty}(H^{(m)}_R)} \le C \Vert g_m \Vert_{L^{\infty}(H^{(m)}_R)} + C \Vert \tilde{h}_m - p_m \Vert_{C^{\eps}_{\ell}(H^{(m)}_R)} \le C(R),
\end{align*}
and therefore
\begin{align}
\label{eq:poly-bd-0}
\Vert p_m \Vert_{L^{\infty}(H^{(m)}_R)} \le \Vert \tilde{h}_m - p_m \Vert_{L^{\infty}(H^{(m)}_R)} + \Vert \tilde{h}_m \Vert_{L^{\infty}(H^{(m)}_R)}  \le C(R).
\end{align}

Altogether, this implies that for some $p_0 \in \cP_{3}$:
\begin{align}
\label{eq:hm-convergence-0}
\Vert \tilde{h}_m \Vert_{L^{\infty}(H^{(m)}_{R})} \le  C(R), \qquad \Vert \tilde{h}_m - p_0 \Vert_{L^{\infty}(H_R^{(m)})} \to 0.
\end{align}

We observe that since all norms in the finite-dimensional space $\tilde{\cP}_{3}$ are equivalent, we have from \eqref{eq:poly-bd-0} 
\begin{align*}
\Vert p_m \Vert_{C^{\eps}_{\ell}(H_R^{(m)})} \le C(R),
\end{align*}
where it is not difficult to see that $C(R)$ can be chosen independently of $m$. Consequently, by \eqref{eq:hm-conv-prelim-0} we have for any $R \le \frac{R_m}{2r_m}$
\begin{align}
\label{eq:hm-bound-0}
\Vert \tilde{h}_m \Vert_{C^{\eps}_{\ell}(H^{(m)}_{R})} \le \Vert \tilde{h}_m - p_m \Vert_{C^{\eps}_{\ell}(H^{(m)}_{R})} + \Vert p_m \Vert_{C^{\eps}_{\ell}(H^{(m)}_{R})} \le C(R).
\end{align}

Moreover, as in the proof of \autoref{lemma:higher-reg-gamma_+} we have by \eqref{eq:blow-up-normalization} and the definitions of $\tilde{a}^{i,j}_m$, $\tilde{b}^i_m$, and $\tilde{c}_m$ that (up to a subsequence)  for any $R > 0$
\begin{align*}
\Vert \tilde{a}^{i,j}_m - a^{i,j} \Vert_{L^{\infty}(H^{(m)}_R)} + \Vert \tilde{b}^i_m \Vert_{L^{\infty}(H^{(m)}_R)} + \Vert \tilde{c}_m \Vert_{L^{\infty}(H^{(m)}_R)} \to 0
\end{align*}
for some constant matrix $(a^{i,j})$ satisfying \eqref{eq:unif-ell}. Moreover, we denote $a := a^{n,n} \ge \lambda$.
Note that $T_{z_m^0,r_m} \to \R \times \{ x_n > 0 \} \times \R^n$ by \autoref{lemma:domains-convergence} and since $R_m r_m^{-1} \to \infty$.

Thus, by the interior regularity from \autoref{lemma:interior-reg} and using again \autoref{lemma:osc-results}, we have for any kinetic cylinder $Q_1(z_1) \subset \R \times \{ x_n > 0  \} \times \R^n$ (for $m$ large enough) using also \eqref{eq:hm-bound-0}
\begin{align*}
\Vert g_m \Vert_{C^{2+\eps}(Q_{1/2}(z_1))} \le C \big( \Vert g_m \Vert_{L^{\infty}(Q_1(z_1))} + [ \tilde{h}_m ]_{C^{\eps}_{\ell}(Q_1(z_1))}  \big) \le C,
\end{align*}
and therefore by the Arzel\`a-Ascoli theorem it holds that $g_m \to g$  and also
\begin{align*}
\partial_t g_m + v \cdot \nabla_x g_m + (-\tilde{a}^{i,j}_m \partial_{v_i,v_j}) g_m + \tilde{b}^{i}_m \partial_{v_i} g_m + \tilde{c}_m g_m \to \partial_t g + v \cdot \nabla_x g + (-a^{i,j} \partial_{v_i,v_j})g
\end{align*}
locally uniformly in $\R \times \{ x_n > 0 \} \times \R^n$.
Hence, by combining \eqref{eq:PDE-gm-0} and \eqref{eq:hm-convergence-0}, we obtain that $g$ solves the following limiting equation (in the classical sense by interior regularity)
\begin{align}
\label{eq:g-equation-0}
\partial_t g + v \cdot \nabla_x g + (-a^{i,j} \partial_{v_i,v_j})g = p_0 ~~ \text{ in } \R \times \{ x_n > 0 \} \times \R^n
\end{align}
for some polynomial $p_0 \in \cP_{3}$ and some constant matrix $(a^{i,j})$.

Moreover, by taking the second property in \eqref{eq:gm-prop-2_0} to the limit, we obtain 
\begin{align}
\label{eq:g-growth_0}
\Vert g \Vert_{L^{\infty}(Q_R)} \le c R^{3+\eps} ~~ \forall R \ge 1.
\end{align}

\textbf{Step 3:} 
We claim that for any  $z_1 = (t_1,x_1,v_1)$ with $t_1 \in \R$, $x_1 \in \{ x_n = 0 \}$, and $v \in \R^n$:
\begin{align}
\label{eq:bdry-conv-higher-0-halfspace}
\Vert g_m - g \Vert_{L^{\infty}(H^{(m)}_1(z_1))} \to 0 ~~ \text{ as } m \to \infty.
\end{align}

To prove it, we apply the boundary $C^{\alpha}$ estimate from \autoref{lemma:boundary-reg-specular} to $g_m$, which is applicable since $g_m$ satisfies the specular reflection boundary condition (see \eqref{eq:PDE-gm-0}). In fact, using \autoref{lemma:boundary-reg-specular}, \autoref{lemma:osc-results}, and \eqref{eq:hm-bound-0} we get for $m$ large enough:
\begin{align*}
[g_m]_{C^{\alpha}(H^{(m)}_1(z_1))} &\le C \big( \Vert g_m \Vert_{L^{\infty}(\mathcal{S}(H^{(m)}_2(z_1)))} + [\tilde{h}_m]_{C^{\eps}_{\ell}(\mathcal{S}(H^{(m)}_2(z_1)))} \big) \\
&\le C \big( \Vert g_m \Vert_{L^{\infty}(H^{(m)}_{8(|z_1| + 1)})} + [\tilde{h}_m]_{C^{\eps}_{\ell}(H^{(m)}_{8(|z_1| + 1)})} \big) \le C,
\end{align*}
where $C$ depends on $z_1$, but not on $m$.
In the second estimate, we used that by \autoref{lemma:kinetic-balls} we have the inclusion of sets
\begin{align*}
Q_{2}(z_1) \subset Q_{2(|z_1|+1)}(z_1) \subset Q_{8(|z_1| + 1)}, \qquad \mathcal{R}(Q_2(z_1)) \subset Q_{8(|z_1| + 1)},
\end{align*}
since $z_1 \in Q_{2(|z_1|+1)}(0)$. Hence, by the  Arzel\`a-Ascoli theorem we deduce  \eqref{eq:bdry-conv-higher-0-halfspace}, as desired.

Moreover, by the energy estimate in \autoref{lemma:energy-est} applied to $g_m$ in $H_{R}^{(m)}(0)$, and using the uniform bounds \eqref{eq:gm-prop-2_0}, \eqref{eq:hm-bound-0}, we get that 
\begin{align*}
\Vert \nabla_v g_m \Vert_{L^{2}(H_R^{(m)}(0))} \le C(R).
\end{align*}
Hence, it must also hold $\nabla_v g \in L^2_{loc}(\R \times \{ x_n > 0 \} \times \R^n)$.

\textbf{Step 4:} 
As an immediate consequence of Step 3 and \eqref{eq:bdry-conv-higher-0-halfspace} we have that $g$ is a weak solution to \eqref{eq:g-equation-0} and also satisfies the specular reflection boundary condition with respect to $\{ x_n > 0 \}$, namely
\begin{align*}
g(t,x',0,v',v_n) = g(t,x',0,v',-v_n) ~~ \forall (t,x,v) \in \R \times \{ x_n = 0 \} \times \R^n.
\end{align*}
Therefore, we can apply the Liouville theorem in \autoref{thm:Liouville-higher-half-space-SR} to deduce that $g \in \tilde{\cP}_{a,5}$, where $a = a^{n,n}$. Hence, we can write
\begin{align*}
g(t,x,v) = P(t,x,v) + \tau \mathcal{T}_{a,3}(x_n,v_n)
\end{align*}
for some $P \in \tilde{\cP}_5$ satisfying the specular reflection condition and some $\tau \in \R$.

 Moreover, it follows from \eqref{eq:bdry-conv-higher-0-halfspace} applied with $z_1 = 0$ and \eqref{eq:gm-prop-2_0} that $g$ also satisfies the following property
\begin{align}
\label{eq:g-prop-2-0-halfspace}
\Vert g \Vert_{L^{\infty}(H_1)} \ge \frac{1}{2}.
\end{align}
Moreover, using again \eqref{eq:bdry-conv-higher-0-halfspace} it follows upon choosing 
\begin{align}
\label{eq:p-choice}
p(t,x,v) = P(r_m^{-2}(t-t_m^0) , r_m^{-3}(x-x_m^0 - (t-t_m^0)v_m^0) , r_m^{-1}(v-v_m^0) ) + \tau \mathcal{T}_{a_m,3}(x_n,v_n) \in \tilde{\cP}_{a_m,5}
\end{align}
in \eqref{eq:gm-prop-1_0}, i.e. $p \in \tilde{\cP}_{a_m,5}$ such that $p(z_m^0 \circ S_{r_m} z) = P(z) + \tau \mathcal{T}_{a_m,3}(x_n,v_n)$, that
\begin{align}
\label{eq:g-prop-1-0-halfspace}
0 = \int_{H_1^{(m)}} g_m(z) p(z_m^0 \circ S_{r_m} z) \d z = \int_{H_1^{(m)}} g_m(z) [P(z) + \tau \mathcal{T}_{a_m,3}(x_n,v_n) ] \d z \to \int_{H_1} g^2(z) \d z,
\end{align}
where we used that $a_m \to a$ and that thus, by definition (see \eqref{eq:Tricomi-sol-def}), it holds
\begin{align*}
\mathcal{T}_{a_m,3} \to \mathcal{T}_{a,3} ~~ \text{ as } m \to \infty
\end{align*}
locally uniformly.
Note that the choice for $p$ in \eqref{eq:p-choice} is valid since $p \in \tilde{\cP}_{a_m,5}$. Indeed it satisfies the specular reflection condition with respect to $\{ x_n > 0 \}$, since $\mathcal{T}_{a_m,3}$ only depends on $x_n,v_n$, and since whenever $(t,x,v) \in \gamma$, i.e. $x_n = 0$, it holds 
\begin{align*}
(r_m^{-3}(x-x_m^0 - (t-t_m^0)v_m^0))_n = x_n = 0, \qquad (r_m^{-1}(v-v_m^0))_n = r_m^{-1} v_n,
\end{align*}
since $(x_m^0)_n = (v_m^0)_n = 0$, and therefore
\begin{align*}
p(t,x',0,v',v_n) & - \tau \mathcal{T}_{a_m,3}(x_n,v_n) \\
&= P(r_m^{-2}(t-t_m^0) , r_m^{-3}(x-x_m^0 - (t-t_m^0)v_m^0) , r_m^{-1}(v-v_m^0) ) \\
&= P(r_m^{-2}(t-t_m^0) , r_m^{-3}(x-x_m^0 - (t-t_m^0)v_m^0) , r_m^{-1}(v-v_m^0) - 2 (r_m^{-1}(v-v_m^0))_n e_n )\\
&= P(r_m^{-2}(t-t_m^0) , r_m^{-3}(x-x_m^0 - (t-t_m^0)v_m^0) , r_m^{-1}((v-2 v_n e_n)-v_m^0)) \\
&= p(t,x',0,v',-v_n) - \tau \mathcal{T}_{a_m,3}(x_n,v_n).
\end{align*}

The proof is complete, since \eqref{eq:g-prop-1-0-halfspace} implies $g \equiv 0$, which contradicts \eqref{eq:g-prop-2-0-halfspace}.
\end{proof}

\subsection{Boundary regularity in the half-space}

We are now in a position to prove the optimal boundary regularity for solution in the half-space subject to the specular reflection boundary condition.

%\begin{theorem}
%\label{thm:regularity-SR-halfspace}
%Let $\Omega = \{ x_n > 0 \}$ and $z_0 \in \gamma$, i.e. $(x_0)_n = 0$. Let $R \in (0,1]$, $\eps \in (0,1)$, and $a^{i,j},b,c,h \in C^{3+\eps}_{\ell}(\mathcal{S}(H_R(z_0)))$ and assume that $a^{i,j},b,c$ satisfy \eqref{eq:unif-ell} and \eqref{eq:bdness} in $\mathcal{S}(H_R(z_0))$.
%
%Let $f$ be a solution to 
%\begin{align*}
%\begin{cases}
%\partial_t f + v \cdot \nabla_x f + (-a^{i,j} \partial_{v_i,v_j})f &= - b \cdot \nabla_v f - c f + h ~~ \text{ in }  \mathcal{S}(H_R(z_0)) , \\
%f(t,x,v) &= f(t,x,\mathcal{R}_x v) ~~ \text{ on } \gamma_- \cap  \mathcal{S}(H_R(z_0)) .
%\end{cases}
%\end{align*}
%Then, it holds:
%\begin{align*}
%[ f ]_{C^{4,1}_{\ell}(H_{R/16}(z_0))} \le C R^{-5} \big( \Vert f \Vert_{L^{\infty}(\mathcal{S}(H_R(z_0)))} + R^{5} [ h ]_{C^{3+\eps}_{\ell}(\mathcal{S}(H_R(z_0)) )} \big).
%\end{align*}
%The constant $C \ge 1$ depends only on $n,\eps,\lambda,\Lambda$, but not on $z_0$ and $R$.
%\end{theorem}

\begin{proof}[Proof of \autoref{thm:regularity-SR-halfspace}]
We closely follow the proof of \autoref{prop:regularity-SR} (resp. \autoref{lemma:regularity-inflow}).

We assume without loss of generality that
\begin{align}
\label{eq:normalization-reg-SR}
R^{-5} \left( \Vert f \Vert_{L^{\infty}(\mathcal{S}(H_R(z_0)))} + R^{5} [ h ]_{C^{3+\eps}_{\ell}( \mathcal{S}(H_R(z_0)))} \right) \le 1,
\end{align}
and claim that for any $z^{\ast} \in H_{R/16}(z_0)$ it holds 
\begin{align}
\label{eq:claim-reg-SR}
[f]_{C^{4,1}_{\ell}(Q_{r/2}(z^{\ast}))} \le C,
\end{align}
where we define $r := d_{\ell}(z^{\ast}) := \sup \{ r > 0 : Q_r(z^{\ast}) \subset (-1,1) \times \Omega \times \R^n \} \le \frac{R}{16}$. 

To prove \eqref{eq:claim-reg-SR}, we note that for any $z^{\ast}$ we can find $z^{\ast}_0 \in \gamma \cap Q_{R/8}(z_0)$ such that $z^{\ast} \in Q_r(z^{\ast}_0)$, where we used \autoref{lemma:kinetic-balls}. 

In case $r \ge \frac{R}{100}$, we apply the interior $C^{5+\eps}_{\ell}$ estimates from \autoref{lemma:interior-reg} and immediately deduce \eqref{eq:claim-reg-SR} by H\"older interpolation (see \autoref{lemma:Holder-interpol}) and \eqref{eq:normalization-reg-SR}.

Hence, we can assume from now on that $r < \frac{R}{100}$. In case, $z_0^{\ast} \in \gamma_{\pm}$ and $Q_{R/40}(z_0^{\ast}) \cap \gamma_0 = \emptyset$ we apply \autoref{prop:regularity-SR} and \eqref{eq:normalization-reg-SR} to deduce \eqref{eq:claim-reg-SR} also in this case.

Otherwise, we define $d := d_{\ell}(z^{\ast},\gamma_0) := \sup\{ r > 0 : Q_{r}(z^{\ast}) \cap  \gamma_0 = \emptyset  \} \in \left[r,\frac{R}{40} \right]$, the distance of $z^{\ast}$ from $\gamma_0$, and distinguish between several cases.

\textbf{Case 1:} We assume that $d \le 4 r$. \\
In that case, there exists $\overline{z_0^{\ast}} \in \gamma_0 \cap Q_{5r}(z^{\ast})$ (using \autoref{lemma:kinetic-balls}) such that
\begin{align}
\label{eq:inclusion-regularity-SR}
Q_{5r}(z^{\ast}) \subset Q_{10r}(\overline{z_0^{\ast}}).
\end{align}
Let us denote $a^{\ast} := a^{n,n}(\overline{z_0^{\ast}})$ and by $p^{\ast} \in \cP_{a^{\ast},5}$ the function in the expansion of $f$ at $\overline{z_0^{\ast}}$ from \autoref{lemma:higher-reg-gamma_0-SR-halfspace}. We write
\begin{align}
\label{eq:split-poly}
p^{\ast} = P^{\ast} + \tau \mathcal{T}_{a^{\ast}}, \qquad P^{\ast} = \sum_{|\beta| \le 5} \alpha^{(\beta)} z^{\beta} \in \cP_5, 
\end{align}
where $\mathcal{T}_{a^{\ast}}$ denotes the Tricomi solution from \autoref{prop:Tricomi} and $\tau , \alpha^{(\beta)} \in \R$. Moreover, note that we can write $P^{\ast} = P^{\ast}_4 + P^{\ast}_5$ such that $P^{\ast}_5 \in \cP_5$ is homogeneous of degree $5$ with respect to the point $z_0^{\ast}$ (i.e. $P^{\ast}_5(z_0^{\ast} \circ S_{r} z) = r^5 P^{\ast}_5(z_0^{\ast} \circ z)$), and $P^{\ast}_4 \in \cP_4$. 

Note that the function $f - P^{\ast}_4$ solves
\begin{align}
\label{eq:eq-difference}
[\partial_t + v \cdot \nabla_x + (-a^{i,j} \partial_{v_i,v_j}) + b \cdot \nabla_v  + c ](f - P^{\ast}_4) = h + \tilde{h} ~~ \text{ in } \mathcal{S}(H_R(z_0)),
\end{align}
where, by the same arguments as in the proof of \eqref{eq:tilde-h-estimate}, using \eqref{eq:inclusion-regularity-SR}, \autoref{lemma:product-rule}, \autoref{lemma:der-Holder}, \autoref{lemma:Holder-interpol}, and $P^{\ast}_4 \in \cP_4$:
\begin{align}
\label{eq:tilde-h-estimate-SR}
\begin{split}
[\tilde{h}]_{C^{3+\eps}_{\ell}(Q_{r}(z^{\ast}))} &\le [\partial_t P^{\ast}_4 + v \cdot \nabla_x P^{\ast}_4 + (-a^{i,j} \partial_{v_i,v_j}) P^{\ast}_4 + b \cdot \nabla_v  P^{\ast}_4 + c P^{\ast}_4]_{C^{3+\eps}_{\ell}(H_{R/5}(\overline{z^{\ast}_0}))} \\
&\le C \Vert P^{\ast}_4 \Vert_{C^{5+\eps}_{\ell}(H_{R/5}(\overline{z^{\ast}_0}))} \\
&\le C R^{-5} \Vert P^{\ast}_4 \Vert_{L^{\infty}(H_{R/5}(\overline{z^{\ast}_0}))} + C [P^{\ast}_4]_{C^{5+\eps}_{\ell}(H_{R/5}(\overline{z^{\ast}_0}))} \\
&\le C R^{-5} \Vert P^{\ast}_4 \Vert_{L^{\infty}(H_{R/5}(\overline{z^{\ast}_0}))}.
\end{split}
\end{align}

Since by construction, it holds $d_{\ell}(\overline{z_0^{\ast}} ,z_0) < \frac{R}{8}$, we have by \autoref{lemma:kinetic-balls}  
\begin{align}
\label{eq:inclusion-regularity-SR-2}
\mathcal{S}(H_{R/4}(\overline{z_0^{\ast}})) \subset \mathcal{S}(H_R(z_0)).
\end{align}

Hence, we can apply the expansion from \autoref{lemma:higher-reg-gamma_0-SR-halfspace} at the boundary point $\overline{z_0^{\ast}}$ for any $\rho \in (0,\frac{R}{5}]$
\begin{align}
\label{eq:appl-exp-SR}
\Vert f - p^{\ast} \Vert_{L^{\infty}(H_{\rho}(\overline{z^{\ast}_0}))} \le C \left(\frac{\rho}{R}\right)^{5+\eps} \left( \Vert f \Vert_{L^{\infty}(H_{R/4}(\overline{z^{\ast}_0}))} + R^{5+\eps} [ h ]_{C^{3+\eps}_{\ell}(H_{R/4}(\overline{z^{\ast}_0}))} \right).
\end{align}

Having established \eqref{eq:tilde-h-estimate-SR}, \eqref{eq:inclusion-regularity-SR-2}, and \eqref{eq:appl-exp-SR}, we are now in a position to prove \eqref{eq:claim-reg-SR} in the situation of Case 1. 

First, by \autoref{lemma:Holder-interpol}, interior Schauder estimates (see \autoref{lemma:interior-reg}) applied to the function $f - P^{\ast}_4$ and \autoref{lemma:osc-results}, we obtain
\begin{align}
\label{eq:Schauder-appl}
\begin{split}
[f]_{C^{4,1}_{\ell}(Q_{r/2}(z^{\ast}))} &= [f - P^{\ast}_4]_{C^{4,1}_{\ell}(Q_{r/2}(z^{\ast}))} \le r^{\eps} [f - P^{\ast}_4]_{C^{5+\eps}_{\ell}(Q_{r/2}(z^{\ast}))} + C r^{-5} \Vert f - P^{\ast}_4 \Vert_{L^{\infty}(Q_{r/2}(z^{\ast}))}\\
&\le C r^{-5} \Vert f - P^{\ast}_4 \Vert_{L^{\infty}(Q_{r}(z^{\ast}))} + C  [\tilde{h} + h]_{C^{3+\eps}_{\ell}(Q_r(z^{\ast}))}  \\
&\le C r^{-5} \Vert f - P^{\ast}_4 \Vert_{L^{\infty}(H_{10r}(\overline{z^{\ast}_0}))} + C  [\tilde{h} + h]_{C^{3+\eps}_{\ell}(Q_r(z^{\ast}))} \\
&\le C r^{-5} \Vert f - p^{\ast} \Vert_{L^{\infty}(H_{10r}(\overline{z^{\ast}_0}))}  + C  [h]_{C^{3+\eps}_{\ell}(Q_r(z^{\ast}))} \\
&\quad + C r^{-5} \Vert P^{\ast}_4 - p^{\ast} \Vert_{L^{\infty}(H_{10r}(\overline{z_0^{\ast}}))} + C R^{-5} \Vert P^{\ast}_4 \Vert_{L^{\infty}(H_{R/5}(\overline{z_0^{\ast}}))},
\end{split}
\end{align}
where we used \eqref{eq:tilde-h-estimate-SR} in the last step.
We estimate the four term on the right-hand side of \eqref{eq:Schauder-appl}, separately. For the first two terms, we apply \eqref{eq:appl-exp-SR} with $\rho = 10r \le \frac{R}{4}$ (which holds by choice of $r$), 
\begin{align*}
r^{-5} \Vert f - p^{\ast} \Vert_{L^{\infty}(H_{10r}(\overline{z^{\ast}_0}))} + [h]_{C^{3+\eps}_{\ell}(Q_r(z^{\ast}))} \le C R^{-5}\Vert f \Vert_{L^{\infty}(H_{R/4}(\overline{z^{\ast}_0})))} + C[h]_{C^{3+\eps}_{\ell}(H_{R/4}(\overline{z^{\ast}_0})))} \le C.
\end{align*}

For the third term we observe that $P^{\ast}_4 - p^{\ast}$ is homogeneous of degree $5$ with respect to $z_0^{\ast}$, and therefore we have $(P^{\ast}_4 - p^{\ast})(z_0^{\ast}) =0$, and thus
\begin{align*}
r^{-5} \Vert P^{\ast}_4 - p^{\ast} \Vert_{L^{\infty}(H_{10r}(\overline{z_0^{\ast}}))} \le C \left(\sum_{|\beta| \le 5} |\alpha^{(\beta)}| + |\tau| \right) \le C R^{-5} \Vert p^{\ast} \Vert_{L^{\infty}(H_{R/5}(\overline{z_0^{\ast}})))}.
\end{align*}
By a similar reasoning, and using also \eqref{eq:appl-exp-SR} with $\rho = R/5$, and \eqref{eq:normalization-reg-SR}, we obtain for the last two terms in \eqref{eq:Schauder-appl}
\begin{align}
\label{eq:P4-bd}
\begin{split}
r^{-5} \Vert P^{\ast}_4 - p^{\ast} \Vert_{L^{\infty}(H_{10r}(\overline{z_0^{\ast}}))} &+ R^{-5} \Vert P^{\ast}_4 \Vert_{L^{\infty}(H_{R/5}(\overline{z_0^{\ast}}))} \le CR^{-5}  \Vert p^{\ast} \Vert_{L^{\infty}(H_{R/5}(\overline{z_0^{\ast}}))}\\
&\le C R^{-5} \left( \Vert f - p^{\ast} \Vert_{L^{\infty}(H_{R/5}(\overline{z_0^{\ast}}))} + \Vert f \Vert_{L^{\infty}(H_{R/5}(\overline{z_0^{\ast}}))} \right) \le C.
\end{split}
\end{align}
Thus, altogether, we have shown \eqref{eq:claim-reg-SR} in this case, as desired.

\textbf{Case 2:} Assume now $4r \le d \le \frac{R}{40}$.

In this case, we split the proof into two steps. First, we will go from scale $r$ to scale $d/4$ by using \autoref{prop:regularity-SR} (instead of interior Schauder estimates, as in Case 1), and second, we will apply the expansion (see \autoref{lemma:higher-reg-gamma_0-SR-halfspace}) as in Case 1 to go from scale $d/4$ to scale $R$.

In fact, since $4r \le d$ we have that $Q_{d/2}(z^{\ast}) \cap \gamma_0 = \emptyset$. Moreover, note that by construction, there is $\overline{z_0^{\ast}} \in \gamma_0 \cap Q_{2d}(z^{\ast})$, and it holds by \autoref{lemma:kinetic-balls}
\begin{align}
\label{eq:inclusion-regularity-SR-3}
Q_{d/4}(z^{\ast}) \subset Q_{8d}(\overline{z_0^{\ast}}) \subset Q_{R/5}(\overline{z_0^{\ast}}).
\end{align} 

As before, we denote by $p^{\ast} \in \cP_{a^{\ast},5}$, where $a^{\ast} = a^{n,n}(\overline{z_0^{\ast}})$, the function in the expansion of $f$ at $\overline{z_0^{\ast}}$, as in Case 1. Moreover, as in Case 1, we split $p^{\ast} = P^{\ast} + \tau \mathcal{T}_{a^{\ast}}$ as in \eqref{eq:split-poly} and find $P^{\ast} = P^{\ast}_4 + P^{\ast}_5$ with the same properties.

Again, the function $f - P^{\ast}_4$ solves \eqref{eq:eq-difference} in $H_R(z_0)$ where by the exact same arguments as in Step 1 (see \eqref{eq:tilde-h-estimate-SR} and \eqref{eq:P4-bd}), and using \eqref{eq:inclusion-regularity-SR-3}:
\begin{align}
\label{eq:tilde-h-estimate-SR-3}
\begin{split}
[\tilde{h}]_{C_{\ell}^{3+\eps}(H_{d/4}(z^{\ast}))} &\le C R^{-5} \Vert P^{\ast}_4 \Vert_{L^{\infty}(H_{R/5}(\overline{z_0^{\ast}}))} \le C.
\end{split}
\end{align}

Hence, we can deduce from H\"older interpolation (see \autoref{lemma:Holder-interpol}) and \autoref{prop:regularity-SR} applied with $f - P^{\ast}_4$, which solves \eqref{eq:eq-difference}, and $z_0 := z^{\ast}$ and $R := d$, using that $Q_{d/2}(z^{\ast}) \cap \gamma_0 = \emptyset$ and $Q_{d/4}(z^{\ast}) \subset Q_R(z_0)$, \eqref{eq:tilde-h-estimate-SR-3}, and taking into account \eqref{eq:normalization-reg-SR},
\begin{align*}
[f]_{C^{4,1}_{\ell}(Q_{r/2}(z^{\ast}))} &\le [f - P_4^{\ast}]_{C^{4,1}_{\ell}(Q_{d/8}(z^{\ast}))} \\
&\le d^{\eps} [f - P_4^{\ast}]_{C^{5+\eps}_{\ell}(Q_{d/8}(z^{\ast}))} + C d^{-5} \Vert f - P_4^{\ast}\Vert_{L^{\infty}(Q_{d/8}(z^{\ast}))}\\
&\le C \left( d^{-5} \Vert f - P^{\ast}_4 \Vert_{L^{\infty}(H_{d/4}(z^{\ast}))} + [ \tilde{h} ]_{C^{3+\eps}_{\ell}(H_{d/4}(z^{\ast}))} + [ h ]_{C^{3+\eps}_{\ell}(H_{d/4}(z_0))} \right)\\
&\le C \left( d^{-5} \Vert f - p^{\ast} \Vert_{L^{\infty}(H_{d/4}(z^{\ast}))} + d^{-5} \Vert P^{\ast}_4 - p^{\ast} \Vert_{L^{\infty}(H_{d/4}(z^{\ast}))} + 1 \right)\\
&\le C  \left( d^{-5} \Vert f - p^{\ast} \Vert_{L^{\infty}(H_{d/4}(z^{\ast}))} + 1 \right).
\end{align*}
In the last step, we applied again \eqref{eq:P4-bd}. It remains to estimate the first term in the previous estimate. However, due to \eqref{eq:inclusion-regularity-SR-3} we can apply \eqref{eq:appl-exp-SR} with $\rho = 8d \le R/5$ and deduce by \eqref{eq:normalization-reg-SR}
\begin{align*}
d^{-5} \Vert f - p^{\ast} \Vert_{L^{\infty}(H_{d/4}(z^{\ast}))} &\le C  d^{-5} \Vert f - p^{\ast} \Vert_{L^{\infty}(H_{8d}(\overline{z_0^{\ast}}))} \\
&\le C \left( R^{-5}\Vert f \Vert_{L^{\infty}(H_{R/4}(\overline{z^{\ast}_0})))} +  [ h ]_{C^{3+\eps}_{\ell}(H_{R/4}(\overline{z^{\ast}_0})))} \right) \le C.
\end{align*}
Altogether, we have shown \eqref{eq:claim-reg-SR} also in the situation of Case 2.

Hence, we have verified \eqref{eq:claim-reg-SR} in all cases. From here, we conclude the proof by \autoref{lemma:covering}.
\end{proof}

The following corollary of \autoref{thm:regularity-SR-halfspace} follows immediately by combination with interior estimates and a covering argument.

\begin{corollary}
\label{cor:regularity-SR-halfspace}
Let $\Omega = \{ x_n > 0 \}$ and $z_0 \in (-1,1) \times \{x_n \ge 0 \} \times \R^n$. Let $R \in (0,1]$, $\eps \in (0,1)$, and $a^{i,j},b,c,h \in C^{3+\eps}_{\ell}( \mathcal{S}(H_R(z_0)))$. Assume that $a^{i,j}$ satisfies \eqref{eq:unif-ell}. Let $f$ be a bounded weak solution to 
\begin{equation*}
\left\{\begin{array}{rcl}
\partial_t f + v \cdot \nabla_x f + (-a^{i,j} \partial_{v_i,v_j})f &=& - b \cdot \nabla_v f - c f + h ~~ \text{ in } \mathcal{S}(H_R(z_0)) , \\
f(t,x,v) &=& f(t,x,\mathcal{R}_x v) ~\qquad\quad \text{ on } \gamma_- \cap \mathcal{S}(H_R(z_0)).
\end{array}\right.
\end{equation*}
Then, it holds $f \in C^{4,1}_{\ell}(H_{R/2}(z_0))$ and 
\begin{align*}
[ f ]_{C^{4,1}_{\ell}(H_{R/2}(z_0))} \le C R^{-5} \big( \Vert f \Vert_{L^{\infty}(\mathcal{S}(H_R(z_0)))} + R^{5} [ h ]_{C^{3+\eps}_{\ell}(\mathcal{S}(H_R(z_0)) )} \big).
\end{align*}
The constant $C$ depends only on $n,\eps,\lambda,\Lambda$, $\Vert a^{i,j} \Vert_{C^{3+\eps}_{\ell}(\mathcal{S}(H_R(z_0)))}, \Vert b \Vert_{C^{3+\eps}_{\ell}(\mathcal{S}(H_R(z_0)))}$, and $\Vert c \Vert_{C^{3+\eps}_{\ell}(\mathcal{S}(H_R(z_0)))}$ but not on $z_0$ and $R$.
\end{corollary}

Note that by application of the local boundedness estimate in \eqref{eq:boundary-reg-specular-Linfty-L1}, we can also replace the $L^{\infty}$ norm of $f$ by the $L^1$ norm.

\subsection{Explicit dependence of the constant on the lower order coefficients}

In the previous subsection we have shown the optimal $C^{4,1}_{\ell}$ regularity in the half-space for kinetic equations with coefficients subject to the specular reflection condition. The constant in the regularity estimate depends on the norm of the coefficients but not on $v_0$.

Our next goal is establish boundary regularity results in general domains by applying the flattening diffeomorphism $\Phi$ from \autoref{lemma:trafo-sr}. This way, we obtain an equation in the half-space with a drift growing like $|v_0|^2$. In this case \autoref{thm:regularity-SR-halfspace} yields regularity of solutions, but unfortunately the estimates deteriorate as $|v_0| \to \infty$. Since the dependence of the constants in \autoref{thm:regularity-SR-halfspace} on the coefficients is not explicit, we cannot use this result to obtain control over the $|v_0|$-dependence in regularity estimates for equations in general domains. This is crucial in order to prove estimates for equations that hold globally in $v$. We overcome this issue by freezing the coefficients and deriving a priori estimates with an explicit dependence. This is carried out in the proof of the following key proposition.

Moreover, note that we are able to replace the $L^{\infty}$ norm of $f$ by the $L^1$ norm. This is achieved by an interpolation and absorption argument, similar to the one in \cite[Section 6.2]{FRW24}.

\begin{proposition}
\label{prop:regularity-SR-halfspace-constants}
Let $\Omega = \{ x_n > 0 \}$ and $z_0 \in (-1,1) \times \{ x_n \ge 0 \} \times \R^n$. Let $R \in (0,1]$, $\eps \in (0,1)$ and $a^{i,j},b,c,h \in C^{3+\eps}_{\ell}(\mathcal{S}(H_R(z_0)))$ and assume that $a^{i,j}$ satisfies \eqref{eq:unif-ell}. Let $f$ be a weak solution to 
\begin{equation*}
\left\{\begin{array}{rcl}
\partial_t f + v \cdot \nabla_x f + (-a^{i,j} \partial_{v_i,v_j})f &=& - b \cdot \nabla_v f - c f + h ~~ \text{ in }  \mathcal{S}(H_R(z_0)) , \\
f(t,x,v) &=& f(t,x,\mathcal{R}_x v) ~\qquad\quad \text{ on } \gamma_- \cap \mathcal{S}(H_R(z_0)).
\end{array}\right.
\end{equation*}
Then, $f \in C^{4,1}_{\ell}(H_{R/2}(z_0))$ and it holds:
\begin{align*}
[ f ]_{C^{4,1}_{\ell}(H_{R/2}(z_0))} &\le C \left[1 + \Vert b^{i} \Vert_{C^{3+\eps}_{\ell}(\mathcal{S}(H_R(z_0)))} \right] R^{-5} R^{-(2+4n)} \Vert f \Vert_{L^{1}( \mathcal{S}(H_R(z_0)))} + C  [ h ]_{C^{3+\eps}_{\ell}(\mathcal{S}(H_R(z_0)) )} .
\end{align*}
The constant $C$ depends only on $n,\eps,\lambda,\Lambda$, $\Vert a^{i,j} \Vert_{C^{3+\eps}_{\ell}(\mathcal{S}(H_R(z_0)))}$, and $\Vert c \Vert_{C^{3+\eps}_{\ell}(\mathcal{S}(H_R(z_0)))}$, but not on $z_0$.
\end{proposition}

\begin{proof}
This proof is based on a covering argument. For this sake, it is much more convenient to work with kinetic balls that are adapted to the kinetic distance. We define
\begin{align*}
Q^d_r(z_0) = \{ z \in \R^{1+2n} : d_{\ell}(z,z_0) < r \}, ~~ H^d_r(z_0) = Q^d_r(z_0) \cap ((-1,1) \times \Omega \times \R^n),
\end{align*}
and recall from \autoref{lemma:kinetic-balls} that 
\begin{align}
\label{eq:metric-geometric}
Q_r(z_0) \subset Q^d_r(z_0) \subset Q_{2r}(z_0).
\end{align}
Moreover, by the triangle inequality for $d_{\ell}$ (see \cite[Proposition 2.2]{ImSi21})
\begin{align*}
z_0 \in Q^d_{r_1}(z_1) \qquad \Rightarrow \qquad Q^d_{r_2}(z_0) \subset Q^d_{r_1 + r_2}(z_1) ~~ \forall z_0,z_1 \in \R^{1+2n}, ~~ \forall r_1, r_2 > 0.
\end{align*}
The proof is split into several steps.

\textbf{Step 1:} Note that $f$ solves in $\mathcal{S}(H_R(z_0))$
\begin{align*}
\partial_t f + v \cdot \nabla_x f + (-a^{i,j} \partial_{v_i,v_j})f &= - c f + \tilde{h}, \quad \text{ where } \quad \tilde{h} = - b^i \partial_{v_i} f + h.
\end{align*}
By \autoref{lemma:product-rule} and \autoref{lemma:der-Holder}, it holds for any $H_r^d(z) \subset \mathcal{S}(H_R(z_0))$ with $r \in (0,1]$, $z \in \R^{1+2n}$:
\begin{align*}
[ \tilde{h} ]_{C^{3+\eps}_{\ell}(H_r^d(z))} \le [ h ]_{C^{3+\eps}_{\ell}(H^d_r(z))} + [ b^i \partial_{v_i} f ]_{C^{3+\eps}_{\ell}(H^d_r(z))} \le  [ h ]_{C^{3+\eps}_{\ell}(H^d_r(z))} + C \Vert b^i  \Vert_{C^{3+\eps}_{\ell}(H^d_r(z))} \Vert f \Vert_{C^{4+\eps}_{\ell}(H^d_r(z))}.
\end{align*}
Moreover, by H\"older interpolation (see \autoref{lemma:Holder-interpol}),
\begin{align*}
 \Vert f \Vert_{C^{4+\eps}_{\ell}(H^d_r(z))} \le C r^{-4} \Vert f \Vert_{L^{\infty}(H^d_r(z))} + [f]_{C^{4+\eps}_{\ell}(H^d_r(z))} \le  r^{-(4+\eps)} r^{-(2+4n)}\Vert f \Vert_{L^{1}(H^d_r(z))} + r^{1-\eps} [f]_{C^{4,1}_{\ell}(H^d_r(z))} ,
\end{align*}
and therefore, altogether
\begin{align}
\label{eq:tilde-h-absorb}
[ \tilde{h} ]_{C^{3+\eps}_{\ell}(H^d_r(z))} \le C \Vert b^i  \Vert_{C^{3+\eps}_{\ell}(H^d_r(z))} \left( r^{-5} r^{-(2+4n)}\Vert f \Vert_{L^{1}(H^d_r(z))} + r^{1-\eps} [f]_{C^{4,1}_{\ell}(H^d_r(z))} \right) + [h]_{C^{3+\eps}_{\ell}(H^d_r(z))}.
\end{align}
Note that the proofs of \autoref{lemma:product-rule}, \autoref{lemma:der-Holder}, and \autoref{lemma:Holder-interpol} can easily be adapted to metric cylinders, using \eqref{eq:metric-geometric} and scaling.

\textbf{Step 2:} By \autoref{cor:regularity-SR-halfspace} rewritten for metric cylinders, we have for any $z_0 \in (-1,1) \times \{ x_n \ge 0 \} \times \R^n$ and $r \le \frac{R}{2}$:
\begin{align*}
[ f ]_{C^{4,1}_{\ell}(H^d_{r}(z_0))} \le C r^{-5} \big( \Vert f \Vert_{L^{\infty}(\mathcal{S}(H^d_{4r}(z_0)))} + r^{5} [ \tilde{h} ]_{C^{3+\eps}_{\ell}( \mathcal{S}(H^d_{4r}(z_0)) )} \big).
\end{align*}

Given $z_0$ and $0 < r_1 < r_2 \le \frac{R}{2}$ we can apply the previous estimate with $z_0 := z^{\ast} \in H^d_{r_1}(z_0)$ and $r = \frac{1}{16}(r_2-r_1)$. Then, since
\begin{align*}
H^d_{r_1}(z_0) \subset  \bigcup_{z^{\ast} \in H^d_{r_1}(z_0)} H^d_{\frac{r_2 - r_1}{4}}(z^{\ast}) \subset \bigcup_{z^{\ast} \in H^d_{r_1}(z_0)} H^d_{r_2 - r_1}(z^{\ast}) \subset H^d_{r_2}(z_0),
\end{align*}
we deduce
\begin{align*}
[ f ]_{C^{4,1}_{\ell}(H^d_{r_1}(z_0))} &\le \sup_{z^{\ast} \in H^d_{r_1}(z_0)} [ f ]_{C^{4,1}_{\ell}(H^d_{\frac{r_2 - r_1}{4}}(z^{\ast}))}\\
&\le C \sup_{z^{\ast} \in H^d_{r_1}(z_0)} (r_2 - r_1)^{-5} \big( \Vert f \Vert_{L^{\infty}( \mathcal{S}(H^d_{r_2 -r_1}(z^{\ast})))} + (r_2 - r_1)^{5} [ \tilde{h} ]_{C^{3+\eps}_{\ell}(\mathcal{S}(H^d_{r_2 - r_1}(z^{\ast})) )} \big)\\
&\le C (r_2 - r_1)^{-5} \big( \Vert f \Vert_{L^{\infty}( \mathcal{S}(H^d_{r_2}(z_0)))} + (r_2 - r_1)^{5} [ \tilde{h} ]_{C^{3+\eps}_{\ell}( \mathcal{S}(H^d_{r_2}(z_0)) )} \big).
\end{align*}

Moreover, by the second claim in \autoref{lemma:Holder-interpol}, applied with $\delta = \frac{r_2 - r_1}{4  Cr_2}  \in (0,1)$, we get
\begin{align*}
\Vert f \Vert_{L^{\infty}(\mathcal{S}(H^d_{r_2}(z_0)))} \le C (r_2 - r_1)^{-(2+4n)} \Vert f \Vert_{L^1(\mathcal{S}(H^d_{r_2}(z_0)))} + \frac{1}{2C}(r_2 - r_1)^5 [f]_{C^{4,1}(\mathcal{S}(H^d_{r_2}(z_0)))},
\end{align*}
which yields
\begin{align}
\label{eq:reg-est-SR-inexplicit-scale}
\begin{split}
[ f ]_{C^{4,1}_{\ell}(H^d_{r_1}(z_0))} &\le C  (r_2 - r_1)^{-5} (r_2 - r_1)^{-(2+4n)} \Vert f \Vert_{L^{1}( \mathcal{S}(H^d_{r_2}(z_0)))} \\
&\quad + C [ \tilde{h} ]_{C^{3+\eps}_{\ell}( \mathcal{S}(H^d_{r_2}(z_0)) )} + \frac{1}{4}[f]_{C^{4,1}_{\ell}(\mathcal{S}(H^d_{r_2}(z_0)))}.
\end{split}
\end{align}

\textbf{Step 3:} By combination of the boundary regularity estimate in \eqref{eq:reg-est-SR-inexplicit-scale}, and the estimate for $\tilde{h}$ from \eqref{eq:tilde-h-absorb}, we deduce that for any $z^{\ast} \in \R^{1+2n}$, $0 < r_1 < r_2 \le \frac{R}{2}$ with $Q^d_{r_2}(z^{\ast}) \subset H_{R/2}(z_0)$:
\begin{align*}
 & [f]_{C_{\ell}^{4,1}( \mathcal{S}(H^d_{r_1}(z^{\ast})))} \\
&\qquad\le C (r_2 - r_1)^{-5} (r_2 - r_1)^{-(2+4n)} \Vert f \Vert_{L^{1}( \mathcal{S}(H^d_{r_2}(z^{\ast})))}  + C [\tilde{h}]_{C_{\ell}^{3+\eps}(\mathcal{S}(H^d_{r_2}(z^{\ast})))}  + \frac{1}{4}[f]_{C^{4,1}_{\ell}(\mathcal{S}(H^d_{r_2}(z_0)))}\\
&\qquad\le C  (r_2 - r_1)^{-5} (r_2 - r_1)^{-(2+4n)} \Vert f \Vert_{L^{1}( \mathcal{S}(H^d_{r_2}(z^{\ast})))} \\
&\qquad\quad + C  \Vert b^i  \Vert_{C^{3+\eps}_{\ell}(H^d_R(z_0))} \left( (r_2-r_1)^{-5} (r_2 - r_1)^{-(2+4n)}\Vert f \Vert_{L^{1}(H^d_{r_2}(z))} + r_2^{1-\eps} [f]_{C^{4,1}_{\ell}(H^d_{r_2}(z))} \right) \\
&\qquad\quad + C [h]_{C^{3+\eps}_{\ell}(H^d_{r_2}(z))} + \frac{1}{4}[f]_{C^{4,1}_{\ell}(\mathcal{S}(H^d_{r_2}(z_0)))}.
\end{align*}

Let us take
\begin{align*}
r_2 \le R_0:= \min \left\{ \left[ 4  C \Vert b^{i} \Vert_{C^{3+\eps}_{\ell}(\mathcal{S}(H_{R}(z_0)))} \right]^{-\frac{1}{1-\eps}} , \frac{R}{2} \right\}.
\end{align*}
Then, by the previous estimate
\begin{align*}
[f]_{C_{\ell}^{4,1}(\mathcal{S}(H^d_{r_1}(z^{\ast})))} &\le C  (r_2 - r_1)^{-5} \left[ 1 + \Vert b^{i} \Vert_{C^{3+\eps}_{\ell}(\mathcal{S}(H_{R}(z_0)))} \right] (r_2 - r_1)^{-(2+4n)} \Vert f \Vert_{L^{1}(\mathcal{S}(H^d_{r_2}(z^{\ast})))} \\
&\quad  + C  [h]_{C_{\ell}^{3+\eps}(\mathcal{S}(H^d_{r_2}(z^{\ast})))} + \frac{1}{2}  [f]_{C_{\ell}^{4,1}(\mathcal{S}(H^d_{r_2}(z^{\ast})))}.
\end{align*}
Let us point out that we already know by \autoref{thm:regularity-SR-halfspace} and \autoref{lemma:boundary-reg-specular} (applied directly to the equation for $f$, without absorbing the drift into $\tilde{h}$), 
\begin{align*}
\Vert f \Vert_{C_{\ell}^{4,1}(\mathcal{S}(H^d_{R_0}(z^{\ast})))} \le C  [f]_{C_{\ell}^{4,1}(\mathcal{S}(H^d_{R_0}(z^{\ast})))} + \Vert f \Vert_{L^{\infty}(\mathcal{S}(H^d_{R_0}(z^{\ast})))} < \infty.
\end{align*}
Therefore, we can apply a standard iteration lemma (see for instance \cite[Lemma 2.3]{KaWe24}, \cite[Lemma 1.1]{GiGi82}) and deduce that for any $0 < r < 2r \le R_0$ with $H^d_{2r}(z^{\ast}) \subset H_R(z_0)$ it holds
\begin{align}
\label{eq:covering-help}
\begin{split}
[f]_{C_{\ell}^{4,1}(\mathcal{S}(H^d_{r}(z^{\ast})))} &\le C r^{-5} \left[ 1 + \Vert b^{i} \Vert_{C^{3+\eps}_{\ell}(\mathcal{S}(H_{R}(z_0)))} \right]  r^{-(2+4n)} \Vert f \Vert_{L^{\infty}(\mathcal{S}(H^d_{2r}(z^{\ast})))} \\
&\quad  + C  r^{5}[h]_{C_{\ell}^{3+\eps}(\mathcal{S}(H^d_{2r}(z^{\ast})))}.
\end{split}
\end{align}
Note that we can cover $Q_{R/2}(z_0)$ by balls $Q^d_{r^l}(z^l)$ with $r^l \le R/2$, $l \in \N$, such that 
\begin{align*}
\sum_{l \in \N} \1_{Q^d_{2r^l}(z^l)} \le C
\end{align*}
for some $C$ depending only on $n$, but not on $R_0,R$. The desired result follows by application of \eqref{eq:covering-help} to each $z^l, r^l$ and summing the corresponding estimates over $l$.
\end{proof}

Clearly, by similar arguments, one can also establish explicit dependencies on $a^{i,j},c$. However, we are only interested in the dependence of the constants on $b^i$, since it is relevant for regularity estimates in general (non-flat) domains (see \autoref{thm:regularity-SR-constants}).

\subsection{Boundary regularity in general domains}
\label{subsec:gen-dom}

By flattening the boundary, we are able to obtain regularity results in more general (non-flat) domains.
The main auxiliary result for this section is \autoref{lemma:SR-preserved}, which deals with basic properties of the flattening diffeomorphism $\Phi$.

\begin{theorem}
\label{thm:regularity-SR-constants}
Let $\eps \in (0,1]$, $\Omega \subset \R^n$ with $\partial \Omega \in C^{\frac{9}{2}}$, and $z_0 \in (-1,1) \times \overline{\Omega} \times \R^n$. Let $a^{i,j},b,c,h \in C^{3+\eps}_{\ell}(\mathcal{S}(H_1(z_0)))$ and assume that $a^{i,j}$ satisfies \eqref{eq:unif-ell}. Let $f$ be a weak solution to 
\begin{equation*}
\left\{\begin{array}{rcl}
\partial_t f + v \cdot \nabla_x f + (-a^{i,j} \partial_{v_i,v_j})f &=& - b \cdot \nabla_v f - c f + h ~~ \text{ in }  \mathcal{S}(H_1(z_0)) , \\
f(t,x,v) &=& f(t,x,\mathcal{R}_x v) ~\qquad\quad \text{ on } \gamma_- \cap \mathcal{S}(H_1(z_0)).
\end{array}\right.
\end{equation*}
Then, $f \in C^{4,1}_{\ell}(H_{1/2}(z_0))$ and it holds:
\begin{align*}
%[ f ]_{C^{4,1}_{\ell}(H_{1/2}(z_0))} &\le C \left(1 + |v_0|^{7} \right)  \Vert f \Vert_{L^{\infty}( \mathcal{S}(H_1(z_0)))} + C \Vert h \Vert_{C^{3+\eps}_{\ell}(\mathcal{S}(H_1(z_0)) )},\\
%[ f ]_{C^{4,1}_{\ell}(H_{1/2}(z_0))} &\le C \left(1 + |v_0|^{9 + 4n} \right)  \Vert f \Vert_{L^{1}( \mathcal{S}(H_1(z_0)))} + C  \Vert h \Vert_{C^{3+\eps}_{\ell}(\mathcal{S}(H_1(z_0)) )}.
[ f ]_{C^{4,1}_{\ell}(H_{1/2}(z_0))} &\le C \left(1 + |v_0|^{\theta_1} \right)  \Vert f \Vert_{L^{\infty}( \mathcal{S}(H_1(z_0)))} + C \left(1 + |v_0|^{4} \right)\Vert h \Vert_{C^{3+\eps}_{\ell}(\mathcal{S}(H_1(z_0)) )}, \\
[ f ]_{C^{4,1}_{\ell}(H_{1/2}(z_0))} &\le C \left(1 + |v_0|^{\theta_2(1+n)} \right)  \Vert f \Vert_{L^{1}( \mathcal{S}(H_1(z_0)))} + C\left(1 + |v_0|^{4} \right) \Vert h \Vert_{C^{3+\eps}_{\ell}(\mathcal{S}(H_1(z_0)) )}.
\end{align*}
The constant $C$ depends only on $n,\eps,\lambda,\Lambda,\Omega$, $\Vert a^{i,j} \Vert_{C^{3+\eps}_{\ell}(\mathcal{S}(H_1(z_0)))}, \Vert b \Vert_{C^{3+\eps}_{\ell}(\mathcal{S}(H_1(z_0)))}$, and \\
$\Vert c \Vert_{C^{3+\eps}_{\ell}(\mathcal{S}(H_1(z_0)))}$, but not on $z_0$, and $\theta_1,\theta_2 > 0$ depend only on $\eps$.
%Moreover, in case $a^{i,j}$ is independent of $x,v$, we have
%\begin{align*}
%[ f ]_{C^{4,1}_{\ell}(H_{1/2}(z_0))} &\le C \left(1 + |v_0|^{22} \right)  \Vert f \Vert_{L^{\infty}( \mathcal{S}(H_1(z_0)))} + C \left(1 + |v_0|^{7} \right)\Vert h \Vert_{C^{3+\eps}_{\ell}(\mathcal{S}(H_1(z_0)) )}, \\
%[ f ]_{C^{4,1}_{\ell}(H_{1/2}(z_0))} &\le C \left(1 + |v_0|^{28 + 12n} \right)  \Vert f \Vert_{L^{1}( \mathcal{S}(H_1(z_0)))} + C\left(1 + |v_0|^{7} \right) \Vert h \Vert_{C^{3+\eps}_{\ell}(\mathcal{S}(H_1(z_0)) )}.
%\end{align*}
\end{theorem}

\begin{proof}
We only prove the second claim. The proof of the first claim follows by obvious modifications.
We can cover $\Omega_{\delta_0/2} \cap B_{1/2}(x_0)$ with finitely many balls $B_{\delta_0}(x^{(l)})$ such that $x^{(l)} \in \partial \Omega \cap B_{1/2}(x_0)$ and $\delta_0 \in (0,1]$ depending only on $\Omega$. Here, $\Omega_{\delta_0/2} = \{ x \in \Omega : \dist(x,\partial \Omega) \le \delta_0/2\}$. Moreover, we can choose the covering in such a way that for any $B_{\delta_0/8}(x^{\ast}) \subset \Omega_{\delta_0/2}$, there is $x^{(l)}$ such that $B_{\delta_0/8}(x^{\ast}) \subset B_{\delta_0/4}(x^{(l)})$.

We can associate any such $x^{(l)}$ with flattening diffeomorphisms $\phi^{(l)},\Phi^{(l)}$ as in \autoref{lemma:SR-preserved} such that $\phi^{(l)}(\Omega \cap B_{\delta_0/2}(x^{(l)})) \subset (\{ x_n > 0 \} \cap B_{\delta_1/2} )$ for some $\delta_1 \in (0,1]$. Here, up to a rotation and a shift, we are working as if $n_{x^{(l)}} = e_n$ and $(x^{(l)})' = 0$ for all $x^{(l)}$.

Next, we cover $H_{1/2}(z_0) \cap ((-1,1) \times \Omega_{\delta_0/4} \times \R^n)$ by finitely many kinetic cylinders $Q_{\delta/2}(z_m)$. By construction, if $\delta < \delta_0/16$, then for any $m$, there exists $l$ such that $B_{\delta}(x_m) \subset B_{\delta_0/4}(x^{(l)})$. Then, note that by the definition of the kinetic cylinders, we will have $Q_{\delta}(z_m) \subset (-1,1) \times B_{\delta_0/2}(x^{(l)}) \times \R^n$, when $\delta < \min \{ \delta_0/16 , (\delta_0/8)^3 |v_0|^{-1/2} \}$. We will choose $\delta = \min \{ \delta_0/32 , (\delta_0/8)^3 |v_0|^{-\frac{9}{\eps}} \}$ even smaller, since it allows us to apply the second part of \autoref{lemma:SR-preserved}(iii) with $\eps' = \eps/2$ and $k = 3$.
Then, $\Phi^{(l)}(Q_{\delta}(z_m))$ is well-defined and we can use \autoref{lemma:SR-preserved} to get an equation for the transformed function $\tilde{f}$, namely
\begin{align*}
\partial_t \tilde{f} + v \cdot \nabla_x \tilde{f} + (- \tilde{a}^{i,j}\partial_{v_i,v_j}) \tilde{f} = - \tilde{b} \cdot \nabla_v \tilde{f} - \tilde{c} \tilde{f} + \tilde{h} ~~ \text{ in }  \Phi^{(l)}(Q_{\delta}(z_m)) \cap ((-1,1) \times \{ x_n > 0 \} \times \R^n ) ,
\end{align*}
satisfying the specular reflection condition. Moreover, by 
\autoref{lemma:SR-preserved}(iii), and using that $\partial \Omega \in C^{\frac{9}{2}}$, it holds
\begin{align*}
\Vert \tilde{a}^{i,j}\Vert_{C^{3+\frac{\eps}{2}}(\mathcal{S}(\Phi(Q_{\delta}(z_j))))} &\le C \Vert a^{i,j} \Vert_{C^{3+\eps}(\mathcal{S}(Q_{\delta}(z_j)))} \le C,\\
\Vert \tilde{b}^i\Vert_{C^{3+\frac{\eps}{2}}(\mathcal{S}(\Phi(Q_{\delta}(z_j))))} &\le C \left( \Vert b^i\Vert_{C^{3+\eps}(\mathcal{S}(Q_{\delta}(z_j)))} + 1 + |v_0|^2 \right) \le C(1 + |v_0|^2),\\
\Vert \tilde{c} \Vert_{C^{3+\frac{\eps}{2}}(\mathcal{S}(\Phi(Q_{\delta}(z_j))))} &\le C \Vert c \Vert_{C^{3+\eps}(\mathcal{S}(Q_{\delta}(z_j)))}\le C,\\
\Vert \tilde{h} \Vert_{C^{3+\frac{\eps}{2}}(\mathcal{S}(\Phi(Q_{\delta}(z_j))))} &\le C \Vert h \Vert_{C^{3+\eps}(\mathcal{S}(Q_{\delta}(z_j)))},
\end{align*}
where $C$ only depends on $n,\eps,\Omega,\Lambda$. Moreover we have by \autoref{lemma:SR-preserved}(iii) with $k = 4$, $\eps = 1$, and since $\partial \Omega \in C^{9/2}:$
\begin{align*}
\Vert f \Vert_{C^{4,1}_{\ell}(Q_{\delta/2}(z_m))} \le C (1 + |v_0|)^{4} \Vert \tilde{f} \Vert_{C^{4,1}_{\ell}(\Phi^{(l)}(Q_{\delta/2}(z_m)))}.
\end{align*}

Hence, we can apply \autoref{prop:regularity-SR-halfspace-constants} with $\tilde{f}$ (and a covering argument) to obtain for any such cylinder
\begin{align*}
[ f ]_{C^{4,1}_{\ell}(Q_{\delta/2}(z_m))} &\le C (1 + |v_0|)^4 \Vert \tilde{f} \Vert_{C^{4,1}_{\ell}(\Phi^{(l)}(Q_{\delta/2}(z_m)))} \\
& \le C (1 + |v_0|)^{4}  \delta^{-5} \left(1 + \Vert \tilde{b}^i \Vert_{C_{\ell}^{3+\frac{\eps}{2}}(\mathcal{S}(\Phi(Q_{\delta}(z_m))))} \right) \delta^{-(2+4n)} \Vert \tilde{f} \Vert_{L^{1}(\mathcal{S}(\Phi^{(l)}(Q_{\delta}(z_m))))} \\
&\quad + C(1 + |v_0|)^{4}  [ \tilde{h} ]_{C^{3+\frac{\eps}{2}}_{\ell}(\mathcal{S}(2 \Phi^{(l)}(Q_{\delta}(z_m))))} \\
& \le C \left(1 + |v_0|^{\theta} \right) \Vert f \Vert_{L^{1}(\mathcal{S}(H_1(z_0)))} + C (1 + |v_0|)^{4} \Vert h \Vert_{C^{3+\eps}_{\ell}(\mathcal{S}(H_1(z_0)) )},
\end{align*}
where $\theta = 10 + \frac{45}{\eps} + \frac{18 + 36n}{\eps}$.
Here, we used crucially that the diameter of $\Phi(Q_{\delta}(z_m))$ is comparable to $\delta$ up to a constant depending only on $\Omega$ in order to obtain the prefactor $\delta^{-5}$ in the estimate.

Altogether, by summing over $l$, we obtain a corresponding regularity estimate in $H_{1/2}(z_0) \cap ((-1,1) \times \Omega_{\delta_0/4} \times \R^n)$.
From here, the desired result follows by combination of the previous estimates with the interior estimates from \autoref{lemma:interior-reg} (using also \autoref{lemma:interior-reg-DGNM} to replace the $L^{\infty}$ by the $L^1$ norm), applied directly to $f$.
%Finally, the claim for $a^{i,j}$ independent of $x,v$ comes from the fact that in this case it suffices to choose $\delta = \min \{ \delta_0/32 , (\delta_0/8)^3 |v_0|^{-3} \}$ in order to have
%\begin{align*}
%\Vert \tilde{a}^i\Vert_{C^{3+\eps}(\mathcal{S}(\Phi(Q_{\delta}(z_j))))} &\le C \Vert a^{i,j} \Vert_{C^{3+\eps}(\mathcal{S}(Q_{\delta}(z_j)))} \le C
%\end{align*}
%in \autoref{lemma:SR-preserved} (due to \eqref{eq:claim-chain-rule-improved-2}). From here, the rest of the proof goes in the same way as before, observing that by \autoref{lemma:SR-preserved},
%\begin{align*}
%\Vert \tilde{b}^i\Vert_{C^{3+\eps}(\mathcal{S}(\Phi(Q_{\delta}(z_j))))} + \Vert \tilde{c} \Vert_{C^{3+\eps}(\mathcal{S}(\Phi(Q_{\delta}(z_j))))} & \le C(1 + |v_0|^3),\\
%\Vert \tilde{h} \Vert_{C^{3+\eps}(\mathcal{S}(\Phi(Q_{\delta}(z_j))))} &\le C (1 + |v_0|^3) \Vert h \Vert_{C^{3+\eps}(\mathcal{S}(Q_{\delta}(z_j)))}.
%\end{align*}
\end{proof}

%\begin{remark}
%Note that in case $\Omega \subset \R^n$ is convex, the exponent $\frac{9}{2}$ in \autoref{thm:regularity-SR-constants} can be improved to $2$, since $\eps$ can be chosen independent of $|v_0|$ in order to guarantee $Q_{\eps}(z_j) \subset (-1,1) \times B_{\delta_0/2}(x^{(i)}) \times \R^n$ inside the proof.
%\end{remark}

As a direct corollary, we deduce the first part of our main result \autoref{thm1}. In fact, we have the following slightly more general theorem:

\begin{corollary}
\label{thm:global-weighted-SR-reg}
Let $\eps \in (0,1]$,  $\Omega \subset \R^n$ with $\partial \Omega \in C^{\frac{9}{2}}$. Let $a^{i,j},b,c,h \in C^{3+\eps}_{\ell}((-1,1) \times \Omega \times \R^n)$ and assume that $a^{i,j}$ satisfies \eqref{eq:unif-ell}. Let $f$ be a weak solution to 
\begin{equation*}
\left\{\begin{array}{rcl}
\partial_t f + v \cdot \nabla_x f + (-a^{i,j} \partial_{v_i,v_j})f &=& - b \cdot \nabla_v f - c f + h ~~ \text{ in }  (-1,1) \times \Omega \times \R^n , \\
f(t,x,v) &=& f(t,x,\mathcal{R}_x v) ~\qquad\quad \text{ on } \gamma_- \cap ((-1,1) \times \Omega \times \R^n).
\end{array}\right.
\end{equation*}
Then, it holds:
\begin{align*}
\Vert f \Vert_{C^{4,1}_{\ell}((-\frac{1}{2} , \frac{1}{2}) \times \Omega \times \R^n)} &\le C \left( \left\Vert (1 + |v|)^{\theta_1}  f \right\Vert_{L^{\infty}((-1,1) \times \Omega \times \R^n)} + \Vert h \Vert_{C^{3+\eps}_{\ell,4} (-1,1) \times \Omega \times \R^n) )} \right),\\
\Vert f \Vert_{C^{4,1}_{\ell}((-\frac{1}{2} , \frac{1}{2}) \times \Omega \times \R^n)} &\le C \left( \left\Vert (1 + |v|)^{\theta_2(1+n)}  f \right\Vert_{L^{1}((-1,1) \times \Omega \times \R^n)} + \Vert h \Vert_{C^{3+\eps}_{\ell,4} (-1,1) \times \Omega \times \R^n) )} \right).
\end{align*}
The constant $C$ depends only on $n,\eps,\lambda,\Lambda,\Omega$, $\Vert a^{i,j} \Vert_{C^{3+\eps}_{\ell}((-1,1) \times \Omega \times \R^n)}, \Vert b \Vert_{C^{3+\eps}_{\ell}((-1,1) \times \Omega \times \R^n)}$, and\\ $\Vert c \Vert_{C^{3+\eps}_{\ell}((-1,1) \times \Omega \times \R^n)}$, and the constants $\theta_1,\theta_2 > 0$ depend only on $\eps$. Moreover, in case $\eps = 1$, we can set $\theta_1 = 55$, and replace $\theta_2(1+n)$ by $73+36n$.
\end{corollary}

Here, we used the following notation for $\theta > 0$: 
%Here, the norms with respect to $L^{\infty}_{\theta}(D)$ and $C^{k,\eps}_{\ell,\theta}(D)$ are defined as follows, for $\theta \ge 0$:
\begin{align}
\label{eq:weighted-Holder-spaces}
\begin{split}
%\Vert f \Vert_{L^{\infty}_{\theta}(D)} &= \sup_{(t,x,v) \in D} |(1 + |v|)^{\theta} f(t,x,v)|,\\
[f]_{C^{k,\eps}_{\ell,\theta}(D)} &= \sup \left\{ (1 + |v|)^{\theta} [f]_{C^{k,\eps}_{\ell}(Q_r(z))} ~\big\vert~ r \in (0,1], ~ z \in D , ~ Q_r(z) \subset D \right\}, \\
\Vert f \Vert_{C_{\ell,\theta}^{k,\eps}(D)} &= \Vert (1 + |v|)^{\theta} f \Vert_{L^{\infty}(D)} + [f]_{C^{k,\eps}_{\ell,\theta}(D)}.
\end{split}
\end{align}

Note that the constant $\theta$ in \autoref{thm:regularity-SR-constants} explodes as $\eps \searrow 0$, which can be easily seen from the proof. This is due to the technical issue that the optimal regularity of solutions is of integer order, namely $C^{4,1}_{\ell}$. In this case, a little more than $C^{2,1}$ regularity is required for $h$, which makes it impossible to carry out an absorption argument as in the proof of \autoref{prop:regularity-SR-halfspace-constants} to obtain an explicit dependence of the constants on $a^{i,j}$. 

The following theorem contains explicit values for the exponents $\theta$, that are well-behaved for $\eps \searrow 0$, whenever $k \in \{ 3,4 \}$. Moreover, it can be seen as an analog of \autoref{thm:global-weighted-SR-reg} but for balls that are away from $\gamma_0$. In that case, solutions are of class $C^{k,\eps}$, as in case $\Omega = \{ x_n > 0 \}$ (see \autoref{prop:regularity-SR}).

\begin{theorem}
\label{thm:regularity-SR-constants-nongrazing}
Let $k \in \{3,4\}$ and $\eps \in (0,1)$. Let $\Omega \subset \R^n$ with $\partial \Omega \in C^{\frac{k+4+\eps}{2}}$ and $z_0 \in (-1,1) \times \overline{\Omega} \times \R^n$. Let $\eps \in (0,1]$ and $a^{i,j},b,c,h \in C^{k+\eps-2}_{\ell}(\mathcal{S}(H_1(z_0)))$ and assume that $a^{i,j}$ satisfies \eqref{eq:unif-ell}. Let $f$ be a weak solution to 
\begin{equation*}
\left\{\begin{array}{rcl}
\partial_t f + v \cdot \nabla_x f + (-a^{i,j} \partial_{v_i,v_j})f &=& - b \cdot \nabla_v f - c f + h ~~ \text{ in }  \mathcal{S}(H_1(z_0)) , \\
f(t,x,v) &=& f(t,x,\mathcal{R}_x v) ~\qquad\quad \text{ on } \gamma_- \cap \mathcal{S}(H_1(z_0)).
\end{array}\right.
\end{equation*}
Then, the following hold true:
\begin{align*}
\text{ if either } Q_2(z_0) \cap \gamma_0 = \emptyset, ~~ \text{ or } k \in \{3,4\},
\end{align*}
then $f \in C^{k+\eps}_{\ell}(H_{1/2}(z_0))$ and it holds:
\begin{align*}
[ f ]_{C^{k+\eps}_{\ell}(H_{1/2}(z_0))} &\le C \left(1 + |v_0|^{\frac{3k}{2}+12} \right)  \Vert f \Vert_{L^{\infty}( \mathcal{S}(H_1(z_0)))} + C \left(1 + |v_0|^{k+4} \right)\Vert h \Vert_{C^{k+\eps-2}_{\ell}(\mathcal{S}(H_1(z_0)) )}, \\
[ f ]_{C^{k+\eps}_{\ell}(H_{1/2}(z_0))} &\le C \left(1 + |v_0|^{\frac{3k}{2} + 14 + 4n} \right)  \Vert f \Vert_{L^{1}( \mathcal{S}(H_1(z_0)))} + C\left(1 + |v_0|^{k+4} \right) \Vert h \Vert_{C^{k+\eps-2}_{\ell}(\mathcal{S}(H_1(z_0)) )}.
\end{align*}
The constant $C$ depends only on $n,k,\eps,\lambda,\Lambda,\Omega$, $\Vert a^{i,j} \Vert_{C^{k+\eps-2}_{\ell}(\mathcal{S}(H_1(z_0)))}, \Vert b \Vert_{C^{k+\eps-2}_{\ell}(\mathcal{S}(H_1(z_0)))}$, and \\$\Vert c \Vert_{C^{k+\eps-2}_{\ell}(\mathcal{S}(H_1(z_0)))}$, but not on $z_0$.
\end{theorem}

We provide a sketch of the proof, as it does not involve any new ideas with respect to the case $k=5$.

\begin{proof}
The proof goes in the same way as the proof of \autoref{thm:global-weighted-SR-reg}. First, one establishes the following estimate for solutions in the half-space, paralleling \autoref{prop:regularity-SR-halfspace-constants},
\begin{align*}
[ f ]_{C^{k+\eps}_{\ell}(H_{R/2}(z_0))} &\le C \left[1 + \Vert a^{i,j} \Vert_{C^{k+\eps - 2}_{\ell}(\mathcal{S}(H_R(z_0)))} + \Vert b^{i} \Vert_{C^{k+\eps - 2}_{\ell}(\mathcal{S}(H_R(z_0)))} + \Vert c \Vert_{C^{k+\eps  -2}_{\ell}(\mathcal{S}(H_R(z_0)))} \right] \times \\
&\qquad \qquad \times R^{-5} R^{-(2+4n)} \Vert f \Vert_{L^{1}( \mathcal{S}(H_R(z_0)))} + C  [ h ]_{C^{k+\eps -2}_{\ell}(\mathcal{S}(H_R(z_0)) )},
\end{align*}
where $C > 0$ only depends on $\lambda, \Lambda$ in \eqref{eq:unif-ell} but not anymore on the regularity norms of the coefficients. 
Note that since $\eps \not=1$, the freezing and absorption argument in the proof of \autoref{prop:regularity-SR-halfspace-constants} can also be carried out for $a^{i,j}$, by writing $a^{i,j}(z) = a^{i,j}(z_0) + (a^{i,j}(z) - a^{i,j}(z_0))$, and treating $(a^{i,j}(z) - a^{i,j}(z_0)) \partial_{v_i,v_j} f$ as a right-hand side. In case $Q_2(z_0) \cap \gamma_0 = \emptyset$, one uses \autoref{prop:regularity-SR} in the proof of the estimate instead of \autoref{cor:regularity-SR-halfspace}.
Otherwise, one observes that \autoref{lemma:higher-reg-gamma_0-SR-halfspace} remains true for $a^{i,j}, b^{i}, c,h \in C^{k+\eps-2}$ and $k \in \{ 3 ,4 \}$ with $P_{z_0} \in \tilde{\cP}_k$, and therefore, also a $C^{k+\eps}_{\ell}$ estimate can be obtained in \autoref{cor:regularity-SR-halfspace} under these weaker assumptions on the coefficients.

From here, by proceeding exactly as in the proof of \autoref{thm:global-weighted-SR-reg}, choosing $\delta = \min\{ \delta_0/32 , (\delta_0/8)^3 |v_0|^{-1} \}$, we can use \autoref{lemma:SR-preserved} in combination with the previous display to get the desired result.
\end{proof}

\section{Generic counterexamples}
\label{sec:counterex}

In this section, we prove that the existence of non-$C^5_{\ell}$ solutions to \eqref{eq:Kolmogorov0}, satisfying a specular reflection condition, is typical, even for homogeneous equations.

The next theorem gives a criterion under which solutions in non-flat, smooth domains $\Omega$ are not $C^5_{\ell}$. In particular, it proves the second part of \autoref{thm1}.

\begin{theorem}
\label{thm:counterexample}
Let $\Omega$ be a smooth domain with $0 \in \partial \Omega$ whose graph is parametrized by a smooth function near $0$. Let $\phi,\Phi$ be as in \eqref{eq:Phi-def} such that $\Phi((-1,1) \times (\Omega \cap B_1(0)) \times \R^n) = (-1,1) \times (\{ x_n > 0 \} \cap B_1(0)) \times \R^n$. Let $z_0 = (t_0,0,v_0) \in \gamma_0$ and $f$ be a weak solution to 
\begin{equation}
\label{eq:PDE-counterexample}
\left\{\begin{array}{rcl}
(\partial_t + v \cdot \nabla_x - \Delta_v) f &=& 0 ~~ ~~ \qquad \qquad  \text{ in } H_1(z_0),\\
f(t,x,v) &=& f(t,x, \mathcal{R}_x v)  ~~~ \text{ in } \gamma \cap Q_1(z_0).
\end{array}\right.
\end{equation}
Assume that the following condition holds true:
%\begin{align}
%\label{eq:counterex-condition}
% \sum_{i = 1}^{n-1} \left( \partial_{x_i,x_n} \phi^n(x_0)  - \partial_{x_n,x_n} \phi^i(x_0) \right) \partial_{v_i,v_n}f(x_0) \not = 0, ~~ \text{ or } ~~ \sum_{i,j = 1}^n \partial_{x_j,x_n} \phi^i(x_0) \partial_{v_i,v_j} f(x_0) \not= 0.
%\end{align}
\begin{align}
\label{eq:counterex-condition}
\sum_{j = 1}^n\sum_{\substack{i = 1 \\ i \not = j}}^{n-1} \partial_{x_i,x_n} \phi^j(0) \partial_{v_i,v_j}f(z_0) + 2\sum_{i = 1}^{n-1}  \partial_{x_i,x_n} \phi^i(0) \partial_{v_i,v_i} f(z_0) \not= \sum_{i = 1}^{n-1} \partial_{x_n,x_n} \phi^i(0) \partial_{v_i,v_n} f(z_0).
\end{align}
Then, $f \not\in C^5_{\ell}(H_{1/2}(z_0))$.
\end{theorem}

\begin{remark}
Note that \eqref{eq:counterex-condition} cannot hold if $\Omega$ is flat near $0$, since in that case $D^2 \phi^i(0) = 0$. In other words, some curvature of $\partial \Omega$ is needed to produce non-$C^5_{\ell}$ solutions.
\end{remark}

\begin{remark}
In case $n=2$, the condition \eqref{eq:counterex-condition} simplifies to
\begin{align*}
\partial_{x_1,x_2} \phi^2(0) \partial_{v_1,v_2} f(z_0) + 2 \partial_{x_1,x_2} \phi^1(0) \partial_{v_1,v_1} f(z_0) \not= \partial_{x_2,x_2} \phi^1(0) \partial_{v_1,v_2} f(z_0).
\end{align*}
\end{remark}

The result in \autoref{thm:counterexample} is useful to construct explicit counterexamples in general domains $\Omega$. Hence, our main result \autoref{thm:global-weighted-SR-reg} is optimal in a generic sense. In fact, given $\Omega$, the diffeomorphism $\phi$ is completely determined by \eqref{eq:Phi-def}. 
%Then, it remains to construct a solution $f$ whose Hessian at $z_0$ is such that \eqref{eq:counterex-condition} fails. 

\begin{example}
If $z_0 = 0$ and the second order expansion of $f$ at $0$ is of the form
\begin{align}
\label{eq:exp-counterex}
f(z) = p_2 + O(|z|^3) := f(0) + p_1(z) + v_1^2 - v_n^2 + O(|z|^3),
\end{align}
for some $p_1 \in \cP_1$, then \eqref{eq:counterex-condition} is satisfied if and only if $\partial_{x_1,x_n} \phi^1(0) \not= 0$.
The existence of such a solution $f$ can be guaranteed by solving the mixed boundary value problem \eqref{eq:PDE-counterexample} in $(-1,1) \times (\Omega \cap B_1) \times \R^n$ and imposing $f = p_2$ in $\gamma_- \cap ((-1,1) \times (\partial B_1 \cap \Omega) \times \R^n)$. Then, if $\partial\Omega$ is flat enough, the expansion in \eqref{eq:exp-counterex} will hold true for $f$.
\end{example}

\begin{proof}[Proof of \autoref{thm:counterexample}]
We split the proof into three steps. Only in the last step, we will make use of the assumptions \eqref{eq:counterex-condition}. We believe the formula in \eqref{eq:rep-p} to be of independent interest.

Let us assume for simplicity that $x_0 = z_0 = 0$. Let $\Phi$ be the diffeomorphism from \autoref{lemma:trafo-sr}. Note that since $z_0 = 0$, we have $\Phi(t_0,x_0,v_0) = 0$ and $\phi(0) = 0$ and $D \phi(0) = \mathrm{Id}_n$.
Let $\tilde{f}$ be the transformed solution as in \autoref{lemma:trafo-sr} to
\begin{equation}
\label{eq:tilde-f-equation}
\left\{\begin{array}{rcl}
(\partial_t + v \cdot \nabla_x - \tilde{a}^{i,j}\partial_{v_i,v_j} + \tilde{b}^i \partial_{v_i}) \tilde{f} &=& 0 ~~ ~~ \qquad \qquad \text{ in } H_1(0),\\
\tilde{f}(t,x,v) &=& \tilde{f}(t,x, \mathcal{R}_x v)  ~~ \text{ in } \gamma \cap Q_1(0),
\end{array}\right.
\end{equation}
where from now on, $H_1(0) := ((-1,1) \times \{ x_n > 0\} \times \R^n) \cap Q_1(0)$.

The coefficients are given as follows
\begin{align}
\label{eq:coeff-counterexample}
\tilde{a}^{i,j}(x) = \sum_{r = 1}^n \partial_{x_i} \phi^{r}(x)\partial_{x_j} \phi^{r}(x), \qquad \tilde{b}^i(x,v) = \langle v , D^2 \phi^i(x) \cdot v \rangle,
\end{align}
where $\phi^{r}$ denotes the $r$th component of $\phi$, and  $\tilde{a}^{i,j}(0) = \delta_{i,j}$, and $\delta_{i,j}$ denotes  Kronecker's delta.

\textbf{Step 1:} Assume that $f \in C^5_{\ell}( ((-1,1) \times \Omega \times \R^n ) \cap Q_{1/2}(0))$. Then, $f \in C^5_{\ell}(H_{1/2}(0))$ and hence there must be $P \in \tilde{\cP}_4$ such that 
\begin{align}
\label{eq:eq-P4}
(\partial_t + v \cdot \nabla_x - \tilde{a}^{i,j}(0)\partial_{v_i,v_j}) P = (\partial_t + v \cdot \nabla_x - \delta_{i,j} \partial_{v_i,v_j}) P = p_0 \in \cP_2,
\end{align}
and $P_5 \in \tilde{\cP}_5$ is homogeneous of degree $5$ such that
\begin{align*}
\tilde{f}_r(z) := \frac{(\tilde{f} - P)(S_r z)}{r^5} \to P_5(z) ~~ \text{ as } r \searrow 0,
\end{align*}
uniformly in $H_{R}$ for any $R > 0$.
In fact, $P + P_5 \in \tilde{\cP}_5$ is the fifth order approximation of $\tilde{f}$ at $0$ and \eqref{eq:eq-P4} follows immediately from the proof of \autoref{lemma:higher-reg-gamma_0-SR-halfspace}. Here, by $\tilde{P}_k$ we denote polynomials in $\cP_k$ satisfying the specular reflection condition with respect to $\{ x_n > 0 \}$.

Let $(r_m)$ be a sequence such that $r_m \searrow 0$ as $m \to \infty$ and define
\begin{align}
\label{eq:not-C5-condition}
-[\mathcal{L}_{m}P](S_{r_m} z) := -r_m^{-3} [(\partial_t + v \cdot \nabla_x - \tilde{a}^{i,j}(S_{r_m} z)\partial_{v_i,v_j} + \tilde{b}^i(S_{r_m} z) \partial_{v_i})P](S_{r_m} z).
\end{align} 
 
By \eqref{eq:not-C5-condition}, we have for $\tilde{f}_m := \tilde{f}_{r_m}$ by \autoref{lemma:scaling} and \eqref{eq:tilde-f-equation}
\begin{equation*}
\left\{\begin{array}{rcl}
(\partial_t + v \cdot \nabla_x - \tilde{a}^{i,j}(S_{r_m} \cdot)\partial_{v_i,v_j} + r_m \tilde{b}^i(S_{r_m} \cdot) \partial_{v_i}) \tilde{f}_m &=& -(\mathcal{L}_{m}P)(S_{r_m} \cdot) ~~ \text{ in } H_{r_m^{-1}}(0), \\
\tilde{f}_m(z) &=& \tilde{f}_m(t,x,\mathcal{R}_x v) ~~~~ \text{ in } \gamma \cap H_{r_m^{-1}}(0).
\end{array}\right.
\end{equation*}
Note that, for any $R \le r_m^{-1}$ and $\eps \in (0,1)$, by the definitions of $\tilde{a}^{i,j}, \tilde{b}^i$ (see \eqref{eq:coeff-counterexample}) and their regularity (since $\Omega$ and thus $\phi$ are smooth) and since $P \in \cP_4$:
\begin{align}
\label{eq:counterex-RHS-Holder}
\begin{split}
[(\mathcal{L}_m P)(S_{r_m} \cdot)]_{C^{3+\eps}_{\ell}(H_R(0))} &\le r_m^{3+\eps} [\mathcal{L}_m P ]_{C^{3+\eps}_{\ell}(H_{1}(0))} \\
&\le r_m^{\eps} [(\tilde{a}^{i,j} -  \delta_{i,j})\partial_{v_i,v_j}P ]_{C^{3+\eps}_{\ell}(H_{1}(0))} + r_m^{\eps}[\tilde{b}^i \partial_{v_i} P ]_{C^{3+\eps}_{\ell}(H_{1}(0))}  \le C.
\end{split}
\end{align}

By interior and boundary estimates (see \autoref{lemma:interior-reg} and \autoref{lemma:boundary-reg-specular}) and using also \eqref{eq:counterex-RHS-Holder} and \autoref{lemma:osc-results} we have for large enough $m$:
\begin{align*}
\Vert \tilde{f}_m \Vert_{C^{2+\eps}_{\ell}(Q_1(z))} &\le C(R) ~~ \forall Q_1(z) \subset H_R(0), ~~ R > 0,\\
\Vert \tilde{f}_m \Vert_{C^{\eps}_{\ell}(H_1(z))} &\le C(R) ~~ \forall z \in \gamma \cap Q_R(0), ~~ R > 0.
\end{align*}
Note that here we also used the boundedness of $\tilde{f}_m$, which follows from its convergence to $P_5$.

Next, we claim that (up to a subsequence)
\begin{align}
\label{eq:claim-convergence}
-(\mathcal{L}_{m}P)(S_{r_m} z) \to p(z) \in \cP_3,
\end{align}
where the convergence holds true locally uniformly. This follows in the same way as in the proof of \autoref{lemma:higher-reg-gamma_+}, using that \autoref{lemma:osc-results} and the boundedness of $\tilde{f}_m$ imply for any $R \le r_m^{-1}$
\begin{align}
\label{eq:counterex-RHS-bd}
\Vert (\mathcal{L}_{m}P)(S_{r_m} \cdot) \Vert_{L^{\infty}(H_R(0))} \le C(R).
\end{align}

Hence, we have that $P_5$ satisfies
\begin{equation}
\label{eq:P5}
\left\{\begin{array}{rcl}
(\partial_t + v \cdot \nabla_x - \tilde{a}^{i,j}(0)\partial_{v_i,v_j}) P_5 &=& p ~~\qquad\qquad\qquad\qquad\qquad ~~ \text{ in } \R \times \{ x_n > 0 \} \times \R^n,\\
P_5(t,x',0,v',v_n) &=& P_5(t,x',0,v',-v_n) ~~ \forall (t,x,v) \in \R \times \{x_n = 0 \} \times \R^n.
\end{array}\right.
\end{equation}

\textbf{Step 2:} The goal of this step is to compute explicitly the polynomial $p$ in \eqref{eq:claim-convergence}. For this we make use of the explicit form of the coefficients $\tilde{a}^{i,j}$ and $\tilde{b}^{i}$ given in \eqref{eq:coeff-counterexample}. Note that since $\partial_{x_i} \phi^r(0) = \delta_{i,r}$
 we have for any $i \in \{1, \dots , n\}$:
\begin{align*}
\tilde{a}^{i,i}(x) &= \sum_{r = 1}^n |\partial_{x_i} \phi^r(x)|^2 = \sum_{\substack{r = 1\\ r \not = i}}^n \left|\sum_{k = 1}^n \partial_{x_i,x_k} \phi^r(0) x_k + O(|x|^2) \right|^2 +   \left|1 + \sum_{k = 1}^n \partial_{x_i,x_k} \phi^i(0) x_k + O(|x|^2) \right|^2 \\
&= 1 + 2\sum_{k = 1}^n \partial_{x_i,x_k} \phi^i(0) x_k + O(|x|^2). 
\end{align*}
Moreover,
 for any $i,j \in \{ 1, \dots , n\}$:
\begin{align*}
\tilde{a}^{i,j}(x) &= \sum_{r = 1}^n \partial_{x_i} \phi^r(x) \partial_{x_j} \phi^r(x)\\
%& = \sum_{\substack{r = 1\\ r \not= i}}^n \left( \sum_{k = 1}^n \partial_{x_i,x_k} \phi^r(x) x_k + O(|x|^2) \right) \partial_{x_j} \phi^r(x)\\
%&\quad + (1 + \sum_{k = 1}^n \partial_{x_i,x_k} \phi^i(x) x_k + O(|x|^2))\partial_{x_j} \phi^i(x) \\
&=\sum_{\substack{r = 1\\ r \not= i,j}}^n \left( \sum_{k = 1}^n \partial_{x_i,x_k} \phi^r(0) x_k + O(|x|^2) \right) \left( \sum_{k = 1}^n  \partial_{x_j,x_k} \phi^r(0) x_k + O(|x|^2)\right)\\
&\quad  + \left( \sum_{k = 1}^n \partial_{x_i,x_k} \phi^j(0) x_k + O(|x|^2) \right) \left(1 + \sum_{k = 1}^n \partial_{x_j,x_k} \phi^j(0) x_k + O(|x|^2) \right) \\
&\quad + \left(1 + \sum_{k = 1}^n \partial_{x_i,x_k} \phi^i(0) x_k + O(|x|^2)\right)\left(\sum_{k = 1}^n \partial_{x_j,x_k} \phi^i(0) x_k + O(|x|^2) \right) \\
&= \sum_{k = 1}^n \partial_{x_i,x_k} \phi^j(0) x_k + \sum_{k = 1}^n \partial_{x_j,x_k} \phi^i(0) x_k  + O(|x|^2).
\end{align*}
Clearly, we also have for any $i \in \{1 , \dots, n\}$
\begin{align*}
\tilde{b}^i(z) = \langle v , D^2 \phi^i(x) v \rangle =  \langle v , D^2 \phi^i(0) v \rangle + O(|x||v|^2).
\end{align*}

%Note that the terms  we denoted by $O(|x|^2)$ and $O(|x||v|^2)$ are smooth, and in particular, their $C^{2+\eps}_{\ell}$ norm in $H_1^{(m)}$-norm remains bounded by the regularity assumption on $\Omega$.

Hence, by \eqref{eq:eq-P4} we deduce
\begin{align*}
-(\mathcal{L}_m P) (S_{r_m} z) &= -r_m^{-3} p_0(S_{r_m} z ) + r_m^{-3}(\tilde{a}^{i,j}(S_{r_m} z) -  \delta_{i,j})\partial_{v_i,v_j}P(S_{r_m}z)  - r_m^{-3}\tilde{b}^i(S_{r_m}z) \partial_{v_i} P(S_{r_m}z) \\
&= \sum_{i,j = 1}^n \left(\sum_{k = 1}^n \left[ \partial_{x_i,x_k} \phi^j(0) + \partial_{x_j,x_k} \phi^i(0) \right] x_k\right) \partial_{v_i,v_j}P(S_{r_m}z) + O(r_m^2) \\
& \quad - r_m^{-1} \sum_{i = 1}^n \sum_{j,k = 1}^n \partial_{x_j,x_k} \phi^i(0) v_j v_k \partial_{v_i} P(S_{r_m}z) - r_m^{-3} p_0(S_{r_m} z ).
%\langle v , D^2 \phi^i(r_m^3 x) \cdot v \rangle \partial_{v_i} P(S_{r_m}z).
\end{align*}
%where $\Vert R_m \Vert_{C_{\ell}^{2+\eps}(H_1^{(m)})} \to 0$, as $m \to \infty$.

Since $P \in \tilde{\cP}_4$ satisfies \eqref{eq:eq-P4}, we can write for some $\alpha'_0,\alpha'_1,\dots,\alpha'_{n-1},\alpha_0,\alpha_{i,j},\alpha^{(\beta)} \in \R$
\begin{align*}
P(t,x,v) = \alpha'_0 + \alpha'_1 v_1 + \dots \alpha'_{n -1 } v_{n - 1} + \alpha_0 t + \sum_{i,j = 1}^n \alpha_{i,j} v_i v_j + \sum_{3 \le |\beta| \le 4} \alpha^{(\beta)} z^{\beta}.
\end{align*}

Then, we can define
\begin{align*}
p_1(z) &:= \sum_{\substack{i,j = 1\\
i \not = j}}^n \sum_{k = 1}^n \left[ \partial_{x_i,x_k} \phi^j(0) + \partial_{x_j,x_k} \phi^i(0) \right] \alpha_{i,j} x_k + 4 \sum_{i = 1}^n \sum_{k = 1}^n \partial_{x_i,x_k} \phi^i(0) \alpha_{i,i} x_k \\
&\quad - \sum_{i,j,k = 1}^n  \partial_{x_j,x_k} \phi^i(0) v_j v_k  \sum_{l = 1}^n (\alpha_{i,l} + \alpha_{l,i}) v_l \in \cP_3, \\
\tilde{p}_m(z) &:= -r_m^{-3} p_0(S_{r_m}z) - r_m^{-1} \sum_{i = 1}^n \sum_{j,k = 1}^n \partial_{x_j,x_k} \phi^i(0) v_j v_k \alpha'_i \in \cP_2,
\end{align*}
and observe that the following identity holds true by the previous considerations:
\begin{align*}
-(\mathcal{L}_m P) (S_{r_m} z) - p_1(z) = O(r_m) + \tilde{p}_m(z).
\end{align*}

By the triangle inequality and \eqref{eq:counterex-RHS-bd}, we have
\begin{align*}
\Vert \tilde{p}_m \Vert_{L^{\infty}(H_1(0))} \le \Vert (\mathcal{L}_m P) (S_{r_m} \cdot) \Vert_{L^{\infty}(H_1(0))} + \Vert p_1 \Vert_{L^{\infty}(H_1(0))} + O(r_m) \le C.
\end{align*}
Therefore, $\tilde{p}_m \to \tilde{p} \in \cP_2$ as $m \to \infty$, and we deduce the following formula for $p$ from \eqref{eq:claim-convergence}:
\begin{align}
\label{eq:rep-p}
-(\mathcal{L}_m P) (S_{r_m} z) \to p(z) := p_1(z) + \tilde{p}(z).
\end{align}

\textbf{Step 3:} Conclusion. 
Let us assume that
\begin{align*}
&\sum_{\substack{i,j = 1 \\
i \not= j}}^n  \left[ \partial_{x_i,x_n} \phi^j(0) + \partial_{x_j,x_n} \phi^i(0) \right] \alpha_{i,j}  + 4 \sum_{i = 1}^n  \partial_{x_i,x_n} \phi^i(0) \alpha_{i,i} \not = 2\sum_{i = 1}^n \partial_{x_n,x_n} \phi^i(0) (\alpha_{i,n} + \alpha_{n,i}) \\
%&\qquad \Leftrightarrow \qquad \sum_{j = 1}^n \sum_{i = 1}^{n-1} \partial_{x_i,x_n} \phi^j(0) \alpha_{i,j} + \sum_{i = 1}^n \sum_{j = 1}^{n-1} \partial_{x_j,x_n} \phi^i(0)  \alpha_{i,j} + 2\sum_{i = 1}^n  \partial_{x_i,x_n} \phi^i(0) \alpha_{i,i} \not = \sum_{i = 1}^n \partial_{x_n,x_n} \phi^i(0) (\alpha_{i,n} + \alpha_{n,i}) \\
&\qquad \Leftrightarrow \qquad \sum_{j = 1}^n\sum_{i = 1}^{n-1} \partial_{x_i,x_n} \phi^j(0) (\alpha_{i,j} + \alpha_{j,i}) + 2 \sum_{i = 1}^n  \partial_{x_i,x_n} \phi^i(0) \alpha_{i,i} \not = \sum_{i = 1}^n \partial_{x_n,x_n} \phi^i(0) (\alpha_{i,n} + \alpha_{n,i}).
\end{align*}
Note that in this case, $p$ given in \eqref{eq:rep-p} violates 
\begin{align}
\label{eq:poly-condition}
p(z) = \alpha x_n + \beta v_n^3 ~~ \text{ for some } \alpha \not= -2\beta , ~~ \forall z = (0,0,x_n,0,v_n) \in \R^{1+2n}.
\end{align}
However, if \eqref{eq:poly-condition} is violated, then we can apply \autoref{thm:Liouville-higher-half-space-SR} to the equation for $P_5$ in \eqref{eq:P5} and deduce that $P_5 \not \in \cP_5$, a contradiction. Hence, we cannot have $f \in C^5_{\ell}$.

To summarize, we need the following assumptions on $\tilde{f}$ in order get a contradiction with $f \in C^5_{\ell}$:
%\begin{align*}
% \sum_{i = 1}^{n-1} \left( \partial_{x_i,x_n} \phi^n(0)  - \partial_{x_n,x_n} \phi^i(0) \right) \partial_{v_i,v_n}\tilde{f}(0) \not = 0, ~~ \text{ or } ~~ \sum_{i,j = 1}^n \partial_{x_j,x_n} \phi^i(0) \partial_{v_i,v_j} \tilde{f}(0) \not= 0.
%\end{align*}
\begin{align*}
\sum_{j = 1}^n\sum_{i = 1}^{n-1} \partial_{x_i,x_n} \phi^j(0) (\partial_{v_i,v_j}\tilde{f}(0) + \partial_{v_j,v_i} \tilde{f}(0)) & + 2\sum_{i = 1}^n  \partial_{x_i,x_n} \phi^i(0) \partial_{v_i,v_i} \tilde{f}(0) \\
&\not= \sum_{i = 1}^n \partial_{x_n,x_n} \phi^i(0) (\partial_{v_i,v_n} \tilde{f}(0) + \partial_{v_n,v_i} \tilde{f}(0)).
\end{align*}

Note that since $D\phi(0) = \mathrm{Id}_n$, we have $\partial_{v_i,v_j} \tilde{f}(0) = \partial_{v_i,v_j} f(0)$, and thus we can recast the above line as a condition on $f$, which is precisely \eqref{eq:counterex-condition}, as desired.
\end{proof}

\end{document}